\documentclass[a4paper,11pt]{amsart} 

\usepackage[utf8]{inputenc}
\usepackage[english]{babel}
\usepackage{amsmath}
\usepackage{amsthm}
\usepackage{amsfonts,mathrsfs,relsize,bigints}
\usepackage{amssymb}
\usepackage{anysize}
\usepackage{hyperref}
\usepackage{appendix}
\usepackage{mathtools}
\usepackage{graphicx}
\usepackage{xcolor}
\usepackage{ulem}
\makeatletter  \@addtoreset{equation}{section} \makeatother
\newtheorem{defi}{Definition}[section]
\newtheorem{lem}{Lemma}[section]
\newtheorem{theo}{Theorem}[section]

\newtheorem{pro}{Proposition}[section]

\newtheorem{rem}{Remark}[section]
\newtheorem{rems}{Remarks}[section]

\DeclareMathOperator{\N}{\mathbb{N}}
\DeclareMathOperator{\R}{\mathbb{R}}
\DeclareMathOperator{\C}{\mathbb{C}}
\DeclareMathOperator{\T}{\mathbb{T}}

\allowdisplaybreaks 

{\thanks{C.G has been partially supported by the MINECO--Feder (Spain) research grant number RTI2018--098850--B--I00, the Junta de Andaluc\'ia (Spain) Project
FQM 954, the MECD (Spain) research grant FPU15/04094, {and the European Research Council through Grant ERC-StG-852741 (CAPA)}, \mbox{T. H.} has been partially supported by the ANR grant ODA (ANR-18-CE40-0020-01), and J. M. has been partially supported by MTM2016--75390 (Mineco, Spain) and { 2017-SGR-395} (Generalitat de Catalunya)}}


\begin{document}

\title{Time periodic solutions for 3D quasi--geostrophic model}

\author[C. Garc\'ia]{Claudia Garc\'ia}
\address{ Departament de Matem\`atiques i Inform\`atica, Universitat de Barcelona, 08007 Barcelona, Spain \& Research Unit ``Modeling Nature'' (MNat), Universidad de Granada, 18071 Granada, Spain}
\email{ claudiagarcia@ub.edu}
\author[T. Hmidi]{Taoufik Hmidi}
\address{Univ Rennes, CNRS, IRMAR -- UMR 6625, 
F-35000 Rennes, France}
\email{thmidi@univ-rennes1.fr}
\author[J. Mateu]{Joan Mateu}
\address{Departament de Matem\`atiques, Universitat Aut\`onoma de Barcelona, 08193 Bellaterra; Barcelona, Catalonia}
\email{mateu@mat.uab.cat}


\begin{abstract}
This paper aims to study time periodic solutions for 3D inviscid quasi--geostrophic model. We show the existence of non trivial rotating patches by suitable perturbation  of  stationary solutions given by {\it generic} revolution shapes around the vertical axis. The construction of those special solutions are done through bifurcation theory.  In general, the spectral problem is very delicate and strongly depends on the shape of the initial stationary solutions. More specifically, the spectral study can be related to an eigenvalue problem of a self--adjoint compact operator. We are able to implement  the bifurcation  only from the largest eigenvalues of the operator, which are simple. Additional  difficulties generated by the singularities of the poles are solved  through the use of suitable function spaces with Dirichlet boundary condition type and refined potential theory with  anisotropic kernels.
\end{abstract}

\maketitle

\tableofcontents

\section{Introduction}
The large scale dynamics of 
an inviscid three--dimensional fluid subject to rapid background rotation and strong stratification can be described through the so--called quasi--geostrophic model.  It is  an asymptotic model  derived from the Boussinesq system for vanishing Rossby and  Froud numbers, for more details about its formal derivation we refer to \cite{Ped}. Rigorous derivation can be found in \cite{B-B,Charv,Iftimie2}.

We point out that this system is a pertinent  model commonly used  in the  ocean and atmosphere circulations   to describe the vortices and to track  the emergence of  long--lived structures. The quasi--geostrophic system is described by the potential vorticity $q$ which is merely advected by the fluid,
\begin{eqnarray}   \label{equation}
       \left\{\begin{array}{ll}
          	\partial_t q+u\partial_1 q+v\partial_2 q=0, & (t,x)\in[0,+\infty)\times\mathbb{R}^3, \\
         	 \Delta \psi=q,& \\
         	 	u=-\partial_2 \psi, \ v=\partial_1 \psi,& \\
         	 q(t=0,x)=q_0(x).& 
       \end{array}\right.
\end{eqnarray}
The second equation involving the standard Laplacian of $\R^3$ can be formally inverted using Green's function  leading to the following representation of the stream function $\psi$,
$$
\psi(t,x)=-\frac{1}{4\pi}\bigintsss_{\R^3}\frac{q(t,y)}{|x-y|}dA(y),
$$
where $dA$ denotes the usual Lebesgue measure. The velocity field $(u,v,0)$  is solenoidal and can be recovered  from $q$ through the Biot--Savart law,
$$
(u,v)(t,x)=\frac{1}{4\pi}\bigintsss_{\R^3}\frac{(x_1-y_1, x_2-y_2)^{\perp}}{|x-y|^3}q(t,y)dA(y).
$$
Notice that the velocity field  is planar but its  components depends on the all spatial  variables and the potential vorticity is transported by the associated flow. The incompressibility of the  velocity allows us to adapt without any difficulties   
the classical results known for 2D Euler equations. For instance, see \cite{MajdaBertozzi}, one may get global unique strong solutions when the initial data $q_0$ belongs to H\"older class $\mathscr{C}^\alpha$, for $\alpha\ge 0.$ Yudovich theory \cite{Yudovich} can also be implemented  and one gets global unique solution when $q_0\in L^1\cap L^\infty.$ This latter context  allows to deal with discontinuous vortices of the patch form, meaning a characteristic function of a bounded domain. This structure is preserved in time and the vortex patch problem consists of studying the regularity of the boundary and to analyze whether singularities can be formed in finite time on the boundary.

For the 2D Euler equations, the $\mathscr{C}^{1,\alpha}$ regularity of the boundary of the patch, with $\alpha\in(0,1)$, is preserved in time, see \cite{Chemin, B-C, Serf}. The contour dynamics of the patch is in general hard to track  and filamentation may occur. Therefore it is of important interest to look for ordered structure in turbulent flows  like relative equilibria. It seems that only a few explicit examples are known in the literature in the  patch form: the circular patches which  are stationary and  the elliptic ones  which rotate uniformly   with a constant angular velocity. This latter example is  known as the {Kirchhoff} ellipses. However, a lot of implicit examples with higher symmetry have been constructed  during the last decades    and the first ones were discovered numerically by  Deem and Zabusky \cite{DeemZabusky}. Having this kind of V-state solutions in mind, {Burbea} \cite{Burbea} designed a rigorous approach to generate them close to Rankine vortices   through complex analysis tools and bifurcation theory. Later this idea was fruitfully  improved and extended in various directions, generating a lot of contributions dealing, for instance, with interesting topics like  the  regularity problem of the relative equilibria, their existence  with different topological structure or for different  active scalar equations and so forth. For more details about this active area we refer the reader to the works \cite{Cas0-Cor0-Gom, Cas-Cor-Gom, CastroCordobaGomezSerrano, C-C-GS-2, DelaHoz-Hassainia-Hmidi, DelaHozHmidiMateuVerdera, DHHM, D-H-R, G-KVS, GHS, GPSY, Hassa-Hmi, HMW,  HmidiHozMateuVerdera, HmidiMateu, H-M, HmidiMateuVerdera} and the references therein.

Coming back to the 3D quasi--geostrophic system, it seems that stationary solutions in the patch form are more abundant than the planar case. Indeed,   any domain with   a revolution shape  about the $z-$axis generates a stationary solution. The analogues to Kirchhoff ellipses still surprisingly survive in the 3d case. In \cite{Meacham} it is shown that a standing ellipsoid  of arbitrary  semi--axis lengths $a, b$  and $c$ rotates steadily about the $z-$axis  with the angular velocity
$$
\Omega=\mu\frac{\lambda^{-1} R_D(\mu^2,\lambda,\lambda^{-1})-\lambda R_D(\mu^2,\lambda^{-1},\lambda)}{3(\lambda^{-1}-\lambda)},
$$
where $\lambda=\frac{a}{b}$ is the horizontal aspect ratio, $\mu:=\frac{c}{\sqrt{ab}}$ the vertical aspect ratio and $R_D$  denotes the elliptic integral of second order
$$
R_D(x,y,z):=\frac32\int_0^{+\infty}\frac{dt}{\sqrt{(t+x)(t+y)(t+z)}}\cdot
$$
For more details about the stability of those ellipsoids we refer to \cite{D-S-R,Dristchel2, Dristchel}.

The main concern of this paper is to investigate the existence of non trivial relative  equilibria close to  the stationary revolution shapes. In our context, we mean by relative equilibria periodic solutions in the patch form, rotating uniformly about the vertical axis without any deformation. Very recently,  Reinaud has explored numerically in \cite{Rein}  the existence and the linear stability of  finite volume relative equilibria distributed around   circular point vortex arrays. Similar analysis has been implemented in \cite{Rein-Drit}  for toroidal vortices.  Apart from the numerical experiments, no analytical results had been yet developed and the main inquiry of this paper is to design some technical material allowing us to construct relative equilibria close to general smooth stationary revolution shapes. 
The basic tool is bifurcation theory but as we shall see its implementation is an involved task which  requires refined and careful analysis.   Let us explain more  our strategy and how to proceed. First, we start with deriving  the contour dynamic equation for rotating finite volume patches ${\bf{1}}_D$. To do so, we look for smooth domains $D$ with the following parametrization,
\begin{align*}
D=&\left\{(re^{i\theta},\cos(\phi))\, :\, 0\leq r\leq r(\phi,\theta), 0\leq \theta\leq 2\pi, 0\leq \phi\leq \pi\right\},
\end{align*}
where the shape is sufficiently close to a revolution shape domain, meaning that
$$r(\phi,\theta)=r_0(\phi)+f(\phi,\theta),
$$ 
with small perturbation $f$. Since the domain is assumed to be smooth then we should prescribe the Dirichlet boundary conditions,
$$r_0(0)=r_0(\pi)=f(0,\theta)=f(\pi,\theta)=0.$$
Notice that without any perturbation, that is,  $f\equiv 0$, the initial data $q_0={\bf 1}_D$ defines a stationary solution for \eqref{equation}, as we will prove in Lemma \ref{Prop-trivial}. Now a rotating solution about the vertical axis is a time--dependent solution taking the form,
$$
q(t,x)=q_0(e^{-i\Omega t}x_h,x_3), \quad q_0={\bf 1}_D, \,x_h=(x_1,x_2).
$$
We shall see later that it is equivalent  to  check that
\begin{align*}
\nonumber {F}(\Omega,f)(\phi,\theta):=\psi_0(r(\phi,\theta)e^{i\theta},\cos(\phi))-\frac{\Omega}{2}r^2(\phi,\theta)-m(\Omega,f)(\phi)=0,
\end{align*}
for any $(\phi,\theta)\in[0,\pi]\times[0,2\pi]$, where
 $$
m(\Omega,f)(\phi):=\frac{1}{2\pi}\bigintsss_0^{2\pi}\left\{{
	\psi_0}(r(\phi,\theta)e^{i\theta},\cos(\phi))-\frac{\Omega}{2}r^2(\phi,\theta)\right\}d\theta,
$$
where $\psi_0$ stands for the stream function associated to $q_0.$
With this reformulation we visualize  the smooth rotating surface as a collection of interacting  stratified   horizontal sections rotating with the same angular velocity but their  size degenerates when we approach the north and south poles corresponding to $\phi\in\{0,\pi\}$.

In order to apply a bifurcation argument, one has to deal with the linearized operator of $F$ around $f=0$. From Proposition \ref{Prop-lin2}
such linearized operator has a compact expression in terms of hypergeometric functions. Indeed, for $h(\phi,\theta)=\sum_{n\geq 1}h_n(\phi)\cos(n\theta),$ one gets
\begin{equation*}
\partial_{f} F(\Omega,0)h(\phi,\theta)=r_0(\phi)\nu_\Omega(\phi)\sum_{n\geq 1}\cos(n\theta)\mathcal{L}_n(h_n)(\phi),
\end{equation*}
where
\begin{align*}
\mathcal{L}_n(h_n)(\phi)=&h_n(\phi)-\mathcal{K}_{n}^{\Omega}(h_n)(\phi)\\
:=&h_n(\phi) -\int_0^\pi \nu_\Omega^{-1}(\phi)H_n(\phi,\varphi)h_n(\varphi)d\varphi,
\end{align*}
with
\begin{equation*}
 \nu_\Omega(\phi):=\int_0^\pi H_1(\phi,\varphi)d\varphi-\Omega,\quad R(\phi,\varphi):=(r_0(\phi)+r_0(\varphi))^2+(\cos(\phi)-\cos(\varphi))^2
 \end{equation*}
and for $n\geq1,$
\begin{align*}
H_n(\phi,\varphi):=&\frac{2^{2n-1}\left(\frac12\right)^2_{n}}{(2n)!}\frac{\sin(\varphi)r_0^{n-1}(\phi)r_0^{n+1}(\varphi)}{\left[R(\phi,\varphi)\right]^{n+\frac12}} F_n\left(\frac{4r_0(\phi)r_0(\varphi)}{R(\phi,\varphi)}\right).
\end{align*}
Here $F_n$ denotes the hypergeometric function
$$
 F_n(x)=F\left(n+\frac{1}{2},n+\frac{1}{2},2n+1,x\right), \quad  x\in[0,1).
$$
An important observation is that  the kernel study of $\partial_{f} F(\Omega,0)$ amounts to checking   whether  $1$ is an eigenvalue for $\mathcal{K}_{n}^{\Omega}$.
To do that we first start with symmetrizing this operator by working  on suitable  weighted Hilbert spaces.  A natural  candidate for that  is  the Hilbert space $L^2_{\mu_\Omega}(0,\pi)$ of square integrable functions with respect to the measure
\begin{equation*}
d\mu_{\Omega}(\varphi):=\sin(\varphi)r_0^2(\varphi)\nu_{\Omega}(\varphi)d\varphi.
\end{equation*}
In general this defines a signed measure and to get a positive one we should restrict the values of  $\Omega$ to the set $(-\infty,\kappa)$, where $\displaystyle{\kappa:=\inf_{\phi\in(0,\pi)}\int_0^\pi H_1(\phi,\varphi)d\varphi}.$

In the next step we prove that  for any $n\geq1$, the operator  $\mathcal{K}_{n}^{\Omega}:L^2_{\mu_\Omega}\rightarrow L^2_{\mu_\Omega}$ acts as  a compact self--adjoint integral operator.  This gives us the structure of the eigenvalues which is a discrete set and we establish from the positivity of the kernel  that the largest eigenvalue $\lambda_n(\Omega)$ giving the spectral radius is positive and simple.   For given integer $n\geq1$, we  define the set
\begin{equation*}
\mathscr{S}_n:=\Big\{\Omega\in(-\infty,\kappa) \quad\textnormal{s.t.}\quad  \lambda_n(\Omega)=1\Big\},
\end{equation*}
and in  Proposition \ref{prop-operator} we shall describe some basic properties of $\lambda_n$ through precise study of the kernel. Those properties  show in particular that the set $\mathscr{S}_n$ is formed by a single point denoted by $\Omega_n$, see Proposition \ref{prop-kernel-onedim} for more details. In addition, we show that the sequence $n\in N^\star \mapsto \Omega_n$ is strictly increasing which ensures that the kernel of the linearized operator is a one--dimensional vector space, see Proposition \ref{cor-fredholm}. Notice that the  weighted space $L^2_{\mu_\Omega}$  is chosen to be so weak in order  to have it be stable under the nonlinear functional $F$. So we need to reinforce the regularity by selecting the   standard  H\"older spaces $\mathscr{C}^{1,\alpha}$ with Dirichlet boundary condition and $\alpha\in(0,1)$.  However this choice  generates  two  delicate problems. The first one is to check that the eigenfunctions constructed in $L^2_{\mu_\Omega}$ are sufficient smooth and belong to the new spaces. To reach  this regularity we need to check that the function $\nu_\Omega$ is $\mathscr{C}^{1,\alpha}$ and this requires more careful analysis due to the logarithmic singularity, see  Proposition \ref{Lem-meas}. Notice that the eigenfunctions satisfy the boundary condition provided that $n\geq2$ and which fails for $n=1$. The second difficulty concerns the stability of the H\"{o}lder  spaces by the nonlinear functional $F$, in fact not $F$ but another modified functional $\tilde{F}$ deduced  from the preceding one by removing the singularities coming from of the north and south poles, \mbox{see \eqref{Ftilde}.} The deformation of the Euclidean kernel through the cylindrical coordinates generates singularities on the poles because the size of  horizontal sections degenerates at those points. That is the central  difficulty when we try to implement potential theory arguments to get the stability of the function spaces and will be  discussed in  Section \ref{sec-regularity}.

Before stating our result, we need to make the following assumptions on the initial profile $r_0$ and denoted throughout this paper by {\bf (H)} :
{\begin{itemize}
\item[{\bf(H1)}] $r_0\in \mathscr{C}^{2}([0,\pi])$, with $r_0(0)=r_0(\pi)=0$ and $r_0(\phi)>0$ for $\phi\in(0,\pi)$.
\item[{\bf(H2)}] There exists $C>0$ such that
$$
\forall\,\phi\in[0,\pi],\quad  C^{-1}\sin\phi\leq r_0(\phi)\leq C\sin(\phi).
$$
\item[{\bf(H3)}] $r_0$ is symmetric with respect to $\phi=\frac{\pi}{2}$, i.e., $r_0\left(\frac{\pi}{2}-\phi\right)=r_0\left(\frac{\pi}{2}+\phi\right)$, for any $\phi\in[0,\frac{\pi}{2}]$.
\end{itemize}}
Now we are ready to give a  short version of the  main result of this paper and the precise one is detailed in  Theorem \ref{theorem}.
 \begin{theo}\label{theo-introduction}
Assume that $r_0$ satisfies the assumptions {\bf (H)}. Then  for any $m\geq 2$, there exists a curve of non trivial rotating solutions with $m$-fold symmetry to the equation \eqref{equation} bifurcating from the trivial revolution shape associated to $r_0$ at the angular velocity $\Omega_m,$ the unique point of the set $\mathscr{S}_m.$
\end{theo}
We specify that by $m$-fold shape symmetry of $\R^3$, we mean a surface  invariant under rotation with the vertical axis and angle $\frac{2\pi}{m}\cdot$ 

There is the particular case of $r_0(\phi)=\sin(\phi)$ defining the unit sphere. Here, its associated stream function can be explicitly computed (see \cite{Kellog}) and it is quadratic inside the shape, that is,
$$
\psi_0(x)=\frac16 (x_1^2+x_2^{2}+x_3^2-3).
$$
That gives us some interesting properties on the eigenvalues $\Omega_m$ of the above Theorem \ref{theo-introduction}. In particular, we achieve that the above eigenvalues $\Omega_m$ belong to $(0,\frac13)$. The same properties occur also in the case of an ellipsoid of equal $x$ and $y$ axes defining a revolution shape around the $z$--axis. In this case, the associated stream function is also quadratic. See Section \ref{Sec-sphere} for a more detailed discussion about those cases. 

The paper  is structured as follows. In Section $2,$ we provide different  reformulations for  the  rotating patch problem and we introduce the appropriate function spaces. Section $3$ is devoted to   different useful expressions of the linearized operator around a stationary solution. The spectral study of the linearized operator will be developed  in  Section $4.$ In Section $5,$ we shall discuss the   well--definition of the nonlinear functional and its regularity. In Section $6,$ we give the general statement of our result and provide its  proof. We end this paper with  three appendices concerning special functions, bifurcation theory and potential theory.

\begin{ackname}
We would like to thank D. G. Dritschel for proposing this problem and for several discussions around it. \end{ackname}

\section{Vortex patch equations}
Take  an initial data uniformly distributed in a bounded domain of $\R^3$, that is, $ q_0={\bf{1}}_{D}$. Then, this structure is preserved by the evolution and one gets for any time $t\geq0$
\begin{align}\label{VP}
q(t,x)={\bf{1}}_{D(t)}(x),
\end{align}
for some bounded domain $D(t)$.  To track the dynamics of the boundary (which is a surface here) we  can implement the contour dynamics method introduced by Deem and Zabusky for Euler equations \cite{DeemZabusky}.  Indeed,  let $\gamma_t: (\phi,\theta)\in\T^2\mapsto \gamma_t(\phi,\theta)\in \R^3$ be  any  parametrization of the boundary $\partial D_t$. Since the boundary is transported by the flow then
\begin{equation}\label{Gene-eq}
\big(\partial_t\gamma_t-U(t,\gamma_t)\big)\cdot n(\gamma_t)=0,
\end{equation}
where $U=(u,v,0)$ and $n(\gamma_t)$ is a normal vector to the boundary at the point $\gamma_t$. There is a special parametrization called Lagrangian parametrization given by 
$$
\partial_t\gamma_ t=U(t,\gamma_t),
$$
which is commonly used to follow the boundary motion. From the Biot--Savart law we deduce that
\begin{align}\label{Bio-Form} U(t,\gamma_t\big(\phi,\theta)\big)=\frac{1}{4\pi}\int_{D_t}\frac{(\gamma_t\big(\phi,\theta)-y)^\perp}{|\gamma_t\big(\phi,\theta)-y|^3}dA(y)
=\frac{1}{4\pi}\int_{\partial D_t}\frac{n^\perp(y)}{|\gamma_t\big(\phi,\theta)-y|}d\sigma(y),
\end{align}
where $d\sigma$ denotes the Lebesgue  surface measure of $\partial D_t$. We have used the notation $x^\perp=(-x_2,x_1,0)\in\R^3$ for $x=(x_1,x_2,x_3)\in\R^3$.

\subsection{Stationary patches} 

Our next goal is to check that any initial patch with revolution shape around the vertical axis generates a stationary solution. More precisely, we have the following result.

\begin{lem}\label{Prop-trivial}
Let   $r:[-1,1]\rightarrow \R_+$  be a continuous function with $r(-1)=r(1)=0$ and let $D$ be the domain enclosed by the surface 
$\left\{(r(z)e^{i\theta},z),\ \theta\in[0,2\pi], z\in[-1,1]\right\}$,  then $q(t,x)={\bf{1}}_{D}(x)$ defines a stationary solution for \eqref{equation}.
\end{lem}

\begin{proof}
Recall from \eqref{Bio-Form} that
\begin{equation}\label{Veloc1}
 U(x)=\frac{1}{4\pi}\bigintsss_{D}\frac{(x-y)^\perp}{|x-y|^3}dA(y).
\end{equation}
Define
$$
G(x):=U(x)\cdot x= -\frac{1}{4\pi}\bigintsss_{D}\frac{x\cdot y^\perp}{|x-y|^3}dA(y), \quad x\in\R^3,
$$
and let us prove that $G\equiv 0$. 
Take  $\theta\in \R$ and denote by $\mathcal{R}_\theta$ the rotation: $ x=(x_h,x_3)\mapsto (e^{i\theta} x_h, x_3)$. Since $D$ is invariant by $\mathcal{R}_\theta$, changing variables leads to 
$$
G(\mathcal{R}_\theta x)= G(x).
$$
Therefore $G(x)=G(|x_h|,0,x_3)$, which means that
$$
G(x)= \frac{-|x_h|}{4\pi}\bigintsss_{D}\frac{y_2\, dA(y)}{((|x_h|-y_1)^2+y_2^2+(x_3-y_3)^2)^{\frac32}}.
$$
Since $D$ is invariant by the reflexion: $y\mapsto (y_1,-y_2,y_3)$ then a change of variables implies that $G(x_1,x_2,x_3)=G(x_1,-x_2,x_ 3)=-G(x_1,x_2,x_3)$ 
and thus $G(x)=0$. Consequently we get in particular that
$$
U(x)\cdot x=0, \quad \forall x\in \partial D.
$$
On the other hand,  we get from the revolution shape property of $D$ that the horizontal component of the normal vector is $\vec{n}_h(x)=(x_1,x_2)$, which implies 
\begin{align}\label{weak_equation}
U(x)\cdot \vec{n}(x)=(u,v)(x)\cdot \vec{n}_h(x)=0, \quad \forall \, x\in \partial D.
\end{align}
This implies that ${\bf{1}}_{D}$ is a stationary solution in the weak sense. 
\end{proof}

\subsection{Reformulations for periodic patches}
In this section, we shall give two ways to write down rotating patches using respectively the velocity field and the stream function.
Assume that we have a rotating patch  around the $x_3$ axis with constant angular velocity $\Omega\in\R$, that is  $D_t=\mathcal{R}_{\Omega t}D$, with $\mathcal{R}_{\Omega t}$ being the rotation of angle $\Omega t$ around the vertical axis. Inserting this expression  into the equation \eqref{Gene-eq} we get
$$
\big(U(x)-\Omega x^\perp\big)\cdot \vec{n}(x)=0, \quad \forall\, x\in \partial D.
$$
Since $U$ is horizontal then this equation 
means also that each horizontal section $D_{x_3}:=\{ y\in\R^2,\, (y,x_3)\in D\}$ rotates with the same  angular velocity $\Omega$. Hence the horizontal sections satisfy the equation 
$$
\big(U(x)-\Omega x^\perp\big)\cdot \vec{n}_{D_{x_3}}(x_h)=0, \quad x_h=(x_1,x_2)\in \partial D_{x_3}, \ x_3\in\R,
$$
where $\vec{n}_{D_{x_3}}$ denotes a normal vector to the planar curve $\partial D_{x_3}$. Next we shall write down this equation in the particular case of  simply connected domains that can be described through polar parametrization  in the following way:
\begin{align}\label{param}
D=&\left\{(re^{i\theta},\cos(\phi))\, :\, 0\leq r\leq r(\phi,\theta), 0\leq \theta\leq 2\pi, 0\leq \phi\leq \pi\right\}.
\end{align}
Notice that we have assumed in this description, and without  any loss of generality, that the orthogonal projection onto the vertical axis is the segment $[-1,1]$.  The horizontal sections are indexed by $\phi$ and parametrized by the polar coordinates as $\theta\mapsto r(\phi,\theta)$   and it is obvious that
$$
\vec{n}_{\partial D_{x_3}}(r(\phi,\theta)e^{i\theta})=\left( i\partial_{\theta}r(\phi,\theta)-r(\phi,\theta)\right)e^{i\theta}.
$$
Then, the equation of the sections reduces to
\begin{equation}\label{vel-form}
\textnormal{Re}\left[\left\{U_h(\phi,\theta)-i\Omega r(\phi,\theta)e^{i\theta}\right\}\left\{\big[i\partial_{\theta}r(\phi,\theta)+r(\phi,\theta)\big]e^{-i\theta}\right\}\right]=0,\quad \forall (\phi,\theta)\in[0,\pi]\times[0,2\pi],
\end{equation}
with, according to \eqref{Bio-Form} and the change of variable $y_3=\cos\varphi$,
\begin{align}\label{VelXW1}
\nonumber U_h(\phi,\theta):=&(U_1,U_2)(r(\phi,\theta),\cos\phi)\\
\nonumber=&\frac{1}{4\pi}\bigintsss_{-1}^{1}\bigintsss_{\partial D_{y_3}}{\frac{n^\perp_{\partial D_{y_3}}(y_h)dy_hdy_3}{|(r(\phi,\theta)e^{i\theta},\cos(\phi))-y|}}\\
=&\frac{1}{4\pi}\bigintsss_{0}^{\pi}\bigintsss_0^{2\pi}\frac{\sin(\varphi)(\partial_{\eta}r(\varphi,\eta)e^{i\eta}+ir(\varphi,\eta)e^{i\eta})}{|(r(\phi,\theta)e^{i\theta},\cos(\phi))-(r(\varphi,\eta)e^{i\eta},\cos(\varphi))|}d\eta d\varphi.
\end{align}
We shall look for a rotating solution close to a stationary one described by a given revolution shape $(\theta,\phi)\mapsto (r_0(\phi) e^{i\theta},\cos(\phi))$. This means that we are looking for a parametrization in the form 
\begin{equation}\label{f}
r(\phi,\theta)=r_0(\phi)+f(\phi,\theta), \quad f(\phi,\theta)=\sum_{n\geq 1}f_n(\phi)\cos(n\theta).
\end{equation}
Implicitly, we have assumed that the domain $D$ is symmetric with respect to the  {{plane $x_2=0$.}}
In addition, we ask the following boundary conditions,
$$
r_0(0)=r_0(\pi)=f(0,\theta)=f(\pi,\theta)=0,
$$
meaning that the domain $D$ intersects the vertical axis at the points $(0,0,-1)$ and $(0,0,1).$

Define the functionals
$$
F_{\bf v}(\Omega,f)(\phi, \theta):=\textnormal{Re}\left[\left\{I_{\bf v}(f)(\phi,\theta)-i\Omega r(\phi,\theta)e^{i\theta}\right\}\left\{i\partial_{\theta}r(\phi,\theta)e^{-i\theta}+r(\phi,\theta)e^{-i\theta}\right\}\right],
$$
with
\begin{equation}\label{v}
I_{\bf v}(f)(\phi,\theta):=U_h(\phi,\theta)=\frac{1}{4\pi}\bigintsss_{0}^{\pi}\bigintsss_0^{2\pi}\frac{\sin(\varphi)(\partial_{\eta}r(\varphi,\eta)e^{i\eta}+ir(\varphi,\eta)e^{i\eta})}{|(r(\phi,
\theta)e^{i\theta},\cos(\phi))-(r(\varphi,\eta)e^{i\eta},\cos(\varphi))|}d\eta d\varphi.
\end{equation}
The subscript $\bf{v}$ refers to the velocity formulation and we use it to compare it later to the stream function formulation. Hence, we need to study the equation:
$$
F_{\bf v}(\Omega,f)(\phi,\theta)=0, \quad (\phi,\theta)\in[0,\pi]\times[0,2\pi].
$$
By Lemma \ref{Prop-trivial}, one has $F_{\bf v}(\Omega,0)(\phi, \theta)\equiv 0$, for any $\Omega\in\R$.

\subsection{Stream function formulation}
There is another way to characterize the rotating solutions described in the previous subsection by virtue of  the stream function formulation.

For $\phi\in [0,\pi],$ let  
$
\theta\in [0,2\pi]\mapsto \gamma_\phi(\theta):=r(\phi,\theta)e^{i\theta},
$
be the parametrization of $\partial D_z$, where $z=\cos(\phi)$. Then one can check without difficulties that \eqref{vel-form} agrees with
$$
 \partial_\theta \left\{\psi_0(\gamma_\phi(\theta),\cos(\phi))-\frac{\Omega}{2}|\gamma_\phi(\theta)|^2\right\}=0, \quad \forall (\phi,\theta)\in[0,\pi]\times[0,2\pi].
$$
Then, the equation can be integrated obtaining
$$
\psi_0(\gamma_\phi(\theta),\cos(\phi))-\frac{\Omega}{2}|\gamma_\phi(\theta)|^2=m(\Omega,f)(\phi),
$$
where $m(\Omega,f)(\phi)$ is a function depending only on $\phi$ and given by
\begin{equation}\label{meanT}
m(\Omega,f)(\phi):=\frac{1}{2\pi}\bigintsss_0^{2\pi}\left\{\psi_0(r(\phi,\theta)e^{i\theta},\cos(\phi))-\frac{\Omega}{2}r^2(\phi,\theta)\right\}d\theta,\quad  \, r=r_0+f.
\end{equation}
Let us consider the functional 
\begin{align}\label{nonlinearfunction2}
\nonumber {F}_{\bf  s}(\Omega,f)(\phi,\theta):=&\psi_0(r(\phi,\theta)e^{i\theta},\cos(\phi))-\frac{\Omega}{2}r^2(\phi,\theta)-m(\Omega,f)(\phi)\\
=& G(\Omega,f)(\phi,\theta)-\frac{1}{2\pi}\int_0^{2\pi} G(\Omega,f)(\phi,\eta)d\eta,
\end{align}
where 
$$
G(\Omega,f)(\phi,\theta):=\psi_0(r(\phi,\theta)e^{i\theta},\cos(\phi))-\frac{\Omega}{2}r^2(\phi,\theta),
$$
and  the stream function is given by 
$$
\psi_0(r(\phi,\theta)e^{i\theta},\cos(\phi))=-\frac{1}{4\pi}\bigintsss_{0}^{\pi}\bigintsss_0^{2\pi}\bigintsss_0^{r(\varphi,\eta)}\frac{\sin(\varphi)rdrd\eta d\varphi}{|(re^{i\eta},\cos(\varphi))-(r(\phi,\theta)e^{i\theta},\cos(\phi))|}\cdot
$$
Then, finding a rotating solution amounts to solving in $f$, for some specific angular velocity  constant $\Omega$,  the equation 
$$
{F}_{\bf  s}(\Omega,f)(\phi,\theta)=0, \quad \forall (\phi,\theta)\in[0,\pi]\times[0,2\pi].
$$
Remark that one may check directly from this reformulation that  any revolution shape is a solution for any angular velocity $\Omega,$ meaning that,  ${F}_{\bf  s}(\Omega,0)=0,$ for any $\Omega$. 
Motivated by the Section $3$ on the structure of the linearized operator, we find it better to get rid of the singularities of the poles  and work with the modified functional 
\begin{equation*}
\tilde{F}(\Omega,f)(\phi,\theta):=\frac{F_{\bf  s}(\Omega,f)(\phi,\theta)}{r_0(\phi)}\cdot
\end{equation*}
Therefore, we get
\begin{equation}\label{Ftilde}
\tilde{F}(\Omega,f)(\phi,\theta)=\frac{1}{r_0(\phi)}\left\{I(f)(\phi,\theta)-\frac{\Omega}{2}r(\phi,\theta)^2-m(\Omega,f)(\phi)\right\},
\end{equation}
with
\begin{equation}\label{def-I}
I(f)(\phi,\theta):=-\frac{1}{4\pi}\bigintsss_{0}^{\pi}\bigintsss_0^{2\pi}\bigintsss_0^{r(\varphi,\eta)}\frac{r\sin(\varphi)drd\eta d\varphi}{|(re^{i\eta},\cos(\varphi))-(r(\phi,\theta)e^{i\theta},\cos(\phi))|},
\end{equation}
and
$$
r(\phi,\theta)=r_0(\phi)+f(\phi,\theta).
$$

\subsection{Functions spaces}
First we shall recall the H\"{o}lder spaces defined on an open nonempty  set $\mathscr{O}\subset \R^d$. Let $\alpha\in (0,1)$ then 
$$
\mathscr{C}^{1,\alpha}(\mathscr{O})=\Big\{ f:\mathscr{O}\mapsto\R, \|f\|_{\mathscr{C}^{1,\alpha}}<\infty\Big\},
$$ 
with
$$
 \|f\|_{\mathscr{C}^{1,\alpha}}=\|f\|_{L^\infty}+\|\nabla f\|_{L^\infty}+\sup_{x\neq y\in \mathscr{O}}\frac{|\nabla f(x)-\nabla f(y)|}{|x-y|^\alpha}\cdot
$$

It is known that $\mathscr{C}^{1,\alpha}(\mathscr{O})$ is a Banach  algebra, meaning a complete space satisfying 
$$
\|fg\|_{\mathscr{C}^{1,\alpha}}\leq C \|f\|_{\mathscr{C}^{1,\alpha}}\|g\|_{\mathscr{C}^{1,\alpha}}.
$$ 
Denote by $\T$ the one--dimensional torus and we  identify the space $\mathscr{C}^{1,\alpha}(\T)$ with the space $\mathscr{C}_{2\pi}^{1,\alpha}(\R)$ of  $2\pi$--periodic functions that belongs to $\mathscr{C}^{1,\alpha}(\R)$. 
Next, we shall introduce  the function spaces  that we use in a crucial way to study  the stability of the functional $\tilde{F}$ defined in \eqref{Ftilde}. For $\alpha\in(0,1)$ and $m\in\N^\star,$ set
\begin{equation}\label{space}
X_m^\alpha:=\Big\{f\in \mathscr{C}^{1,\alpha}\big((0,\pi)\times\T\big) 
 , \, f\left(\phi,\theta\right)=\sum_{n\geq 1}f_n(\phi)\cos(nm\theta)\Big\},
\end{equation}
supplemented with the conditions
\begin{equation}\label{spaceX1}
\forall \theta\in[0,2\pi]\quad  f(0,\theta)=f(\pi,\theta)= 0\quad\hbox{and}\quad \forall (\phi,\theta)\in[0,\pi]\times[0,2\pi] \quad f\left(\pi-\phi,\theta\right)=f\left(\phi,\theta\right).
\end{equation}
This space is equipped with the same norm as $\mathscr{C}^{1,\alpha}((0,\pi)\times{\T}).$ 
The first assumption in \eqref{spaceX1} is a kind  of partial Dirichlet condition and the second one is a symmetry property with respect to the equatorial $\phi=\frac\pi2$.
Notice that any function $f\in\mathscr{C}^{1,\alpha}\big((0,\pi)\times\T\big) $ admits a continuous  extension up to the boundary, so the foregoing conditions are  meaningful.  Furthermore, the Dirichlet boundary conditions allow us to use Taylor's formula to get a constant $C>0$ such that for   any $f\in X_m^\alpha$
\begin{align}\label{platit1}
|f(\varphi,\eta)|\leq& C\|f\|_{\textnormal{Lip}}\sin \varphi,\nonumber\\
\partial_\eta f(0,\eta)=\partial_\eta f(\pi,\eta)=0&\quad\hbox{and}\quad |\partial_\eta f(\varphi,\eta)|\leq C\|f\|_{\mathscr{C}^{1,\alpha}}\sin^\alpha(\varphi).
\end{align}
The notation  $B_{X_m^\alpha}(\varepsilon)$ means the ball in $X_m^\alpha$ centered in $0$ with radius $\varepsilon$.

Next we shall discuss quickly some consequences needed for later purposes and  following from  the assumptions ${\bf{(H)}}$ on $r_0$, given in the Introduction before our main statement. 
\begin{itemize}
\item From {\bf(H2)} we have that $r_0'(0)>0$ and by continuity of the derivative  there exists $\delta>0$ such that $r_0'(\phi)>0$ for $\phi\in[0,\delta]$. 
Combining this with the mean value theorem, we deduce the arc-chord estimate: there exists $C>0$ such that
\begin{equation}\label{Chord}
C^{-1}(\phi-\varphi)^2\leq (r_0(\varphi)-r_0(\phi))^2+(\cos(\phi)-\cos(\varphi))^2\leq C(\phi-\varphi)^2,
\end{equation}
for any $\phi,\varphi\in[0,\pi]$.
\item We have that $\frac{\sin(\cdot)}{r_0(\cdot)}\in \mathscr{C}^\alpha([0,\pi])$, and then { $\phi\in[0,\frac{\pi}{2}]\mapsto\frac{\phi}{r_0(\phi)}$
$\in \mathscr{C}^\alpha([0,\frac{\pi}{2}])$}.

\end{itemize}

\section{Linearized operator}
This section is devoted to show different expressions of the linearized operator around a revolution shape. We can find an useful one in terms of hypergeometric functions. See Appendix \ref{Ap-spfunctions} for details about these special functions.

From now on, we will use the stream function formulation and then we omit the subscript $\bf s$ from $F_{\bf s}$  in order to alleviate the notation. The linearized operator of the velocity formulation is closely related to this one, see the previous section.
\subsection{First representation}
In the following, we provide the structure of the  linearized operator of $F$ around the trivial solution $(\Omega,0)$.
\begin{pro}\label{Prop-lin1}
Let $\tilde{F}$ be as in \eqref{Ftilde} and $ (\phi,\theta)\in[0,\pi]\times[0,2\pi]\mapsto h(\phi,\theta)=\sum_{n\geq 1}h_n(\phi)\cos(n\theta)$ be a smooth function.
Then,
\begin{align}\label{exp-lin1}
\partial_{f} \tilde{F}&(\Omega,0)h(\phi,\theta)\nonumber=-\Omega\sum_{n\geq 1}h_n(\phi)\cos(n\theta)\nonumber\\
+&\sum_{n\geq 1}\cos(n\theta)\left\{\frac{h_n(\phi)}{4\pi{{ r_0(\phi)}}}\bigintss_{0}^\pi\bigintss_0^{2\pi}\frac{\sin(\varphi)r_0(\varphi)\cos(\eta)\,\,d\eta d\varphi}{\sqrt{r_0^2(\phi)+r_0^2(\varphi)+(\cos\phi-\cos\varphi)^2-2r_0(\phi)r_0(\varphi)\cos(\eta)}}\right.\nonumber\\
&\left.-\frac{1}{4\pi r_0(\phi)}\bigintss_{0}^\pi\bigintss_0^{2\pi}\frac{\sin(\varphi)h_n(\varphi)r_0(\varphi)\cos(n\eta)}{\sqrt{r_0^2(\phi)+r_0^2(\varphi)+(\cos\phi-\cos\varphi)^2-2r_0(\phi)r_0(\varphi)\cos(\eta)}}d\eta d\varphi\right\}.
\end{align}
\end{pro}
\begin{proof}
First, note that
$$
|(re^{i\eta},\cos(\varphi))-(r_0(\phi)e^{i\theta},\cos(\phi))|^2=r^2+r_0^2(\phi)+(\cos(\phi)-\cos(\varphi))^2-2rr_0(\phi)\cos(\theta-\eta).
$$
The linearized operator at a state $r_0$ is defined by Gateaux derivative,
\begin{align*}
\partial_f \tilde{F}(\Omega,0)h(\phi,\theta):=&\frac{d}{dt}\tilde{F}(\Omega,th)\Big|_{t=0}(\phi,\theta)\\
=&\frac{1}{r_0(\phi)}\left(\frac{d}{dt}G(\Omega,th)\Big|_{t=0}(\phi,\theta)-\frac{1}{2\pi}\int_0^{2\pi}\frac{d}{dt}G(\Omega,th)\Big|_{t=0}(\phi,\eta)d\eta\right).
\end{align*}
Thus straightforward computations yield
\begin{align*}
&\frac{d}{dt}G(\Omega,th)\Big|_{t=0}(\phi,\theta)
=-\frac{1}{4\pi}\bigintss_{0}^\pi\bigintss_0^{2\pi}\frac{\sin(\varphi)r_0(\varphi)h(\varphi,\eta)d\eta d\varphi}{A(\phi,\theta,\varphi,\eta)^\frac12}-\Omega r_0(\phi)h(\phi,\theta)\\
&+\frac{h(\phi,\theta)}{4\pi}\bigintss_{0}^\pi\bigintss_0^{2\pi}\bigintss_0^{r_0(\varphi)}\frac{\sin(\varphi)r(r_0(\phi)-r\cos(\eta))\,drd\eta d\varphi}{(r^2+r_0^2(\phi)+(\cos(\phi)-\cos(\varphi))^2-2rr_0(\phi)\cos(\eta))^{\frac32}},
\end{align*}
with
$$
A(\phi,\theta,\varphi,\eta):=r_0^2(\varphi)+r_0^2(\phi)+(\cos(\phi)-\cos(\varphi))^2-2r_0(\varphi)r_0(\phi)\cos(\theta-\eta).
$$
By expanding $h$ in Fourier series we get
\begin{align*}
&\partial_f G(\Omega,0)h(\phi,\theta)=-\sum_{n\geq 1}\left.\frac{1}{4\pi}\bigintsss_{0}^\pi\bigintsss_0^{2\pi}\frac{\sin(\varphi)r_0(\varphi)\cos(n\eta)}{A(\phi,\theta,\varphi,\eta)^\frac12}  h_n(\varphi)d\eta d\varphi\left.+\Omega r_0(\phi)h_n(\phi)\cos(n\theta)\right.\right.\\
&\left.+\frac{1}{4\pi}\sum_{n\geq 1}{h_n(\phi)\cos(n\theta)}\bigintsss_{0}^\pi\bigintsss_0^{2\pi}\bigintsss_0^{r_0(\varphi)}\frac{\sin(\varphi)r(r_0(\phi)-r\cos(\eta))drd\eta d\varphi}{(r^2+r_0^2(\phi)+(\cos(\phi)-\cos(\varphi))^2-2rr(\phi,\theta)\cos(\eta))^{\frac32}}\cdot\right.
\end{align*}
Let us analyze every term. For the first one, making the change of variable $\theta-\eta\mapsto \eta$ we get  using a symmetry argument,
\begin{align*}
\bigintsss_{0}^\pi\bigintsss_0^{2\pi}&\frac{\sin(\varphi)r_0(\varphi)h_n(\varphi)\cos(n\eta)d\eta d\varphi}{A(\phi,\theta,\varphi,\eta)^\frac12}=
\bigintsss_{0}^\pi\bigintsss_0^{2\pi}\frac{\sin(\varphi)r_0(\varphi)h_n(\varphi)\cos(n(\eta-\theta))d\eta d\varphi}{A(\phi,\theta,\varphi,{\theta-\eta})^\frac12}\\
=&\cos(n\theta)\bigintsss_{0}^\pi\bigintsss_0^{2\pi}\frac{\sin(\varphi)r_0(\varphi)h_n(\varphi)\cos(n\eta)d\eta d\varphi}{(r_0^2(\varphi)+r_0^2(\phi)
+(\cos(\phi)-\cos(\varphi))^2-2r_0(\varphi)r_0(\phi)\cos(\eta))^{\frac12}}\cdot
\end{align*}
Concerning the last integral term, we first use the identity
\begin{align*}
&\partial_r \frac{r}{(r^2+r_0^2(\phi)+(\cos(\phi)-\cos(\varphi))^2-2rr_0(\phi)\cos(\eta))^{\frac12}}=\\&\frac{{ 1}}{(r^2+r_0^2(\phi)+(\cos(\phi)-\cos(\varphi))^2-2rr_0(\phi)\cos(\eta))^{\frac12}}\\
&-\frac{r(r-r_0(\phi)\cos\eta)}{(r^2+r_0^2(\phi)+(\cos(\phi)-\cos(\varphi))^2-2rr_0(\phi)\cos(\eta))^{\frac32}}\cdot
\end{align*}
Consequently 
\begin{align*}
\mathcal{I}(\phi,\varphi):=&\bigintsss_0^{2\pi}\bigintsss_0^{r_0(\varphi)}\partial_r \frac{r\cos(\eta)drd\eta }{(r^2+r_0^2(\phi)+(\cos(\phi)-\cos(\varphi))^2-2rr_0(\phi)\cos(\eta))^{\frac12}}\\
=&\bigintsss_0^{2\pi}\bigintsss_0^{r_0(\varphi)} \frac{\cos(\eta)drd\eta }{(r^2+r_0^2(\phi)+(\cos(\phi)-\cos(\varphi))^2-2rr_0(\phi)\cos(\eta))^{\frac12}}\\
&-\bigintsss_0^{2\pi}\bigintsss_0^{r_0(\varphi)}\frac{r^2\cos(\eta)-rr_0(\phi)(1-\sin^2(\eta))drd\eta }{(r^2+r_0^2(\phi)+(\cos(\phi)-\cos(\varphi))^2-2rr_0(\phi)\cos(\eta))^{\frac32}}\cdot
\end{align*}
Thus
\begin{align*}
\mathcal{I}(\phi,\varphi)
=&\bigintsss_0^{2\pi}\bigintsss_0^{r_0(\varphi)}\frac{r(r_0(\phi)-r\cos(\eta))drd\eta}{(r^2+r_0^2(\phi)+(\cos(\phi)-\cos(\varphi))^2-2rr_0(\phi)\cos(\eta))^{\frac32}}\\
&+\bigintsss_0^{2\pi}\bigintsss_0^{r_0(\varphi)} \frac{\cos(\eta)drd\eta }{(r^2+r_0^2(\phi)+(\cos(\phi)-\cos(\varphi))^2-2rr_0(\phi)\cos(\eta))^{\frac12}}\\
&-\bigintsss_0^{2\pi}\bigintsss_0^{r_0(\varphi)}\frac{rr_0(\phi)\sin^2(\eta)drd\eta }{(r^2+r_0^2(\phi)+(\cos(\phi)-\cos(\varphi))^2-2rr_0(\phi)\cos(\eta))^{\frac32}}\cdot
\end{align*}
Integrating by parts with respect to $\eta$ gives \begin{align*}
\bigintsss_0^{2\pi}\bigintsss_0^{r_0(\varphi)} &\frac{\cos(\eta)drd\eta }{(r^2+r_0^2(\phi)+(\cos(\phi)-\cos(\varphi))^2-2rr_0(\phi)\cos(\eta))^{\frac12}}\\
&-\bigintsss_0^{2\pi}\bigintsss_0^{r_0(\varphi)}\frac{rr_0(\phi)\sin(\eta)^2drd\eta }{(r^2+r_0^2(\phi)+(\cos(\phi)-\cos(\varphi))^2-2rr_0(\phi)\cos(\eta))^{\frac32}}=0.
\end{align*}
Putting together the preceding identities allows to get
\begin{align*}
\bigintsss_0^\pi\bigintsss_0^{2\pi}&\bigintsss_0^{r_0(\varphi)}\frac{\sin(\varphi)r(r_0(\phi)-r\cos(\eta))drd\eta d\varphi}{(r^2+r_0^2(\phi)+(\cos(\phi)-\cos(\varphi))^2-2rr_0(\phi)\cos(\eta))^{\frac32}}\\
=&\bigintsss_0^\pi\bigintsss_0^{2\pi}\bigintsss_0^{r_0(\varphi)}\partial_r \frac{\sin(\varphi)r\cos(\eta)drd\eta d\varphi}{(r^2+r_0^2(\phi)+(\cos(\phi)-\cos(\varphi))^2-2rr_0(\phi)\cos(\eta))^{\frac12}}\\
=&\int_0^\pi\int_0^{2\pi} \frac{\sin(\varphi)r_0(\varphi)\cos(\eta)drd\eta d\varphi}{(r_0^2(\varphi)+r_0^2(\phi)+(\cos(\phi)-\cos(\varphi))^2-2r_0(\varphi)r_0(\phi)\cos(\eta))^{\frac12}}.
\end{align*}
Therefore we obtain 
\begin{align*}
&\partial_f G(\Omega,0)h(\phi,\theta)=-\sum_{n\geq 1}\left.
\frac{1}{4\pi}\bigintsss_{0}^\pi\bigintsss_0^{2\pi}\frac{\sin(\varphi)r_0(\varphi)\cos(n\eta)}{A(\phi,\theta,\varphi,{\theta-\eta})^\frac12}
h_n(\varphi)d\eta d\varphi\left.\cos(n\theta)\right.\right.\\
&+\frac{1}{4\pi}\sum_{n\geq 1}{h_n(\phi)\cos(n\theta)}\bigintsss_0^\pi\bigintsss_0^{2\pi}
\frac{\sin(\varphi)r_0(\varphi)\cos(\eta)drd\eta d\varphi}{(r_0^2(\varphi)+r_0^2(\phi)+(\cos(\phi)-\cos(\varphi))^2-2r_0(\varphi)r_0(\phi)\cos(\eta))^{\frac12}}\\
&+\Omega r_0(\phi)\sum_{n\geq 1}h_n(\phi)\cos(n\theta).
\end{align*}
Now it is clear that
$$
\frac{1}{2\pi}\int_0^{2\pi}\partial_f G(\Omega,0)h(\phi,\eta)d\eta=0,
$$
and so \eqref{exp-lin1} is obtained. 
\end{proof}

{
\begin{rem}\label{rem-streamfunction0}
Notice that the local part of the linearized operator \eqref{exp-lin1} can be directly related to the stream function $\psi_0$ associated to the domain   parametrized by $(\phi,\theta)\mapsto (r_0(\phi)e^{i\theta},\cos(\phi))$. Indeed, by differentiating  the functional  $\widetilde{\psi}_0:f\mapsto \psi_0((r_0(\phi)+f(\phi,\theta))e^{i\theta},\cos(\phi))$  at $f\equiv 0$ and in the direction $h$  one gets
\begin{align*}
(\partial_f \widetilde{\psi}_0)&h(\phi,\theta)=\frac{h(\phi,\theta)}{4\pi}\bigintss_{0}^\pi\bigintss_0^{2\pi}\frac{\sin(\varphi)r_0(\varphi)\cos(\eta)\,\,d\eta d\varphi}{\sqrt{r_0^2(\phi)+r_0^2(\varphi)+(\cos\phi-\cos\varphi)^2-2r_0(\phi)r_0(\varphi)\cos(\eta)}}\cdot
\end{align*}
This form is useful later   for spherical and   ellipsoidal shapes where the stream functions admit  explicit expressions inside these domains, see Section \ref{Sec-sphere}.
\end{rem}
}
\subsection{Second representation with hypergeometric functions}
The main purpose of this subsection is to provide a suitable representation of the linearized operator.  First we need to  use some notations. For $n\geq1$,  set 
$$
 F_n(x):=F\left(n+\frac{1}{2},n+\frac{1}{2};2n+1;x\right), \quad  x\in[0,1), 
$$
 where the hypergeometric functions are defined in the  Appendix \ref{Ap-spfunctions}. Other   useful notations are listed below,
\begin{equation}\label{RefR}
 R(\phi,\varphi):=(r_0(\phi)+r_0(\varphi))^2+(\cos(\phi)-\cos(\varphi))^2, \quad  0<\phi,\varphi<\pi ,
 \end{equation}
and 
\begin{align}\label{H-1}
H_n(\phi,\varphi):=&\frac{2^{2n-1}\left(\frac12\right)^2_{n}}{(2n)!}\frac{\sin(\varphi)r_0^{n-1}(\phi)r_0^{n+1}(\varphi)}{\left[R(\phi,\varphi)\right]^{n+\frac12}} F_n\left(\frac{4r_0(\phi)r_0(\varphi)}{R(\phi,\varphi)}\right).
\end{align}

Now we are ready to state the main result of this section.
\begin{pro}\label{Prop-lin2}
Let $\tilde{F}$ be as in \eqref{Ftilde} and $h(\phi,\theta)=\sum_{n\geq 1}h_n(\phi)\cos(n\theta), (\phi,\theta)\in[0,\pi]\times [0,2\pi],$ be a smooth function.
Then,
\begin{equation}\label{Prop-lin2-expression}
\partial_{f} \tilde{F}(\Omega,0)h(\phi,\theta)=\sum_{n\geq 1}\cos(n\theta)\mathcal{L}_n^\Omega(h_n)(\phi),
\end{equation}
where
\begin{align*}
\mathcal{L}_n^\Omega (h_n)(\phi)=&h_n(\phi)\left[\bigintsss_0^\pi
H_1(\phi,\varphi)d\varphi-\Omega\right]-\bigintsss_0^\pi H_n(\phi,\varphi)h_n(\varphi)d\varphi, \quad  \phi\in(0,\pi) .
\end{align*}
\end{pro}

\begin{proof}
With the help of Lemma \ref{Lem-integral}, we can simplify more the expression of the linearized operator given in Proposition \ref{Prop-lin1}. We shall first give another representation of the first integral of \eqref{exp-lin1},
\begin{align*}
\mathscr{I}_1(\phi):=&\frac{1}{4\pi}\bigintss_0^\pi\bigintss_0^{2\pi}\frac{\sin(\varphi)r_0(\varphi)\cos(\eta)}{\sqrt{r_0^2(\phi)+r_0^2(\varphi)+(\cos(\phi)-\cos(\varphi))^2-2r_0(\phi)r_0(\varphi)\cos(\eta)}}d\eta d\varphi\\
=&\frac{1}{4\pi}\bigintss_0^\pi\frac{\sin(\varphi)r_0(\varphi)}{\sqrt{2r_0(\phi)r_0(\varphi)}}\bigintss_0^{2\pi}\frac{\cos(\eta)}{\sqrt{\frac{r_0^2(\phi)+r_0^2(\varphi)+(\cos(\phi)-\cos(\varphi))^2}{2r_0(\phi)r_0(\varphi)}-\cos(\eta)}}d\eta d\varphi.
\end{align*}
From Lemma \ref{Lem-integral} we infer 
\begin{align*}
\bigintss_0^{2\pi}\frac{\cos(\eta)\,\,d\eta }{\sqrt{\frac{r_0^2(\phi)+r_0^2(\varphi)+(\cos(\phi)-\cos(\varphi))^2}{2r_0(\phi)r_0(\varphi)}-\cos(\eta)}}=2\pi \frac{2\left(\frac12\right)_1^2}{2!}\frac{F_1\left(\frac{2}{1+\frac{r_0^2(\phi)+r_0^2(\varphi)+(\cos(\phi)-\cos(\varphi))^2}{2r_0(\phi)r_0(\varphi)}}\right) }{\left(1+\frac{r_0^2(\phi)+r_0^2(\varphi)+(\cos(\phi)-\cos(\varphi))^2}{2r_0(\phi)r_0(\varphi)}\right)^{\frac{3}{2}}}\cdot
\end{align*}
Thus we deduce 
\begin{align*}
\mathscr{I}_1(\phi)
&= r_0(\phi)\frac14\bigintss_0^\pi\frac{\sin(\varphi)r_0^2(\varphi)}{R^{\frac32}(\phi,\varphi)}F_1\left(\frac{4r_0(\phi)r_0(\varphi)}{R(\phi,\varphi)}\right) d\varphi=
{r_0(\phi)\int_0^\pi H_1(\phi,\varphi)d\varphi}.
\end{align*}
{Remark that  the validity of   Lemma \ref{Lem-integral} is guaranteed since  the inequality 
$$
{\frac{r_0^2(\phi)+r_0^2(\varphi)+(\cos(\phi)-
\cos(\varphi))^2}{2r_0(\phi)r_0(\varphi)}}>1,
$$
}
is satisfied provided that $\phi\neq\varphi$ which leads to a negligible set. For the last integral in \eqref{exp-lin1}, we apply once again  Lemma \ref{Lem-integral},
\begin{align*}
\bigintss_0^{2\pi}&\frac{\cos(n\eta)}{\sqrt{r_0^2(\phi)+r_0^2(\varphi)+(\cos(\phi)-\cos(\varphi))^2-2r_0(\phi)r_0(\varphi)\cos(\eta)}}d\eta\\
=&\frac{2^{2n+1}\pi\left(\frac12\right)^2_{n}}{(2n)!}\frac{r_0^n(\phi)r_0^n(\varphi)}{R^{n+\frac12}(\phi,\varphi)} F_n\left(\frac{4r_0(\phi)r_0(\varphi)}{R(\phi,\varphi)}\right).
\end{align*}
It follows that
\begin{align*}
\frac{1}{4\pi}\bigintss_0^\pi\bigintss_0^{2\pi}&\frac{\sin(\varphi)h_n(\varphi)r_0(\varphi)\cos(n\eta)}
{\sqrt{r_0^2(\phi)+r_0^2(\varphi)+(\cos(\phi)-\cos(\varphi))^2-2r_0(\phi)r_0(\varphi)\cos(\eta)}}d\eta d\varphi\\
=&\frac{2^{2n-1}\left(\frac12\right)^2_{n}}{(2n)!}\bigintss_0^\pi\frac{r_0^n(\phi)r_0^{n+1}(\varphi){\sin(\varphi)}}{R^{n+\frac12}(\phi,\varphi)} 
F_n\left(\frac{4r_0(\phi)r_0(\varphi)}{R(\phi,\varphi)}\right)h_n(\varphi) d\varphi\\
=& {r_0(\phi)\int_0^\pi H_n(\phi,\varphi)h_n(\varphi)d\varphi,}
\end{align*}
which gives  the  announced result.
\end{proof}

{
\begin{rem}\label{rem-streamfunction}
By virtue of Remark \ref{rem-streamfunction0} and the previous expression one has that
\begin{equation}\label{Ident-2}
\int_0^\pi H_1(\phi,\varphi)d\varphi=\frac{1}{r_0(\phi)}\partial_R \psi_0(Re^{i\theta},\cos(\phi))\left|_{R=r_0(\phi)}\right.,
\end{equation}
where $\psi_0$ is the stream function  associated to the  domain parametrized by $(r_0(\phi)e^{i\theta}
,\cos(\phi))$, for $(\phi,\theta)\in[0,\pi]\times[0,2\pi]$.\end{rem}}

{

\subsection{Qualitative properties of some auxiliary functions}
In the following lemma, we shall study some specific  properties of the sequence of functions $\{H_n\}_n$ introduced in \eqref{H-1}.
We shall study the monotonicity of  the sequence $n\mapsto H_n(\phi,\varphi)$ which will be  crucial later in the study of the monotonicity 
of the eigenvalues associated to the operators  family $\{\mathcal{L}_n, n\geq1\}$.  We will also study the decay rate of this sequence for large $n$.
\begin{lem}\label{Lem-Hndecreasing}
For any $\varphi\neq\phi\in(0,\pi)$, the sequence
$
n\in \N^\star\mapsto H_n(\phi,\varphi)
$
is strictly decreasing.

Moreover, if we assume that $r_0$ satisfies {\bf{(H2)}}, then, for any $0\leq\alpha<\beta\leq1 $
there exists  a  {constant $C>0$} such that
\begin{equation}\label{estim-asym}
|H_n(\phi,\varphi)|\leq C n^{-\alpha}\frac{\sin(\varphi) r_0^{\frac12}(\varphi)}{r_0^{\frac32}(\phi)}|\phi-\varphi|^{-\beta},
 \quad \forall  n\geq1, \phi\neq \varphi\in(0,\pi). 
\end{equation}
\end{lem}
\begin{proof}

By virtue  of \eqref{H-1} we may write
\begin{align*}
H_n(\phi,\varphi)=\frac{2^{2n-1}\Gamma^2\left(n+\frac12\right)}{(2n)!\pi}\frac{\sin(\varphi)r_0(\varphi)^{\frac12}}{4^{n+\frac12}r_0(\phi)^{\frac32}}x^{n+\frac12}F_n(x),
\end{align*}
where $x:=\frac{4r_0(\phi)r_0(\varphi)}{R(\phi,\varphi)}$ belongs to
$[0,1)$ provided that $\varphi\neq\phi$. Now using the integral representation of hypergeometric functions \eqref{Ap-spfunctions-integ} we obtain
\begin{align}\label{trax1}
\nonumber H_n(\phi,\varphi)=&\frac{2^{2n-1}\Gamma^2\left(n+\frac12\right)}{(2n)!\pi}\frac{\sin(\varphi)r_0^{\frac12}
(\varphi)}{4^{n+\frac12}r_0^{\frac32}(\phi)}\frac{(2n)!}{\Gamma^2\left(n+\frac12\right)}x^{n+\frac12}\bigintsss_0^1 t^{n-\frac12}(1-t)^{n-\frac12}(1-tx)^{-n-\frac12}dt\\
\nonumber=&\frac{2^{2n-1}\Gamma^2\left(n+\frac12\right)}{(2n)!\pi}\frac{\sin(\varphi)r_0^{\frac12}(\varphi)}{4^{n+\frac12}r_0^{\frac32}(\phi)}
\frac{(2n)!}{\Gamma^2\left(n+\frac12\right)}\mathcal{H}_n(x)\\
=&\frac{1}{4\pi}\frac{\sin(\varphi)r_0(\varphi)^{\frac12}}{r_0^{\frac32}(\phi)}\mathcal{H}_n(x),
\end{align}
with the notation
$$
\mathcal{H}_n(x):= x^{n+\frac12}\int_0^1 t^{n-\frac12}(1-t)^{n-\frac12}(1-tx)^{-n-\frac12}dt.
$$
Therefore the desired result amounts to checking that   $n\mapsto \mathcal{H}_n(x)$ is strictly decreasing for any $x\in(0,1)$. This follows from the fact that  $n\mapsto x^{n+\frac12}$ is strictly decreasing combined with  the identity 
$$
\int_0^1 t^{n-\frac12}(1-t)^{n-\frac12}(1-tx)^{-n-\frac12}dt=\int_0^1 t^{-\frac12}(1-t)^{-\frac12}(1-tx)^{-\frac12}\left(\frac{t(1-t)}{1-tx}\right)^{n}dt,
$$
which shows the strict decreasing of this sequence since $0<\frac{t(1-t)}{1-tx}<1$, for any $t,x\in(0,1)$.

It follows that for any $\phi\neq\varphi$,  the sequence  $n\mapsto H_n(\phi,\varphi)$ is  strictly decreasing.

It remains to prove the decay estimate of $H_n$ for large $n$. It is an immediate consequence of the following more precise estimate: for any  $\alpha\in[0,1]$, we get
{\begin{equation}\label{trix1}
|\mathcal{H}_n(x)|\leq x^{n-\frac12+\alpha }\frac{|\ln(1-x)|^{1-\alpha}}{n^\alpha(1-x)^\alpha},
\end{equation}}
for $n\geq1$ and  $0\leq x<1.$  To see the connection with \eqref{estim-asym}  recall first from \eqref{trax1}  that
$$
|H_n(\phi,\varphi)|\lesssim  \frac{\sin(\varphi)r_0(\varphi)^\frac12}{r_0(\phi)^\frac32} |\mathcal{H}_n(x)|.
$$
Since $0\leq x\leq 1$ then  we obtain from \eqref{trix1} that for any $1\geq\beta>\alpha\geq0$
\begin{align*}
|\mathcal{H}_n(x)|\lesssim\frac{|\ln(1-x)|^{1-\alpha}}{n^\alpha (1-x)^\alpha}\lesssim{n^{-\alpha} (1-x)^{-\beta}}\cdot
\end{align*}
According to \eqref{Est-LogX} we deduce that
\begin{align*}
|\mathcal{H}_n(x)|
\lesssim&\frac{1}{n^\alpha |\phi-\varphi|^\beta}.
\end{align*}
which is the desired inequality. Let us now turn to the proof of \eqref{trix1}. We write
{\begin{align*}
| \mathcal{H}_n(x)|&\leq x^{n+\frac12}\bigintsss_0^1 t^{n-\frac12}\frac{(1-t)^{n-\frac12}}{(1-tx)^{n-\frac12}}\frac{dt}{(1-tx)}\leq x^{n+\frac12}\bigintsss_0^1 t^{n-\frac12}\frac{dt}{1-tx},
\end{align*}}
where we have used that
$$
1-tx\geq  1-t,
$$
for any $t\in[0,1]$ and $0\leq x<1$. Observe that we easily get the identity
\begin{equation}
\bigintsss_0^1 t^{n-\frac12}\frac{dt}{1-tx}=\sum_{k\geq0}\frac{x^k}{n+\frac12+k},
\end{equation}
which implies
$$
\bigintsss_0^1 t^{n-\frac12}\frac{dt}{1-tx}\le\frac{1}{n(1-x)},
$$
and 
$$\bigintsss_0^1 t^{n-\frac12}\frac{dt}{1-tx}\leq-\frac{\ln(1-x)}{x}.
$$
By using interpolation, we obtain
$$
\bigintsss_0^1 t^{n-\frac12}\frac{dt}{1-tx}\leq\frac{1}{n^\alpha}\frac{|\ln(1-x)|^{1-\alpha}}{x^{1-\alpha}(1-x)^\alpha},
$$
which gives us
{$$
| \mathcal{H}_n(x)|\le x^{n+\frac12}\frac{|\ln(1-x)|^{1-\alpha}}{x^{1-\alpha}(1-x)^\alpha}\frac{1}{n^\alpha},
$$}
for $n\in \mathbb{N}^\star$ and $0\le x<1.$
\end{proof}

\section{Spectral study}
In this section, we aim to investigate some fundamental  spectral properties of the linearized operator $\partial_f \widetilde{F}(\Omega,0)$ in order 
to apply the Crandall--Rabinowitz theorem. For this goal  one must check that the kernel and the co--image of the linearized operator are  
one dimensional vector spaces. Noting that the study of the kernel agrees with the eigenvalue problem of a Hilbert--Schmidt operator,
we achieve that the dimension is one. Moreover, we will study the Fredholm structure of the linearized operator, which will imply that the codimension of the image is one. 
At the end of the section, we characterize also the transversal condition.

\subsection{Symmetrization  of the linearized operator}\label{Symmetriz}
The main strategy to explore some spectral properties of the linearized operator at  each frequency level $n$ is to  construct a suitable Hilbert space, basically an  $L^2$ space with respect to a special  Borel measure,  on which it acts as a self--adjoint compact operator. Later we investigate the eigenspace associated with the largest eigenvalue and prove in particular   that this space is one--dimensional.

Let us explain how to symmetrize the operator. Recall  from \eqref{Prop-lin2-expression} that for any smooth function 
$\displaystyle{h(\phi,\theta)=\sum_{n\geq1}h_n(\phi) \cos(n\theta)}$, we may write the operator $\mathcal{L}_n$ under the form
\begin{equation}\label{operator}
\mathcal{L}_n^\Omega(h_n)(\phi)=\nu_{\Omega}(\phi)\left\{h_n(\phi)-\int_0^\pi K_n(\phi,\varphi)h_n(\varphi)d\mu_{\Omega}(\varphi)\right\}, \quad \phi\in[0,\pi], 
\end{equation}
with
\begin{equation}\label{K-kernel}
K_n(\phi,\varphi):=\frac{H_n(\phi,\varphi)}{\sin(\varphi)\nu_{\Omega}(\phi)\nu_{\Omega}(\varphi)r_0^2(\varphi)},
\end{equation}
\begin{equation}\label{nu-function}
\nu_{\Omega}(\phi):=\int_0^\pi H_1(\phi,\varphi)d\varphi-\Omega,
\end{equation}
and the signed measure
\begin{equation}\label{signed-meas}
d\mu_{\Omega}(\varphi):=\sin(\varphi)r_0^2(\varphi)\nu_{\Omega}(\varphi)d\varphi.
\end{equation}
Define the quantity \begin{equation}\label{kappa}
\kappa:=\inf_{\phi\in[0,\pi]}\int_0^\pi H_1(\phi,\varphi)d\varphi.
\end{equation}
We shall discuss in Proposition \ref{Lem-meas} below the existence of $\kappa$ which allows to guarantee the positivity of the measure $d\mu_\Omega$ provided that the parameter $\Omega$ is  restricted to lie in the interval $(-\infty,\kappa)$. We shall also study the regularity  of   the function $\nu_\Omega$ which is delicate and more involved. In particular, we  prove 
that, under reasonable  assumptions  on the profile $r_0$, this  function is at least in the H\"{o}lder  space $\mathscr{C}^{1,\alpha}$ 
for any $\alpha\in(0,1)$.

Notice that the  kernel $K_ n$  is symmetric. Indeed,  according to  \eqref{H-1} we get the formula 
\begin{align}\label{form-K}
K_n(\phi,\varphi)=&\frac{2^{2n-1}\left(\frac12\right)^2_{n}}{(2n)!}\frac{r_0^{n-1}(\phi)r_0^{n-1}(\varphi)}{\nu_\Omega(\phi)\nu_\Omega(\varphi)\left[R(\phi,\varphi)\right]^{n+\frac12}} F_n\left(\frac{4r_0(\phi)r_0(\varphi)}{R(\phi,\varphi)}\right),
\end{align}
which gives the desired property  in view of  the symmetry of $R$, that is, $R(\phi,\varphi)=R(\varphi,\phi)$.

We shall explore  in  Section \ref{Eigenvalue problem}  more spectral properties of the symmetric operator associated to the kernel  $K_n$.

\subsection{Regularity of $\nu_\Omega$}\label{Sec-assym}

This section is devoted to the study of the regularity of the function $\nu_{\Omega}$ that arises in \eqref{operator}, which turns to be a very delicate problem. This is needed for getting enough regularity   for the kernel elements that should belong to the function spaces where bifurcation will be applicable. For lower regularities than Lipschitz class, this can be implemented in a standard way using some boundary behavior of the hypergeometric functions. However for higher regularity of type $\mathscr{C}^{1,\alpha}$, the problem turns out to be more delicate due to some logarithmic singularity induced by  $H_1.$ To get rid of  this singularity   we use some specific cancellation coming from the structure of the kernel. We shall also develop the local structure of $\nu_\Omega$ near its minimum which appears to be crucial later especially in Proposition \ref{prop-operator}.\\
The main result  of this section reads as follows.

\begin{pro} \label{Lem-meas}
Let  $r_0$ be a profile  satisfying {\bf{(H1)}} and {\bf{(H2)}}.  Then  the following properties hold true.
\begin{enumerate}
\item  The function $\phi\in[0,\pi]\mapsto \nu_\Omega(\phi)$ belongs to $\mathscr{C}^\beta([0,\pi]),$ for all $\beta\in [0,1)$.
\item We have $\kappa>0$ and  for any $ \Omega\in(-\infty,\kappa)$ we get
$$
\forall  \phi\in [0,\pi],\quad \nu_{\Omega}(\phi)\geq\kappa-\Omega>0.
$$
\item The function $\nu_\Omega$ belongs  to $\mathscr{C}^{1,\alpha}([0,\pi])$, for any $\alpha\in(0,1)$, with 
$$
\nu_\Omega^\prime(0)=\nu_\Omega^\prime(\pi)=0.
$$
\item Let  $\Omega\in(-\infty,\kappa]$ and assume that  $\nu_\Omega$ reaches its  minimum at a point $\phi_0\in[0,\pi]$ then there exists  $C>0$ independent of $\Omega$  such that, $$
\forall \phi\in[0,\pi],\quad 0\leq \nu_\Omega\big(\phi\big)-\nu_\Omega\big(\phi_0\big)\leq C|\phi-\phi_0|^{1+\alpha}.
$$
Moreover, for  $\Omega=\kappa$ this result becomes 
$$
\forall \phi\in[0,\pi],\quad 0\leq \nu_\kappa\big(\phi\big)\leq C|\phi-\phi_0|^{1+\alpha}.
$$
\end{enumerate}
\end{pro}
\begin{proof}
\medskip
\noindent
${\bf{(1)}}$ { To start, notice   that according to \eqref{H-1}
\begin{align}\label{KK1}
\nonumber H_1(\phi,\varphi)=&\frac{1}{4}\frac{\sin(\varphi)r_0^{2}(\varphi)}{\left[R(\phi,\varphi)\right]^{\frac32}} F_1\left(\frac{4r_0(\phi)r_0(\varphi)}{R(\phi,\varphi)}\right), \quad \forall \, \varphi\neq\phi\in (0,\pi)\\
=&\mathscr{K}_1(\phi,\varphi)\mathscr{K}_2(\phi,\varphi),
\end{align}
where
\begin{equation}\label{K1}
\mathscr{K}_1(\phi,\varphi):=\frac{1}{4}\frac{\sin(\varphi)r_0^{2}(\varphi)}{R^{\frac32}(\phi,\varphi)},
\end{equation}
and
{\begin{equation}\label{K2}
\mathscr{K}_2(\phi,\varphi):=F_1\left(\frac{4r_0(\phi)r_0(\varphi)}{R(\phi,\varphi)}\right).
\end{equation}}
Therefore we may write 
$$
\forall \phi\in(0,\pi),\quad\nu_\Omega(\phi)=\int_0^\pi \mathscr{K}_1(\phi,\varphi) \mathscr{K}_2(\phi,\varphi)d\varphi-\Omega.
$$
Using the boundary behavior of hypergeometric functions stated in Proposition \ref{Prop-behav} we deduce that
\begin{align*}
1\leq \mathscr{K}_2(\phi,\varphi)\leq &C+C\left|\ln\left(1-\frac{4r_0(\phi)r_0(\varphi)}{R(\phi,\varphi)}\right)\right|\\
\le& C+C \left|\ln\left(\frac{(r_0(\phi)+r_0(\varphi))^2+(\cos\phi-\cos\varphi)^2}{(r_0(\phi)-r_0(\varphi))^2+(\cos\phi-\cos\varphi)^2}\right)\right|.
\end{align*}
From the assumption {\bf{(H2)}} on $r_0$ we can write, using the mean value theorem 
$$
(r_0(\phi)+r_0(\varphi))^2+(\cos\phi-\cos\varphi)^2\leq C(\sin\phi+\sin\varphi)^2+ (\phi-\varphi)^2.
$$
In view of  \eqref{Chord}   we get for all $\phi\neq \varphi\in[0,\pi],$
\begin{equation}\label{Est-LogX}
1\le \frac{(r_0(\phi)+r_0(\varphi))^2+(\cos\phi-\cos\varphi)^2}{(r_0(\phi)-r_0(\varphi))^2+(\cos\phi-\cos\varphi)^2}\leq C\frac{(\sin\phi+\sin\varphi)^2}{(\phi-\varphi)^2}+C.
\end{equation}
Consequently, we get
\begin{equation}\label{Est-Log}
 1\leq \mathscr{K}_2(\phi,\varphi)\le C+C \left|\ln\left(\frac{\sin(\phi)+\sin(\varphi)}{|\phi-\varphi|}\right)\right|.
\end{equation}
On the other hand, it is obvious  using the assumption {\bf{(H2)}} on $r_0$ that
\begin{align}\label{bound-K1}
0\le\mathscr{K}_1(\phi,\varphi)\leq\frac{\sin(\varphi)r_0^{2}(\varphi)}{\left[R(\phi,\varphi)\right]^{\frac32}}\leq& \frac{\sin\varphi}{r_0(\varphi)}\leq C, \quad \forall \varphi,\phi\in(0,\pi).
\end{align}
It follows that
\begin{align*}
\sup_{\phi\in(0,\pi)}|\nu_\Omega(\phi)|\leq& C+C\sup_{\phi\in(0,\pi)}\int_0^\pi\ln\left(\frac{\sin(\phi)+\sin(\varphi)}{|\phi-\varphi|}\right)d\varphi\le C,
\end{align*}
which ensures that $\nu_\Omega$ is bounded.

Now  let us check the H\"{o}lder continuity by estimating
\begin{align*}
|\nu_\Omega(\phi_1)-\nu_\Omega(\phi_2)|\leq& \int_0^\pi|\mathscr{K}_2(\phi_1,\varphi)||\mathscr{K}_1(\phi_1,\varphi)-\mathscr{K}_1(\phi_2,\varphi)|d\varphi\\
&+\int_0^\pi|\mathscr{K}_1(\phi_2,\varphi)||\mathscr{K}_2(\phi_1,\varphi)-\mathscr{K}_2(\phi_2,\varphi)|d\varphi\\
&=:\mathscr{H}_1(\phi_1,\phi_2)+\mathscr{H}_2(\phi_1,\phi_2).
\end{align*}
Let us begin with $\mathscr{H}_1$. Notice that
\begin{equation}\label{partialR}
\partial_\phi R(\phi,\varphi)=2r_0^\prime(\phi)(r_0(\phi)+r_0(\varphi))+2\sin\phi(\cos\varphi-\cos\phi),
\end{equation}
which implies that
$$
|\partial_\phi R(\phi,\varphi)|\leq CR^{\frac12}(\phi,\varphi).
$$
It follows that 
\begin{align*}
|\partial_\phi R^{-\frac32}(\phi,\varphi)|\lesssim& R^{-2}(\phi,\varphi)\lesssim  r_0^{-4}(\varphi).
\end{align*}
Differentiating $\mathscr{K}_1$ with respect to $\phi$ yields 
$$
4\pi \partial_\phi\mathscr{K}_1(\phi,\varphi)=\partial_\phi(R^{-\frac32})(\phi,\varphi)\sin\varphi \,r_0^2(\varphi).
$$
Hence using {\bf{(H2)}} we deduce that
\begin{align}\label{Tixa}
\sup_{\phi\in(0,\pi)}|\partial_\phi\mathscr{K}_1(\phi,\varphi)|&\leq  C\frac{\sin\varphi}{r_0^2(\varphi)}\lesssim\frac{1}{\sin\varphi}\cdot
\end{align}
From an interpolation argument using the boundedness of $\mathscr{K}_1$ we find, according to the mean value theorem,  
\begin{align}\label{K1-1}
\nonumber |\mathscr{K}_1(\phi_1,\varphi)-\mathscr{K}_1(\phi_2,\varphi)|=& |\mathscr{K}_1(\phi_1,\varphi)-\mathscr{K}_1(\phi_2,\varphi)|^{1-\beta}|\mathscr{K}_1(\phi_1,\varphi)-\mathscr{K}_1(\phi_2,\varphi)|^\beta\\
\nonumber \le& \big(2\|\mathscr{K}_1\|_{L^\infty}\big)^{1-\beta}\|\mathscr{K}_1\|_{\textnormal{Lip}}^{\beta} |\phi_1-\phi_2|^\beta\\
\lesssim& \frac{|\phi_1-\phi_2|^\beta}{\sin^\beta\varphi}\cdot \end{align}
Using \eqref{Est-Log}, we obtain
\begin{equation}\label{bound-H1}
\mathscr{H}_1(\phi_1,\phi_2)\leq C|\phi_1-\phi_2|^\beta\int_0^\pi\frac{1}{\sin^\beta\varphi}\left\{1+\ln\left(\frac{\sin(\phi_1)+\sin(\varphi)}{|\phi_1-\varphi|}\right)\right\}d\varphi\leq C|\phi_1-\phi_2|^\beta,
\end{equation}
for any $\beta\in(0,1)$.\\
Next we shall proceed to the estimate  $\mathscr{H}_2$.  From \eqref{bound-K1}, one finds that
$$
\mathscr{H}_2(\phi_1,\phi_2)\leq C\int_0^\pi|\mathscr{K}_2(\phi_1,\varphi)-\mathscr{K}_2(\phi_2,\varphi)|d\varphi.
$$ 
We separate the last integral as follows:
\begin{align*}
\mathscr{H}_2(\phi_1,\phi_2)\leq& C\bigintsss_0^\pi|\mathscr{K}_2(\phi_1,\varphi)-\mathscr{K}_2(\phi_2,\varphi)|d\varphi\\
=&C\bigintsss_{B_{\phi_1}(d)\cup B_{\phi_2}(d)}|\mathscr{K}_2(\phi_1,\varphi)-\mathscr{K}_2(\phi_2,\varphi)|d\varphi\\
&+C\bigintsss_{B_{\phi_1}^c(d)\cap B_{\phi_2}^c(d)}|\mathscr{K}_2(\phi_1,\varphi)-\mathscr{K}_2(\phi_2,\varphi)|d\varphi\\
=:& \mathscr{H}_{2,1}(\phi_1,\phi_2)+\mathscr{H}_{2,2}(\phi_1,\phi_2),
\end{align*}
where $d=|\phi_1-\phi_2|$, $B_{\phi}(r)=\{\varphi\in[0,\pi]: |\varphi-\phi|<r\}$ and $B^c_{\phi}(r)$ denotes its complement set. For the first term, $\mathscr{H}_{2,1}$,  we simply use \eqref{Est-Log}
\begin{align*}
\mathscr{H}_{2,1}(\phi_1,\phi_2)\leq& C\bigintsss_{B_{\phi_1}(d)\cup B_{\phi_2}(d)}\left\{1+\left|\ln\left(\frac{\sin(\phi_1)+\sin(\varphi)}{|\phi_1-\varphi|}\right)\right|+\left|\ln\left(\frac{\sin(\phi_2)+\sin(\varphi)}{|\phi_2-\varphi|}\right)\right|\right\}d\varphi.
\end{align*}
{Since for $\varphi\in B_{\phi_1}(d)\cup B_{\phi_2}(d)$ one has $|\phi_1-\varphi|\leq2 |\phi_1-\phi_2|$ and $|\phi_2-\varphi|\leq2 |\phi_1-\phi_2|$, then one achieves
\begin{align}\label{bound-H21}
\mathscr{H}_{2,1}(\phi_1,\phi_2)\leq& C|\phi_1-\phi_2|^\beta\bigintsss_{B_{\phi_1}(d)\cup B_{\phi_2}(d)}\left\{\frac{1}{|\phi_1-\varphi|^\beta}+\left|\ln\left(\frac{\sin(\phi_1)+\sin(\varphi)}{|\phi_1-\varphi|}\right)\right|\frac{1}{|\phi_1-\varphi|^\beta}\right.\nonumber\\
&\left.\qquad\qquad\qquad +\left|\ln\left(\frac{\sin(\phi_2)+\sin(\varphi)}{|\phi_2-\varphi|}\right)\right|\frac{1}{|\phi_2-\varphi|^\beta}\right\}d\varphi\nonumber\\
\leq & C|\phi_1-\phi_2|^\beta,
\end{align}
for any $\beta\in(0,1)$.
For the second term of $\mathscr{H}_2$, we observe that for any ${\varphi \in B_{\phi_1}^c(d)\cap B_{\phi_2}^c(d)}$ one has
\begin{align}\label{Es-eq1}
\frac12|\phi_1-\varphi|\leqslant|\phi_2-\varphi|\leqslant 2|\phi_1-\varphi|.
\end{align} 
On the other hand, direct computations yield
$$
\partial_\phi \mathscr{K}_2(\phi,\varphi)=4r_0(\varphi)\frac{r_0^\prime(\phi)R(\phi,\varphi)-r_0(\phi)\partial_\phi R(\phi,\varphi)}{R^2(\phi,\varphi)}F_1^\prime\left(\frac{4r_0(\phi)r_0(\varphi)}{R(\phi,\varphi)}\right).
$$
}
We know that
$$
\forall x\in(-1,1),\quad F_1^\prime(x)=\frac{3}{4}F(5/2,5/2;4;x),
$$
and by virtue of  the boundary behavior stated in  Proposition \ref{Prop-behav}  we get
$$
\forall x\in [0,1),\quad|F_1^\prime(x)|\lesssim (1-x)^{-1}.
$$
It follows that, using \eqref{Chord}, 
\begin{align}\label{TixaZ}
\forall \phi\neq \varphi\in(0,\pi)\quad \Big|F_1^\prime\left(\frac{4r_0(\phi)r_0(\varphi)}{R(\phi,\varphi)}\right)\Big|\lesssim&\frac{(r_0(\phi)+r_0(\varphi))^2+(\cos\phi-\cos\varphi)^2}{(r_0(\phi)-r_0(\varphi))^2+(\cos\phi-\cos\varphi)^2}\lesssim \frac{R(\phi,\varphi)}{|\phi-\varphi|^2}\cdot
\end{align}
By explicit calculation using \eqref{partialR} we get
\begin{align*}
r_0^\prime(\phi)R(\phi,\varphi)-r_0(\phi)\partial_\phi R(\phi,\varphi)=&r_0^\prime(\phi)\left(r_0^2(\varphi)-r_0^2(\phi)+(\cos \phi-\cos\varphi)^2\right)\\
&+2r_0(\phi)\sin\phi(\cos\phi-\cos\varphi).
\end{align*}
Then using the mean value theorem we find
\begin{align*}
|r_0^\prime(\phi)R(\phi,\varphi)-r_0(\phi)\partial_\phi R(\phi,\varphi)|\lesssim&|\phi-\varphi|\big(r_0(\varphi)+r_0(\phi)+|\cos \phi-\cos\varphi|\big)\\
&+r_0(\phi)\sin\phi|\cos\phi-\cos\varphi|\\
\lesssim&|\phi-\varphi| R^{\frac12}(\phi,\varphi).
\end{align*}
Putting together the preceding estimates we find
\begin{align*}
|\partial_\phi \mathscr{K}_2(\phi,\varphi)|\lesssim&r_0(\varphi)|\phi-\varphi| R^{-\frac32}(\phi,\varphi)\Big|F_1^\prime
\left(\frac{4r_0(\phi)r_0(\varphi)}{R(\phi,\varphi)}\right)\Big|\lesssim|\phi-\varphi|^{-1}.
\end{align*}
{Then applying once  again  the mean value theorem, we get for any   ${\varphi \in B_{\phi_1}^c(d)\cap B_{\phi_2}^c(d)}$ a value $\phi\in(\phi_1,\phi_2)$ such that 
\begin{align*}
|\mathscr{K}_2(\phi_1,\varphi)-\mathscr{K}_2(\phi_2,\varphi)|
&\lesssim {|\phi_1-\phi_2|}{|\phi-\varphi|^{-1}}\\
&\lesssim {|\phi_1-\phi_2|}{|\phi_1-\varphi|^{-1}}.
\end{align*}
Combining this estimate with \eqref{Est-Log} and \eqref{Es-eq1}, and using an interpolation argument we get for any $\varepsilon>0,$
\begin{align}\label{K1-2}
\nonumber |\mathscr{K}_2(\phi_1,\varphi)-\mathscr{K}_2(\phi_2,\varphi)|\lesssim &|\phi_1-\phi_2|^\beta {|\phi_1-\varphi|^{-\beta}}\\
&\times\left(C+\left|\ln\left(\frac{\sin \phi_1+\sin \varphi}{|\phi_1-\varphi|}\right)\right|+\left|\ln\left(\frac{\sin\phi_2+\sin\varphi}{|\phi_2-\varphi|}\right)\right|\right)^{1-\beta}\\
\nonumber\lesssim &|\phi_1-\phi_2|^\beta \left(1+\big|\ln(\sin \varphi)\big|^{1-\beta}\right)
\left({|\phi_1-\varphi|^{-\beta-\varepsilon}}\right).
\end{align}}
Hence, we find that
\begin{align}\label{bound-H22}
\nonumber \mathscr{H}_{2,2}(\phi_1,\phi_2)&\leq C|\phi_1-\phi_2|^\beta\bigintsss_0^\pi \left(1+\big|\ln(\sin \varphi)\big|^{1-\beta}\right)
\left({|\phi_1-\varphi|^{-\beta-\varepsilon}}\right)d\varphi\\
&\leq C|\phi_1-\phi_2|^\beta,
\end{align}
for any $\beta\in(0,1)$. Finally, putting together \eqref{bound-H21}-\eqref{bound-H22} we get
$$
\mathscr{H}_2(\phi_1,\phi_2)\leq C|\phi_1-\phi_2|^\beta,
$$
for any $\beta\in(0,1)$. Then, we find that $\nu_\Omega\in\mathscr{C}^\beta$, for any $\beta\in(0,1)$.

}

\medskip
\noindent
${\bf{(2)}}$ The function  $\phi\mapsto\Omega+\nu_\Omega(\phi)$ is continuous over the compact set $[0,\pi]$ then it reaches its minimum at some point $\phi_0\in[0,\pi].$  Thus from the definition of $\kappa$ in \eqref{kappa} we deduce that
$$
\kappa=\inf_{\phi\in(0,\pi)}\int_0^\pi H_1(\phi,\varphi)d\varphi=\int_0^\pi H_1(\phi_0,\varphi)d\varphi>0,
$$
which implies that
\begin{align*}
\forall \,\phi\in[0,\pi],\quad \nu_\Omega(\phi)\geq &\int_0^\pi H_1(\phi_0,\varphi)d\varphi-\Omega\geq\kappa-\Omega.
\end{align*}
Hence we infer that for any $\Omega\in(-\infty,\kappa)$ 
$$
\forall \,\phi\in[0,\pi],\quad \nu_\Omega(\phi)\geq \kappa-\Omega>0.
$$

\medskip
\noindent
${\bf{(3)}}$ The proof is  long and  technical and for clarity of presentation it will be divided into two steps. In the first one  we prove that $\nu_\Omega$ is $\mathscr{C}^1$ in the full closed interval $[0,\pi]$. This is mainly  based  on two principal  ingredients. The first one  is  an important algebraic    cancellation in the integrals  allowing us to get rid of the logarithmic singularity coming from the boundary and  the second one is  the boundary behavior of the hypergeometric functions allowing us to deal with the diagonal singularity lying inside the domain of integration. Notice that in order to apply Lebesgue theorem and recover  the continuity of the derivative up to the boundary  we use a rescaling argument.  This rescaling argument shows in addition  a surprising effect concerning   the derivative at the boundary points $\nu_\Omega^\prime(0)$ and $\nu_\Omega^\prime(\pi)$: they  are independent of the global structure of the profile $r_0$ and they depend only on the derivative $r_0^\prime(0)$. {This property allows us to compute $\nu_\Omega'(0)$ in the special case of $r_0(\phi)=\sin(\phi)$ by using the special geometry of the sphere and observe that this derivative vanishes. }
As to the second step, it is devoted to the proof of  $\nu_\Omega^\prime\in \mathscr{C}^\alpha(0,\pi)$ which is  involved and requires more refined analysis.

\medskip
\noindent
{\bf{$\bullet$ Step 1: $\nu_\Omega\in \mathscr{C}^1([0,\pi])$}}. The first step is to  check that $\nu_\Omega$ is $\mathscr{C}^1$ on $[0,\pi]$. 
Define
$$
\varrho(\phi,\varphi):=\frac{4 r_0(\phi)r_0(\varphi)}{R(\phi,\varphi)},
$$ then we can check that
\begin{align*}
\partial_\phi H_1(\phi,\varphi)=&\mathscr{K}_1(\phi,\varphi)\left(-\frac32 R^{-1}(\phi,\varphi) \partial_\phi R(\phi,\varphi)F_1(\varrho(\phi,\varphi))+F_1^\prime(\varrho(\phi,\varphi))\partial_\phi\varrho(\phi,\varrho)\right),\end{align*}
which implies after simple manipulations  that
\begin{align*}
\nonumber\partial_\phi H_1=&\mathscr{K}_1\left(-\frac32 \frac{\partial_\phi R}{R}+\frac34\partial_\phi\rho-\frac32\frac{\partial_\phi R}{R}\big(F_1(\rho)-1\big)+\big(F_1^\prime(\rho)-3/4\big)\big[\partial_\phi\rho+\partial_\varphi\rho\big]\right)\\
&-\mathscr{K}_1\big(F_1^\prime(\rho)-3/4\big)\partial_\varphi \rho.
\end{align*}
In addition using the identity 
\begin{align*}
\mathscr{K}_1\big(F_1^\prime(\rho)-3/4\big)\partial_\varphi\rho=&\mathscr{K}_1\partial_\varphi\big[F(\rho)-3/4\rho-1\big]\\
=&\partial_\varphi\left(\mathscr{K}_1\big[F(\rho)-3/4\rho-1\big]\right)-(\partial_\varphi\mathscr{K}_1)\big[F(\rho)-3/4\rho-1\big],
\end{align*}
we find
\begin{equation*}
\partial_\phi H_1=\varkappa_0+\varkappa_1\big(F_1(\varrho)-1\big)+\varkappa_2\big(F_1^\prime(\varrho)-3/4\big)-\partial_\varphi\left(\mathscr{K}_1\big[F(\rho)-3/4\rho-1\big]\right),
\end{equation*}
with
\begin{align}\label{ExpF1}
\varkappa_0:=&\mathscr{K}_1\left(-\frac32 \frac{\partial_\phi R}{R}+\frac34\partial_\phi\rho\right)-\frac34\rho\partial_\varphi \mathscr{K}_1,\\
\nonumber\varkappa_1:=&-\frac32\frac{\partial_\phi R}{R}\mathscr{K}_1+\partial_\varphi\mathscr{K}_1,\\
\nonumber\varkappa_2:=&\mathscr{K}_1\left(\partial_\phi\rho+\partial_\varphi\rho\right).
\end{align}
Notice that  $F_1(0)=1,\,F_1^\prime(0)=\frac34$. Assuming that the following functions are well--defined and using the boundary conditions then we can write
\begin{align}\label{Deriv1}
\nonumber \nu_\Omega^\prime(\phi)=&\int_0^\pi\Big(\varkappa_0(\phi,\varphi)+\varkappa_1(\phi,\varphi)\big[F_1(\varrho(\phi,\varphi))-1\big]
+\varkappa_2(\phi,\varphi)\big[F_1^\prime(\varrho(\phi,\varphi))-3/4\big]\Big)d\varphi\\
:=&\zeta_1(\phi)+\zeta_2(\phi)+\zeta_3(\phi),
\end{align}
with
$$
\zeta_1(\phi):=\int_0^\pi\varkappa_0(\phi,\varphi)d\varphi,\quad \zeta_2(\phi):=\int_0^\pi\varkappa_1(\phi,\varphi)\big[F_1(\varrho(\phi,\varphi))-1\big]d\varphi
$$
and
$$
\zeta_3(\phi)=\int_0^{\pi} \varkappa_2(\phi,\varphi)\big[F_1^\prime(\varrho(\phi,\varphi))-3/4\big]d\varphi.
$$
Direct computations show that
\begin{align*}
\partial_\phi\rho(\phi,\varphi)=&\frac{4 r_0(\varphi)r_0^\prime(\phi)}{R^2(\phi,\varphi)}\Big(R(\phi,\varphi)-2r_0(\phi)\big(r_0(\varphi)+r_0(\phi)\big)\Big)-\frac{8 r_0(\phi)r_0(\varphi)}{R^2(\phi,\varphi)}\sin\phi(\cos\varphi-\cos\phi).
\end{align*}
According to \eqref{partialR} and using some  cancellation, it implies that

\begin{align}\label{CancelF1}
\nonumber \varkappa_0(\phi,\varphi)=&-3r_0^\prime(\phi)\frac{r_0(\phi)\mathscr{K}_1(\phi,\varphi)}{R(\phi,\varphi)}\left( 1+2\frac{r_0(\varphi)(r_0(\phi)+r_0(\varphi))}{R(\phi,\varphi)}\right)\\
\nonumber &+6\frac{\mathscr{K}_1(\phi,\varphi)}{R^2(\phi,\varphi)}\sin(\phi)\, r_0(\phi)\,r_0(\varphi)\big(\cos\phi-\cos\varphi\big)\\
&+3\frac{\mathscr{K}_1(\phi,\varphi)}{R(\phi,\varphi)}\sin(\phi)\big(\cos\phi-\cos\varphi\big)
-3\partial_\varphi \mathscr{K}_1\frac{r_0(\phi)r_0(\varphi)}{R(\phi,\varphi)}.
\end{align}
We point out that this simplification is crucial and allows to get rid of the logarithmic singularity.

Now we shall start with the regularity of the function  
$$
\zeta_1:\phi\in(0,\pi)\mapsto \int_0^{\pi}\varkappa_0(\phi,\varphi)d\varphi,
$$
and prove first that it is 
continuous in $[0,\pi]$. It is obvious from \eqref{CancelF1} that $\varkappa_0$ is 
$\mathscr{C}^1$ over any compact set contained in  $(0,\pi)\times [0,\pi]$ and therefore  $\zeta_1$ is $\mathscr{C}^1$ over any compact set contained in $(0,\pi).$ 
Thus   it remains to check that this function  is  continuous at the points $0$ and $\pi$. The proofs for both cases  are quite similar and we shall only
check the continuity  at the origin. For this purpose   it is enough to check that $\zeta_1$ admits a limit at zero. Before that let us check that $\zeta_1$ 
is bounded in $(0,\pi).$ From the definition of $R$ stated in \eqref{RefR}  and using elementary inequalities it is easy to verify  the following estimates:
for any $(\phi,\varphi)\in(0,\pi)^2$
\begin{align*}
\frac{r_0(\phi)(r_0(\phi)+r_0(\varphi))}{R(\phi,\varphi)}\le&1,\\
\frac{r_0(\phi)r_0(\varphi)}{R(\phi,\varphi)}\leq&{\frac{1}{2}},\\
 r_0(\phi)r_0(\varphi)\big|\cos\phi-\cos\varphi\big|\le& R(\phi,\varphi).
\end{align*}
In addition, the assumption {\bf{(H2)}} implies that
$$
\sup_{\phi,\varphi\in(0,\pi)}\mathscr{K}_1(\phi,\varphi)<\infty.
$$
Thus we find according to \eqref{Tixa} and {\bf{(H2)}} 
\begin{align}\label{TranX1}
\forall(\phi,\varphi)\in(0,\pi)^2,\,|\varkappa_0(\phi,\varphi)|\lesssim&\frac{\sin(\phi)}{R(\phi,\varphi)}+|\partial_\varphi \mathscr{K}_1(\phi,\varphi)|\frac{r_0(\phi)r_0(\varphi)}{R(\phi,\varphi)}\lesssim\frac{\sin(\phi)}{R(\phi,\varphi)}\cdot
\end{align}
Hence we deduce that
\begin{align*}
\forall \phi\in(0,\pi),\quad |\zeta_1(\phi)|\lesssim& \int_0^\pi \frac{\sin(\phi)}{R(\phi,\varphi)} d\varphi\lesssim \int_0^{\frac\pi2} \frac{\sin\phi}{(\sin\phi+\sin\varphi)^2} d\varphi.
\end{align*}
Making the change of variables $\sin\varphi=x$ we get
\begin{align*}
 \int_0^{\frac\pi2} \frac{\sin\phi}{(\sin\phi+\sin\varphi)^2} d\varphi=&\sin\phi\int_0^1\frac{1}{(\sin\phi+x)^2}\frac{dx}{\sqrt{1-x^2}}\\
 \lesssim&\sin\phi\int_0^{\frac12}\frac{1}{(\sin\phi+x)^2}{dx}+\sin\phi
 \lesssim1.
\end{align*}
Thus 
\begin{equation}\label{Tri01}
\sup_{\phi\in(0,\pi)}|\zeta_1(\phi)|<\infty.\end{equation}
 Let us now prove that $\zeta_1$ admits a limit at the origin and compute its value. For this goal, take  $0<\delta<<1$ small enough and write
$$
\zeta_1(\phi)=\int_0^{\delta}\varkappa_0(\phi,\varphi)d\varphi+\int_{\delta}^\pi\varkappa_0(\phi,\varphi)d\varphi.
$$
The assumption ${\textbf{(H2) }}$ combined with standard trigonometric formula  allow to get the estimate 
\begin{align}\label{Lowbound}
 \nonumber R(\phi,\varphi)\gtrsim& (\sin\phi+\sin\varphi)^2+(\cos\phi-\cos\varphi)^2=2\big(1-\cos(\phi+\varphi)\big)\\
&\gtrsim1-\cos(\phi+\varphi).
\end{align}
From this we infer that
\begin{equation}\label{TRamm1}
\forall \,\phi\in[0,\pi/2],\forall \varphi\in(\delta,\pi),\quad R(\phi,\varphi)\gtrsim1-\cos\delta.
\end{equation}
Thus we get from \eqref{TranX1}
{\begin{align*}
\left|\int_{\delta}^\pi\varkappa_0(\phi,\varphi)d\varphi\right|\lesssim&  \int_\delta^\pi
 \frac{\phi}{(1-\cos\delta)}d\varphi\lesssim\phi\delta^{-2}.
 \end{align*}}
 This implies that for given small parameter $\delta$ one has
 $$
 \lim_{\phi\to0}\int_{\delta}^\pi\varkappa_0(\phi,\varphi)d\varphi=0.
 $$
 Therefore 
 $$
 \lim_{\phi\to0}\zeta_1(\phi)=\lim_{\phi\to0}\int_0^{\delta}\varkappa_0(\phi,\varphi)d\varphi.
 $$
 Making the change of variables $\varphi=\phi \theta$ we get
 $$
 \int_0^{\delta}\varkappa_0(\phi,\varphi)d\varphi=\int_0^{\frac{\delta}{\phi}}\phi\varkappa_0(\phi,\phi \theta)\,d\theta.
 $$
 From \eqref{TranX1} and {\bf{(H2)}} one may write
 $$
 \forall(\phi,\varphi)\in(0,\delta)^2,\quad|\varkappa_0(\phi,\varphi)|\lesssim \frac{\phi}{(\phi+\varphi)^2},
 $$
 which yields after simplification to the uniform bound on $\phi$,
 $$
\forall \theta\in[0,\delta/\phi],\quad  \phi\varkappa_0(\phi,\phi \theta)\lesssim \frac{1}{(1+\theta)^2}\cdot
 $$
 This gives a domination  which is integrable over $(0,+\infty)$. In order to apply classical dominated Lebesgue theorem, 
 it remains to check the {convergence almost everywhere in $\theta$}  as $\phi$ goes to zero. This can be done through  the first--order Taylor expansion around zero.
 First one has the expansion
 $$
 r_0(\phi\theta)=c_0\phi \theta+\phi\theta\epsilon(\phi \theta);\quad R(\phi, \phi \theta)=c_0^2\phi^2\big(1+\theta+\theta\epsilon(\phi\theta)\big)^2,
 $$
with $c_0=r_0^\prime(0)$ and $\displaystyle{\lim_{x\to0}\epsilon(x)=0.}$ Thus, from the definitions \eqref{RefR} and  \eqref{K1} it is straightforward that
\begin{equation}\label{TimX1}
4\pi \lim_{\phi\to 0}\mathscr{K}_1(\phi,\phi\theta)=\frac{{{c_0^{-1}}}\theta^3}{(1+\theta)^3},\quad \lim_{\phi\to 0}\frac{\phi r_0(\phi)}{R(\phi,\phi\theta)}=\frac{c_0^{-1}}{(1+\theta)^2}\cdot
 \end{equation}
 Hence
\begin{equation}\label{Tahya0}
4\pi \lim_{\phi\to 0} r_0^\prime(\phi)\frac{\phi r_0(\phi)\mathscr{K}_1(\phi,\phi \theta)}{R(\phi,\phi \theta)}
\left( 1+2\frac{r_0(\phi\theta)(r_0(\phi)+r_0(\phi \theta))}{R(\phi,\phi \theta)}\right)
=\frac{{{c_0^{-1}}}\theta^3(1+3\theta)}{(1+\theta)^6}\cdot
\end{equation}
Similarly we get
\begin{equation}\label{Est1}
 4\pi\lim_{\phi\to 0}\frac{\mathscr{K}_1(\phi,\phi \theta)}{R^2(\phi,\phi\theta)}\sin\phi\, r_0(\phi)r_0(\phi\theta)\big(\cos\phi-\cos(\phi \theta)\big)=0,
 \end{equation}
and
 \begin{equation}\label{Est2}
  4\pi\lim_{\phi\to 0}\frac{\mathscr{K}_1(\phi,\phi \theta)}{R(\phi,\phi \theta)}\sin(\phi)\big(\cos\phi-\cos(\phi \theta)\big)=0.
 \end{equation}
 Standard computations yield
 \begin{align}\label{Tahya1}
 \nonumber 4\pi\partial_\varphi \mathscr{K}_1(\phi,\varphi)=&
 - 3 R^{-\frac52}(\phi,\varphi)\sin(\varphi)\, r_0^2(\varphi)\Big(r_0^\prime(\varphi)(r_0(\phi)+r_0(\varphi))+\sin(\varphi)\,(\cos\phi-\cos\varphi)\Big)\\
&+ \frac{\cos(\varphi )\,r_0^2(\varphi)+2\sin(\varphi) r_0(\varphi)\, r_0^\prime(\varphi)}{R^\frac32(\phi,\varphi)}.
\end{align}
Thus 
 \begin{align}\label{Tramm1}
 4\pi \lim_{\phi\to0}\phi\partial_\varphi \mathscr{K}_1(\phi,\phi\theta)
 =&c_0^{-1}\left(-3\frac{\theta^3}{(1+\theta)^4}+{\frac{3\theta^2}{(1+\theta)^3}}\right)={3 c_0^{-1}\frac{\theta ^2}{(1+\theta)^4}}.
\end{align}
Therefore
\begin{equation}\label{Tahya2}
4\pi\lim_{\phi\to0}\phi\partial_\varphi \mathscr{K}_1(\phi,\phi\theta)\frac{r_0(\phi)r_0({\phi\theta})}{R(\phi,\phi\theta)}=
{3c_0^{-1}\frac{\theta ^3}{(1+\theta)^6}}\cdot
\end{equation}
Plugging \eqref{Tahya0}, \eqref{Est1}, \eqref{Est2} and \eqref{Tahya2} into \eqref{CancelF1}
\begin{align*}
4\pi\lim_{\phi\to0}\phi\varkappa_0(\phi,\phi\theta)=&-\frac{3 c_0^{-1}\theta^3}{(1+\theta)^6}\Big({4}+3\theta\Big).
\end{align*}
Using Lebesgue dominated theorem we deduce that
$$
4\pi\lim_{\phi\to0}\int_0^{\frac{\delta}{\phi}}\phi\varkappa_0(\phi,\phi\theta)d\theta=-3 c_0^{-1}\bigintsss_0^{+\infty}\frac{{4}\theta^3+3\theta^4}{(1+\theta)^6}d\theta. 
$$
Computing the integrals we finally get 
\begin{equation}\label{Tahya12}
4\pi\lim_{\phi\to0}\zeta_1(\phi)=-{\frac{4}{5}} c_0^{-1}.
\end{equation}
Let us now move to  the regularity of the function $\zeta_2$ defined in \eqref{Deriv1} through
$$
\phi\in(0,\pi),\quad \zeta_2(\phi)=\int_0^\pi\varkappa_1(\phi,\varphi)\big(F_1(\rho(\phi,\varphi))-1\big)d\varphi,
$$
where $\varkappa_1$ is defined in \eqref{ExpF1}. From direct computations using  $|\partial_\phi R|\lesssim R^\frac12,$ the boundedness of $\mathscr{K}_1$, the assumption  {\bf{(H2)}}  and \eqref{Tahya1} one can check that
\begin{align*}
|\varkappa_1(\phi,\varphi)|\lesssim& \frac{|\partial_\phi R(\phi,\varphi)|\mathscr{K}_1(\phi,\varphi)}{R(\phi,\varphi)}+|\partial_\varphi\mathscr{K}_1(\phi,\varphi)|\\
\lesssim& {R^{-\frac12}(\phi,\varphi)}.
\end{align*}
Using Proposition \ref{Prop-behav}  we get
\begin{align}\label{ZNX1}
\nonumber|F_1(\varrho(\phi,\varphi))-1|\lesssim& \varrho(\phi,\varphi)\left(1+|\ln(1-\varrho(\phi,\varphi))|\right)\\
\lesssim &\frac{r_0(\phi)r_0(\varphi)}{R(\phi,\varphi)}
\left(1+\ln\left(\frac{(r_0(\phi)+r_0(\varphi))^2+(\cos(\phi)-\cos(\varphi))^2}{(r_0(\phi)-r_0(\varphi))^2+(\cos(\phi)-\cos(\varphi))^2}\right)\right).
\end{align}
Thus
\begin{align*}
|\varkappa_1(\phi,\varphi)[F_1(\varrho(\phi,\varphi))-1]|\lesssim& \frac{r_0(\phi)r_0(\varphi)}{R^{\frac32}(\phi,\varphi)}
\left(1+\ln\left(\frac{(r_0(\phi)+r_0(\varphi))^2+(\cos(\phi)-\cos(\varphi))^2}{(r_0(\phi)-r_0(\varphi))^2+(\cos(\phi)-\cos(\varphi))^2}\right)\right).
\end{align*}
Hence  using the arc-chord property  \eqref{Chord}
we  find 
\begin{equation}\label{Tita1}
|\varkappa_1(\phi,\varphi)[F_1(\varrho(\phi,\varphi))-1]|\lesssim\frac{r_0(\phi)r_0(\varphi)}{R^{\frac32}(\phi,\varphi)}
\left(C+\ln\left(\frac{R(\phi,\varphi)}{|\phi-\varphi|^2}\right)\right).
\end{equation}
for some constant $C>0$.
In addition, using \eqref{Lowbound} we get   
\begin{equation}\label{Rbound}\inf_{\phi\in[0,{\frac\pi2}],\\
\atop\varphi\in[{\frac\pi2},\pi]}R(\phi,\varphi)> 0,
\end{equation}
which leads to 
$$
\forall\,\phi \in[0,{\pi/2}],\,\varphi\in[{\pi/2},\pi],\quad |\varkappa_1(\phi,\varphi)[F_1(\varrho(\phi,\varphi))-1]|\lesssim 1+|\ln|\phi-\varphi||.
$$
This implies that
\begin{equation}\label{Tri1}
\sup_{\phi\in[0,\pi/2]}\int_{\frac\pi2}^\pi|\varkappa_1(\phi,\varphi)[F_1(\varrho(\phi,\varphi))-1]|d\varphi\lesssim 1+\sup_{\phi\in[0,\pi/2]}\int_0^\pi|\ln|\phi-\varphi||d\varphi<\infty.
\end{equation}

Now in the region $\phi,\varphi\in[0,\pi/2], $ we use the estimate {{\bf{(H2)}}}  leading to
$$
(\phi+\varphi)^2\lesssim R(\phi,\varphi)\lesssim
{(\phi+\varphi)^2}.
$$
Plugging this  into \eqref{Tita1} we find
\begin{align*}
\sup_{\phi\in[0,\pi/2]}\bigintsss_0^{\frac\pi2}|\varkappa_1(\phi,\varphi)[F_1(\varrho(\phi,\varphi))-1]|d\varphi\lesssim 1+\sup_{\phi\in[0,\pi/2]}\bigintsss_0^{\frac\pi2} \frac{\phi\varphi}{(\phi+\varphi)^3}\left|\ln\left(\frac{\phi+\varphi}{|\phi-\varphi|}\right)\right|d\varphi.\end{align*}
Making the change of variables $\varphi=\phi \theta$ we obtain
\begin{align*}
\bigintsss_0^{\frac\pi2} \frac{\phi\varphi}{(\phi+\varphi)^3}\left|\ln\left(\frac{\phi+\varphi}{|\phi-\varphi|}\right)\right|d\varphi=&\bigintsss_0^{\frac{\pi}{2\phi}} \frac{\theta}{(1+\theta)^3}\ln\Big(\frac{1+\theta}{|1-\theta|}\Big)d\theta\\
\leq&\bigintsss_0^{+\infty} \frac{\theta}{(1+\theta)^3}\ln\Big(\frac{1+\theta}{|1-\theta|}\Big)d\theta<\infty,\end{align*}
which implies that
\begin{equation*}
\sup_{\phi\in[0,\pi/2]}\int_0^{\frac\pi2}|\varkappa_1(\phi,\varphi)[F_1(\varrho(\phi,\varphi))-1]|d\varphi<\infty.
\end{equation*}
Therefore we obtain by virtue of \eqref{Tri1}
\begin{equation*}
\sup_{\phi\in[0,\pi/2]}\int_0^{\pi}|\varkappa_1(\phi,\varphi)[F_1(\varrho(\phi,\varphi))-1]|d\varphi<\infty.
\end{equation*}
By symmetry we get similar estimate for $\phi\in[\frac\pi2,\pi]$ and hence
\begin{equation}\label{Tri2}
\sup_{\phi\in(0,\pi)}|\zeta_2(\phi)|<\infty.\end{equation}
Let us now calculate the limit when $\phi$ goes to $0$ of $\zeta_2$. We shall proceed in a similar way to $\zeta_1$. Let $0<\delta\ll1$ be enough small, then using \eqref{Tita1} combined with \eqref{TRamm1} we obtain
$$
\lim_{\phi\to0}\int_\delta^\pi|\varkappa_1(\phi,\varphi)||F_1(\rho(\phi,\varphi))-1|d\varphi=0.
$$
Hence
$$
\lim_{\phi\to0}\zeta_2(\phi)=\lim_{\phi\to0}\int_0^\delta\varkappa_1(\phi,\varphi)\big(F_1(\rho(\phi,\varphi))-1\big)d\varphi.
$$
Now we make the change of variables $\varphi=\phi \theta$ and then
$$
\lim_{\phi\to0}\zeta_2(\phi)=\lim_{\phi\to0}\int_0^{\frac\delta\phi}\phi\varkappa_1(\phi,\phi\theta)\big(F_1(\rho(\phi,\phi\theta))-1\big){d\theta}.
$$
According to \eqref{ExpF1} one has
$$
\phi\varkappa_1(\phi,\phi\theta)=-\frac32\frac{\phi\partial_\phi R(\phi,\phi\theta)}{R(\phi,\phi\theta)}\mathscr{K}_1(\phi,\phi\theta)+\phi\partial_\varphi\mathscr{K}_1(\phi,\phi\theta).
$$
From the differentiating of the expression of $R$ stated in \eqref{RefR} we get
\begin{align*}
\frac{\phi\partial_\phi R(\phi,\phi\theta)}{R(\phi,\phi\theta)}=&2\frac{\phi r_0^\prime(\phi)(r_0(\phi)+r_0(\phi\theta))+
	\phi\sin\phi(\cos(\phi\theta)-\cos\phi)}{(r_0(\phi)+
	r_0(\phi\theta))^2
	+(\cos\phi-\cos(\phi\theta))^2}\cdot
\end{align*}
Taking  Taylor expansion to first order we deduce the pointwise convergence,
\begin{align*}
\lim_{\phi\to0}\frac{\phi \partial_\phi R(\phi,\phi\theta)}{R(\phi,\phi\theta)}=&\frac{2}{1+\theta}\cdot
\end{align*}
Combined with \eqref{TimX1} it implies that
$$
4\pi\lim_{\phi\to0}-\frac32\frac{\phi\partial_\phi R(\phi,\phi\theta)}{R(\phi,\phi\theta)}\mathscr{K}_1(\phi,\phi\theta)=-\frac{3{c_0^{-1}}\theta^3}{(1+\theta)^4}\cdot
$$
Plugging \eqref{Tramm1} and   the preceding estimates  into the expression of $\varkappa_1$ given by \eqref{ExpF1} we find
\begin{align*}
4\pi\lim_{\phi\to0}\phi\varkappa_1(\phi,\phi\theta)=&-\frac{3{c_0^{-1}}\theta^3}{(1+\theta)^4}+{3c_0^{-1}\frac{\theta^2}{(1+\theta)^4}}\\
=&{3c_0^{-1}\frac{\theta^2(1-\theta)}{(1+\theta)^4}}\cdot
\end{align*}
From the result
$$
\lim_{\phi\to0} \frac{r_0(\phi)r_0(\phi\theta)}{R(\phi,\phi\theta)}=\frac{\theta}{(1+\theta)^2},
$$
we deduce the point-wise convergence
$$
\lim_{\phi\to0}F_1(\rho(\phi,\phi\theta))=F_1\left(\frac{4\theta}{(1+\theta)^2}\right).
$$
Consequently,
$$
4\pi \lim_{\phi\to0}\phi\varkappa_1(\phi,\phi\theta)\Big[F_1(\rho(\phi,\phi\theta))-1\Big]=
{3c_0^{-1}\frac{\theta^2(1-\theta)}{(1+\theta)^4}}
\left(F_1\left(\frac{4\theta}{(1+\theta)^2}\right)-1\right)\cdot
$$
Therefore, 
\begin{equation}\label{Lund1}
4\pi \lim_{\phi\to0}\zeta_2(\phi)=3c_0^{-1}\mathlarger{\bigintsss_0^{+\infty}
{\frac{\theta^2(1-\theta)}{(1+\theta)^4}}
\left(F_1\left(\frac{4\theta}{(1+\theta)^2}\right)-1\right)}
d\theta.
\end{equation}
{Note that we can apply the dominated convergence theorem in the previous integral since
$$
|\phi\varkappa_1(\phi,\phi\theta)[F_1(\rho(\phi,\phi\theta))-1]|\leq C\frac{\theta}{(1+\theta)^3}\left(1+\ln\left(\frac{1+\theta}{|1-\theta|}\right)\right),
$$
which is integrable.
}
Next, we shall implement a similar study for $\zeta_3$ defined in \eqref{Deriv1}.
Straightforward computations yield
\begin{equation}\label{Sing0}
\partial_\phi \varrho(\phi,\varphi)+\partial_\varphi \varrho(\phi,\varphi)=\varrho_1(\phi,\varphi)+\varrho_2(\phi,\varphi),
\end{equation}
with 
\begin{equation}\label{Sing1}
\varrho_1(\phi,\varphi):=4\frac{r_0^2(\varphi)-r_0^2(\phi)}{R^2(\phi,\varphi)}r_0^\prime(\phi)\big(r_0(\varphi)-r_0(\phi)\big),
\end{equation}
and
\begin{align}\label{Sing2}
\nonumber\varrho_2(\phi,\varphi):=&4\frac{r_0^2(\varphi)-r_0^2(\phi)}{R^2(\phi,\varphi)}r_0(\phi)\Big(r_0^\prime(\phi)-r_0^\prime(\varphi)\Big)\\
\nonumber &+8\frac{r_0(\varphi)r_0(\phi)}{R^2(\phi,\varphi)}\big(\cos\phi-\cos\varphi)\big(\sin\phi-\sin\varphi\big)\\
&+4\frac{(\cos\phi-\cos\varphi)^2}{R^2(\phi,\varphi)}\Big(r_0(\varphi)r_0^\prime(\phi)+r_0(\phi)r_0^\prime(\varphi)\Big).\end{align}
Since $r_0^\prime$ is Lipschitz then  using the mean value theorem we get
\begin{align}\label{Sing3X}
\forall\,\phi,\varphi\in(0,\pi),\quad|\varrho_1(\phi,\varphi)|+|\varrho_2(\phi,\varphi)|\lesssim\frac{(\phi-\varphi)^2}{R^{\frac32}(\phi,\varphi)}\cdot
\end{align}
From Proposition \ref{Prop-behav} combined with \eqref{Chord}  we get
\begin{align}\label{Sing4}
|F_1^\prime(\varrho(\phi,\varphi))-3/4|=&\frac34|F(5/2,5/2;4;\rho(\phi,\varphi))-1|\nonumber\\
\lesssim& \frac{\varrho(\phi,\varphi) R(\phi,\varphi)}{(r_0(\phi)-r_0(\varphi)^2+(\cos \phi-\cos\varphi)^2}\nonumber\\
\lesssim& \frac{r_0(\phi) r_0(\varphi)}{(\phi-\varphi)^2}.
\end{align}
In addition
\begin{align*}
|\varkappa_2(\phi,\varphi))|=& |\mathscr{K}_1(\phi,\varphi)||\varrho_1(\phi,\varphi)+\varrho_2(\phi,\varphi)|\\
\lesssim&\frac{\sin\varphi\, r_0^2(\varphi)}{R^{\frac32}(\phi,\varphi)}\frac{(\phi-\varphi)^2}{R^{\frac32}(\phi,\varphi)}\\
\lesssim&\frac{\sin\varphi\, (\varphi-\phi)^2}{R^{2}(\phi,\varphi)}\cdot
\end{align*}
Consequently, we obtain in view of {\bf{(H2)}}
\begin{align}\label{Hila1}
|\varkappa_2(\phi,\varphi)[F_1^\prime(\varrho(\phi,\varphi))-3/4]|\lesssim&\frac{\sin(\varphi)r_0(\varphi)r_0(\phi)}{R^{2}(\phi,\varphi)}\lesssim\frac{ \sin\phi}{R(\phi,\varphi)}\cdot
\end{align}
As before, we can assume without  any loss of generality  that  $\phi\in[0,\,\pi/2]$, then by {\bf{(H2)}}  
\begin{align*}
\int_0^\pi|\varkappa_2(\phi,\varphi)[F_1^\prime(\varrho(\phi,\varphi))-3/4]|d\varphi\lesssim&\int_0^{\frac\pi2}\frac{ \sin\phi}{(\sin\phi+\sin\varphi)^2}d\varphi\\
\lesssim& \int_0^{\frac\pi2}\frac{\phi}{\big(\phi+\varphi\big)^2}d\varphi.
\end{align*}
By the change of variables $\varphi=\phi\theta$ we get
\begin{align*}
\int_0^{\frac\pi2}\frac{\phi}{\big(\phi+\varphi\big)^2}d\varphi=&\int_0^{\frac{\pi}{2\phi}}\frac{1}{\big(1+\theta\big)^2}d\theta\leq\int_0^{+\infty}\frac{1}{\big(1+\theta\big)^2}d\theta<\infty.
\end{align*}
Therefore
\begin{equation*}
\sup_{\phi\in[0,\pi/2]}\int_0^\pi|\varkappa_2(\phi,\varphi)[F_1^\prime(\varrho(\phi,\varphi))-3/4]|d\varphi<\infty.
\end{equation*}
Consequently
\begin{equation}\label{Tita0}
\sup_{\phi\in[0,\pi]}|\zeta_3(\phi)|<\infty.
\end{equation}
Now, we shall calculate the limit of $\zeta_3$ at the origin. Let $0<\delta\ll1$ be enough small, then using \eqref{Hila1} combined with \eqref{TRamm1} and {\bf{(H2)}} we obtain
$$
\lim_{\phi\to0}\int_\delta^\pi\varkappa_2(\phi,\varphi) [F_1^\prime(\varrho(\phi,\varphi))-3/4]|d\varphi=0.
$$
It follows that
$$
\lim_{\phi\to0}\zeta_3(\phi)=\lim_{\phi\to0}\int_0^\delta\varkappa_2(\phi,\varphi) [F_1^\prime(\varrho(\phi,\varphi))-3/4]|d\varphi.
$$
Making  the change of variables $\varphi=\phi \theta$ yields
$$
\lim_{\phi\to0}\zeta_3(\phi)=\lim_{\phi\to0}\int_0^{\frac\delta\phi}\phi\varkappa_2(\phi,\phi\theta)\big[F_1^\prime(\rho(\phi,\phi\theta))-3/4\big]{d\theta}.
$$
Using Taylor expansion to first order one in \eqref{Sing1} and \eqref{Sing2} we can check that
$$
\lim_{\phi\to0}\phi\varrho_1(\phi,\phi\theta)=4\frac{(\theta-1)^2}{(1+\theta)^3},\quad 
\lim_{\phi\to0}\phi\varrho_2(\phi,\phi\theta)=0.
$$
Hence  we get in view of the definition of $\varkappa_2$  and \eqref{TimX1}   the point-wise limit
\begin{align*}
4\pi\lim_{\phi\to0}\phi\varkappa_2(\phi,\phi\theta)=&4\pi\lim_{\phi\to0}\phi\mathscr{K}_1(\phi,\phi\theta)\Big(\varrho_1(\phi,\phi\theta)+\varrho_2(\phi,\phi\theta)\Big)=4{{c_0^{-1}}}\frac{\theta^3(\theta-1)^2}{(1+\theta)^6}\cdot
\end{align*}
It follows that
\begin{align*}
4\pi\lim_{\phi\to0}\phi\varkappa_2(\phi,\phi\theta)\Big(F_1^\prime(\rho(\phi,\phi\theta))-3/4\Big)&=4{{c_0^{-1}}}\frac{\theta^3(\theta-1)^2}{(1+\theta)^6}\left(F_1^\prime\Big(\frac{4\theta}{(1+\theta)^2}\Big)-3/4\right).
\end{align*}
{ Moreover, by \eqref{Hila1} 
$$|\phi\varkappa_2(\phi,\phi\theta)\Big(F_1^\prime(\rho(\phi,\phi\theta))-3/4\Big)|\le C\frac{\theta^2}{(1+\theta)^4},
$$ so we can apply dominated convergence theorem obtaining}

\begin{equation}\label{Lund2}
4\pi\lim_{\phi\to0}\zeta_3(\phi)=4{{c_0^{-1}}}\bigintsss_0^{+\infty}\frac{\theta^3(\theta-1)^2}{(1+\theta)^6}\left(F_1^\prime\Big(\frac{4\theta}{(1+\theta)^2}\Big)-3/4\right) d\theta.
\end{equation}
Putting together \eqref{Deriv1}, \eqref{Tahya12}, \eqref{Lund1} and \eqref{Lund2} we find 
\begin{align*}
4\pi\lim_{\phi\to0}\nu_\Omega^\prime(\phi)=&{-{\frac{4}{5}}c_0^{-1}}+4{{c_0^{-1}}}\bigintsss_0^{+\infty}\frac{\theta^3(\theta-1)^2}{(1+\theta)^6}\left(F_1^\prime\Big(\frac{4\theta}{(1+\theta)^2}\Big)-3/4\right) d\theta\\
&+{3{{c_0^{-1}}} {\bigintsss_0^{+\infty}\frac{\theta^2(1-\theta)}{(1+\theta)^4}}
\left(F_1\left(\frac{4\theta}{(1+\theta)^2}\right)-1\right)
d\theta:=\eta c_0^{-1}}.
\end{align*}
Notice that the real  number $\eta$ is well-defined since all the integrals converge.
This shows  the existence of the derivative of $\nu_\Omega$ at the origin. 
It is important to emphasize that  number $\eta$  is independent of the  profile $r_0$ and  we claim that the number  $\eta$ is zero. It is slightly difficult to check this result directly  from the integral representation of $
\eta$, however we shall check it in a different way by calculating its value for the unit ball 
$$
\big\{ (r e^{i\theta},z),\,r^2+z^2\leq1, \theta\in\R\big\}
$$
whose boundary can be  parametrized by $(\phi,\theta)\mapsto (r_0(\phi) e^{i\theta},\cos \phi)$ with $r_0(\phi)=\sin\phi.$ 
 Now  according to the identity \eqref{Ident-2} one has 
\begin{equation*}
\int_0^\pi H_1(\phi,\varphi)d\varphi=\frac{1}{r_0(\phi)}\partial_r \psi_0(re^{i\theta},\cos(\phi))\left|_{r=r_0(\phi)}\right..
\end{equation*}
However it is known \cite{Kellog} that   the stream function $\psi_0$ is radial and  quadratic inside the domain taking the form
$$
 0\leq r\leq \sin \phi,\quad \psi_0(re^{i\theta},\cos(\phi))=\frac{1}{6}\big( r^2+\cos^2\phi\big).
$$
Consequently, with this  special geometry the function $\displaystyle{\nu_\Omega}$  is constant and therefore
$$
\nu_\Omega^\prime(0)=\nu_\Omega^\prime(\pi)=0.
$$

\medskip
\noindent
{\bf{$\bullet$ Step 2: $\nu_\Omega^\prime\in \mathscr{C}^\alpha(0,\pi)$}}.
{ We shall prove that $\nu_\Omega^\prime$ is $\mathscr{C}^\alpha(0,\pi)$ and for this purpose we  start with the first term in \eqref{Deriv1}, i.e., $\zeta_1$. According to \eqref{CancelF1}  it can be split into several terms and to fix the ideas let us describe how to proceed with   the first term given by 
$$
\phi\mapsto 4\pi r_0^\prime(\phi) r_0(\phi)\bigintsss_0^\pi \frac{\mathscr{K}_1(\phi,\varphi)}{R(\phi,\varphi)}d\varphi=r_0'(\phi)r_0(\phi)\bigintsss_0^\pi \frac{\sin(\varphi) r_0^2(\varphi)}{R^{\frac52}(\phi,\varphi)}d\varphi,
$$
and check that it belongs to $\mathscr{C}^\alpha(0,\pi).$ The remaining terms of $\zeta_1$ can be treated in a similar way and to alleviate the discussion we leave them to the reader.

From the assumptions {\bf(H)} on $r_0$ we have $r_0^\prime,\,\phi\mapsto \frac{r_0(\phi)}{\sin\phi}\in \mathscr{C}^\alpha(0,\pi)$ , then using the fact that $\mathscr{C}^\alpha$ is an algebra, it suffices to verify that 
$$
\phi\mapsto \bigintsss_0^\pi \frac{\sin(\phi)\sin(\varphi) r_0^2(\varphi)}{R^{\frac52}(\phi,\varphi)}d\varphi\in \mathscr{C}^\alpha(0,\pi).
$$
This function is locally  $\mathscr{C}^1$ in $(0,\pi)$ and so the problem reduces to check the regularity close to the boundary $\{0,\pi\}$. By symmetry it suffices to check the regularity near the origin.
Decompose the integral as follows 
$$
 \bigintsss_0^\pi \frac{\sin(\phi)\sin(\varphi) r_0^2(\varphi)}{R^{\frac52}(\phi,\varphi)}d\varphi= \bigintsss_0^{\frac{\pi}{2}} \frac{\sin(\phi)\sin(\varphi) r_0^2(\varphi)}{R^{\frac52}(\phi,\varphi)}d\varphi+ \bigintsss_{\frac{\pi}{2}}^\pi \frac{\sin(\phi)\sin(\varphi) r_0^2(\varphi)}{R^{\frac52}(\phi,\varphi)}d\varphi.
$$
Since we are considering $\phi\in(0,\pi/2)$, it is easy to check that the last integral term defines a $\mathscr{C}^1$ function in $[0,\pi/2]$. {Since $\sin(\phi)/\phi$ is $\mathscr{C}^\alpha(0,\pi/2)$}, then the problem amounts  to checking that the
 function 
$$
\zeta_{1,1}:\phi\mapsto \bigintsss_0^{\frac{\pi}{2}}\frac{\phi\sin(\varphi) r_0^2(\varphi)}{R^{\frac52}(\phi,\varphi)}d\varphi,
$$ is $\mathscr{C}^\alpha$ close to zero.
Making the change of variables $\varphi=\phi \theta$ we get
\begin{align*}
\zeta_{1,1}(\phi)= \bigintsss_0^{\frac{\pi}{2\phi}}\frac{\phi^2\sin(\phi \theta) r_0^2(\phi \theta)}{R^{\frac52}(\phi,\phi\theta)}d\theta= \bigintss_0^{\frac{\pi}{2\phi}}\frac{\frac{\sin(\phi \theta)}{\phi} \left(\frac{r_0(\phi \theta)}{\phi}\right)^2}{\left(\big(\frac{r_0(\phi)+r_0(\phi\theta)}{\phi}\big)^2+(\frac{\cos(\phi)-\cos(\phi\theta)}{\phi})^2\right)^{\frac52}}d\theta.
\end{align*}
Let us now define the following functions
$$
\forall \,s\in[\phi,\pi/2],\quad T_{1,\phi}(s):=\bigintsss_0^{\frac{\pi}{2s}}\frac{\phi^2\sin(\phi \theta) r_0^2(\phi \theta)}{R^{\frac52}(\phi,\phi\theta)}d\theta,
$$
and
\begin{equation}\label{T2phi1}
\forall \,s\in(0,\phi],\quad T_{2,\phi}(s):=\bigintss_0^{\frac{\pi}{2\phi}}\frac{\frac{\sin(s \theta)}{s} \left(\frac{r_0(s \theta)}{s}\right)^2}{\left(\big(\frac{r_0(s)+r_0(s\theta)}{s}\big)^2+(\frac{\cos(s)-\cos(s\theta)}{s})^2\right)^{\frac52}}d\theta.
\end{equation}
We will show that $T_{1,\phi}\in \mathscr{C}^\alpha[\phi,\pi/2]$ and $T_{2,\phi}\in \mathscr{C}^\alpha(0,\phi]$  uniformly in {$\phi\in(0,\frac\pi2)$}. Thus we get in particular a constant $C>0$ such that for any  $0<\phi_1\leq \phi_2< \frac\pi2$,
\begin{equation}\label{TM1}
|T_{1,\phi_1}(\phi_1)-T_{1,\phi_1}(\phi_2)|\leq C|\phi_1-\phi_2|^\alpha
\end{equation}
and
\begin{equation}\label{TM2}
{|T_{2,\phi_2}(\phi_1)-T_{2,\phi_2}(\phi_2)|\leq C|\phi_1-\phi_2|^\alpha.}
\end{equation}
By combining \eqref{TM1} and \eqref{TM2}, 
{ since  $T_{1,\phi_1}(\phi_ 2)= T_{2,\phi_2}(\phi_1)$ we get}
$$
|\zeta_{1,1}(\phi_1)-\zeta_{1,1}(\phi_2)|\leq |T_{1,\phi_1}(\phi_1)-T_{1,\phi_1}(\phi_2)|+{{|T_{2,\phi_2}(\phi_1)-T_{2,\phi_2}(\phi_2)|}}\leq C|\phi_1-\phi_2|^\alpha.
$$
This  ensures that $\zeta_{1,1}\in \mathscr{C}^\alpha(0,\pi/2).$

It remains to show that $T_{1,\phi}\in \mathscr{C}^\alpha([\phi,\pi/2])$ and $T_{2,\phi}\in \mathscr{C}^\alpha((0,\phi])$ uniformly in $\phi\in (0,\frac\pi2)$.  We start with the term  $T_{1,\phi}$. Then straightforward computations imply
\begin{align*}
\forall\, s\in [\phi,\pi/2],\quad |\partial_sT_{1,\phi}(s)|=& \frac{\pi}{2s^2}\frac{\phi^2\sin(\phi\pi/2s)r_0^2(\phi\pi/2s)}{R^\frac52(\phi,\phi\pi/2s)}\\
\leq& C\frac{1}{s^5}\frac{\phi^5}{\phi^5(1+\frac{\pi}{2s})^5}
\leq C,
\end{align*}
for any $\phi,s\in(0,\pi/2]$. Notice that we have used in the last line the following  inequalities which follow from the assumptions {\bf{(H2)}},
$$
\phi \lesssim r_0(\phi)\lesssim \phi, \quad \forall \phi\in[0,\pi/2]
$$
and
\begin{equation}\label{r0phi}
\theta\lesssim \frac{r_0(\phi\theta)}{\phi}\lesssim \theta, \quad \forall \theta\in[0,\pi/2\phi].
\end{equation}
Hence $T_{1,\phi}\in\textnormal{Lip}([\phi,\pi/2])$, uniformly with respect to $\phi\in(0,\pi/2)$.

Let us move to the term $T_{2,\phi}$.
First, we  write
$$
\frac{\sin(\phi \theta)}{\phi}=\theta\int_0^1\cos(\phi\theta\tau)d\tau,
$$
and taking the derivative with  respect to $\phi$ we obtain
$$
\partial_\phi\left(\frac{\sin(\phi \theta)}{\phi} \right)=-\theta^2\int_0^1\sin(\phi\theta\tau)\tau\,d\tau.
$$
Hence,
\begin{equation}\label{sinephi}
\Big|\partial_\phi\left(\frac{\sin(\phi \theta)}{\phi} \right)\Big|\leq \theta^2.
\end{equation}
By the mean value theorem we infer
\begin{equation}\label{sinephi2}
\left|\frac{\sin(s_1\theta)}{s_1}-\frac{\sin(s_2\theta)}{s_2}\right|\leq |s_1-s_2|\theta^2
\end{equation}
Interpolating between \eqref{r0phi}, which is also true  for $r_0=\sin$,  and \eqref{sinephi2} we obtain
\begin{equation}\label{sinephi2}
\left|\frac{\sin(s_1\theta)}{s_1}-\frac{\sin(s_2\theta)}{s_2}\right|\leq C|s_1-s_2|^\alpha \theta^{1-\alpha}\theta^{2\alpha}=C|s_1-s_2|^\alpha \theta^{1+\alpha}.
\end{equation}
Using Taylor's formula
\begin{equation*}
r_0(\phi\theta)=\phi\theta\int_0^1r_0'(\tau \phi\theta)d\tau,
\end{equation*}
one finds that if $0\leq\phi \theta\leq\pi/2$ then
\begin{equation}\label{r0stheta2}
\left|\partial_\phi\left(\frac{r_0(\phi\theta)}{\phi}\right)\right|\leq C\theta^2.
\end{equation}
As before, one gets that if $0\leq s_1 \theta,\, s_2\theta \leq\pi/2$ hence
\begin{equation}\label{r0stheta3}
\left|\frac{r_0(s_1\theta)}{s_1}-\frac{r_0(s_2\theta)}{s_2}\right|\leq C|s_1-s_2|^\alpha\theta^{1+\alpha}.
\end{equation}
Now, let us check that $T_{2,\phi}$ is $\mathscr{C}^\alpha(0,\phi]$ uniformly in $\phi\in(0,\pi/2)$. Let $s_1,s_2\in(0,\phi],$ then  using the estimates \eqref{sinephi2} and \eqref{r0phi}, we achieve for any $s\in(0,\phi],$
{\begin{align*}
\left|\bigintss_0^{\frac{\pi}{2\phi}}\frac{\left(\frac{\sin(s _1\theta)}{s_1}-\frac{\sin(s_2 \theta)}{s_2}\right) \left(\frac{r_0(s_1 \theta)}{s_1}\right)^2}{\left(\big(\frac{r_0(s_1)+r_0(s_1\theta)}{s_1}\big)^2+(\frac{\cos(s_1)-\cos(s_1\theta)}{s_1})^2\right)^{\frac52}}d\theta\right|&\leq C |s_1-s_2|^\alpha\bigintss_0^{\frac{\pi}{2\phi}}\frac{\theta^{1+\alpha}\theta^2}{(1+\theta)^5}d\theta\\
&\leq C|s_1-s_2|^\alpha,
\end{align*}}
for {$\alpha\in(0,1)$}. In the same way
{
\begin{align*}
\left|\bigintss_0^{\frac{\pi}{2\phi}}
\frac{\frac{\sin(s_2 \theta)}{s_2} \left(\frac{r_0(s_1 \theta)}{s_1}-\frac{r_0(s_2 \theta)}{s_2}\right)
\left(\frac{r_0(s_1 \theta)}{s_1}+\frac{r_0(s_2 \theta)}{s_2}\right)}
{((\frac{r_0(s_1)+r_0(s_1\theta)}{s_1})^2+(\frac{\cos(s_1)-\cos(s_1\theta)}{s_1})^2)^{\frac52}} d\theta
\right|
&\leq C|s_1-s_2|^\alpha\bigintss_0^{\frac{\pi}{2\phi}}\frac{\theta^{1+\alpha}
\theta^2}{(1+\theta)^5}d\theta\\
&\leq C|s_1-s_2|^\alpha.
\end{align*}}
To analyze the difference  of  the denominator in $T_{2,\phi}$  we first write that for any $0\leq s\theta\leq\frac\pi2,$
$$
s^5R^{-\frac52}(s,s\theta)=\left(\left(\frac{r_0(s)+r_0(s\theta)}{s}\right)^2+\left(\frac{\cos(s)-\cos(s\theta)}{s}\right)^2\right)^{-\frac52}\lesssim (1+\theta)^{-5}. 
$$
{ Using an slight variant of the argument in \eqref{sinephi} one gets 
\begin{equation}\label{Cosphi}
 |1-\cos(\phi)|\le |\phi|
\end{equation}
and
\begin{equation}\label{cosphi}
\Big|\partial_\phi\left(\frac{\cos(\phi \theta)-1}{\phi} \right)\Big|\leq \theta^2.
\end{equation}
By  differentiation, and using \eqref{cosphi} and \eqref{r0stheta2}, we find that} if  $ 0\leq s\theta\leq\frac\pi2$ hence
$$
 \left| \partial_s\left(s^5R^{-\frac52}(s,s\theta)\right)\right|\lesssim(1+\theta)^{-4}.
$$
This implies in view of the mean value theorem
$$
\forall\, 0\leq s_1\theta, s_2\theta\leq\frac\pi2,\quad  \left|s_1^5R^{-\frac52}(s_1,s_1\theta)-s_2^5R^{-\frac52}(s_2,s_2\theta)\right|\lesssim(1+\theta)^{-4}|s_1-s_2|.
$$
Moreover
$$
\forall\, 0\leq s_1\theta, s_2\theta\leq\frac\pi2,\quad  \left|s_1^5R^{-\frac52}(s_1,s_1\theta)-s_2^5R^{-\frac52}(s_2,s_2\theta)\right|\lesssim(1+\theta)^{-5}.
$$
Then by interpolation we get
\begin{equation}\label{partialR5}
\forall\, 0\leq s_1\theta, s_2\theta\leq\frac\pi2,\quad  \left|s_1^5R^{-\frac52}(s_1,s_1\theta)-s_2^5R^{-\frac52}(s_2,s_2\theta)\right|\lesssim(1+\theta)^{\alpha-5}|s_1-s_2|^\alpha.
\end{equation}
Therefore we obtain
{
\begin{align*}
\bigintss_0^{\frac{\pi}{2\phi}}{\frac{\sin(s_2 \theta)}{s_2} \left(\frac{r_0(s_2 \theta)}{s_2}\right)^2}\left|s_1^5R^{-\frac52}(s_1,s_1\theta)-s_2^5R^{-\frac52}(s_2,s_2\theta)\right|d\theta&\lesssim|s_1-s_2|^\alpha\bigintss_0^{\infty}(1+\theta)^{\alpha-2}d\theta,
\end{align*}}
which converges since $\alpha\in(0,1).$\\
Combining the preceding estimates  one deduces that 
$$
\forall\, s_1, s_0\in(0,\phi],\quad |T_{2,\phi}(s_1)-T_{2,\phi}(s_2)|\leq C|s_1-s_2|^\alpha,
$$
uniformly in $\phi\in(0,\pi/2)$. Hence, we conclude that $\zeta_{1,1}$ is $\mathscr{C}^\alpha(0,\pi/2)$, for any $\alpha\in(0,1)$.\\
{ The argument used to prove that the other terms in \eqref{CancelF1} are in  $\mathscr{C}^\alpha(0,\pi/2)$, for any $\alpha\in(0,1),$ are quite similar, but for the reader convenience  we will sketch some  details.   The second term in the sum is 
$$
\phi\mapsto r_0'(\phi)\frac{r_0(\phi)}{\sin(\phi)}\bigintsss_0^\pi \frac{\sin(\phi)\sin(\varphi) r_0^4(\varphi)}{R^{\frac72}(\phi,\varphi)}d\varphi.
$$
Let us consider the functions 
$$
\forall \,s\in[\phi,\pi/2],\quad T_{1,\phi}(s):=\bigintsss_0^{\frac{\pi}{2s}}\frac{\phi^2\sin(\phi \theta) r_0^4(\phi \theta)}{R^{\frac72}(\phi,\phi\theta)}d\theta
$$
and
$$
\forall \,s\in(0,\phi],\quad T_{2,\phi}(s):=\bigints_0^{\frac{\pi}{2\phi}}\frac{\frac{\sin(s \theta)}{s} \left(\frac{r_0(s \theta)}{s}\right)^4}{\left(\big(\frac{r_0(s)+r_0(s\theta)}{s}\big)^2+(\frac{\cos(s)-\cos(s\theta)}{s})^2\right)^{\frac72}}d\theta.
$$
To prove that this second term is in $\mathscr{C}^\alpha(0,\pi/2)$ we will see that $T_{1,\phi}\in\textnormal{Lip}([\phi,\pi/2])$ and $T_{2, \phi}$ is in $\mathscr{C}^\alpha(0,\phi)$. 
Using the mean value theorem and the estimate 
$$
\forall\, s\in [\phi,\pi/2],\quad |\partial_sT_{1,\phi}(s)|
\leq C\frac{1}{s^7}\frac{\phi^7}{\phi^7(1+\frac{\pi}{2s})^7}
\leq C,
$$
we can easily see that $T_{1,\phi}$ is in the desired space. 
Now we will check that $T_{2,\phi}$ is $\mathscr{C}^\alpha(0,\phi]$ uniformly in $\phi\in(0,\pi/2)$. Let $s_1,s_2\in(0,\phi],$ then  using the estimates \eqref{r0phi} and \eqref{sinephi2}, we achieve for any $s\in(0,\phi],$
\begin{align*}
\left|\bigints_0^{\frac{\pi}{2\phi}}\frac{\left(\frac{\sin(s _1\theta)}{s_1}-\frac{\sin(s_2 \theta)}{s_2}\right) \left(\frac{r_0(s_1 \theta)}{s_1}\right)^4}{\left(\big(\frac{r_0(s_1)+r_0(s_1\theta)}{s_1}\big)^2+(\frac{\cos(s_1)-\cos(s_1\theta)}{s_1})^2\right)^{\frac72}}d\theta\right|&\leq C |s_1-s_2|^\alpha\bigintss_0^{\frac{\pi}{2\phi}}\frac{\theta^{1+\alpha}\theta^4}{(1+\theta)^7}d\theta\\
&\leq C|s_1-s_2|^\alpha,
\end{align*}
for {$\alpha\in(0,1)$}. In the same way
\begin{align*}
&\left|\bigints_0^{\frac{\pi}{2\phi}}
\frac{\frac{\sin(s_2 \theta)}{s_2} \left((\frac{r_0(s_1 \theta)}{s_1})^3+(\frac{r_0(s_1 \theta)}{s_1})^2(\frac{r_0(s_2 \theta)}{s_2})+(\frac{r_0(s_1 \theta)}{s_1})(\frac{r_0(s_2 \theta)}{s_2})^2+(\frac{r_0(s_2 \theta)}{s_2})^3\right)
\left(\frac{r_0(s_1 \theta)}{s_1}-\frac{r_0(s_2 \theta)}{s_2}\right)}
{((\frac{r_0(s_1)+r_0(s_1\theta)}{s_1})^2+(\frac{\cos(s_1)-\cos(s_1\theta)}{s_1})^2)^{\frac72}} d\theta
\right|\\
&\leq C|s_1-s_2|^\alpha\bigintss_0^{\frac{\pi}{2\phi}}\frac{\theta^{1+\alpha}
\theta^2}{(1+\theta)^5}d\theta\leq C|s_1-s_2|^\alpha.
\end{align*}
Now by differentiation,  and using \eqref{r0stheta2} and \eqref{cosphi} we find that if  $ 0\leq s\theta\leq\frac\pi2$ then
\begin{equation}\label{BR}
 \left| \partial_s\left(s^7R^{-\frac72}(s,s\theta)\right)\right|\lesssim(1+\theta)^{-6}.
\end{equation}
Therefore, using  interpolation argument we obtain
\begin{align*}
\bigintss_0^{\frac{\pi}{2\phi}}{\frac{\sin(s_2 \theta)}{s_2} \left(\frac{r_0(s_2 \theta)}{s_2}\right)^4}\left|s_1^7R^{-\frac72}(s_1,s_1\theta)-s_2^7R^{-\frac72}(s_2,s_2\theta)\right|d\theta&\lesssim|s_1-s_2|^\alpha\bigintss_0^{\infty}(1+\theta)^{\alpha-2}d\theta,
\end{align*}
and the last integral converges since $\alpha\in(0,1).$\\
Another  term to consider in \eqref{CancelF1}  is given by
$$
\phi\mapsto  r_0'(\phi)\left(\frac{r_0(\phi)}{\sin(\phi)}\right) ^2\bigintsss_0^\pi \frac{\sin^2(\phi)\sin(\varphi) r_0^3(\varphi)}{R^{\frac72}(\phi,\varphi)}d\varphi,
$$
which will be treated exactly in the same way as the previous one. For this reason we will not repeat again the arguments. 
The next term in \eqref{CancelF1} can be written as 
$$
\phi\mapsto \frac{r_0(\phi)}{\sin(\phi)}\bigintsss_0^\pi \frac{\sin^2(\phi)\sin(\varphi) r_0^3(\varphi)(\cos(\phi)-\cos(\varphi))}{R^{\frac72}(\phi,\varphi)}d\varphi.
$$
As in the previous cases we can consider the auxiliary functions 
$$
\forall \,s\in[\phi,\pi/2],\quad T_{1,\phi}(s):=\bigintss_0^{\frac{\pi}{2s}}\frac{\phi^2\sin(\phi \theta) r_0^3(\phi \theta)((1-\cos(\phi)-(1-\cos(\theta\phi)))}{R^{\frac72}(\phi,\phi\theta)}d\theta
$$
and
$$
\forall \,s\in(0,\phi],\quad T_{2,\phi}(s):=\bigints_0^{\frac{\pi}{2\phi}}\frac{\big(\frac{\sin(s \theta)}{s}\big) \left(\frac{(r_0(s \theta)}{s}\right)^3\big(\frac{(1-\cos(s))-(1-\cos(s\theta)}{s}\big)}
{\left(\big(\frac{r_0(s)+r_0(s\theta)}{s}\big)^2+(\frac{\cos(s)-\cos(s\theta)}{s})^2\right)^{\frac72}}d\theta.
$$
For this term it is enough to check that the auxiliary functions are in the $\mathscr{C}^\alpha$ space. For $T_{1,\phi}$ we will use the mean value theorem. Thus the inequality  
$$
\forall\, s\in [\phi,\pi/2],\quad |\partial_sT_{1,\phi}(s)|
\leq C\frac{1}{s^6}\frac{\phi^7}{\phi^7(1+\frac{\pi}{2s})^6}
\leq C,
$$
will be enough. 
To estimate $T_{2,\phi}$ we will follow the same arguments developed  in the previous cases. Hence, using inequalities \eqref{r0stheta2}, \eqref{r0stheta3}, \eqref{sinephi}, \eqref{sinephi2}, \eqref{r0phi}, \eqref{cosphi} and \eqref{Cosphi} one gets the following estimates  for $\alpha\in(0,1)$.

\begin{align*}
&\left|\bigintss_0^{\frac{\pi}{2\phi}}
\frac{\left(\frac{\sin(s _1\theta)}{s_1}-\frac{\sin(s_2 \theta)}{s_2}\right)\left(\frac{r_0(s_1 \theta)}{s_1}\right)^3
\big(\frac{(1-\cos(s_1))-(1-\cos(s_1\theta)}{s_1}\big)}
{\left(\big(\frac{r_0(s_1)+r_0(s_1\theta)}{s_1}\big)^2+(\frac{\cos(s_1)-\cos(s_1\theta)}{s_1})^2\right)^{\frac72}}d\theta\right|\\
&\leq C |s_1-s_2|^\alpha\bigintss_0^{\frac{\pi}{2\phi}}\frac{\theta^{1+\alpha}\theta^3}{(1+\theta)^6}d\theta\\
&\leq C|s_1-s_2|^\alpha
\end{align*}
and
\begin{align*}
&\left|\bigintss_0^{\frac{\pi}{2\phi}}\frac{\big(\frac{\sin(s_2 \theta)}{s_2}\big) \left(\left(\frac{(r_0(s_1 \theta)}{s_1}\right)^3-\left(\frac{(r_0(s_2 \theta)}{s_2}\right)^3\right)\big(\frac{(1-\cos(s_2))-(1-\cos(s_2\theta)}{s_2}\big)}
{\left(\big(\frac{r_0(s_2)+r_0(s_2\theta)}{s_2}\big)^2+(\frac{\cos(s_2)-\cos(s\theta)}{s_2})^2\right)^{\frac72}}d\theta\right|\\
&\leq C |s_1-s_2|^\alpha\bigintss_0^{\frac{\pi}{2\phi}}\frac{\theta^{1+\alpha}\theta^3}{(1+\theta)^6}d\theta\\
&\leq C|s_1-s_2|^\alpha. 
\end{align*}

On the other hand, using \eqref{cosphi}, \eqref{Cosphi} and interpolation, one obtains
\begin{equation}\label{incos}
\left|\frac{1-\cos(s_1\theta)}
{s_1}-\frac{1-\cos(s_2\theta)}{s_2}\right|\le C\theta^{1+\alpha}|s_1-s_2|^{\alpha}. 
\end{equation}

Then, by \eqref{incos} we obtain

\begin{align*}
&\left|\bigintss_0^{\frac{\pi}{2\phi}}
\frac{\left(\frac{\sin(s _2\theta)}{s_2}\right)
\left(\frac{r_0(s_2 \theta)}{s_2}\right)^3\big(\frac{(1-\cos(s_1))-(1-\cos(s_1\theta))}{s_1}-\frac{(1-\cos(s_2))-(1-\cos(s_2\theta))}{s_2}\big)
}
{\left(\big(\frac{r_0(s_1)+r_0(s_1\theta)}{s_1}\big)^2+(\frac{\cos(s_1)-\cos(s_1\theta)}{s_1})^2\right)^{\frac72}}d\theta
\right| \\
&\leq C |s_1-s_2|^\alpha\bigintss_0^{\frac{\pi}{2\phi}}\frac{(1+\theta^{1+\alpha})\theta^4}{(1+\theta)^7}d\theta\\
&\leq C|s_1-s_2|^\alpha,
\end{align*}

To estimate the last term of $T_{2,\phi},$ using \eqref{BR} 

\begin{align*}
&\left|\bigintss_0^{\frac{\pi}{2\phi}}{\frac{\sin(s_2 \theta)}{s_2} \left(\frac{r_0(s_2 \theta)}{s_2}\right)^3}\frac{(1-\cos(s_2))-(1-\cos(s_2\theta)}{s_2}\left(s_1^7R^{-\frac72}(s_1,s_1\theta)-s_2^7R^{-\frac72}(s_2,s_2\theta)\right)d\theta\right|\\
&\lesssim|s_1-s_2|^\alpha\bigintss_0^{\infty}\theta ^4(1+\theta)^{\alpha-6}d\theta\\
&\lesssim|s_1-s_2|^\alpha.  \end{align*} 
The next function  to analyze in \eqref{CancelF1} is 
$$
\phi\mapsto \bigintsss_0^\pi \frac{\sin(\phi)\sin(\varphi) r_0^2(\varphi)(\cos(\phi)-\cos(\varphi))}{R^{\frac52}(\phi,\varphi)}d\varphi.
$$
We will not repeat the arguments for this function because they are quite similar to the preceding  case. 
The last function to consider   in \eqref{CancelF1} is given by 
\begin{align*}
 \partial_{\varphi}\mathscr{K}_1(\phi,\varphi)\frac{r_0(\phi)r_0(\varphi)}{R(\phi, \varphi)}=&2\frac{r_0(\phi)r_0^2(\varphi)r_0'(\varphi)\sin(\varphi)}{(R(\phi, \varphi))^{5/2}}+ \frac{r_0(\phi)r_0^3(\varphi)\cos(\phi)}{(R(\phi, \varphi))^{5/2}}\\
 &-3\frac{r_0(\phi)r_0^3(\varphi) \sin(\varphi)}{(R(\phi, \varphi))^{7/2}}\big((r_0(\phi)+r_0(\varphi))r_0'(\varphi)+(\cos(\phi)-\cos(\varphi))\sin(\varphi)\big)
\end{align*}
This term generates several functions. Some of them are similar to the functions estimated in the previous cases and the others are similar between them. For this reason we will only check the first one. Let us prove that the function 

$$
\phi\mapsto  \frac{r_0(\phi)}{\sin(\phi)} \bigintsss_0^\pi \frac{\sin(\phi)\sin(\varphi) r_0^2(\varphi)r_0'(\varphi)}{R^{\frac52}(\phi,\varphi)}d\varphi,
$$
is in  $\mathscr{C}^\alpha(0,\pi/2)$, for any $\alpha\in(0,1).$ Since the integral in the interval $[\frac{\pi}{2},\pi]$ provides a function in $\mathscr{C}^\alpha,$  as in the above cases we can reduce the integral to the interval $[0,\pi/2]. $
Now the strategy is again to consider the auxilary functions 

$$
\forall \,s\in[\phi,\pi/2],\quad T_{1,\phi}(s):=\bigintsss_0^{\frac{\pi}{2s}}\frac{\phi^2\sin(\phi \theta) r_0^2(\phi \theta)r'_0(\phi\theta)}{R^{\frac72}(\phi,\phi\theta)}d\theta
$$
and
$$
\forall \,s\in(0,\phi],\quad T_{2,\phi}(s):=\bigintss_0^{\frac{\pi}{2\phi}}\frac{\frac{\sin(s \theta)}{s} \left(\frac{r_0(s \theta)}{s}\right)^2r_0'(s\theta)}{\left(\big(\frac{r_0(s)+r_0(s\theta)}{s}\big)^2+(\frac{\cos(s)-\cos(s\theta)}{s})^2\right)^{\frac52}}d\theta.
$$
Since 
$$
\forall\, s\in [\phi,\pi/2],\quad |\partial_sT_{1,\phi}(s)|
\leq C\frac{1}{s^5}\frac{\phi^5}{\phi^5(1+\frac{\pi}{2s})^5}
\leq C, 
$$
 the function $T_{1,\phi}$ is in $\mathscr{C}^\alpha,$ for any $\alpha\in (0,1).$
 To establish that $T_{2,\phi}$ is in the same space we need the following estimates. 
 \begin{align*}
 &\bigintss_0^{\frac{\pi}{2\phi}}\frac{(\frac{\sin(s_1 \theta)}{s_1}-\frac{\sin(s_2 \theta)}{s_2}) \left(\frac{r_0(s_1 \theta)}{s_1}\right)^2r_0'(s_1\theta)}{\left(\big(\frac{r_0(s_1)+r_0(s_1\theta)}{s_1}\big)^2+(\frac{\cos(s_1)-\cos(s_1\theta)}{s_1})^2\right)^{\frac52}}d\theta \\
  &\lesssim|s_1-s_2|^\alpha\bigintss_0^{\infty}\theta ^{3+\alpha}(1+\theta)^{-5}d\theta\\
&\lesssim|s_1-s_2|^\alpha,
 \end{align*}
in the first inequality we have used that $r_0'$ is  a bounded function. 
The next estimate is 
\begin{align*}
 &\bigints_0^{\frac{\pi}{2\phi}}\frac{(\frac{\sin(s_2 \theta)}{s_2})
 \left( \left(
 \frac{r_0(s_1 \theta)}{s_1}\right)^2-\left(
 \frac{r_0(s_2 \theta)}{s_2}\right)^2
  \right)
   r_0'(s_1\theta)}{\left(\big(\frac{r_0(s_1)+r_0(s_1\theta)}{s_1}\big)^2+(\frac{\cos(s_1)-\cos(s_1\theta)}{s_1})^2\right)^{\frac52}}d\theta \\
  &\lesssim|s_1-s_2|^\alpha\bigintss_0^{\infty}\theta ^{3+\alpha}(1+\theta)^{-5}d\theta\\
&\lesssim|s_1-s_2|^\alpha.
 \end{align*}
Using the inequality $|r_0'(s_1\theta)-r_0'(s_2\theta)|\le C\theta^{\alpha}|s_1-s_2|^{\alpha}$ we get 
  \begin{align*}
 &\bigints_0^{\frac{\pi}{2\phi}}\frac{(\frac{\sin(s_1 \theta)}{s_1})
  \left(\frac{r_0(s_1 \theta)}{s_1}\right)^2
 (r_0'(s_1\theta)-r_0'(s_2\theta))}{\left(\big(\frac{r_0(s_1)+r_0(s_1\theta)}{s_1}\big)^2+(\frac{\cos(s_1)-\cos(s_1\theta)}{s_1})^2\right)^{\frac52}}d\theta \\
  &\lesssim|s_1-s_2|^\alpha\bigintss_0^{\infty}\theta ^{3+\alpha}(1+\theta)^{-5}d\theta\\
&\lesssim|s_1-s_2|^\alpha.
 \end{align*}
 The last term to estimate is 
\begin{align*}
&\bigintss_0^{\frac{\pi}{2\phi}}{\frac{\sin(s_2 \theta)}{s_2} \left(\frac{r_0(s_2 \theta)}{s_2}\right)^2}r_0'(s_2\theta)\left|s_1^5R^{-\frac52}(s_1,s_1\theta)-s_2^5R^{-\frac52}(s_2,s_2\theta)\right|d\theta \\
&\lesssim|s_1-s_2|^\alpha\bigintss_0^{\infty}(1+\theta)^{\alpha-2}d\theta\\
&\lesssim |s_1-s_2|^\alpha.
\end{align*}
Hence, we obtain the announced result. The remaining terms can be studied using  the same inequalities and for this reason we will avoid them. 
}

Let us now move to the regularity of $\zeta_2$ defined in \eqref{Deriv1} which takes the form
$$
\zeta_2(\phi)=-{\frac{3}{2}}\zeta_{2,1}(\phi)+\zeta_{2,2}(\phi), \quad \zeta_{2,1}(\phi):=\bigintss_0^\pi \frac{\partial_\phi R(\phi,\varphi)}{R^\frac52(\phi,\varphi)}\sin(\varphi)r_0^2(\varphi)[F_1(\rho(\phi,\varphi))-1]d\varphi,
$$
and 
$$
\zeta_{2,2}(\phi):=\bigintss_0^\pi\partial_\varphi\mathscr{K}_1(\phi,\varphi)\big(F_1(\rho(\phi,\varphi))-1\big)d\varphi.
$$

The first function  can be split into two parts as follows

\begin{align*}
\zeta_{2,1}(\phi)=&\bigintss_0^{\frac{\pi}{2}} \frac{\partial_\phi R(\phi,\varphi)}{R^\frac52(\phi,\varphi)}\sin(\varphi)r_0^2(\varphi)[F_1(\rho(\phi,\varphi))-1]d\varphi\\
&+\bigintss_{\frac{\pi}{2}}^\pi\frac{\partial_\phi R(\phi,\varphi)}{R^\frac52(\phi,\varphi)}\sin(\varphi)r_0^2(\varphi)[F_1(\rho(\phi,\varphi))-1]d\varphi\\
=:&I_1(\phi)+I_2(\phi).
\end{align*}
As before, by evoking the symmetry property of $r$  we can restrict the study to  $\phi\in[0,\frac{\pi}{2}]$. The second term is the easiest  one and we claim that  $I_2\in W^{1,\infty}$. Indeed,
\begin{align*}
I_2'(\phi)=&\bigintss_{\frac{\pi}{2}}^\pi\partial_\phi\left( \frac{\partial_\phi R(\phi,\varphi)}{R^\frac52(\phi,\varphi)}\right)\sin(\varphi)r_0^2(\varphi)[F_1(\rho(\phi,\varphi))-1]d\varphi\\
&+\bigintss_{\frac{\pi}{2}}^\pi \frac{\partial_\phi R(\phi,\varphi)}{R^\frac52(\phi,\varphi)}\sin(\varphi)r_0^2(\varphi)F_1^\prime(\rho(\phi,\varphi))\partial_\phi \rho(\phi,\varphi)d\varphi.
\end{align*}
It can be  transformed into
 \begin{align*}
I_2'(\phi)=&\bigintsss_{\frac{\pi}{2}}^\pi\partial_\phi\left( \frac{\partial_\phi R(\phi,\varphi)}{R^\frac52(\phi,\varphi)}\right)\sin(\varphi)r_0^2(\varphi)\Big(F_1(\rho(\phi,\varphi))-1\Big)d\varphi\\
&+\bigintsss_{\frac{\pi}{2}}^\pi \frac{\partial_\phi R(\phi,\varphi)}{R^\frac52(\phi,\varphi)}\sin(\varphi)r_0^2(\varphi)F_1'(\rho(\phi,\varphi))(\partial_\phi \rho(\phi,\varphi)+\partial_\varphi \rho(\phi,\varphi))d\varphi\\
&-\bigintsss_{\frac{\pi}{2}}^\pi \frac{\partial_\phi R(\phi,\varphi)}{R^\frac52(\phi,\varphi)}\sin(\varphi)r_0^2(\varphi)\partial_\varphi({F_1}(\rho(\phi,\varphi)-1)d\varphi.
\end{align*}
Integrating by parts yields
 \begin{align*}
I_2'(\phi)=&\bigintsss_{\frac{\pi}{2}}^\pi\partial_\phi\left( \frac{\partial_\phi R(\phi,\varphi)}{R^\frac52(\phi,\varphi)}\right)\sin(\varphi)r_0^2(\varphi)\Big(F_1(\rho(\phi,\varphi))-1\Big)d\varphi\\
&+\bigintsss_{\frac{\pi}{2}}^\pi \frac{\partial_\phi R(\phi,\varphi)}{R^\frac52(\phi,\varphi)}\sin(\varphi)r_0^2(\varphi)F_1^\prime(\rho(\phi,\varphi))(\partial_\phi \rho(\phi,\varphi)+\partial_\varphi \rho(\phi,\varphi))d\varphi\\
&+\bigintsss_{\frac{\pi}{2}}^\pi \partial_\varphi\left(\frac{\partial_\phi R(\phi,\varphi)}{R^\frac52(\phi,\varphi)}\sin(\varphi)r_0^2(\varphi)\right)
\Big(F_1(\rho(\phi,\varphi)-1\Big)d\varphi\\
&{+ \frac{\partial_\phi R(\phi,\frac{\pi}{2})}{R^\frac52(\phi,\frac{\pi}{2})}r_0^2\left({\pi}/{2}\right)
\Big(F_1\left(\rho\left(\phi,{\pi}/{2}\right)\right)-1\Big).}
\end{align*}

{Notice that the last term is bounded uniformly on $\phi\in[0,\pi/2]$. In fact,  one has from the definition of $R$ in \eqref{RefR}
$$
\forall \phi\in [0,\pi/2],\quad \frac{1}{R(\phi,\frac{\pi}{2})}\leq {\frac{1}{r_0^2(\pi/2)}}\cdot
$$
Using \eqref{partialR} we get
$$
\partial_\phi R(\phi,\pi/2)=2r_0^\prime(\phi)(r_0(\phi)+r_0(\pi/2))-2\sin\phi\cos\phi.
$$
Moreover,  since $r_0$ is symmetric with respect to $\pi/2$ then we get $r_0^\prime\left(\frac{\pi}{2}\right)=0$, which implies that  $\partial_\phi R(\pi/2,\pi/2)=0$ and by the mean value theorem,  
$$
\forall \phi\in (0,\pi),\quad \Big|(\partial_\phi R)\left(\phi,\pi/2\right)\Big|\lesssim\left|\phi-\frac{\pi}{2}\right|.
$$
}
Hence, combining  \eqref{ZNX1} and \eqref{Est-LogX} we find
\begin{align*}
\forall \phi\in\left(0,\frac{\pi}{2}\right),\quad \left|F_1\left(\rho\left(\phi,\frac{\pi}{2}\right)\right)-1\right|\lesssim&\rho\left(\phi,\frac{\pi}{2}\right)\Big(1{+}\ln\Big[1-\rho\left(\phi,\frac{\pi}{2}\right)\Big]\Big)\\
\lesssim& { 1+}\left|\ln\left(\frac{\pi}{2}-\phi\right)\right|.
\end{align*}
Consequently
\begin{equation}\label{TunX1}
\forall \phi\in\left(0,\frac{\pi}{2}\right),\quad \frac{\Big|\partial_\phi R(\phi,\frac{\pi}{2})\Big|}{R^\frac52(\phi,\frac{\pi}{2})}r_0^2\left(\frac{\pi}{2}\right)
\Big|F_1\left(\rho\left(\phi,\frac{\pi}{2}\right)\right)-1\Big|\lesssim \big(\frac{\pi}{2}-\phi\big)\left|{1+}\ln\left(\frac{\pi}{2}-\phi\right)\right|,
\end{equation}
which ensures that this quantity is bounded in the interval $(0,\frac{\pi}{2})$.

Next,  let us check the boundedness of the  integral terms of $I_2^\prime$. Inequality  {\eqref{Rbound}} allows to get 
$$
 \sup_{\phi\in [0,\pi/2]\\
\atop \varphi\in [\pi/2,\pi]}  \left|\partial_\phi\left( \frac{\partial_\phi R(\phi,\varphi)}{R(\phi,\varphi)^\frac52}\right)\right|+\left|\frac{\partial_\phi R(\phi,\varphi)}{R(\phi,\varphi)^\frac52}\right|+\left|\partial_\varphi\left(\frac{\partial_\phi R(\phi,\varphi)}{R(\phi,\varphi)^\frac52}\sin(\varphi)r_0(\varphi)^2\right)\right|<\infty,
$$
which implies
 \begin{align*}
|I_2'(\phi)|\lesssim&1+\bigintsss_{\frac{\pi}{2}}^\pi\Big|F_1(\rho(\phi,\varphi))-1\Big|d\varphi+\bigintsss_{\frac{\pi}{2}}^\pi\big|F_1^\prime(\rho(\phi,\varphi))(\partial_\phi \rho(\phi,\varphi)+\partial_\varphi \rho(\phi,\varphi))\big|d\varphi.
\end{align*}
Therefore, \eqref{Est-LogX} combined with \eqref{TixaZ} and  \eqref{Sing3X} yield
\begin{align*}
\forall \phi\in[0,\pi/2],\quad |I_2^\prime(\phi)|\leq&C+C\bigintsss_{\frac\pi2}^\pi\ln\left(\frac{\phi+\varphi}{|\phi-\varphi|}\right)d\varphi+C\bigintsss_{\frac\pi2}^\pi\frac{R(\phi,\varphi)}{(\phi-\varphi)^2}\frac{(\phi-\varphi)^2}{R^\frac32(\phi,\varphi)}d\varphi\leq C.
\end{align*}
Let us move to $I_1$. First, we do the change of variables $\varphi=\phi\theta$ leading to
\begin{align*}
I_1(\phi)=\bigintsss_0^{\frac{\pi}{2\phi}} \frac{\phi (\partial_\phi R)(\phi,\phi\theta)}{R^\frac52(\phi,\phi\theta)}\sin(\phi\theta)r_0^2(\phi\theta)\Big(F_1(\rho(\phi,\phi\theta))-1\Big)d\theta.
\end{align*}
We will check that $I_1$ is $\mathscr{C}^\alpha(0,\pi/2)$, for any $\alpha\in(0,1)$. Indeed, take $\phi_1\leq \phi_2\in(0,\frac{\pi}{2})$, then
\begin{align}\label{DecoWX}
\nonumber & I_1(\phi_1)-I_1(\phi_2)=\bigintsss_{\frac{\pi}{2\phi_2}}^{\frac{\pi}{2\phi_1}} \frac{\phi_1 (\partial_\phi R)(\phi_1,\phi_1\theta)}{R^\frac52(\phi_1,\phi_1\theta)}\sin(\phi_1\theta)r_0^2(\phi_1\theta)\Big(F_1(\rho(\phi_1,\phi_1\theta))-1\Big)d\theta\\
\nonumber&+\bigintsss_0^{\frac{\pi}{2\phi_2}} \frac{\phi_1 (\partial_\phi R)(\phi_1,\phi_1\theta)}{R^\frac52(\phi_1,\phi_1\theta)}\sin(\phi_1\theta)r_0^2(\phi_1\theta)\Big(F_1(\rho(\phi_1,\phi_1\theta))-F_1(\rho(\phi_2,\phi_2\theta))\Big)d\theta\\
\nonumber&+\bigintss_0^{\frac{\pi}{2\phi_2}} \left(\frac{\phi_1 (\partial_\phi R)(\phi_1,\phi_1\theta)}{R^\frac52(\phi_1,\phi_1\theta)}\sin(\phi_1\theta)r_0^2(\phi_1\theta)-\frac{\phi_2 (\partial_\phi R)(\phi_2,\phi_2\theta)}{R^\frac52(\phi_2,\phi_2\theta)}\sin(\phi_2\theta)r_0^2(\phi_2\theta)\right)\\
\nonumber&\qquad\qquad\times\Big(F_1(\rho(\phi_2,\phi_2\theta))-1\Big)d\theta
\\
=:&I_{1,1}+I_{1,2}+I_{1,3},
\end{align}
where
\begin{align*}
I_{1,1}:=&\bigintss_{\frac{\pi}{2\phi_2}}^{\frac{\pi}{2\phi_1}} \frac{\phi_1 (\partial_\phi R)(\phi_1,\phi_1\theta)}{R^\frac52(\phi_1,\phi_1\theta)}\sin(\phi_1\theta)r_0^2(\phi_1\theta)\Big(F_1(\rho(\phi_1,\phi_1\theta))-1\Big)d\theta.
\end{align*}
We follow the ideas done for $\zeta_1$. In order to estimate $I_{1,1}$, define
$$
G_{1,\phi}(s):=\bigintsss_0^{\frac{\pi}{2s}} \frac{\phi (\partial_\phi R)(\phi,\phi\theta)}{R^\frac52(\phi,\phi\theta)}\sin(\phi\theta)r_0^2(\phi\theta)\Big(F_1(\rho(\phi,\phi\theta))-1\Big)d\theta.
$$
Then
$$
\forall\, s\in[\phi,\pi/2),\quad \partial_sG_{1,\phi}(s)=-\frac{\pi}{2s^2}\frac{\phi(\partial_\phi R)(\phi,\frac{\pi \phi}{2s})}{R^\frac52(\phi,\frac{\pi \phi}{2s}) }\sin\left(\frac{\pi \phi}{2s}\right)r_0^2\left(\frac{\pi \phi}{2s}\right)\left[F_1\left(\rho\left(\phi,\frac{\pi \phi}{2s}\right)\right)-1\right],
$$
for $\phi\in(0,\pi/2)$. 
{ Taking the derivative in $\phi$ of the function $R$} 
$$
(\partial_\phi R)(\phi,\phi\theta)=2r_0'(\phi)(r_0(\phi)+r_0(\phi\theta))+2\sin(\phi)(\cos(\phi\theta)-\cos(\phi)),
$$
we get
\begin{equation}\label{Thm1}
{|(\partial_\phi R)(\phi,\phi\theta)|}\leq C(\phi(1+\theta)+\phi|1-\theta|).
\end{equation}
Moreover, proceeding as before in \eqref{ZNX1}  combined with  the assumptions {\bf{(H)}} we find
\begin{equation}\label{F1-1}
\big|F_1(\rho(\phi,\phi\theta))-1\big|\lesssim \frac{1}{1+\theta}\Big(1+\ln\left|\frac{1+\theta}{1-\theta}\right|\Big).
\end{equation}
{Putting  together the preceding estimates allows to get
\begin{align*}
|\partial_sG_{1,\phi}(s)|\lesssim& \frac{1}{s^2}\frac{\phi\left\{\phi(1+\frac{\pi}{2s})+\phi|1-\frac{\pi}{2s}|\right\}}{\phi^5(1+\frac{\pi}{2s})^5}\frac{\phi^3}{s^3}\frac{1}{1+\frac{\pi}{2s}}\left(1+\ln\left|\frac{1+\frac{\pi}{2s}}{1-\frac{\pi}{2s}}\right|\right)\\
\lesssim& \frac{(s+\frac{\pi}{2})+|s-\frac{\pi}{2}|}{(s+\frac{\pi}{2})^6}\left(1+\ln\left|\frac{1+\frac{\pi}{2s}}{1-\frac{\pi}{2s}}\right|\right).
\end{align*}
It follows that 
\begin{align*}
\forall s\in(0,\pi/2),\quad\sup_{\phi\in(0,s]}\left|\partial_sG_{1,\phi}(s)\right|\lesssim& 1+\left|\ln\left(\frac{\pi}{2}-s\right)\right|.
\end{align*}
 Now using this estimate combined with  the mean value theorem we get for $0<\phi_1\leq\phi_2<\frac\pi2$
\begin{align*}
|I_{1,1}|\le&\int_{\phi_1}^{\phi_2}\left|\partial_sG_{1,\phi_1}(s)\right|ds\\
\le &C|\phi_1-\phi_2|+\int_{\phi_1}^{\phi_2}\left|\ln\left(\frac{\pi}{2}-s\right)\right|ds.
\end{align*}}
Using H\"{o}lder inequality yields for any $\alpha\in(0,1)$,
\begin{align*}
\int_{\phi_1}^{\phi_2}\left|\ln\left(\frac{\pi}{2}-s\right)\right|ds\le&|\phi_1-\phi_2|^\alpha\left(\int_{0}^{\frac\pi2}\left|\ln\left(\frac{\pi}{2}-s\right)\right|^{\frac{1}{1-\alpha}}ds\right)^{1-\alpha}\leq C_\alpha|\phi_1-\phi_2|^\alpha.\end{align*}
Notice that the constant $C_\alpha$ blows up when $\alpha$ approaches $1.$ Thus
$$
\forall\, \phi_1,\phi_2\in(0,\pi/2),\quad |I_{1,1}|\le C_\alpha|\phi_1-\phi_2|^\alpha.
$$

Next, let us move to the estimate of  $I_{1,2}$. Using \eqref{TixaZ} we arrive at
\begin{equation}\label{DerW1}
|F_1'(\rho(\phi,\phi\theta))|\leq C\frac{R(\phi,\phi\theta)}{\phi^2(1-\theta)^2}\leq C\frac{(1+\theta)^2}{(1-\theta)^2}\cdot
\end{equation}
Set
$$
\mathscr{R}(\theta,\phi):=\rho(\phi,\phi\theta)=\frac{4r_0(\phi) r_0(\phi\theta)}{R(\phi,\phi\theta)},
$$
then differentiating with respect to  $\theta$ we get
\begin{align*}
 \partial_\theta\mathscr{R}(\theta,\phi)=
&-\frac{8r_0(\phi)r_0(\phi\theta)}{R(\phi,\phi\theta)^2}\Big((r_0(\phi)+r_0(\phi\theta))\phi r_0'(\phi\theta)+(\cos(\phi)-\cos(\phi\theta))\phi\sin(\phi\theta)\Big)\\
&+\frac{4r_0(\phi) \phi r_0^\prime(\phi\theta)}{R(\phi,\phi\theta)}\cdot
\end{align*}
Using the assumption {\bf{(H2)}} we may  check that
$$
\forall\, 0\leq \phi\theta\le\pi/2,\quad|\partial_\theta\mathscr{R}(\theta,\phi)|\leq\frac{C}{(1+\theta)^2}, 
$$
where $C$ depends only on $\|r_0^\prime\|_{L^\infty}.$ Now by  rewriting
\begin{align*}
\partial_\theta\mathscr{R}(\theta,\phi)=&\frac{\frac{4r_0(\phi)}{\phi}  r_0^\prime(\phi\theta)}{\left(\frac{r_0(\phi)+r_0(\phi\theta)}{\phi}\right)^2+\left(\frac{\cos(\phi)-\cos(\phi\theta)}{\phi}\right)^2}\\
&-\frac{8\frac{r_0(\phi)}{\phi}\frac{r_0(\phi\theta)}{\phi}}{\left\{\left(\frac{r_0(\phi)+r_0(\phi\theta)}{\phi}\right)^2+\left(\frac{\cos(\phi)-\cos(\phi\theta)}{\phi}\right)^2\right\}^2}\\
&\times\left[\frac{(r_0(\phi)+r_0(\phi\theta))}{\phi} r_0'(\phi\theta)+{(\cos(\phi)-\cos(\phi\theta))}\frac{\sin(\phi\theta)}{\phi}\right],
\end{align*}
and differentiating in $\phi$ we get the estimate
$$
\forall\, 0\leq \phi\theta\le\pi/2,\quad|\partial_\phi\partial_\theta\mathscr{R}(\theta,\phi)|\leq\frac{C}{(1+\theta)}, 
$$
where  $C$ depends only on $\|r_0\|_{C^2}.$ Taylor's formula
$$
\mathscr{R}(\theta,\phi)=\mathscr{R}(1,\phi)+\int_1^\theta\partial_\theta\mathscr{R}(\tau,\phi)d\tau
$$
combined with  $\mathscr{R}(1,\phi)=1$ yields
$$
\partial_\phi\mathscr{R}(\theta,\phi)=\int_1^\theta\partial_\phi\partial_\theta\mathscr{R}(\tau,\phi)d\tau.
$$
This  implies in turn that
\begin{equation}\label{est-R}
\sup_{\phi\in(0,\frac{\pi}{2\theta})}|\partial_\phi\mathscr{R}(\theta,\phi)|\leq C\left|\ln\left(\frac{1+\theta}{2}\right)\right|.
\end{equation}
Combining this estimate with \eqref{DerW1} we deduce that
$$
\sup_{\phi\in(0,\frac{\pi}{2\theta})}\big|\partial_\phi\big[F_1(\rho(\phi,\phi\theta))\big]\big|\leq C\frac{(1+\theta)^2}{(1-\theta)^2}\left|\ln\left(\frac{1+\theta}{2}\right)\right|.
$$
Following an interpolation argument combining   the preceding estimate with \eqref{F1-1} yields \mbox{for any  $\alpha\in[0,1]$} and for $0<\phi_1\leq\phi_2\leq\frac{\pi}{2\theta}$

\begin{equation}\label{tvmf2}
\big|F_1(\rho(\phi_1,\phi_1\theta))-F_1(\rho(\phi_2,\phi_2\theta))\big|\leq C|\phi_1-\phi_2|^\alpha \frac{(1+\theta)^{3\alpha-1}}{|1-\theta|^{2\alpha}}\left|\ln\left(\frac{1+\theta}{2}\right)\right|^\alpha\left(1+\ln\left|\frac{1+\theta}{1-\theta}\right|\right)^{1-\alpha}.
\end{equation}

Plugging this estimate into the definition of $I_{1,2}$ given in \eqref{DecoWX} implies
$$
|I_{1,2}|\leq C|\phi_1-\phi_2|^\alpha{\bigintss_0^{\infty}\frac{\theta^3}{(1+\theta)^{{5}}}\frac{(1+\theta)^{ 3\alpha}}{|1-\theta|^{2\alpha}}\left|\ln\left(\frac{1+\theta}{2}\right)\right|^\alpha\left(1+\ln\left|\frac{1+\theta}{1-\theta}\right|\right)^{1-\alpha}d\theta}.
$$
This integral converges, close to $1$ and at $\infty,$ provided that $0\leq \alpha<1.$ We mention that to get the integrability close to $1$ we use the approximation
$$
\ln\left(\frac{1+\theta}{2}\right)\overset{1}{ \sim} \frac{\theta-1}{2}\cdot
$$
As to the estimate of the term $I_{1,3}$ described in  \eqref{DecoWX}  we roughly implement similar ideas. For that purpose, we introduce the function
\begin{align*}
\forall \, 0\leq s\leq \phi\leq\frac\pi2,\quad G_{2,\phi}(s)=:&\bigintsss_0^{\frac{\pi}{2\phi}} \frac{s (\partial_\phi R)(s,s\theta)}{R^\frac52(s,s\theta)}\sin(s\theta)r_0^2(s\theta)[F_1(\rho(\phi,\phi\theta))-1]d\theta\\
=&\bigintss_0^{\frac{\pi}{2\phi}} \frac{\frac{(\partial_\phi R)(s,s\theta)}{s}}{\frac{R^{\frac52}(s,s\theta)}{s^5}}\frac{\sin(s\theta)}{s}\Big(\frac{r_0(s\theta)}{s}\Big)^2\Big(F_1(\rho(\phi,\phi\theta))-1\Big)d\theta.
\end{align*}
Then combining  \eqref{sinephi2}, \eqref{Thm1} and \eqref{F1-1}, we deduce that
\begin{align*}
\bigintss_0^{\frac{\pi}{2\phi}} &\frac{\frac{|\partial_\phi R(s,s\theta)|}{s}\left|\frac{\sin(s_1\theta)}{s_1}-\frac{\sin(s_2\theta)}{s_2}\right|}{\left(\left(\frac{r_0(s)+r_0(s\theta)}{s}\right)^2+\left(\frac{\cos(s)-\cos(s\theta)}{s}\right)^2\right)^\frac52}\frac{r_0(s\theta)^2}{s^2}\big|F_1(\rho(\phi,\phi\theta))-1\big|d\theta\\
\leq&C|s_1-s_2|^\alpha\bigintsss_0^\infty \frac{1+\theta}{(1+\theta)^5}\theta^{1+\alpha}\theta^2\frac{1}{1+\theta}\left(1+\ln\left|\frac{1+\theta}{1-\theta}\right|\right)d\theta\\
\leq &C |s_1-s_2|^\alpha,
\end{align*}
provided that  $\alpha\in(0,1)$. Implementing the same analysis for the remaining terms and {using \eqref{r0stheta3} and  \eqref{partialR5} } as for $\zeta_{1}$, we find
$$
\forall\, 0\le s_1,s_2\leq\phi,\quad |G_{2,\phi}(s_1)-G_{2,\phi}(s_2)|\leq C|s_1-s_2|^\alpha,
$$
uniformly for $\phi\in(0,\pi/2).$ Therefore from the definition \eqref{DecoWX} we obtain for \mbox{any $0\leq \phi_1\leq\phi_2\leq\frac\pi2,$}
\begin{align*}
|I_{1,3}|=&\big|G_{2,\phi_2}(\phi_1)- G_{2,\phi_2}(\phi_2)\big|\leq C|\phi_1-\phi_2|^\alpha.
\end{align*}
{
Now let us consider the next  term in $\zeta_2(\phi)$
\begin{align*}
\zeta_{2,2}(\phi)=&\bigintsss_0^{\pi}\partial_{\varphi}\mathscr{K}_1(\phi,\varphi)(F_1(\rho(\phi,\varphi))-1)d\varphi \\
=&2r_0(\phi)\bigintss_0^{\pi}
\frac{r_0^2(\varphi)r_0'(\varphi)\sin(\varphi)}{R(\phi,\varphi)^{\frac{5}{2}}}
(F_1(\rho(\phi,\varphi))-1)d\varphi \\
&+r_0(\phi)\bigintss_0^{\pi}
\frac{r_0^3(\varphi)\cos(\varphi)}{R(\phi,\varphi)^{\frac{5}{2}}}
(F_1(\rho(\phi,\varphi))-1)d\varphi \\
&-3r_0(\phi)\bigintss_0^{\pi}
\frac{r_0^2(\varphi)\sin(\varphi)}{R(\phi,\varphi)^{\frac{7}{2}}}
\big((r_0(\phi)+r_0(\varphi))r_0'(\varphi)-(\cos(\varphi)-\cos(\phi))\sin(\varphi)\big)
(F_1(\rho(\phi,\varphi))-1)d\varphi. \\
\end{align*}
The first function that we intend to study is 
$$\phi \mapsto \frac{r_0(\phi)}{\phi}\bigintsss_0^{\pi}
\frac{\phi r_0^2(\varphi)r_0'(\varphi)\sin(\varphi)}{R(\phi,\varphi)^{\frac{5}{2}}}
(F_1(\rho(\phi,\varphi))-1)d\varphi= \frac{r_0(\phi)}{\phi}I(\phi). 
$$
Since the function $\frac{r_0(\phi)}{\phi}$ has bounded derivatives,  it is enough to estimate the function $I$. Thus, 
\begin{align*}
I=&\bigintsss_0^{\pi/2}
\frac{\phi r_0^2(\varphi)r_0'(\varphi)\sin(\varphi)}{R(\phi,\varphi)^{\frac{5}{2}}}
(F_1(\rho(\phi,\varphi))-1)d\varphi+
\bigintsss_{\pi/2}^{\pi}
\frac{\phi r_0^2(\varphi)r_0'(\varphi)\sin(\varphi)}{R(\phi,\varphi)^{\frac{5}{2}}}
(F_1(\rho(\phi,\varphi))-1)d\varphi
\\
=&I_1(\phi)+I_2(\phi). 
\end{align*}
{The arguments to estimate the terms $I_1(\phi)$ and $I_2(\phi)$ are similar to the case of the function $\zeta_1,$ but we will repeat them for the reader convenience. First we will prove that $I_2'(\phi)$ is a bounded function. By direct computations we infer
\begin{align*}
 I_2'(\phi)=&\bigintsss_{\pi/2}^{\pi}
\frac{r_0^2(\varphi)r_0'(\varphi)\sin(\varphi)}{R(\phi,\varphi)^{\frac{5}{2}}}
(F_1(\rho(\phi,\varphi))-1)d\varphi \\
&-\frac{5}{2}\bigintsss_{\pi/2}^{\pi}
\frac{\partial_{\phi}R(\phi,\varphi)\phi r_0^2(\varphi)r_0'(\varphi)\sin(\varphi)}{R(\phi,\varphi)^{\frac{7}{2}}}
(F_1(\rho(\phi,\varphi))-1)d\varphi\\
&\bigintsss_{\pi/2}^{\pi}
\frac{\phi r_0^2(\varphi)r_0'(\varphi)\sin(\varphi)}{R(\phi,\varphi)^{\frac{5}{2}}}
F_1'(\rho(\phi,\varphi))\partial_{\phi}\rho(\phi,\varphi)d\varphi \\
=&I_{2,1}+I_{2,2}+I_{2,3}. 
 \end{align*}
The estimate of $I_{2,2}$ can be done  using \eqref{estimat-8}, \eqref{Rbound} and \eqref{bound-K1}, leading to  
$$
 |I_{2,2}|\leqslant C+\int_{\pi/2}^{\pi}\ln \big|\frac{\phi+\varphi}{\phi-\varphi}\big|d\varphi\le C. 
$$
For the term $I_{2,1}$ we may apply  \eqref{Sing3X}, \eqref{Chord} and \eqref{Rbound} in order to get 
$$
|I_{2,1}|\leqslant \int_{\pi/2}^{\pi}\frac{R(\phi,\varphi)|\phi-\varphi|^2}{|\phi-\varphi|^2R(\phi,\varphi)^4}d\varphi\le C.
$$
To estimate the term $I_{2,3}$ we decompose it in different terms. 
\begin{align*}
  I_{2,3}=&\bigintsss_{\pi/2}^{\pi}
\frac{\phi r_0^2(\varphi)r_0'(\varphi)\sin(\varphi)}{R(\phi,\varphi)^{\frac{5}{2}}}
F_1'(\rho(\phi,\varphi))(\partial_{\phi}\rho(\phi,\varphi)+\partial_{\varphi}\rho(\phi,\varphi))d\varphi\\
&-\bigintsss_{\pi/2}^{\pi}
\frac{\phi r_0^2(\varphi)r_0'(\varphi)\sin(\varphi)}{R(\phi,\varphi)^{\frac{5}{2}}}
F_1'(\rho(\phi,\varphi))\partial_{\varphi}\rho(\phi,\varphi)d\varphi. 
 \end{align*}
Since $r_0'(\pi/2)=r_0(\pi)=0$ then integration by parts allows to get
\begin{align*}
  I_{2,3}=&\bigintsss_{\pi/2}^{\pi}
\frac{\phi r_0^2(\varphi)r_0'(\varphi)\sin(\varphi)}{R(\phi,\varphi)^{\frac{5}{2}}}
F_1'(\rho(\phi,\varphi))(\partial_{\phi}\rho(\phi,\varphi)+\partial_{\varphi}\rho(\phi,\varphi))d\varphi\\
&-\bigintsss_{\pi/2}^{\pi}
\partial_{\varphi}\Big(\frac{\phi r_0^2(\varphi)r_0'(\varphi)\sin(\varphi)}{R(\phi,\varphi)^{\frac{5}{2}}}\Big)
(F_1(\rho(\phi,\varphi))-1)\partial_{\varphi}d\varphi. 
 \end{align*}
 We can then proceed similarly to  the  terms $I_{2,1}$ and $I_{2,2}$ in order to get that  $I_{2,2}$ is bounded.}
 The next step will be to check that the function $I_1$ is in $\mathscr C^{\alpha}(0,\pi/2),$ for all $\alpha\in (0,1).$ If we do the change of variable $\varphi=\phi\theta$ we find
$$
I_1(\phi)=\int_0^{\frac{\pi}{2\phi}}
\frac{\phi^2 r_0^2(\theta\phi)r_0'(\theta\phi)\sin(\theta\phi)}{R(\phi,\theta\phi)^{\frac{5}{2}}}
(F_1(\rho(\phi,\theta\phi))-1)d\theta. 
$$
If $\phi_1$ and $\phi_2$ are in $(0,\pi/2),$ then 
\begin{align*}
 &|I_1(\phi_1)-I_1(\phi_2)|=\bigintss_{\frac{\pi}{2\phi_2}}^{\frac{\pi}{2\phi_1}}
\frac{\phi_1^2 r_0^2(\theta\phi_1)r_0'(\theta\phi_1)\sin(\theta\phi_1)}{R(\phi_1,\theta\phi_1)^{\frac{5}{2}}}
(F_1(\rho(\phi_1,\theta\phi_1))-1)d\theta\\
&+\bigintss_0^{\frac{\pi}{2\phi_2}}
\frac{\phi_1^2 r_0^2(\theta\phi_1)r_0'(\theta\phi_1)\sin(\theta\phi_1)}{R(\phi_1,\theta\phi_1)^{\frac{5}{2}}}
(F_1(\rho(\phi_1,\theta\phi_1)-F_1(\rho(\phi_2,\theta\phi_2))d\theta \\
&+ \bigintss_0^{\frac{\pi}{2\phi_2}}
\Big(\frac{\phi_1^2 r_0^2(\theta\phi_1)r_0'(\theta\phi_1)\sin(\theta\phi_1)}{R(\phi_1,\theta\phi_1)^{\frac{5}{2}}}-
\frac{\phi_2^2 r_0^2(\theta\phi_2)r_0'(\theta\phi_2)\sin(\theta\phi_2)}{R(\phi_2,\theta\phi_2)^{\frac{5}{2}}} \Big)
(F_1(\rho(\phi_2,\theta\phi_2))-1)d\theta\\
=&I_{1,1}+I_{1,2}+I_{1,3}.
\end{align*}
To estimate the function $I_{1,1}$,  let us consider the auxiliary function 
$$
G_{1,\phi}(s)=\bigintsss_0^{\frac{\pi}{2s}}
\frac{\phi^2 r_0^2(\theta\phi)r_0'(\theta\phi)\sin(\theta\phi)}{R(\phi,\theta\phi)^{\frac{5}{2}}}
(F_1(\rho(\phi,\theta\phi))-1)d\theta. 
$$
{First we may write   for $0<\phi_1\leqslant \phi_2 <\pi/2,$ 
$$
|I_{1,1}|=|G_{1,\phi_1}(\phi_2)-G_{1,\phi_1}(\phi_1)|\le \int_{\phi_1}^{\phi_2}|\partial_sG_{1,\phi_1}(s)|ds. 
$$
Hence,  it is enough to obtain an appropriate  bound for $\partial_sG_{1,\phi_1}.$ Taking the derivative in $s$, applying \eqref{F1-1} and some standard estimates used before we have 
$$
|\partial_sG_{1,\phi}(s)|\le C\left(1+\Big|\ln\left(\frac{\pi}{2}-s\right)\Big|\right), 
$$
and so by H\"older inequality
$$
|I_{1,1}|\le C\bigintsss_{\phi_1}^{\phi_2}\left(1+\Big|\ln\left(\frac{\pi}{2}-s\right)\Big|\right)ds \le C |\phi_1-\phi_2|^{\alpha}, 
$$
for all $\alpha \in (0,1).$
Let us now move to the estimate of  the term $I_{1,2}.$ By \eqref{tvmf2} and some standard estimates used along the work we obtain, for any $\alpha\in (0,1)$ 
\begin{align*}
|I_{1,2}|&\le  C|\phi_1-\phi_2|^{\alpha} \bigintsss_0^{\infty}\frac{\theta^3 (1+\theta)^{3\alpha-1}}{(1+\theta)^5(1-\theta)^{2\alpha}}\ln^{\alpha}(\frac{1+\theta}{2})\left(1+\ln|\frac{1+\theta}{1-\theta}|\right)^{1-\alpha}d\theta\\
& \le C|\phi_1-\phi_2|^{\alpha}, 
\end{align*}
where in the last inequality we have used that $\ln(\frac{1+\theta}{2})\simeq \frac{\theta-1}{2}$ if $\theta$ is enough close to $1.$}\\
To estimate the term $I_{1,3}, $ let us take the function 
\begin{align*}
G_{2,\phi}(s)=& \bigintss_0^{\frac{\pi}{2\phi}}
\frac{s^2 r_0^2(\theta s)r_0'(\theta s)\sin(\theta s)}{R(s,\theta s)^{\frac{5}{2}}}
(F_1(\rho(\phi,\theta\phi))-1)d\theta\\ 
=&\bigintss_0^{\frac{\pi}{2\phi}}
\frac{(\frac{r_0(\theta s)}{s})^2r_0'(\theta s)(\frac{\sin(\theta s)}{s})}{\frac{R( s,\theta s)^{\frac{5}{2}}}{s^5}}
(F_1(\rho(\phi,\theta\phi))-1)d\theta. 
\end{align*}
Now, $|I_{2,3}|=|G_{2,\phi_2}(\phi_1)-G_{2,\phi_2}(\phi_2)|. $ As in the case of the function $\zeta_1$ we will get the desired estimate through decomposing the integral in several terms. 
The first term to consider is 
$$
\bigintsss_0^{\frac{\pi}{2\phi}}
\frac{\big((\frac{r_0(\theta s_1)}{s_1})^2-(\frac{r_0(\theta s_2)}{s_2})^2\big)
r_0'(\theta s_1)(\frac{\sin(\theta s_1)}{s_1})}{\frac{R(s_1,\theta s_1)^{\frac{5}{2}}}{s_ 1^5}}
(F_1(\rho(\phi,\theta\phi))-1)d\theta. 
$$
Using \eqref{r0stheta3} and \eqref{F1-1} one can see that this term is bounded by 
$$
C |s_1-s_2|^{\alpha}\bigintsss_0^{\infty}\frac{\theta^{3+\alpha}}{(1+\theta)^6}(1+\ln\big|\frac{1+\theta}{1-\theta}\big|)d\theta \le C |s_1-s_2|^{\alpha}.
$$
For the next adding term, using \eqref{sinephi2} and \eqref{F1-1} we obtain 
\begin{align*}
&\bigintsss_0^{\frac{\pi}{2\phi}}
\frac{(\frac{r_0(\theta s_2)}{s_2})^2
r_0'(\theta s_1)(\frac{\sin(\theta s_1)}{s_1}-\frac{\sin(\theta s_2)}{s_2})}{\frac{R(s_1,\theta s_1)^{\frac{5}{2}}}{s_ 1^5}}
(F_1(\rho(\phi,\theta\phi))-1)d\theta \\
& \le C |s_1-s_2|^{\alpha}\int_0^{\infty}\frac{\theta^{3+\alpha}}{(1+\theta)^6}\left(1+\ln\big|\frac{1+\theta}{1-\theta}\big|\right)d\theta\\
& \le C |s_1-s_2|^{\alpha}.
\end{align*}
For the next term we will use the regularityof $r_0$ and \eqref{F1-1},  obtaining 
\begin{align*}
&\bigintsss_0^{\frac{\pi}{2\phi}}
\frac{(\frac{r_0(\theta s_2)}{s_2})^2
(r_0'(\theta s_1)-r_0'(\theta s_2))(\frac{\sin(\theta s_2)}{s_2})}{\frac{R(s_1,\theta s_1)^{\frac{5}{2}}}{s_ 1^5}}
(F_1(\rho(\phi,\theta\phi))-1)d\theta \\
& \le C |s_1-s_2|^{\alpha}\bigintsss_0^{\infty}\frac{\theta^{3+\alpha}}{(1+\theta)^6}\left(1+\ln\big|\frac{1+\theta}{1-\theta}\big|\right)d\theta\\
& \le C |s_1-s_2|^{\alpha}.
\end{align*}
Let us move to the last term in $I_{1,3}.$ By \eqref{F1-1} and \eqref{partialR5}
\begin{align*}
&\left|\bigintss_0^{\frac{\pi}{2\phi}}{\frac{\sin(s_2 \theta)}{s_2} \left(\frac{r_0(s_2 \theta)}{s_2}\right)^2}(r_0'(s_2\theta))\left(s_1^5R^{-\frac52}(s_1,s_1\theta)-s_2^5R^{-\frac52}(s_2,s_2\theta)\right)(F_1(\rho(\phi,\theta\phi))-1)d\theta\right|\\
&\lesssim|s_1-s_2|^\alpha\bigintss_0^{\infty}\frac{\theta ^3}{(1+\theta)^{6-\alpha}}(1+\ln\big|\frac{1+\theta}{1-\theta})d\theta\\
&\lesssim|s_1-s_2|^\alpha. 
\end{align*} 
The two remaining terms in $\zeta_2$ are left to the reader because they are very similar to the first one. 
}It remains to estimate the term  $\zeta_3$ defined by \eqref{Deriv1}. It can be split as follows,
\begin{align*}
\zeta_3(\phi)=&{\frac{1}{4\pi}}\bigintsss_0^\pi\frac{\sin(\varphi)r_0^2(\varphi)}{R^\frac32(\phi,\varphi)}(\partial_\phi\rho(\phi,\varphi)+\partial_\varphi\rho(\phi,\varphi))\left[F_1'(\rho(\phi,\varphi))-\frac34\right]d\varphi\\
=&\frac{1}{4\pi}\bigintsss_0^{\frac{\pi}{2}}\frac{\sin(\varphi)r_0^2(\varphi)}{R^\frac32(\phi,\varphi)}(\partial_\phi\rho(\phi,\varphi)+\partial_\varphi\rho(\phi,\varphi))\left[F_1'(\rho(\phi,\varphi))-\frac34\right]d\varphi\\
&+\frac{1}{4\pi}\bigintsss_{\frac{\pi}{2}}^\pi\frac{\sin(\varphi)r_0^2(\varphi)}{R^\frac32(\phi,\varphi)}(\partial_\phi\rho(\phi,\varphi)+\partial_\varphi\rho(\phi,\varphi))\left[F_1'(\rho(\phi,\varphi))-\frac34\right]d\varphi\\
=&I_{3}+I_4.
\end{align*}
{Recall that we restrict to check the regularity for $\phi\in(0,\pi/2)$ without loss of generality and later we extend it to $\phi\in(0,\pi)$. Then, we} will show that $I_3$ belongs to $\mathscr{C}^\alpha(0,\frac\pi2)$. {We will skip here the details for the regularity of $I_4$ since this term is less singular and the same procedure works, see the estimates for $I_2$ previously done.} To estimate  $I_3$ we proceed as before through the use of  the change of variables $\varphi=\phi\theta$,
\begin{align*}
I_3(\phi)=&\frac{1}{4\pi}\bigintsss_0^{\frac{\pi}{2\phi}}\frac{\phi\sin(\phi\theta)r_0(\phi\theta)^2}{R(\phi,\phi\theta)^\frac32}((\partial_\phi\rho)(\phi,\theta\phi)+(\partial_\varphi\rho)(\phi,\phi\theta))\left[F_1'(\rho(\phi,\phi\theta))-\frac34\right]d\theta.
\end{align*}
Define the functions
\begin{align*}
H_{1,\phi}(s):=&\bigintsss_0^{\frac{\pi}{2s}}\frac{\phi\sin(\phi\theta)r_0^2(\phi\theta)}{R^\frac32(\phi,\phi\theta)}\Big((\partial_\phi\rho)(\phi,\theta\phi)+(\partial_\varphi\rho)(\phi,\phi\theta)\Big)\left[F_1'(\rho(\phi,\phi\theta))-\frac34\right]d\theta,\\
H_{2,\phi}(s):=&\bigintsss_0^{\frac{\pi}{2\phi}}\frac{s\sin(s\theta)r_0^2(s\theta)}{R^2(s,s\theta)}R^\frac12(\phi,\phi\theta)\Big((\partial_\phi\rho)(\phi,\theta\phi)+(\partial_\varphi\rho)(\phi,\phi\theta)\Big)\left[F_1'(\rho(\phi,\phi\theta))-\frac34\right]d\theta,\\
H_{3,\phi}(s):=&\bigintsss_0^{\frac{\pi}{2\phi}}\frac{\phi\sin(\phi\theta)r_0^2(\phi\theta)}{R^2(\phi,\phi\theta)}R^\frac12(s,s\theta)\Big((\partial_\phi\rho)(s,\theta s)+(\partial_\varphi\rho)(s,s\theta)\Big)\left[F_1'(\rho(\phi,\phi\theta))-\frac34\right]d\theta,\\
H_{4,\phi}(s):=&\bigintsss_0^{\frac{\pi}{2\phi}}\frac{\phi\sin(\phi\theta)r_0^2(\phi\theta)}{R^2(\phi,\phi\theta)}R^\frac12(\phi,\phi\theta)\Big((\partial_\phi\rho)(\phi,\theta\phi)+(\partial_\varphi\rho)(\phi,\phi\theta)\Big)\left[F_1'(\rho(s,s\theta))-\frac34\right]d\theta.
\end{align*}
In order to check that $I_3$ belongs to  $\mathscr{C}^\alpha(0,\frac\pi2)$, it suffices to prove that each function $H_{i,\phi}$ is  in $\mathscr{C}^\alpha(0,\frac\pi2)$ uniformly in $\phi\in(0,\pi/2)$, for any $i=1,\dots, 4$. Let us start with $H_{1,\phi}$ showing that its derivative is bounded. 

From straightforward calculus it is easy to check that for any $0\leq\phi\leq s<\frac\pi2,$
\begin{align*}
|H_{1,\phi}'(s)|\leq &\frac{\pi}{2s^2}\frac{\phi\sin\left(\frac{\phi\pi}{2s}\right)r_0^2\left(\frac{\phi\pi}{2s}\right)}{R^\frac32\left(\phi,\frac{\phi\pi}{2s}\right)}\Big|(\partial_\phi\rho)\left(\phi,{\phi\pi}/{2s}\right)+(\partial_\varphi\rho)\left(\phi,{\phi\pi}/{2s}\right)\Big|\left|F_1'\left(\rho\big(\phi,{\phi\pi}/{2s}\big)\right)-\frac34\right|.
\end{align*}
Hence, we obtain 
\begin{align*}
|H_{1,\phi}'(s)|\lesssim &\frac{1}{s^5}\frac{\phi^4}{\phi^3(1+\frac{\pi}{2s})^3}\Big|(\partial_\phi\rho)\left(\phi,{\phi\pi}/{2s}\right)+(\partial_\varphi\rho)\left(\phi,{\phi\pi}/{2s}\right)\Big|\left|F_1'\left(\rho\big(\phi,{\phi\pi}/{2s}\big)\right)-\frac34\right|.
\end{align*}
Using \eqref{Sing0}--\eqref{Sing3X}--\eqref{Sing4} allows to get 
\begin{align*}
|H_{1,\phi}'(s)|\lesssim &\frac{1}{s^5}\frac{\phi^4}{\phi^3(1+\frac{\pi}{2s})^3}\frac{\phi^2\left(1-\frac{\pi}{2s}\right)^2}{\phi^3\left(1+\frac{\pi}{2s}\right)^3}\frac{\phi^2\frac{\pi}{2s}}{\phi^2\left(1-\frac{\pi}{2s}\right)^2}\lesssim 1,
\end{align*}
which is uniformly bounded  on $0\leq\phi\leq s<\frac\pi2$. We shall skip the details   for $H_{2,\phi}$  which can be analyzed following the same lines of  the term  $T_{2,\phi}$ introduced in \eqref{T2phi1}.

Let us now focus on the estimate of  $H_{3,\phi}$. Set
$$
\mathscr{T}(\theta,s):=R^\frac12(s,s\theta)\big((\partial_\phi\rho)(s, s\theta)+(\partial_\varphi\rho)(s,s\theta)\big),
$$
then  using \eqref{Sing3X}, we  deduce 
\begin{equation}\label{T-1}
|\mathscr{T}(\theta,s)|\leq C\frac{(1-\theta)^2}{(1+\theta)^2}\cdot
\end{equation}
By ranging the expression of $\mathscr{T}$ as follows 
\begin{align*}
\mathscr{T}(\theta,s)=&4\frac{\frac{r_0^2(s\theta)}{s^2}-\frac{r_0^2(s)}{s^2}}{\frac{R^\frac32(s,s\theta)}{s^{3}}}r_0^\prime(s)\left(\frac{r_0(s\theta)}{s}-\frac{r_0(s)}{s}\right)\\
&+4\frac{\frac{r_0^2(s\theta)}{s^2}-\frac{r_0^2(s)}{s^2}}{\frac{R^\frac32(s,s\theta)}{s^3}}\frac{r_0(s)}{s}\Big(r_0^\prime(s)-r_0^\prime(s\theta)\Big)\\
\nonumber &{+8s\frac{\frac{r_0(s\theta)}{s}\frac{r_0(s)}{s}}{\frac{R^\frac32(s,s\theta)}{s^3}}\frac{\big(\cos s-\cos s\theta)}{s}\left(\frac{\sin s}{s}-\frac{\sin s\theta}{s}\right)}\\
&+4\frac{\frac{(\cos s-\cos s\theta)^2}{s^2}}{\frac{R^\frac32(s,s\theta)}{s^3}}\left(\frac{r_0(s\theta)}{s}r_0^\prime(s)+\frac{r_0(s)}{s}{r_0^\prime(s\theta)}\right),
\end{align*}
and differentiating with respect to $s$  we find
\begin{equation}\label{T-2}
|\partial_s \mathscr{T}(\theta,s)|\leq C\frac{|1-\theta|}{1+\theta}.
\end{equation}
{We will not give the full details for this estimate because the computations are long and tedious, but to get a more precise idea how this works we shall just explain  the estimate of  the first term in $\partial_s \mathscr{T}$ given by 
$$
\forall \,\,0<\theta s< \frac\pi2{, s<\frac\pi2},\quad \mathscr{T}_1(\theta,s)=:4\frac{\partial_s \left(\frac{r_0(s\theta)-r_0(s)}{s}\right)\frac{r_0(s\theta)+r_0(s)}{s}}{\frac{R^\frac32(s,s\theta)}{s^3}}r_0^\prime(s)\frac{r_0(s\theta)-r_0(s)}{s}\cdot
$$
{Note that the other terms can be treated similarly and we use similar estimates but with $\sin$, $\cos$ and $r_0'$ instead of $r_0$. In particular, here we use $r_0\in \mathscr{C}^2$ in order to bound $r_0''$. }
Define 
$$
\forall \,\,0<\theta s< \frac\pi2,{s<\frac\pi2}\quad g(\theta,s):=\frac{r_0(s\theta)-r_0(s)}{s}\cdot$$
Then, one has $\partial_\theta g(\theta,s)=r_0'(s\theta)$ and then
 \begin{equation}\label{g-st}|\partial_s\partial_\theta g(\theta,s)|=|\theta r_0''(s\theta)|\leq C\theta.\end{equation}
Since $g(1,s)=0$, we can write by Taylor's formula
$$
g(\theta,s)=\int_1^\theta\partial_\theta g(\tau,s)d\tau,
$$
and hence
$$
\partial_s g(\theta,s)=\int_1^\theta \partial_s \partial_\theta g(\tau,s)d\tau.
$$
Using \eqref{g-st}, we achieve
$$
\forall \,\,0<\theta s< \frac\pi2,\quad |\partial_s g(\theta,s)|\leq C|1-\theta|{(1+\theta)}.
$$
Plugging this into the the definition of  $\mathscr{T}_1$  and using the mean value theorem yields to the estimate
$$
|\mathscr{T}_1|\leq C\frac{|1-\theta|{(1+\theta)^2}|1-\theta|}{(1+\theta)^3}\leq C\frac{(1-\theta)^2}{(1+\theta)}\cdot
$$
Now, interpolating between \eqref{T-1} and \eqref{T-2}, we find that for any $\alpha\in(0,1)$
\begin{equation}\label{T-3}
|\mathscr{T}(\theta,s_1)-\mathscr{T}(\theta,s_2)|\leq C|s_1-s_2|^\alpha\frac{|1-\theta|^{2-\alpha}}{(1+\theta)^{2-\alpha}}.
\end{equation}
Using \eqref{Sing4} we get
\begin{equation}\label{est-f1prime}
\forall \, 0\leq \phi\theta\leq\frac\pi2,\quad \left|F_1'(\rho(\phi,\phi\theta))-\frac34\right|\leq C\frac{\theta}{(1-\theta)^2}.
\end{equation}
Combining this estimate with  \eqref{T-3} and \eqref{Sing4}, we conclude that for any $0\leq s_1,s_2\leq\phi\leq\frac\pi2$
\begin{align}\label{T-4}
 |H_{3,\phi}(s_1)-H_{3,\phi}(s_2)|\leq& C|s_1-s_2|^\alpha\int_0^{+\infty}\frac{\phi^4\theta^3}{\phi^4(1+\theta)^4}\frac{(1-\theta)^{2-\alpha}}{(1+\theta)^{2-\alpha}}\frac{\theta}{(1-\theta)^2} d\theta \leq C|s_1-s_2|^\alpha,
\end{align}
for any $\alpha\in(0,1)$. 
Let us finish working with $H_{4,\phi}$.
{ Using \eqref{r0phi} and the standard inequality $$
\Big|\frac{\cos(\theta\phi)-1}{\phi}\Big|\le \theta,$$
one gets 

\begin{equation}\label{AAA}
\Big|\frac{1}{1-\rho(\phi\theta,\phi)}\Big|\le\frac{(1+\theta)^2}{(1-\theta)^2}.
\end{equation}
As a consequence of \eqref{estimat-7}, \eqref{est-R} and \eqref{AAA} one has}

\begin{equation}\label{est-f1prime2}
\left|\partial_s \left(F_1'(\rho(s,s\theta))-\frac34\right)\right|\leq C|F_1''(\rho(s,s\theta))||\partial_s(\rho(s,s\theta))| \leq C\frac{(1+\theta)^4}{(1-\theta)^4}
\left|\ln\left(\frac{1+\theta}{2}\right)\right|.
\end{equation}
Interpolating between \eqref{est-f1prime} and \eqref{est-f1prime2} we achieve
\begin{align}\label{est-f1prime3}
\left|F_1'(\rho(s_1,s_1\theta))-F_1'(\rho(s_2,s_2\theta))\right|\leq& C|s_1-s_2|^\alpha\frac{(1+\theta)^{4\alpha}}{(1-\theta)^{4\alpha}}\left|\ln\left(\frac{1+\theta}{2}\right)\right|^\alpha\frac{\theta^{1-\alpha}}{(1-\theta)^{2(1-\alpha)}}\nonumber\\
\leq &C|s_1-s_2|^\alpha\frac{(1+\theta)^{1+3\alpha}}{(1-\theta)^{2+2\alpha}}\left|\ln\left(\frac{1+\theta}{2}\right)\right|^\alpha.
\end{align}
Finally, using  \eqref{T-1} and \eqref{est-f1prime3} we obtain for any  $0\leq s_1,s_2\leq\phi$
\begin{align*}
|H_{4,\phi}(s_1)-H_{4,\phi}(s_2)|\leq& C|s_1-s_2|^\alpha\bigintsss_0^{\infty}\frac{\theta^3}{(1+\theta)^4}\frac{(1-\theta)^2}{(1+\theta)^{2}}\frac{(1+\theta)^{1+3\alpha}}{(1-\theta)^{2+2\alpha}}\left|\ln\left(\frac{1+\theta}{2}\right)\right|^\alpha d\theta\\
\leq& C|s_1-s_2|^\alpha\bigintsss_0^{\infty}{(1+\theta)^{3\alpha-2}|1-\theta|^{-2\alpha}}\left|\ln\left(\frac{1+\theta}{2}\right)\right|^\alpha d\theta\\
\leq& C|s_1-s_2|^\alpha,
\end{align*}
the convergence of the integral is guaranteed pro\begin{equation}\label{sinephi}
\Big|\partial_\phi\left(\frac{\sin(\phi \theta)}{\phi} \right)\Big|\leq \theta^2.
\end{equation}vided that $\alpha\in(0,1)$. This achieves the proof of $\nu_\Omega\in \mathscr{C}^{1,\alpha}(0,\pi)$ for any $\alpha\in(0,1).$
}
}

\medskip
\noindent
{\bf{(4)}} Since the function  $\nu_\Omega$ reaches its minimum at a point $\phi_0\in[0,\pi]$, we have that if this point belongs to the open set $(0,\pi)$ then necessary $\nu_\Omega^\prime(\phi_0)=0$.  However when $\phi_0\in\{0,\pi\}$ then from the point {\bf{(3)}} of Proposition \ref{Lem-meas} we deduce also that the derivative is vanishing at $\phi_0.$ Using the mean value theorem, we obtain for any $\phi\in[0,\pi]$
\begin{align*}
\nu_\Omega(\phi)=&\nu_\Omega(\phi_0)+\nu_\Omega^\prime\big(\overline\phi\big)(\phi-\phi_0)=\nu_\Omega(\phi_0)+\big(\nu_\Omega^\prime\big(\overline\phi\big)-\nu_\Omega^\prime\big(\phi_0\big)\big)(\phi-\phi_0),
\end{align*}
for some $\overline\phi\in(\phi_0,\phi)$. Since $\nu_\Omega^\prime\in \mathscr{C}^\alpha$ then 
$$
\Big|\nu_\Omega^\prime\big(\overline\phi\big)-\nu_\Omega^\prime\big(\phi_0\big)\Big|\leq\|\nu_\Omega^\prime\|_{\mathscr{C}^\alpha}|\phi-\phi_0|^\alpha.
$$
Notice that $\|\nu_\Omega^\prime\|_{\mathscr{C}^\alpha}$ is independent of $\Omega.$ Consequently
$$
\forall \phi\in[0,\pi],\quad 0\leq \nu_\Omega\big(\phi\big)-\nu_\Omega\big(\phi_0\big)\leq C|\phi-\phi_0|^{1+\alpha},
$$
for some absolute constant $C$. In the particular case $\Omega=\kappa$ we get from the definition \eqref{kappa} that  $\nu_\kappa(\phi_0)=0$ and therefore the preceding result becomes
$$
\forall \phi\in[0,\pi],\quad 0\leq \nu_\kappa\big(\phi\big)\leq C|\phi-\phi_0|^{1+\alpha},\quad \nu_\kappa\big(\phi_0\big)=0.
$$
\end{proof}

\subsection{Eigenvalue problem}\label{Eigenvalue problem}
In Section \ref{Symmetriz} we have checked that the operator $\mathcal{L}_n^\Omega$ defined in \eqref{operator} is of integral type.  Then studying the kernel of this operator reduces to solving the integral equation
\begin{equation}\label{kerneleq}
\mathcal{K}_n^{{\Omega}} h_n(\phi):=\int_0^\pi K_n(\phi,\varphi)h_n(\varphi)d\mu_{\Omega}(\varphi)=h_n(\phi),\quad \forall \phi\in [0,\pi],
\end{equation}
where the kernel $K_n$  and the measure $d\mu_\Omega$ are  defined successively in \eqref{K-kernel} and \eqref{signed-meas}. The parameter $\Omega$ ranges  over the interval $(-\infty,\kappa)$. This latter  condition is imposed to guarantee the positivity of the measure $d\mu_\Omega$ through the positivity of $\nu_\Omega$ according to \mbox{Lemma $\ref{Lem-meas}$.} 
We point out that  studying the kernel of $\mathcal{L}_n^\Omega$  amounts to finding the values of $\Omega$ such that  $1$ is an eigenvalue of $\mathcal{K}_n^{{\Omega}}$.
To investigate the spectral study of $\mathcal{K}_n^{{\Omega}}$ we need to 
introduce the Hilbert space $L^2_{\mu_\Omega}$ of measurable functions $f:[0,\pi]\to \mathbb{R}$ such that
\begin{equation}\label{muomega}
\|f\|_{\mu_\Omega}:=\left(\int_0^\pi|f(\varphi)|^2d\mu_{\Omega}(\varphi)\right)^{\frac12}<\infty.
\end{equation}

Notice that the space $L^2_{\mu_\Omega}$ is equipped with the usual inner product:
\begin{align}\label{scalar-prod1}
 \langle f,g \rangle_{{\Omega}} =\int_{0}^{\pi}f(\varphi)g(\varphi)d\mu_{\Omega}(\varphi), \quad \forall\, f,g\in L^2_{\mu_\Omega}.
\end{align}
\begin{rems}\label{remark-imp1}
\begin{enumerate}
\item Since $d\mu_\Omega$ is a nonnegative bounded Borel measure for any $\Omega\in(-\infty,\kappa)$, then the Hilbert space  $L^2_{\mu_\Omega}$ is  separable.
\item For any $\Omega\in(-\infty,\kappa)$, the space $L^2_{\mu_\Omega}$ is isomorphic to the space $L^2_{\mu}$ where 
$$
d\mu(\varphi)=\sin(\varphi)\, r_0^2(\varphi)\, d\varphi.
$$
This follows from Proposition $\ref{Lem-meas}$-$(2)$which ensures that $\nu_\Omega$ is nowhere vanishing. However this property fails for the critical value $\Omega=\kappa$ because $\nu_\kappa$ is vanishing at some points.

\end{enumerate}
\end{rems}
The next proposition deals with some basic  properties of the operator $\mathcal{K}_n^{{\Omega}}.$

\begin{pro}\label{prop-operator}
Let $\Omega\in(-\infty,\kappa)$ and $r_0$ satisfies the assumptions {\bf{(H1)}} and {\bf{(H2)}}. Then, the following assertions hold true.
\begin{enumerate}
\item For any $n\geq1$, the operator $\mathcal{K}_n^{{\Omega}}:L^2_{\mu_\Omega}\rightarrow L^2_{\mu_\Omega}$ is   Hilbert--Schmidt   and self-adjoint.
\item For any $n\geq1$, the eigenvalues of $\mathcal{K}_n^{{\Omega}}$ form a countable family of real numbers. Let  $\lambda_n(\Omega)$ be  the largest eigenvalue, then it is strictly positive   and satisfies
$$
\bigintsss_0^\pi\bigintsss_0^\pi\frac{H_n(\phi,\varphi)\sin^\frac12 (\phi)r_0(\phi)\varrho(\varphi)\varrho(\phi)}{\nu_\Omega^{\frac12}(\varphi)\nu_\Omega^{\frac12}(\phi)\sin^\frac12(\varphi) r_0(\varphi)} d\varphi d\phi\leq \lambda_n(\Omega)\leq {\left\{\bigintsss_0^\pi\bigintsss_0^\pi K_n^2(\phi,\varphi)d\mu_\Omega(\varphi)\,d\mu_\Omega(\phi)\right\}^\frac12,}$$
for any function $\varrho$ such that $\displaystyle{\int_0^\pi \varrho^2(\varphi)d\varphi=1.}$
\item We have the following decay: for any $\alpha\in[0,1)$ there exists $C>0$ such that 
$$
\forall\, \Omega\in(-\infty,\kappa),\,\forall\, n\geq1,\quad \bigintsss_0^\pi\bigintsss_0^\pi K_n^2(\phi,\varphi)d\mu_\Omega(\varphi)\,d\mu_\Omega(\phi)\leq {C(\kappa-\Omega)^{-2} n^{-2\alpha}}.
$$
\item The eigenvalue $\lambda_n(\Omega)$ is simple {{and the associated nonzero 
 eigenfunctions do not vanish in $(0,\pi)$.}}
\item For any $\Omega\in(-\infty,\kappa)$, the sequence $n\in\N^{\star}\mapsto \lambda_n(\Omega)$ is strictly decreasing.
\item For any $n\geq1$ the map $\Omega\in(-\infty,\kappa)\mapsto \lambda_n(\Omega)$ is differentiable and strictly increasing.
\end{enumerate}
\end{pro}

\begin{proof}

\medskip
\noindent
{\bf (1)}  In order to check that $\mathcal{K}_n^{{\Omega}}$ is a Hilbert--Schmidt operator, we need to verify 
that the kernel $K_n$  satisfies the integrability condition
$$
\|\mathcal{K}_n^{{\Omega}}\|_{\mu_{\Omega}}:=\left(\int_0^\pi\int_0^\pi |K_n(\phi,\varphi)|^2d\mu_{\Omega}(\varphi)d\mu_{\Omega}(\phi)\right)^{\frac12}<+\infty.
$$
Indeed, by \eqref{K-kernel} and \eqref{H-1}, one gets
\begin{align*}
\|\mathcal{K}_n^{{\Omega}}\|_{\mu_{\Omega}}^2=&C_n\bigintsss_0^\pi\bigintsss_0^\pi\frac{\sin(\varphi)\sin(\phi)\,r_0^{2n}(\phi)r_0^{2n}(\varphi)}{R^{2n+1}(\phi,\varphi)\,\nu_{\Omega}(\varphi)\nu_{\Omega}(\phi)}\,F_n^2\left(\frac{4r_0(\phi)r_0(\varphi)}{R(\phi,\varphi)}\right)d\varphi d\phi,
\end{align*}
for some constant $C_n$ and $R$ was defined in \eqref{RefR}. Remark that 
$$
\left|\frac{4r_0(\phi)r_0(\varphi)}{R(\phi,\varphi)}\right|\leq 1.
$$
Moreover, according to Proposition \ref{Lem-meas} the function  $\nu_{\Omega}(\varphi)$ is not vanishing in the interval $[0,\pi]$ provided that $\Omega<\kappa$. Therefore we get 
\begin{align*}
\|\mathcal{K}_n^{{\Omega}}\|_{\mu_{\Omega}}^2\lesssim&\bigintsss_0^\pi\bigintsss_0^\pi\frac{\sin(\varphi)\sin(\phi)}
{{ R(\phi, \varphi)}}
F_n^2\left(\frac{4r_0(\phi)r_0(\varphi)}{R(\phi,\varphi)}\right)d\varphi d\phi.
\end{align*}
By \eqref{estimat-1} and the assumption {\bf{(H2)}} we deduce that
\begin{align*}
\|\mathcal{K}_n^{{\Omega}}\|_{\mu_{\Omega}}^2\leq&C+C\bigintsss_0^\pi\bigintsss_0^\pi
\ln^2\left(1-\frac{4r_0(\phi)r_0(\varphi)}{R(\phi,\varphi)}\right)d\varphi d\phi\\
\leq&C+C\bigintsss_0^\pi\bigintsss_0^\pi
\ln^2\left(\frac{(r_0(\phi)-r_0(\varphi))^2+(\cos\phi-\cos\varphi)^2}{R(\phi,\varphi)}\right)d\varphi d\phi.
\end{align*}
It suffices now to use the inequality \eqref{Est-LogX} to get
\begin{align*}
\|\mathcal{K}_n^{{\Omega}}\|_{\mu_{\Omega}}^2\leq&C+C\bigintsss_0^\pi\bigintsss_0^\pi
\ln^2\left(\frac{\sin\phi+\sin\varphi}{|\phi-\varphi|}\right)d\varphi d\phi<\infty.
\end{align*}

This concludes that the operator $\mathcal{K}_n^{{\Omega}}$ is bounded and is of  Hilbert--Schmidt type. 
As a consequence from the general theory this operator is necessarily compact.

On the other hand, as we have mentioned before the kernel $K_n$ is symmetric in view of the formula \eqref{form-K} and the symmetry of $R$ defined in \eqref{RefR}. 
Therefore we deduce that  $\mathcal{K}_n^{{\Omega}}$ is a self--adjoint operator

\medskip
\noindent
{\bf (2)}
From the spectral theorem on self-adjoint compact operators, we know that the eigenvalues of $\mathcal{K}_n^{{\Omega}}$ form a countable family of real numbers. Define the real numbers 
$$
m=\inf_{\|h\|_{\mu_{\Omega}}=1} \langle\mathcal{K}_n^{{\Omega}} h, h\rangle_{{\Omega}} \quad \textnormal{and} \quad M=\sup_{\|h\|_{\mu_{\Omega}}=1} \langle\mathcal{K}_n^{{\Omega}} h, h\rangle_{{\Omega}}.
$$
Since $\mathcal{K}$ is self-adjoint, we obtain $\sigma (\mathcal{K}_n^{{\Omega}})\subset[m,M]$, with $m\in\sigma (\mathcal{K}_n^{{\Omega}})$ and $M\in\sigma(\mathcal{K}_n^{{\Omega}})$, where the set  $\sigma(\mathcal{K}_n^{{\Omega}})$ denotes  the spectrum of $\mathcal{K}_n^{{\Omega}}$. Since $\lambda_n(\Omega)$ is the largest eigenvalue, then
\begin{equation}\label{eigenv}
\lambda_n(\Omega)=M=\sup_{\|h\|_{\mu_{\Omega}}=1} \langle\mathcal{K}_n^{{\Omega}} h, h\rangle_{{\Omega}}.
\end{equation}
We shall prove that $M>0$ and $|m|\leq M$. Indeed, for any $h\in L^2_{\mu_\Omega}$, the positive  function $|h|$ belongs also  to $ L^2_{\mu_\Omega}$ with  the same norm  and using the positivity of the kernel $K_n$ we obtain
$$
\sup_{\|h\|_{\mu_{\Omega}}=1} \langle\mathcal{K}_n^{{\Omega}} h, h\rangle_{{\Omega}}=\sup_{h\geq0, \|h\|_{\mu_{\Omega}}=1} \langle\mathcal{K}_n^{{\Omega}} h, h\rangle_{{\Omega}}.
$$
Using once again  the positivity of the kernel one deduces that 
$$
\forall h\geq 0, \|h\|_{\mu_{\Omega}}=1\Longrightarrow \langle\mathcal{K}_n^{{\Omega}} h, h\rangle_{{\Omega}}>0.
$$
Consequently, we obtain that $M>0$. In order to prove that $|m|\le M$, we shall proceed as follows. Using the positivity of the kernel, we achieve
$$
 |m| \leq \langle\mathcal{K}_n^{{\Omega}} |h|, |h|\rangle_{{\Omega}}\leq M,\quad \forall \,\|h\|_{\mu_\Omega}=1.
$$
This implies that $M$ is nothing but the spectral radius of the operator $\mathcal{K}_n^{{\Omega}}$, that is,
$$
M=\|\mathcal{K}_n^{{\Omega}}\|_{\mathcal{L}(L^2_{\mu_\Omega})}.
$$
From the  Cauchy--Schwarz inequality, one deduces that 
$$
\|\mathcal{K}_n^{{\Omega}}\|_{\mathcal{L}(L^2_{\mu_\Omega})}^2
\leq \int_0^\pi\int_0^\pi|K_n(\phi,\varphi)|^2d\mu_{\Omega}(\phi)d\mu_{\Omega}(\varphi),
$$
which implies that
$$
\lambda_n^2(\Omega)\leq \int_0^\pi\int_0^\pi|K_n(\phi,\varphi)|^2d\mu_{\Omega}(\phi)d\mu_{\Omega}(\varphi).
$$
For the lower bound, we shall work with the special function
{$$
f(\varphi)=\frac{\varrho(\varphi)}{\sin(\varphi)^\frac12 r_0(\varphi)\nu_{\Omega}(\varphi)^{\frac12}},\quad \varphi\in(0,\pi),
$$}
with the  normalized condition $\|f\|_{\mu_{\Omega}}=1$ which is equivalent to 
$$
\int_0^\pi\varrho^2(\varphi)d\varphi=1
$$
and
$$
\lambda_n(\Omega)\geq\langle\mathcal{K}_n^{{\Omega}}f,f\rangle_{{\Omega}}=\bigintsss_0^\pi\bigintsss_0^\pi\frac{H_n(\phi,\varphi)}{\nu_\Omega^{\frac12}(\varphi)\nu_\Omega^{\frac12}(\phi)}\frac{\sin^\frac12(\phi) r_0(\phi)}{\sin^\frac12(\varphi) r_0(\varphi)} \varrho(\varphi)\varrho(\phi)d\varphi d\phi.
$$
This gives the announced lower bound for the largest eigenvalue.

\medskip
\noindent
{\bf (3)}
From the expression of $K_n$ given by \eqref{K-kernel} we easily get
$$
 \|\mathcal{K}_n^{{\Omega}}\|_{\mu_\Omega}^2{ \le}\bigintsss_0^\pi\bigintsss_0^\pi 
 K_n^2(\phi,\varphi)d\mu_\Omega(\varphi)\,d\mu_\Omega(\phi)=\bigintsss_0^\pi\bigintsss_0^\pi 
 \frac{H_n^2(\phi,\varphi)}{\nu_{\Omega}(\phi)\nu_{\Omega}(\varphi)}\frac{\sin(\phi)r_0^2(\phi)}{\sin(\varphi) r_0^2(\varphi)}d\phi d\varphi.
$$
Using the definition \eqref{kappa} of $\kappa$ we infer
$$
\forall \, \phi\in[0,\pi],\quad \nu_\Omega(\phi)\geq \kappa-\Omega
$$
and we obtain

$$
\|\mathcal{K}_n^{{\Omega}}\|_{\mu_\Omega}^2\lesssim (\kappa-\Omega)^{-2}
\bigintsss_0^\pi\bigintsss_0^\pi {H_n^2(\phi,\varphi)}\frac{\sin(\phi)r_0^2(\phi)}{\sin(\varphi) r_0^2(\varphi)}d\phi  d\varphi.
$$ 
Applying Lemma \ref{Lem-Hndecreasing} combined with the assumption {\bf{(H2)}} yields for any $0\leq \alpha<\beta\leq1$
\begin{align*}
\|\mathcal{K}_n^{{\Omega}}\|_{\mu_\Omega}^2\lesssim& (\kappa-\Omega)^{-2} n^{-2\alpha}\bigintsss_0^\pi\bigintsss_0^\pi|\phi-\varphi|^{-2\beta}d\phi d\varphi.
\end{align*}
By taking $\beta<\frac12$ we get the convergence of the integral and consequently we obtain the desired result,
\begin{align}\label{Est-kern-P}
{\|\mathcal{K}_n^{{\Omega}}\|_{\mu_\Omega}^2\lesssim (\kappa-\Omega)^{-2} n^{-2\alpha}.}
\end{align}

\medskip
\noindent
{\bf (4)}
First, let us check that any nonzero eigenfunction associated to the largest eigenvalue $\lambda_n(\Omega)$ should be with  a constant sign. Indeed, let $f$ 
be a nonzero normalized eigenfunction and assume that it changes the sign over a non negligible set. From the strict positivity of the kernel in the interval  $(0,\pi)$, 
we deduce that 
$$
 \mathcal{K}_n^{{\Omega}} f(\phi)<\mathcal{K}_n^{{\Omega}} |f|(\phi),\quad \forall \phi\in(0,\pi).
$$
First, by the assumption on $f$ we get
$$
\int_0^\pi \mathcal{K}_n^{{\Omega}}(f)(\phi)f(\phi)d\mu_{\Omega}(\phi)=\lambda_n(\Omega)\int_0^\pi f^2(\phi) d\mu_{\Omega}(\phi)=\lambda_n(\Omega).
$$
Second, from \eqref{eigenv} we have that
$$
\int_0^\pi \mathcal{K}_n^{{\Omega}}(|f|)(\phi)|f(\phi)|d\mu_{\Omega}(\phi)\leq \lambda_n(\Omega).
$$
Consequently,
$$
\lambda_n(\Omega)=\int_0^\pi \mathcal{K}_n^{{\Omega}}(f)(\phi)f(\phi)d\mu_{\Omega}(\phi)< \int_0^\pi \mathcal{K}_n^{{\Omega}}(|f|)(\phi)|f(\phi)|d\mu_{\Omega}(\phi)\leq \lambda_n(\Omega),
$$
achieving a contradiction. 
Hence, any nonzero eigenfunction of $\lambda_n(\Omega)$ must have a constant sign. Now let us check that $f$ is not vanishing in $(0,\pi).$ First we write 
$$
f(\phi)=\frac{1}{\lambda_n(\Omega)}\mathcal{K}_n^{{\Omega}} f(\phi)=\frac{1}{\lambda_n(\Omega)\nu_\Omega(\phi)}\int_0^\pi H_n(\phi,\varphi) f(\varphi) d\varphi.
$$

From \eqref{H-1} and Proposition \ref{Lem-meas} we get
$$
\forall \phi,\varphi\in(0,\pi),\quad H_n(\phi,\varphi)>0, \quad \nu_\Omega(\phi)>0.
$$
The first assertion  follows from the strict positivity of the associated hypergeometric function. Combined with the positivity of $f$ we deduce that
$$
\forall \phi\in(0,\pi),\quad f(\phi)>0.
$$

Finally, we shall  check  that the subspace generated by the eigenfunctions associated to $\lambda_n(\Omega)$ is one--dimensional. Assume that we have two independent eigenfunctions $f_0$ and $f_1$, which are necessarily with constant sign,  then there exists $a, b\in\R$ such that the eigenfunction $af_0+bf_1$ changes its sign. This is a  contradiction.

\medskip
\noindent
{\bf (5)}
Using \eqref{K-kernel} combined with  Lemma \ref{Lem-Hndecreasing}, we get that $n\in\N^{\star}\mapsto K_n(\phi,\varphi)$ is strictly decreasing for any $\varphi\neq \phi\in(0,\pi)$. Then, for any $\Omega\in(-\infty,\kappa)$ and for any nonnegative function $f$, we get
$$
\forall\, \phi\in(0,\pi),\quad \mathcal{K}_n^{{\Omega}} f(\phi)>\mathcal{K}_{n+1}^{{\Omega}}f(\phi),
$$
which implies in turn  that
$$
\int_0^\pi \mathcal{K}_n^{{\Omega}}(f)(\phi)f(\phi)d\mu_\Omega(\phi)>\int_0^\pi \mathcal{K}_{n+1}^{{\Omega}}(f)(\phi)f(\phi)d\mu_\Omega(\phi).
$$
Since the largest eigenvalue $\lambda_{n+1}(\Omega)$ is reached at some positive normalized function $f_{n+1}\geq0$, then
 \begin{align*}
 \lambda_{n+1}(\Omega)&=\int_0^\pi\mathcal{K}_{n+1}^{{\Omega}}(f_{n+1})(\phi) f_{n+1}(\phi)d\mu_{\Omega}(\phi)\\
 &<\int_0^\pi\mathcal{K}_n^{{\Omega}}(f_{n+1})(\phi) f_{n+1}(\phi)d\mu_{\Omega}(\phi)\\
 &<\sup_{\|f\|_{\mu_\Omega}=1}\int_0^\pi\mathcal{K}_n^{{\Omega}}(f)(\phi) f(\phi)d\mu_{\Omega}(\phi)\\
 &<\lambda_{n}(\Omega).
\end{align*}
This provides the announced result.

\medskip
\noindent
{\bf (6)} {  Fix $\Omega_0\in(-\infty,\kappa)$ and denote by $f_n^\Omega$ the  positive normalized  eigenfunction associated to the eigenvalue $\lambda_n(\Omega)$. Using the definition of the eigenfunction yields
\begin{align}\label{normaliz-1}
\lambda_n(\Omega)=\frac{\langle \mathcal{K}_n^{{\Omega}} f_n^\Omega, f_n^{\Omega_0}\rangle_{{\Omega_0}}}{\langle f_n^\Omega, f_n^{\Omega_0}\rangle_{{\Omega_0}}}, \quad \|f_n^\Omega\|_{\mu_{\Omega_0}}=1.
\end{align}
The regularity follows from the general theory using the fact that this eigenvalue is simple. However we can in our special case give a direct proof for its differentiability in the following way. 
From the decomposition 
$$
\frac{1}{\nu_\Omega(\phi)}=\frac{1}{\nu_{\Omega_0}(\phi)}+\frac{\Omega-\Omega_0}{\nu_\Omega(\phi)\nu_{\Omega_0}(\phi)},
$$ 
we get according to the expression of $\mathcal{K}_n^{{\Omega}}$
\begin{align}\label{F-deco1}
\nonumber \mathcal{K}_n^{{\Omega}} f(\phi)=& {\frac{1}{\nu_{\Omega_0}(\phi)}}\int_0^\pi H_n(\phi,\varphi) f(\varphi) d\varphi+ 
{(\Omega-\Omega_0)}\bigintsss_0^\pi \frac{H_n(\phi,\varphi)}{\nu_\Omega(\phi)\nu_{\Omega_0}(\phi)} f(\varphi) d\varphi\\
=&\mathcal{K}_n^{{\Omega_0}} f(\phi)+(\Omega-\Omega_0)\mathscr{R}_n^{\Omega_0,\Omega} f(\phi)
\end{align}
with
\begin{align}\label{Form-kern-d}
\nonumber\mathscr{R}_n^{\Omega_0,\Omega} f(\phi)&:=\bigintsss_0^\pi \frac{H_n(\phi,\varphi)}{\nu_\Omega(\phi)\nu_{\Omega_0}(\phi)} f(\varphi) d\varphi\\
&={\frac{1}{\nu_{\Omega_0}(\phi)}\mathcal{K}_n^{{\Omega}} f(\phi).}
\end{align}
Therefore we obtain
\begin{align*}
\lambda_n(\Omega)=\frac{\langle \mathcal{K}_n^{{\Omega_0}} f_n^\Omega, f_n^{\Omega_0}\rangle_{{\Omega_0}}}{\langle f_n^\Omega, f_n^{\Omega_0}\rangle_{{\Omega_0}}}+(\Omega-\Omega_0)\frac{\langle \mathscr{R}_n^{\Omega_0,\Omega}f_n^\Omega, f_n^{\Omega_0}\rangle_{{\Omega_0}}}{\langle f_n^\Omega, f_n^{\Omega_0}\rangle_{{\Omega_0}}}.
\end{align*}
{As $\mathcal{K}_n^{{\Omega_0}}$ is self-adjoint on the Hilbert space $L^2_{\mu_{\Omega_0}}$  then
$$
\frac{\langle \mathcal{K}_n^{{\Omega_0}} f_n^\Omega, f_n^{\Omega_0}\rangle_{{\Omega_0}}}{\langle f_n^\Omega, f_n^{\Omega_0}\rangle_{{\Omega_0}}}=\frac{\langle  f_n^\Omega, \mathcal{K}_n^{{\Omega_0}}f_n^{\Omega_0}\rangle_{{\Omega_0}}}{\langle f_n^\Omega, f_n^{\Omega_0}\rangle_{{\Omega_0}}}=\lambda_n(\Omega_0).
$$}
Let us assume for a while that 
\begin{align}\label{con-R}
{\lim_{\Omega\to \Omega_0}\frac{\langle \mathscr{R}_n^{\Omega_0,\Omega}f_n^\Omega, f_n^{\Omega_0}\rangle_{{\Omega_0}}}{\langle f_n^\Omega, f_n^{\Omega_0}\rangle_{{\Omega_0}}}={\langle \mathscr{R}_n^{\Omega_0,\Omega_0}f_n^{\Omega_0}, f_n^{\Omega_0}}\rangle_{{\Omega_0}}}.
\end{align}
Then we deduce that  $\Omega\mapsto \lambda_n(\Omega)$ is differentiable at $\Omega_0$ with 
\begin{align*}
\lambda_n^\prime(\Omega_0)=&{\langle \mathscr{R}_n^{\Omega_0,\Omega_0}f_n^{\Omega_0}, f_n^{\Omega_0}\rangle_{{\Omega_0}}}=\bigintsss_0^\pi \frac{H_n(\phi,\varphi)}{\nu_{\Omega_0}(\phi)} f_n^{\Omega_0}(\varphi) f_n^{\Omega_0}(\phi) \sin(\phi) r_0^2(\phi)d\phi d\varphi.
\end{align*}
Since
$$
\forall \varphi, \phi\in(0,\pi), \quad H_n(\phi, \varphi)>0, \,\,f_n^{\Omega_0}(\phi)>0, \,\,\nu_{\Omega_0}(\phi)>0,
$$
we find that
$ \lambda_n^\prime(\Omega_0)>0$, which achieves the proof of the suitable result.\\
It remains to prove \eqref{con-R}. First, 
for the numerator of the  left hand side  we first make the splitting
\begin{align}\label{XP-1}
\nonumber\langle \mathscr{R}_n^{\Omega_0,\Omega}f_n^\Omega, f_n^{\Omega_0}\rangle_{{\Omega_0}}&=\langle \mathscr{R}_n^{\Omega_0,\Omega_0}f_n^\Omega, f_n^{\Omega_0}\rangle_{{\Omega_0}}+\langle ( \mathscr{R}_n^{\Omega_0,\Omega}-\mathscr{R}_n^{\Omega_0,\Omega_0})f_n^\Omega, f_n^{\Omega_0}\rangle_{{\Omega_0}}\\
&:= \mathcal{I}_1(\Omega)+\mathcal{I}_2(\Omega).
\end{align}
To estimate the second term $\mathcal{I}_2$  we use the identities \eqref{F-deco1} and \eqref{Form-kern-d} leading to
$$
\mathscr{R}_n^{\Omega_0,\Omega}-\mathscr{R}_n^{\Omega_0,\Omega_0}=\frac{\Omega-\Omega_0}{\nu_{\Omega_0}^2(\phi)}\mathcal{K}_n^{{\Omega}}.
$$
It follows that
\begin{align*}
\mathcal{I}_2(\Omega)=(\Omega-\Omega_0)\langle  {\nu^{-2}_{\Omega_0}} \mathcal{K}_n^{\Omega}f_n^\Omega,f_n^{\Omega_0}\rangle_{{\Omega_0}}.
\end{align*}
Hence, applying  Proposition \ref{Lem-meas}-(2) combined with Cauchy-Schwarz inequality and the normalization assumption  of the eigenfunctions  in \eqref{normaliz-1} we infer
\begin{align}
\big|\mathcal{I}_2(\Omega)\big|\lesssim |\Omega-\Omega_0| \|\mathcal{K}_n^{\Omega}f_n^\Omega\|_{\mu_{\Omega_0}}.
\end{align}
From Remarks \ref{remark-imp1}-(2),  \eqref{Est-kern-P} and \eqref{normaliz-1} we may write for $\Omega$ close to  $\Omega_0$
\begin{align*}
\big|\mathcal{I}_2(\Omega)\big|&\lesssim |\Omega-\Omega_0| \|\mathcal{K}_n^{\Omega}f_n^\Omega\|_{\mu_{\Omega}}\\
&\lesssim |\Omega-\Omega_0|.
\end{align*}
This obviously gives
\begin{align}\label{I-2-es}
\lim_{\Omega\to\Omega_0} \mathcal{I}_2(\Omega)=0.
\end{align}
Let us move to the  term $\mathcal{I}_1$ introduced in \eqref{XP-1}. Then 
combining  \eqref{Form-kern-d} with the fact that $\mathcal{K}_n^{{\Omega_0}}$ is self-adjoint on the Hilbert space $L^2_{\mu_{\Omega_0}}$  allows to get
\begin{align}\label{I-1-P1}
\mathcal{I}_1(\Omega)&=\langle f_n^\Omega, g\rangle_{{\Omega_0}}\quad\hbox{with}\quad g:=\mathcal{K}_n^{{\Omega_0}}\big(\nu^{-1}_{\Omega_0}f_n^{\Omega_0}\big).
\end{align}
Applying Proposition \ref{Lem-meas}-(2) we  easily get that $g\in L^2_{\mu_{\Omega_0}}$. Now we claim that 
\begin{align}\label{strong-conv-1}
\lim_{\Omega\to \Omega_0}\|f_n^\Omega-f_n^{\Omega_0}\|_{\mu_{\Omega_0}}=0.
\end{align}
Before giving its proof, let us see how to conclude. It is  easy  to check from \eqref{I-1-P1} and \eqref{strong-conv-1} that 
$$
\lim_{\Omega\to\Omega_0} \mathcal{I}_1(\Omega)=\mathcal{I}_1(\Omega_0)\quad \hbox{and}\quad \lim_{\Omega\to\Omega_0}\langle f_n^\Omega, f_n^{\Omega_0}\rangle_{{\Omega_0}}=1.
$$
Thus combining this result with  \eqref{XP-1} and \eqref{I-2-es} yields \eqref{con-R}.  It remains to check \eqref{strong-conv-1} which is a consequence of classical results on perturbation theory.  One can use for instance   \cite{Nagy} or \cite[Chapter XII]{reed}, where the analytic dependence of the eigenvalues and the associated eigenfunctions is analyzed. Let us briefly discuss the main arguments used  to get the continuity of the eigenfunctions  with respect to the parameter $\Omega$.  First  we set
$$
A(\Omega):=\mathcal{K}_n^{{\Omega}}.
$$
Then  using \eqref{nu-function} we finds
\begin{align*}
\frac{1}{\nu_{\Omega}(\phi)}=\sum_{m\in\N}\frac{(\Omega-\Omega_0)^m}{\nu^{m+1}_{\Omega_0}(\phi)}\cdot
\end{align*} 
Then similarly to \eqref{F-deco1} we obtain the decomposition
\begin{align}\label{analytic-1}
A(\Omega)=\sum_{m\in\N}(\Omega-\Omega_0)^mA_m,\quad \quad\hbox{with}\quad A_m:=\nu^{-m}_{\Omega_0}(\phi){\mathcal{K}_n^{{\Omega_0}}}.
\end{align}
By applying  the lower bound of Proposition \ref{Lem-meas}-(2) we get that $A_m$ is bounded with
$$
\|A_m\|_{\mathcal{L}(L^2_{\mu_{\Omega_0}})}\lesssim (\kappa-\Omega_0)^{-m}
$$
This shows that $A(\Omega)$ is analytic for $\Omega$ close enough to $\Omega_0$. Now define the sets
$$
\rho(A(\Omega)):=\Big\{z\in\mathbb{C}, A(\Omega)-z\textnormal{Id}\, \hbox{ is invertible }\Big\}\quad\hbox{and}\quad \Gamma:=\Big\{(\Omega, z)\in\mathbb{C}^2, z\in \rho(A(\Omega))\Big\}.
$$
Then it is known that the resolvent set $A(\Omega)$ and $\Gamma$ are open, see Theorem XII.7 in \cite{reed}. Now, since $\lambda_n(\Omega_0)$ is an isolated   simple eigenvalue and $\Gamma$ is open then we may  find  $\delta>0$ such that the oriented circle $\gamma:=\big\{z, |z- \lambda_n(\Omega_0)|=\delta\big\}$ is contained in $\rho(A(\Omega))$ for all $|\Omega-\Omega_0|\leqslant\delta.$  Thus, for $|\Omega-\Omega_0|\leqslant\delta$ the operator
$$
P_n(\Omega)=-\frac{1}{2i \pi}\int_{\gamma}\big(A(\Omega)-z\textnormal{Id}\big)^{-1} dz
$$
is well-defined and it is  analytic in view of \eqref{analytic-1}. Notice that one may get the estimate
\begin{align}\label{proj-diff}
\|P_n(\Omega)-P_n(\Omega_0)\|_{\mathcal{L}(L^2_{\mu_{\Omega_0}})}\lesssim |\Omega-\Omega_0|.
\end{align}
Then  from classical results on spectral theory, see  for instance Theorems XII.5-XII.6-XII.8 in \cite{reed},  $P_n(\Omega_0)$ is a projection on the one dimensional eigenspace associate to $\lambda_n(\Omega_0).$  In addition, $A_n(\Omega)$ admits only one eigenvalue inside the circle $\Gamma$  which necessary coincides  with  $\lambda_n(\Omega)$. Notice that this latter claim can be proved using the continuity in $\Omega$ of the largest eigenvalue  which can be checked from \eqref{eigenv}.  Furthermore, $P_n(\Omega)$ is  still a one dimensional projection on the eigenspace associated to the  $\lambda_n(\Omega)$. As a consequence, if $f_n^{\Omega_0}$ is a normalized eigenfunction of the operator $A(\Omega_0)=\mathcal{K}_n^{{\Omega_0}}$ associated to $\lambda_n(\Omega_0)$, then
$
P_n(\Omega)f_n^{\Omega_0}
$ is an eigenfunction of $A(\Omega)$  associated to $\lambda_n(\Omega)$. Applying \eqref{proj-diff} yields
$$
\lim_{\Omega\to \Omega_0}\|P_n(\Omega)f_n^{\Omega_0}-f_n^{\Omega_0}\|_{{\mu_{\Omega_0}}}=0,
$$
where we have used the fact that $P_n(\Omega_0)f_n^{\Omega_0}=f_n^{\Omega_0}.$ Now, by taking
$$
f_n^{\Omega}:=\frac{P_n(\Omega)f_n^{\Omega_0}}{\|P_n(\Omega)f_n^{\Omega_0}\|}_{{\mu_{\Omega_0}}}
$$
we get a normalized eigenfunction in the sense of \eqref{normaliz-1} and the family $\Omega\mapsto f_n^\Omega\in L^2_{\mu_{\Omega_0}}$ is continuous at $\Omega_0$, which ensures \eqref{strong-conv-1}. }This ends the proof of the desired result. 
 \end{proof}
Next we shall establish  the following result.
\begin{pro}\label{prop-kernel-onedim}
Let $n\geq1$  and $r_0$ satisfies the assumptions {\bf{(H1)}} and {\bf{(H2)}}. Set
\begin{equation}\label{set-eigenvalues}
\mathscr{S}_n:=\Big\{\Omega\in(-\infty,\kappa) \quad\textnormal{s.t.}\quad \lambda_n(\Omega)=1\Big\}.
\end{equation} Then the following holds true
\begin{enumerate}
\item The set $\mathscr{S}_n$ is formed by a single point denoted by $\Omega_n$ .
\item The sequence $(\Omega_n)_{n\geq1}$ is strictly increasing and satisfies 
$$
\lim_{n\to\infty}\Omega_n=\kappa.
$$
\end{enumerate}

\end{pro}
\begin{proof}
\medskip
\noindent
{\bf (1)}
To check that the set $\mathscr{S}_n$ is non empty we shall use the mean value theorem.
 From the upper bound in Proposition \ref{prop-operator}--(2) and \eqref{K-kernel} we find that
 $$
 0\leq\lambda_n(\Omega)\leq { \left\{ \bigintss_0^\pi\bigintss_0^\pi \frac{H_n^2(\phi,\varphi)\sin \phi\, r_0^2(\phi)}{\nu_\Omega(\phi)\nu_\Omega(\varphi)\, \sin(\varphi)\, r_0^2(\varphi)}d\varphi d\phi\right\}^\frac12.}
 $$
Thus by  taking the limit as $\Omega\to-\infty$ we deduce that 
\begin{equation}\label{liminf0}
\lim_{\Omega\rightarrow -\infty} \lambda_n(\Omega)=0.
\end{equation}
Next, we intend to show that 
\begin{equation}\label{liminf}
\lim_{\Omega\to \kappa}\lambda_n(\Omega)=\infty.
\end{equation}

Using the lower bound of $\lambda_n(\Omega)$ in Proposition \ref{prop-operator}--(2),  we find by virtue of Fatou Lemma
\begin{equation*}
\bigintsss_0^\pi\bigintsss_0^\pi\frac{H_n(\phi,\varphi)}{\nu_\kappa^{\frac12}(\varphi)\nu_\kappa^{\frac12}(\phi)}\frac{\sin^\frac12(\phi) r_0(\phi)}{\sin^\frac12(\varphi) r_0(\varphi)} \varrho(\phi)\varrho(\varphi)d\varphi d\phi\leq \liminf_{\Omega\rightarrow \kappa} \lambda_n(\Omega),
\end{equation*}
for any nonnegative $\varrho$ satisfying $\displaystyle{\int_0^\pi\varrho^2(\phi)d\phi=1}$.  According to Proposition \ref{Lem-meas}--$(4)$, the function $\nu_\kappa$ reaches its minimum at a point $\phi_0\in[0,\pi]$ and
$$
\forall\, \phi\in[0,\pi],\quad 0\leq \nu_\kappa(\phi)\leq C|\phi-\phi_0|^{1+\alpha}.
$$
There are two possibilities: $\phi_0\in(0,\pi)$ or $\phi_0\in\{0,\pi\}$. Let us start with the first case and we shall  take $\varrho$ as follows
$$
\varrho(\phi)=\frac{c_\beta}{|\phi-\phi_0|^\beta},
$$
with $\beta<\frac12$ and the constant $c_\beta$ is chosen such that $\varrho$ is normalized. Hence using the preceding estimates  we get
\begin{equation}\label{Tramb1}
C\bigintsss_0^\pi\bigintsss_0^\pi\frac{H_n(\phi,\varphi)}{|\phi-\phi_0|^{\frac{1+\alpha}{2}+\beta}|\varphi-\phi_0|^{\frac{1+\alpha}{2}+\beta}}\frac{\sin^\frac12(\phi) r_0(\phi)}{\sin^\frac12(\varphi) r_0(\varphi)} d\varphi d\phi\leq \liminf_{\Omega\rightarrow \kappa} \lambda_n(\Omega).
\end{equation}
Let $\varepsilon>0$ such that $[\phi_0-\varepsilon,\phi_0+\varepsilon]\subset (0,\pi)$. According to \eqref{H-1} the function    $H_n$ is strictly positive in the domain $(0,\pi)^2$, hence  there exists $\delta>0$ such 
$$
\forall\, (\phi,\varphi)\in[\phi_0-\varepsilon,\phi_0+\varepsilon]^2, \quad \frac{H_n(\phi,\varphi)\sin^\frac12(\phi) r_0(\phi)}{\sin^\frac12(\varphi) r_0(\varphi)}\geq \delta.
$$
Thus we obtain
\begin{equation*}
C\bigintsss_{\phi_0-\varepsilon}^{\phi_0+\varepsilon}\bigintsss_{\phi_0-\varepsilon}^{\phi_0+\varepsilon}\frac{d\phi\,d\varphi}{|\phi-\phi_0|^{\frac{1+\alpha}{2}+\beta}|\varphi-\phi_0|^{\frac{1+\alpha}{2}+\beta}}\leq \liminf_{\Omega\rightarrow \kappa} \lambda_n(\Omega).
\end{equation*}
By taking $\frac{1+\alpha}{2}+\beta>1$, which is an admissible configuration, we
find
$$
\lim_{\Omega\rightarrow \kappa} \lambda_n(\Omega)=+\infty.
$$
Now let us move to the second possibility  where $\phi_0\in\{0,\pi\}$ and without 
any loss of generality we can only deal with the case $\phi_0=0.$ From \eqref{H-1} and using the inequality
$$
\forall \, x\in[0,1),\quad F_n(x)\geq 1,
$$
we obtain
\begin{align*}
\forall \phi,\varphi\in(0,\pi),\quad H_n(\phi,\varphi)\geq&c_n\frac{\sin(\varphi)r_0^{n-1}(\phi)r_0^{n+1}(\varphi)}{\left[R(\phi,\varphi)\right]^{n+\frac12}}.
\end{align*}
Combined with the assumption ${\bf{(H2)}}$, it implies
\begin{align*}
\forall \phi,\varphi\in(0,\pi),\quad H_n(\phi,\varphi)\geq&c_n\frac{\sin^{n+2}(\varphi)\sin^{n-1}(\phi)}{\left[R(\phi,\varphi)\right]^{n+\frac12}}.
\end{align*}
 Plugging this into \eqref{Tramb1} we find
\begin{equation*}
C_n\bigintsss_0^\pi\bigintsss_0^\pi\frac{1}{\phi^{\frac{1+\alpha}{2}+\beta}\varphi^{\frac{1+\alpha}{2}+\beta}} \frac{\sin^{n+\frac12}(\varphi)\sin^{n+\frac12}(\phi)}{\left[R(\phi,\varphi)\right]^{n+\frac12}}d\varphi d\phi\leq \liminf_{\Omega\rightarrow \kappa} \lambda_n(\Omega).
\end{equation*}
Let $\varepsilon>0$ sufficiently small, then using Taylor expansion we get according to \eqref{RefR}
$$
0\leq \phi,\varphi\leq\varepsilon\Longrightarrow R(\phi,\varphi)\leq C(\phi+\varphi)^2.
$$
 Thus
 \begin{equation*}
C_n\bigintsss_0^\varepsilon\bigintsss_0^\varepsilon\frac{1}{\phi^{\frac{1+\alpha}{2}+\beta}\varphi^{\frac{1+\alpha}{2}+\beta}} \frac{\varphi^{n+\frac12}\phi^{n+\frac12}}{(\phi+\varphi)^{2n+1}}d\varphi d\phi\leq \liminf_{\Omega\rightarrow \kappa} \lambda_n(\Omega).
\end{equation*}
which gives after simplification
\begin{equation*}
C_n\bigintsss_0^\varepsilon\bigintsss_0^\varepsilon\frac{\varphi^{n-\frac\alpha2-\beta}\phi^{n-\frac\alpha2-\beta}}{(\phi+\varphi)^{2n+1}}d\varphi d\phi\leq \liminf_{\Omega\rightarrow \kappa} \lambda_n(\Omega).
\end{equation*}
Making the change of variables $\varphi=\phi \theta$ we obtain
\begin{align*}
\bigintsss_0^\varepsilon\bigintsss_0^\varepsilon\frac{\varphi^{n-\frac\alpha2-\beta}\phi^{n-\frac\alpha2-\beta}}{(\phi+\varphi)^{2n+1}}d\varphi d\phi=&\bigintsss_0^\varepsilon\phi^{-\alpha-2\beta}\bigintsss_0^{\frac\varepsilon\phi}\frac{\theta^{n-\frac\alpha2-\beta}}{(1+\theta)^{2n+1}}d\theta d\phi.
\end{align*}
This integral diverges provided that $\alpha+2\beta>1$ and thus  under this assumption
$$
\lim_{\Omega\rightarrow \kappa} \lambda_n(\Omega)=+\infty.
$$
Hence we obtain \eqref{liminf}.
By the intermediate mean value, we achieve the existence of at least one solution for the equation  
$$
\lambda_n(\Omega)=1.
$$
Consequently, using Proposition \ref{prop-operator} we deduce by the mean value theorem that  the set $\mathscr{S}_n$ contains only one element.

\medskip
\noindent
{\bf (2)} Since  $\Omega_n$ satisfies the equation
$$
\lambda_n(\Omega_n)=1.
$$
According to Proposition \ref{prop-operator}--$(5)$ the sequence $k\mapsto \lambda_k(\Omega_n)$ is strictly decreasing. It implies in particular that
$$
\lambda_{n+1}(\Omega_n)<\lambda_n(\Omega_n)=1.
$$
Hence by \eqref{liminf} one may apply  the mean value theorem and find  an element of the set $\mathscr{S}_{n+1}$
in the interval $(\Omega_n,\kappa)$. This means that $\Omega_{n+1}>\Omega_n$ and thus this sequence is strictly increasing. 
It remains to prove that this sequence is converging to $\kappa.$ The convergence of this sequence to some element $\overline{\Omega}\leq \kappa$ is clear. 
To prove that $\overline{\Omega}=\kappa$ we shall argue by contradiction by assuming that $\overline{\Omega}<\kappa.$
By the construction of $\Omega_n$ one has necessarily 
$$
\forall n\geq1,\quad \lambda_n(\overline{\Omega})>1.
$$
Using the upper-bound estimate stated in Proposition \ref{prop-operator}--$(2)$ combined with the point $(3)$ we obtain for any $\alpha\in(0,1)$
\begin{equation}\label{low-contra}
\forall\, n\geq1,\quad 0< \lambda_n(\overline{\Omega})\lesssim {(\kappa-\overline{\Omega})^{-1} n^{-\alpha}}.
\end{equation}
By taking the limit as $n\to+\infty$ we find
$$
\lim_{n\to\infty}\lambda_n(\overline{\Omega})=0.
$$
This contradicts \eqref{low-contra} which achieves the proof.
\end{proof}
\subsection{Eigenfunctions regularity}\label{Eigenfunctions regularity}
This section is devoted to the strong regularity  of the eigenfunctions associated to the operator $\mathcal{K}_n^{{\Omega}}$ and 
constructed in Proposition \ref{prop-operator}. We have already  seen that these eigenfunctions belong to a weak function space $L^2_{\mu_\Omega}$. 
Here we shall show first their continuity and later their H\"{o}lder regularity.

\subsubsection{Continuity}
The main result of this section reads as follows.
\begin{pro}\label{prop-higher-reg}
Let $\Omega\in(-\infty,\kappa),$ $n\geq1$, $r_0$ satisfies the assumptions {\bf{(H1)}} and {\bf{(H2)}}, and $f$ be
an eigenfunction for $\mathcal{K}_n^{{\Omega}}$ associated to a non-vanishing  eigenvalue. Then $f$ is continuous
\mbox{over $[0,\pi]$},  and for  $n\geq 2$ it satisfies the boundary condition $f(0)=f(\pi)=0$. However this boundary condition fails 
for $n=1$ at least with the eigenfunctions associated to the largest eigenvalue $\lambda_1(\Omega).$ 
\end{pro}
 
\begin{proof}
Let $f\in L^2_{\mu_{\Omega}}$ be any non trivial eigenfunction of the operator  $\mathcal{K}_n^{{\Omega}}$ defined in \eqref{kerneleq} and associated to an eigenvalue $\lambda\neq0$, then 
\begin{equation}\label{Eigen-V1}
f(\phi)=\frac{1}{\lambda\, \nu_\Omega(\phi)}\int_0^\pi H_n(\phi,\varphi) f(\varphi)d\varphi, \quad \forall \phi\in (0,\pi)\,\textnormal{a.e.}
\end{equation}
Since $f\in L^2_{\mu_\Omega},$ then the function $g: \varphi\in [0,\pi]\mapsto r_0^{\frac32}(\varphi)f(\varphi)$belongs to  $ L^2((0,\pi);d\varphi)$. Therefore the equation \eqref{Eigen-V1} can be written in terms of $g$ as follows
 $$
 g(\phi)=\frac{1}{\lambda\,\nu_\Omega(\phi)}\bigintsss_0^\pi r_0^{-\frac32}(\varphi){r_0^{\frac32}(\phi)H_n(\phi,\varphi)} g(\varphi)d\varphi, \quad \forall \phi\in (0,\pi)\,\textnormal{a.e.}
 $$
 Coming back to the definition of $H_n$ in \eqref{H-1} we obtain for some constant $c_n$ the formula
 \begin{align*}
r_0^{-\frac32}(\varphi){r_0^{\frac32}(\phi)}H_n(\phi,\varphi)=&c_n\frac{\sin(\varphi) r_0^{n-\frac12}(\varphi)r_0^{n+\frac12}(\phi)}{\left[R(\phi,\varphi)\right]^{n+\frac12}} F_n\left(\frac{4r_0(\phi)r_0(\varphi)}{R(\phi,\varphi)}\right).
\end{align*}
Using \eqref{estimat-1} and the assumption {\bf{(H2)}}   yields
\begin{align}\label{H-XX1}
r_0^{-\frac32}(\varphi){r_0^{\frac32}(\phi)}H_n(\phi,\varphi)\lesssim&\frac{r_0^{n+\frac12}(\varphi)r_0^{n+\frac12}(\phi)}{R^{n+\frac12}(\phi,\varphi)}\left(1+\ln\left(\frac{\sin(\phi)+\sin(\varphi)}{|\phi-\varphi|}\right) \right)\\
\nonumber\lesssim&1+\ln\left(\frac{\sin(\phi)+\sin(\varphi)}{|\phi-\varphi|}\right). 
\end{align}
This implies, using Cauchy-Schwarz inequality  and the fact that $\nu_\Omega$ is bounded away from zero
\begin{align*}
 |g(\phi)|\lesssim& \bigintsss_0^\pi\left(1+\ln\left(\frac{\sin(\phi)+\sin(\varphi)}{|\phi-\varphi|}\right)\right)g(\varphi)d\varphi\\
 \lesssim& \|g\|_{L^2(d\varphi)}\left(1+\bigintsss_0^\pi\ln^2\left(\frac{\sin(\phi)+\sin(\varphi)}{|\phi-\varphi|}\right)d\varphi\right)^{\frac12}\\
  \lesssim& \|f\|_{\mu_\Omega},\quad \forall \phi\in (0,\pi)\,\textnormal{a.e.}
\end{align*}
It follows that $g$ is bounded. Now inserting this estimate into  \eqref{Eigen-V1} allows to get
\begin{align}\label{split}
 |f(\phi)|\lesssim&\|g\|_{L^\infty}\int_0^\pi  r_0^{-\frac32}(\varphi)H_n(\phi,\varphi) d\varphi,\nonumber\\
 \lesssim&  \|f\|_{\mu_\Omega}\bigintss_0^\pi\frac{\sin(\varphi) r_0^{n-\frac12}(\varphi)r_0^{n-1}(\phi)}{\left[R(\phi,\varphi)\right]^{n+\frac12}} F_n\left(\frac{4r_0(\phi)r_0(\varphi)}{R(\phi,\varphi)}\right)d\varphi.
\end{align}
Using once again {estimat-1} and the assumption {\bf{(H2)}} we deduce that
 \begin{align*}
 |f(\phi)|\lesssim& \|f\|_{\mu_\Omega}\bigintss_0^\pi\frac{\sin^{n-1}(\phi)\sin^{n+\frac12}(\varphi)}{\big((\sin(\phi)+\sin(\varphi))^2+(\cos\phi-\cos\varphi)^2\big)^{n+\frac12}}\Big(1+\ln\Big(\frac{\sin(\phi)+\sin(\varphi)}{|\phi-\varphi|}\Big)\Big)d\varphi.
\end{align*}
By symmetry we may  restrict the analysis to $\phi\in[0,\frac\pi2]$. Thus, splitting the integral given in \eqref{split} and using that  $$
\inf_{\varphi\in[\pi/2,\pi]\\
\atop\phi\in[0,\pi/2]}R(\phi,\varphi)>0,
$$ we obtain 
\begin{align*}
 |f(\phi)|\lesssim& \|f\|_{\mu_\Omega}\bigintss_0^{\frac\pi2}\frac{\sin^{n-1}(\phi)\sin^{n+\frac12}(\varphi)}{\big(\sin(\phi)+\sin(\varphi)
 \big)^{2n+1}}\Big(1+\ln\Big(\frac{\sin(\phi)+\sin(\varphi)}{|\phi-\varphi|}\Big)\Big)d\varphi\\
 +&\|f\|_{\mu_\Omega}\bigintss_{\frac\pi2}^\pi\Big(1+\ln\Big(\frac{\sin(\phi)+\sin(\varphi)}{|\phi-\varphi|}\Big)\Big)d\varphi.
\end{align*}
It follows that
\begin{align*}
 |f(\phi)|\lesssim& \|f\|_{\mu_\Omega}\bigintss_0^{\frac\pi2}\frac{\phi^{n-1}\varphi^{n+\frac12}}{\big(\phi+\varphi\big)^{2n+1}}\Big(1+\ln\Big(\frac{\phi+\varphi}{|\phi-\varphi|}\Big)\Big)d\varphi
 +\|f\|_{\mu_\Omega}.
\end{align*}
Using the change of variables $\varphi=\phi \theta$ we get
\begin{align*}
 |f(\phi)|\lesssim& \|f\|_{\mu_\Omega}\phi^{-\frac12}\bigintss_0^{\frac{\pi}{2\phi}}\frac{\theta^{n+\frac12}}{\big(1+\theta\big)^{2n+1}}\Big(1+\ln\Big(\frac{1+\theta}{|1-\theta|}\Big)\Big)d\theta
 +\|f\|_{\mu_\Omega}\\
 \lesssim& \|f\|_{\mu_\Omega}\phi^{-\frac12}.
\end{align*}
Consequently we find 
$$
\sup_{\phi\in(0,\pi)}r_0^{\frac12}(\phi)| f(\phi)|\lesssim  \|f\|_{\mu_\Omega}.
$$
Inserting this estimate into \eqref{Eigen-V1} and using \eqref{H-XX1} yields
\begin{align*}
|f(\phi)|\lesssim& \|f\|_{\mu_\Omega} \bigintsss_0^\pi r_0^{-\frac12}(\varphi) H_n(\phi,\varphi) d\varphi\\
\lesssim& \|f\|_{\mu_\Omega} \bigintss_0^\pi\frac{r_0^{n+\frac32}(\varphi)r_0^{n-1}(\phi)}{R^{n+\frac12}(\phi,\varphi)}\left(1+\ln\left(\frac{\sin(\phi)+\sin(\varphi)}{|\phi-\varphi|}\right) \right) d\varphi.
\end{align*}
As before we can restrict $\phi\in [0,\frac\pi2]$ and by using the fact
$$
\inf_{\varphi\in[\pi/2,\pi]\\
\atop\phi\in[0,\pi/2]}R(\phi,\varphi)>0,
$$
we deduce after splitting the integral
\begin{align*}
|f(\phi)|\lesssim& \|f\|_{\mu_\Omega} \bigintsss_0^\pi r_0^{-\frac12}(\varphi) H_n(\phi,\varphi) d\varphi\\
\lesssim& \|f\|_{\mu_\Omega} r_0^{n-1}(\phi)+\|f\|_{\mu_\Omega} \bigintss_0^{\frac\pi2}\frac{\varphi^{n+\frac32}\phi^{n-1}}{(\phi+\varphi)^{2n+1}}\left(1+\ln\left(\frac{\phi+\varphi}{|\phi-\varphi|}\right) \right) d\varphi.
\end{align*}
Making the change of variables $\varphi=\phi \theta$ leads to
\begin{align*}
\forall \phi\in[0,\pi/2],\quad |f(\phi)|\lesssim& \|f\|_{\mu_\Omega} r_0^{n-1}(\phi)+\|f\|_{\mu_\Omega}\phi^{\frac12} \bigintss_0^{\frac{\pi}{2\phi}}
\frac{\theta^{n+\frac32}}{(1+\theta)^{2n+1}}\left(1+\ln\left(\frac{\theta+1}{|\theta-1|}\right) \right) d\theta\\
\lesssim& \|f\|_{\mu_\Omega}(\phi^{n-1}+{ \phi^{\frac{1}{2}}}).
\end{align*}
Consequently we get
\begin{align*}
\forall \phi\in(0,\pi),\quad |f(\phi)|\lesssim& \|f\|_{\mu_\Omega}(r_0^{n-1}(\phi)+{r_0^{\frac{1}{2}}(\phi)}).
\end{align*}

This shows that $f$ is bounded over $(0,\pi)$ and  by the dominated convergence theorem one can show that $f$ is in fact continuous on $[0,\pi]$ and satisfies for $n\geq2$ 
the boundary condition
$$
f(0)=f(\pi)=0.
$$
Last, we shall check that this boundary condition fails  for $n=1$  with the largest eigenvalue $\lambda_1(\Omega).$ Indeed, according to \eqref{Eigen-V1} we have
$$
f(0)=\frac{1}{\lambda \nu_\Omega(0)}\int_0^\pi H_1(0,\varphi) f(\varphi)d\varphi.
$$
However, from \eqref{H-1} we get 
$$
\forall\, \varphi\in(0,\pi),\quad H_1(0,\varphi)=c_1 \frac{r_0^2(\varphi)\sin(\varphi)}{R^\frac32(0,\varphi)}>0.
$$
Combining this with the fact that  $f$ does not change the sign allows to get that $f(0)\neq0.$ 
\end{proof}

\subsubsection{H\"{o}lder continuity}
The main goal of this section is to prove the H\"{o}lder regularity of the eigenfunctions. 
{{ 
{\begin{pro}\label{HoldX1}
Assume that $r_0$ satisfies the conditions {\bf (H)}  and let  {$\Omega\in(-\infty,\kappa)$}, then any solution $h$ of the equation
\begin{equation}\label{EigenV}
h(\phi)=\frac{1}{\lambda\,\nu_\Omega(\phi)}\int_0^\pi H_n(\varphi,\phi)h(\varphi)d\varphi,\quad \forall \phi\in(0,\pi),
\end{equation}
with $\lambda\neq0$, belongs to $\mathscr{C}^{1,\alpha}(0,\pi)$, for any $n\geq 2$. The functions involved in the above expression can be found in \eqref{H-1}--\eqref{K-kernel}--\eqref{nu-function}.
\end{pro}}

\begin{proof}
{From the initial expression of the linearized operator \eqref{exp-lin1} in Proposition \ref{Prop-lin1} and combining it with Proposition \ref{Prop-lin2}, one has}
\begin{align}\label{FMZ1}
\mathcal{F}_n(h)(\phi):=&\int_0^\pi H_n(\varphi,\phi)h(\varphi)d\varphi{=}
\frac{1}{4\pi}\frac{1}{r_0(\phi)}
\int_0^\pi\int_0^{2\pi}\mathcal{H}_n(\phi,\varphi,\eta)h(\varphi)d\eta d\varphi,
\end{align}
with
\begin{align*}
 \widehat{R}(\phi,\varphi,\eta):=&(r_0(\phi)-r_0(\varphi))^2+
2r_0(\phi)r_0(\varphi)(1-\cos\eta)+(\cos\phi-\cos\varphi)^2,\\
\mathcal{H}_n(\phi,\varphi,\eta):=&\frac{\sin(\varphi)r_0(\varphi)\cos(n\eta)}{\widehat{R}^{\frac12}(\phi,\varphi,\eta)}.
\end{align*}
It is clear that  any solution $h$ of \eqref{EigenV} is equivalent to a solution of 
$$
\forall\, \phi\in(0,\pi),\quad h(\phi)=\frac{\mathcal{F}_n(h)(\phi)}{\lambda \,\nu_\Omega(\phi)}\cdot
$$
From Proposition \ref{Lem-meas} we know  that $\nu_\Omega\in \mathscr{C}^{1,\alpha}(0,\pi)$ and does not vanish when $\Omega\in(-\infty,\kappa)$. Therefore to check the regularity  $h\in \mathscr{C}^{1,\alpha}(0,\pi)$  it is enough to establish that $\mathcal{F}_n(h)\in \mathscr{C}^{1,\alpha}(0,\pi),$ due to the fact that $\mathscr{C}^\alpha$ is an algebra.   Since $h$ is symmetric  with respect to $\phi=\frac{\pi}{2}$, then one can verify that $\mathcal{F}_n(h)$ preserves this symmetry and hence  we shall only study the regularity  in the interval $[0,\frac{\pi}{2}]$ { and check that the left and right derivative at $\pi/2$ coincide}.  Notice that Proposition \ref{prop-higher-reg} tells us that $h$ is continuous in $[0,\pi]$,  for any $n\geq 1$. \\
In order to prove such regularity, let us  first check  that
\begin{equation}\label{estim-den}
\widehat{R}(\phi,\varphi,\eta)\geq C\left\{(\phi-\varphi)^2+(\sin^2(\phi)+\sin^2(\varphi))\sin^2 (\eta/2)\right\}, \forall\,\phi,\varphi\in(0,\pi), \eta\in(0,2\pi),
\end{equation}
which is the key point in this proof. In order to do so, recall first from \eqref{Chord} that 
\begin{equation}\label{estim-den-1}
\widehat{R}(\phi,\varphi,\eta)\geq C(\phi-\varphi)^2.
\end{equation}
On the other hand, define the function
$$
g_1(x)=x^2+r_0(\varphi)^2-2xr_0(\varphi)\cos(\eta)+(\cos(\varphi)-\cos(\phi))^2,
$$
which obviously verifies  $g_1(r_0(\phi))=\widehat{R}(\phi,\varphi,\eta)$. Such function has a minimum located at
$$
x_c=r_0(\varphi)\cos(\eta).
$$
Now we shall distinguish two cases: $\cos\eta\in[0,1]$ and $\cos\eta\in[-1,0].$ In the first case we get
\begin{align*}
g_1(x)\geq& g_1(x_c)=r_0^2(\varphi)\sin^2(\eta)+(\cos(\varphi)-\cos(\phi))^2\\
\geq & r_0^2(\varphi)\sin^2(\eta).
\end{align*}
From elementary  trigonometric relations we deduce that
$$
\sin^2\eta= 2\sin^2(\eta/2)\big(1+\cos(\eta)\big)\geq 2\sin^2(\eta/2).
$$
This implies in particular that, for $\cos\eta\in[0,1]$
$$
\widehat{R}(\phi,\varphi,\eta)\geq 2 r_0^2(\varphi)\sin^2(\eta/2).
$$
As to the second case $\cos\eta\in[-1,0]$, we simply notice that  the critical point $x_c$ is negative and therefore the second degree polynomial  $g_1$  is strictly increasing in $\R_+$. This implies that 
\begin{align*}
\widehat{R}(\phi,\varphi,\eta)=g_1(r_0(\phi))\geq& g_1(0)
\geq r_0^2(\varphi)
\geq r_0^2(\varphi)\sin^2(\eta/2).
\end{align*}
Therefore we get in both cases
\begin{equation}\label{estim-den-2}
\widehat{R}(\phi,\varphi,\eta)\geq  r_0^2(\varphi)\sin^2(\eta/2).
\end{equation}
By the  symmetry property $\widehat{R}(\phi,\varphi,\eta)=\widehat{R}(\varphi,\phi,\eta)$ we also get 
\begin{equation}\label{estim-den-3}
\widehat{R}(\phi,\varphi,\eta)\geq r_0^2(\phi)\sin^2(\eta/2).
\end{equation}
Adding  together \eqref{estim-den-1}--\eqref{estim-den-2}--\eqref{estim-den-3}, we achieve
\begin{equation}\label{estim-den-4}
3\widehat{R}(\phi,\varphi,\eta)\geq C(\phi-\varphi)^2+(r_0^2(\phi)+r_0^2(\varphi))\sin^2(\eta/2).
\end{equation}
It suffices now to combine this inequality with the assumption    {\bf (H2)} on $r_0$ in order to get  the desired estimate \eqref{estim-den}.

Let us now prove that  $\mathcal{F}_n(h)\in \mathscr{C}^{1,\alpha}$ and for this aim we shall  proceed in four steps.

\medskip
\noindent
$\bullet$ {\bf Step 1:} If $h\in L^\infty$ then  $\mathcal{F}_n(h)\in \mathscr{C}^\alpha(0,\pi)$.

Here we check that $\mathcal{F}_n(h)\in \mathscr{C}^\alpha(0,\pi)$ for any $n\geq 1$. 
In order to avoid the singularity in the denominator coming from $r_0$, we integrate by parts in the variable $\eta$  
\begin{align*}
\mathcal{F}_n(h)(\phi)=-\frac{1}{4\pi n}\bigintss_0^\pi\bigintss_0^{2\pi}\frac{\sin(\varphi)r_0^2(\varphi)\sin(n\eta)\sin(\eta) h(\varphi)}{\widehat{R}^{\frac32}(\phi,\varphi,\eta)}d\eta d\varphi.
\end{align*}
Introduce
$$
K_1(\phi,\varphi,\eta):=\frac{\sin(\varphi)r_0^2(\varphi)\sin(n\eta)\sin(\eta) 
h(\varphi)}{\widehat{R}^{\frac32}(\phi,\varphi,\eta)},
$$
and according to  Chebyshev polynomials we know  that 
\begin{equation}\label{Cheby1}\sin (n\eta)=\sin (\eta) \,U_{n-1}(\cos \eta),
\end{equation}
with $U_{n}$ being a polynomial of degree $n$. Thus
$$
K_1(\phi,\varphi,\eta)=\frac{\sin(\varphi)r_0^2(\varphi)U_{n-1}(\cos \eta)\sin^2(\eta) 
h(\varphi)}{\widehat{R}^{\frac32}(\phi,\varphi,\eta)}\cdot
$$
Using the assumption {\bf (H2)} combined with the  estimate \eqref{estim-den} for the denominator $\widehat{R}(\phi,\varphi,\eta)$, we achieve
\begin{align*}
|K_1(\phi,\varphi,\eta)|\lesssim &\frac{\|h\|_{L^\infty}\sin^3(\varphi) \sin^2(\eta/2)}{\left((\phi-\varphi)^2+(\sin^2(\phi)+\sin^2(\varphi))\sin^2 (\eta/2)\right)^\frac32}\\
\lesssim &\frac{\sin(\varphi)}{\left((\phi-\varphi)^2+(\sin^2(\phi)+\sin^2(\varphi))\sin^2 (\eta/2)\right)^\frac12}\cdot
\end{align*}
Interpolating between the two inequalities
$$
\frac{\sin(\varphi)}{\left((\phi-\varphi)^2+(\sin^2(\phi)+\sin^2(\varphi))\sin^2 (\eta/2)\right)^\frac12}\le |\varphi-\phi|^{-1}
$$
and
$$ \frac{\sin(\varphi)}{\left((\phi-\varphi)^2+(\sin^2(\phi)+\sin^2(\varphi))\sin^2 (\eta/2)\right)^\frac12}\leq\sin^{-1} (\eta/2),
$$
we deduce that for any $\beta\in[0,1]$
 \begin{equation}\label{Tad11}
 \frac{\sin(\varphi)}{\left((\phi-\varphi)^2+(\sin^2(\phi)+\sin^2(\varphi))\sin^2 (\eta/2)\right)^\frac12}\lesssim |\varphi-\phi|^{-(1-\beta)}\sin^{-\beta} (\eta/2).
 \end{equation}
Then,
\begin{align*}
|K_1(\phi,\varphi,\eta)|\lesssim &\frac{1}{|\phi-\varphi|^{1-\beta}\sin^\beta(\eta/2)}\cdot
\end{align*}
Let us now bound the derivative $\partial_\phi K_1(\phi,\varphi,\eta)$. For this purpose,  let us first  show that
\begin{equation}\label{zamma1}
{|\partial_\phi \widehat{R}(\phi,\varphi,\eta)|}\lesssim \widehat{R}^{\frac12}(\phi,\varphi,\eta).
\end{equation}
Indeed
\begin{align*}
\frac{\partial_\phi\widehat{R}(\phi,\varphi,\eta)}{\widehat{R}^\frac12(\phi,\varphi,\eta)}=\frac{2r_0'(\phi)(r_0(\phi)-r_0(\varphi))+2r_0'(\phi)r_0(\varphi)(1-\cos(\eta))+2\sin(\phi)(\cos(\varphi)-\cos(\phi))}{\widehat{R}^\frac12(\phi,\varphi,\eta)}\cdot
\end{align*}
Using the identity  $1-\cos(\eta)=2 \sin^2(\eta/2)$ and \eqref{estim-den}, we get a constant $C$ such that 
\begin{align*}
\frac{\big|\partial_\phi\widehat{R}(\phi,\varphi,\eta)\big|}{\widehat{R}^\frac12(\phi,\varphi,\eta)}\lesssim \frac{|\phi-\varphi|+\sin(\varphi)\sin^2(\eta/2)}{\left((\phi-\varphi)^2+(\sin^2(\phi)+\sin^2(\varphi))\sin^2(\eta/2)\right)^\frac12}\leq C,
\end{align*}
achieving \eqref{zamma1}. Therefore, taking the  derivative in $\phi$ of $K_1$ yields
\begin{align*}
\left|\partial_\phi K_1(\phi,\varphi,\eta)\right|\leq &C\|h\|_{L^\infty}\frac{\sin^3(\varphi)\sin^2(\eta/2)}{\left((\phi-\varphi)^2+(\sin^2(\phi)+\sin^2(\varphi))\sin^2(\eta/2)\right)^2}\\
\lesssim &\frac{\sin(\varphi)}{(\phi-\varphi)^2+(\sin^2(\phi)+\sin^2(\varphi))\sin^2(\eta/2)}\\
\lesssim &\frac{|\phi-\varphi|^{-1}\sin(\varphi)}{\left((\phi-\varphi)^2+(\sin^2(\phi)+\sin^2(\varphi))\sin^2(\eta/2)\right)^{\frac12}}\cdot
\end{align*}
Hence  \eqref{Tad11} allows to get or any $\beta\in(0,1)$,
\begin{align*}
\left|\partial_\phi K_1(\phi,\varphi,\eta)\right|
\lesssim &\frac{1}{|\phi-\varphi|^{2-\beta}\sin^{\beta}(\eta/2)}\cdot
\end{align*}
{ Here, we can use Proposition \ref{prop-potentialtheory} to the case where the operator  $\mathcal{K}$  depends  only on one variable by taking $g_1(\theta,\eta)=g_3(\theta,\eta)=\sin^{-\beta}(\eta/2)$. Then,  we infer  $\mathcal{F}_n(h)\in \mathscr{C}^\beta(0,\pi)$ for any $\beta\in(0,1)$.
}

\medskip
\noindent
$\bullet$ {\bf Step 2:} For $n\geq 2$, if $h$ is bounded then {$h(0)=h(\pi)=0$}.

Notice that this property was shown in Proposition \ref{prop-higher-reg} and we give here an alternative proof. 
Since   $\nu_\Omega$ is not vanishing then this amounts to checking that $\mathcal{F}_n(h)(0)=0$. By continuity, it is clear by Fubini  that 
\begin{align*}
\mathcal{F}_n(h)(0)=&-\frac{1}{4\pi n }\bigintss_0^\pi\bigintss_0^{2\pi}\frac{\big(\sin(\varphi)r_0^2(\varphi)\sin(n\eta)\sin(\eta)h(\varphi)}{\big(r_0^2(\varphi)+(1-\cos\varphi)^2\big)^\frac32}d\eta d\varphi\\
=&-\frac{1}{4\pi n }\bigintss_0^\pi\frac{\sin(\varphi)r_0^2(\varphi)h(\varphi)d\varphi}{\big(r_0^2(\varphi)+(1-\cos\varphi)^2\big)^\frac32}\bigintsss_0^{2\pi}\sin(n\eta)\sin(\eta) d\eta,
\end{align*}
which is vanishing if $n\geq 2$. Hence, $h(0)=0$, for any $n\geq 2$. {In the same way, we achieve $h(\pi)=0$.}

\medskip
\noindent
$\bullet$ {\bf Step 3:} If $h\in \mathscr{C}^\alpha(0,\pi)$ and {$h(0)=h(\pi)=0$}, then $\mathcal{F}_n(h)\in W^{1,\infty}(0,\pi)$.

We have shown before that $\mathcal{F}_n(h)\in \mathscr{C}^\beta(0,\pi)$ for any $\beta\in(0,1)$.
Then it   is enough to check that $\mathcal{F}_n(h)^\prime\in L^\infty(0,\pi)$. For this aim, we write 
\begin{align}\label{Fnprime}
\mathcal{F}_n(h)^\prime(\phi)=&\frac38\frac{1}{\pi n}\bigintsss_0^\pi\bigintsss_0^{2\pi}\frac{\sin(\varphi)r_0(\varphi)^2\sin(n\eta)\sin(\eta) h(\varphi)\partial_\phi \widehat{R}(\phi,\varphi,\eta)}{\widehat{R}^\frac52(\phi,\varphi,\eta)}d\eta d\varphi.
\end{align}
Adding and subtracting some appropriate terms, we find
\begin{align}\label{Splitt11}
\mathcal{F}_n(h)^\prime(\phi)=&\frac{3}{8\pi n}\bigintsss_0^\pi\bigintsss_0^{2\pi}\frac{\sin(\varphi)r_0^2(\varphi)\sin(n\eta)\sin(\eta) \big(h(\varphi)-h(\phi)\big)\partial_\phi \widehat{R}(\phi,\varphi,\eta)}{\widehat{R}^\frac52(\phi,\varphi,\eta)}d\eta d\varphi \nonumber\\
&+\frac{3h(\phi)}{8\pi n}\bigintsss_0^\pi\bigintsss_0^{2\pi}\frac{\sin(\varphi)r_0^2(\varphi)\sin(n\eta)\sin(\eta) \big[\partial_\phi \widehat{R}(\phi,\varphi,\eta)+\partial_\varphi \widehat{R}(\phi,\varphi,\eta)\big]}{\widehat{R}^\frac52(\phi,\varphi,\eta)}d\eta d\varphi \nonumber\\
&-\frac{3h(\phi)}{8\pi n}\bigintsss_0^\pi\bigintsss_0^{2\pi}\frac{\sin(\varphi)r_0^2(\varphi)\sin(n\eta)\sin(\eta) \partial_\varphi \widehat{R}(\phi,\varphi,\eta)}{\widehat{R}^\frac52(\phi,\varphi,\eta)}d\eta d\varphi\nonumber\\
:=& \frac{3}{8\pi n}(I_1+I_2-I_3)(\phi).
\end{align}
Let us bound each term separately. Using \eqref{estim-den}, \eqref{Cheby1}  and \eqref{zamma1} we achieve
\begin{align*}
|I_1(\phi)|\leq& C\|h\|_{\mathscr{C}^\alpha}\bigintsss_0^\pi\bigintsss_0^{2\pi}\frac{\sin^3(\varphi)\sin^2(\eta)|\varphi-\phi|^\alpha d\eta d\varphi}{\left((\phi-\varphi)^2+(\sin^2(\phi)+\sin^2(\varphi))\sin^2(\eta/2)\right)^2}\\
\leq& C\|h\|_{\mathscr{C}^\alpha}\bigintsss_0^\pi\bigintsss_0^{2\pi}\frac{\sin(\varphi)|\varphi-\phi|^\alpha d\eta d\varphi}{(\phi-\varphi)^2+(\sin^2(\phi)+\sin^2(\varphi))\sin^2(\eta/2)}
.\end{align*}
We write in view of \eqref{Tad11}
\begin{align*}
\frac{\sin(\varphi)}{(\phi-\varphi)^2+(\sin^2(\phi)+\sin^2(\varphi))\sin^2(\eta/2)}\leq&\frac{|\phi-\varphi|^{-1}\sin(\varphi)}{\left((\phi-\varphi)^2+(\sin^2(\phi)+\sin^2(\varphi))\sin^2(\eta/2)\right)^\frac12}\\
\lesssim&\frac{1}{|\phi-\varphi|^{2-\beta}\sin^\beta(\eta/2)}\cdot
\end{align*}
Therefore by imposing $1-\alpha<\beta<1$ we get
\begin{align*}
|I_1(\phi)| 
\leq C\|h\|_{\mathscr{C}^\alpha}\bigintsss_0^\pi\bigintsss_0^{2\pi}\frac{d\eta d\varphi}{|\phi-\varphi|^{2-\alpha-\beta}\sin^{\beta}(\eta/2)}
\leq C\|h\|_{\mathscr{C}^\alpha},
\end{align*}
which implies immediately  that  $I_1\in L^\infty$. Now let us move to the boundedness of $I_2$. From direct computations we get  
\begin{align}\label{MahmaX1}
\Big|\big(\partial_\phi +\partial_\varphi\big)\widehat{R}(\phi,\varphi,\eta)\Big|=&\Big|2(r_0(\phi)-r_0(\varphi))(r_0^\prime(\phi)-r_0^\prime(\varphi))+2(\cos\phi-\cos\varphi)(\sin\varphi-\sin\phi)\nonumber\\
+&2(1-\cos\eta)\big(r_0^\prime(\phi)r_0(\varphi)+r_0^\prime(\varphi) r_0(\phi)\big)\Big|.
\end{align}

Combining the assumption {\bf{(H2)}} with  $r_0'\in W^{1,\infty}$ and the mean value theorem  yields
\begin{align}\label{dif-den}
\Big|\big(\partial_\phi +\partial_\varphi\big)\widehat{R}(\phi,\varphi,\eta)\Big|
\leq & C\Big(  |\phi-\varphi |^2+(\sin\varphi+\sin\phi)\sin^2(\eta/2) \Big).
\end{align}
Hence, using \eqref{estim-den} and \eqref{Cheby1} we obtain
\begin{align*}
|I_2(\phi)|\leq& C{|h(\phi)-h(0)|}\bigintsss_0^\pi\bigintsss_0^{2\pi}\frac{\sin^3(\varphi)\sin^2(\eta)
\left(|\phi-\varphi|^{2}+(\sin\varphi+\sin\phi)\sin^2(\eta/2)\right) }{\left((\phi-\varphi)^2+(\sin^2(\phi)+\sin^2(\varphi))\sin^2(\eta/2)\right)^\frac{5}{2}}d\eta d\varphi\\
\leq& 
C {\|h\|_{\mathscr{C}^\alpha}}
{ \bigintsss_0^\pi\bigintsss_0^{2\pi}\frac{\phi ^{\alpha}d\eta d\varphi}
{\left((\phi-\varphi)^2+(\sin^2(\phi)+\sin^2(\varphi))\sin^2(\eta/2)\right)^{\frac12}}}\cdot
\end{align*}
Interpolation inequalities imply 
\begin{equation}\label{Low-bdd}
\frac{1}{\big((\phi-\varphi)^2+(\sin^2(\varphi)+\sin^2(\phi))\sin^2(\eta/2)\big)^{\frac12}}\le |\phi-\varphi|^{\alpha-1}\sin^{-\alpha}(\phi)\sin^{-\alpha}(\eta/2).\end{equation}
Therefore we get for any $\phi\in(0,\pi/2), \varphi\in(0,\pi)$ and $\eta\in(0,2\pi),$
\begin{align}\label{Tahy1}
\frac{\phi ^{\alpha}}
{\big((\phi-\varphi)^2+(\sin(\phi)+\sin(\varphi))^2\sin^2(\eta/2)\big)^{\frac12}}
\lesssim |\phi-\varphi|^{\alpha-1}\sin^{-\alpha}(\eta/2).
\end{align}
It follows that
\begin{align*}
|I_2(\phi)|\leq& C\|h\|_{\mathscr{C}^\alpha}
\bigintsss_0^\pi\bigintsss_0^{2\pi}\frac{d\eta d\varphi}{|\phi-\varphi|^{1-\alpha}\sin^\alpha(\eta/2)}
\leq C||h||_{\mathscr{C}^\alpha},
\end{align*}
which gives the boundedness of $I_2.$ It remains to bound 
the last  term $I_3$. Then  integrating by parts we infer
\begin{align*}
I_3(\phi)={{\frac23}}h(\phi)\bigintsss_0^\pi\bigintsss_0^{2\pi}\frac{\partial_{\varphi}\big(\sin(\varphi)r_0^2(\varphi)\big)\sin(n\eta)\sin(\eta) }
{\widehat{R}^\frac32(\phi,\varphi,\eta)}d\eta d\varphi.
\end{align*}
{To make the previous integration by parts rigorously, we should split the integral in $\varphi\in(\varepsilon,\phi-\varepsilon)$ and $\varphi\in(\pi+\varepsilon,\pi-\varepsilon)$ and later taking the limit as $\varepsilon\rightarrow 0$. }

Then, since $h(0)=0$ and $h\in \mathscr{C}^\alpha$ we find according to the assumptions {\bf (H)}, \eqref{estim-den} and \eqref{Cheby1}
\begin{align*}
|I_3(\phi)|\leq&C \|h\|_{\mathscr{C}^\alpha}\bigintss_0^\pi\bigintss_0^{2\pi}\frac{\phi^\alpha\sin^2(\varphi)\sin^2(\eta) d\eta d\varphi}{\big[(\phi-\varphi)^2
+(\sin^2(\phi)+\sin^2(\varphi))\sin^2(\eta/2)\big]^\frac{3}{2}}\\
\leq& C\|h\|_{\mathscr{C}^\alpha}\bigintss_0^\pi\bigintss_0^{2\pi}\frac{\phi^\alpha d\eta d\varphi}{\big[(\phi-\varphi)^2
+(\sin^2(\phi)+\sin^2(\varphi))\sin^2(\eta/2)\big]^\frac{1}{2}}\cdot\end{align*}
Applying  \eqref{Tahy1} yields
\begin{align*}
|I_3(\phi)|
\leq& C\|h\|_{\mathscr{C}^\alpha}\bigintss_0^\pi\bigintss_0^{2\pi}\frac{d\eta d\varphi}{|\phi-\varphi|^{1-\alpha}\sin^\alpha(\eta/2)}\leq C\|h\|_{\mathscr{C}^\alpha}.
\end{align*}
That implies that $h\in W^{1,\infty}(0,\pi/2)$.

{Moreover, since $r_0'(\pi/2)=0$ (this comes from the symmetry of $r_0$ with respect to $\pi/2$) and using \eqref{Fnprime} we find that $h'(\pi/2)=0$, which justifies why we can check the regularity only on $\phi\in(0,\pi/2)$. Finally, we get the desired result, that is, $h\in W^{1,\infty}(0,\pi)$.}

\medskip
\noindent
$\bullet$ {\bf Step 4:} If $h^\prime\in L^\infty(0,\pi)$ and {$h(0)=h(\pi)=0$}, then $\mathcal{F}_n(h)^\prime\in \mathscr{C}^\beta(0,\pi)$ for any $\beta\in(0,1)$.

Coming back to \eqref{Splitt11} and integrating by parts in the last integral we deduce 
 \begin{align}\label{Splitt17}
\mathcal{F}_n(h)^\prime(\phi)=&\frac{3}{8\pi n}\bigintss_0^\pi\bigintss_0^{2\pi}\frac{\sin(\varphi)r_0^2(\varphi)h(\varphi)\sin(n\eta)\sin(\eta) \big[\partial_\phi \widehat{R}(\phi,\varphi,\eta)+\partial_\varphi \widehat{R}(\phi,\varphi,\eta)\big]}{\widehat{R}^\frac52(\phi,\varphi,\eta)}d\eta d\varphi \nonumber\\
&-\frac{1}{4\pi n}\bigintss_0^\pi\bigintss_0^{2\pi}\frac{\partial_\varphi\big(\sin(\varphi)r_0^2(\varphi) h(\varphi)\big)\sin(n\eta)\sin(\eta) }{\widehat{R}^\frac32(\phi,\varphi,\eta)}d\eta d\varphi\nonumber\\
&\qquad \qquad :=\frac{3}{8\pi n}\bigintsss_0^\pi\bigintsss_0^{2\pi} \big(T_1-{\frac23} T_2\big)(\phi,\varphi,\eta) d\eta d\varphi,
\end{align}
 with 
 $$
  T_2(\phi,\varphi,\eta)=\frac{\partial_\varphi\big(\sin(\varphi)
r_0(\varphi)^2h(\varphi)\big)\sin(n\eta)\sin(\eta) }{\widehat{R}^{\frac32}(\phi,\varphi,\eta)}\cdot
$$
{As in the previous step, integration by parts can be justified by splitting the integral in $\varphi\in(\varepsilon,\phi-\varepsilon)$ and $\varepsilon\in(\phi+\varepsilon,\pi-\varepsilon)$ and later taking limits as $\varepsilon\rightarrow 0$.}

We want to apply Proposition \ref{prop-potentialtheory} to  each of those terms. First, for $T_1$ we use {\bf{(H)}} and  \eqref{dif-den} combined with 
 \eqref{estim-den}, we arrive at
\begin{align*}
|T_1(\phi,\varphi,\eta)|\lesssim &\,\,\frac{\sin^3(\varphi)|h(\varphi)|\sin^2(\eta)\big(|\phi-\varphi|^2+(\sin\phi+\sin\varphi)\sin^2({\eta/2})\big)}{\left\{(\phi-\varphi)^2+(\sin^2(\phi)+\sin^2(\varphi))\sin^2(\eta/2)\right\}^\frac52}\\
\lesssim &\,\,\frac{|h(\varphi)|}{\left\{(\phi-\varphi)^2+(\sin^2(\phi)+\sin^2(\varphi))\sin^2(\eta/2)\right\}^\frac12}\cdot
\end{align*}
Since $h'\in L^\infty$ and $h(0)=h(\pi)=0$ then we can write $h(\varphi)=\sin(\varphi)\widehat{h}(\varphi),$ with $\widehat{h}\in L^\infty(0,\pi)$.

Consequently,
\begin{align*}
|T_1(\phi,\varphi,\eta)|
\lesssim &||\widehat{h}||_{L^\infty}\frac{\sin(\varphi)}{\left\{(\phi-\varphi)^2+\sin^2(\varphi)\sin^2(\eta/2)\right\}^\frac12}\\
\lesssim &\frac{1}{\left\{(\phi-\varphi)^2+\sin^2(\eta/2)\right\}^\frac12}\cdot
\end{align*}
Interpolating again, we find that for any  $\beta\in[0,1]$
$$
|T_1(\phi,\varphi,\eta)|\lesssim \frac{1}{|\phi-\varphi|^{1-\beta}\sin^\beta(\eta/2)}\cdot
$$
Let us mention that we have proven that
\begin{align}\label{estim-h}
\frac{|h(\varphi)|}{\left\{(\phi-\varphi)^2+(\sin^2(\phi)+\sin^2(\varphi))\sin^2(\eta/2)\right\}^\frac12}\leq C\frac{1}{|\phi-\varphi|^{1-\beta}\sin^\beta(\eta/2)},
\end{align}
for any $\beta\in(0,1)$, $\phi\in(0,\pi/2), \varphi\in(0,\pi)$ and $\eta\in(0,2\pi)$, which will be useful later.

Now we shall  estimate the derivative  of $T_1$ with respect to $\phi$. We start with 
\begin{align}\label{MahmaX1-2}
\Big|\partial_\phi\left\{\big(\partial_\phi +\partial_\varphi\big)\widehat{R}(\phi,\varphi,\eta)\right\}\Big|=&\Big|2r_0'(\phi)(r_0^\prime(\phi)-r_0^\prime(\varphi))+2(r_0(\phi)-r_0(\varphi))r_0''(\phi)\nonumber\\
&-2\sin(\phi)(\sin\varphi-\sin\phi)-2(\cos\phi-\cos\varphi)\cos(\phi)\nonumber\\
&+2(1-\cos\eta)\big(r_0''(\phi)r_0(\varphi)+r_0^\prime(\varphi) r_0'(\phi)\big)\Big|.
\end{align}
Using that $r_0''\in L^\infty$, we find
\begin{align}\label{MahmaX1-3}
\Big|\partial_\phi\left\{\big(\partial_\phi +\partial_\varphi\big)\widehat{R}(\phi,\varphi,\eta)\right\}\Big|\leq&C\left(|\phi-\varphi|+\sin^2(\eta/2)\right).
\end{align}
Thus,
\begin{align*}
|\partial_\phi T_1(\phi,\varphi,\eta)|\lesssim &\frac{\sin^3(\varphi)|h(\varphi)|\sin^2(\eta)\Big|\partial_\phi\left\{\big(\partial_\phi +\partial_\varphi\big)\widehat{R}(\phi,\varphi,\eta)\right\}\Big|}{\widehat{R}^\frac52(\phi,\varphi,\eta)}\\
&+\frac{\sin^3(\varphi)|h(\varphi)|\sin^2(\eta)\left|\big(\partial_\phi +\partial_\varphi\big)\widehat{R}(\phi,\varphi,\eta)\right|\left|\partial_\phi \widehat{R}(\phi,\varphi,\eta)\right|}{\widehat{R}^\frac72(\phi,\varphi,\eta)}\cdot
\end{align*}
Using \eqref{estim-den}--\eqref{dif-den}--\eqref{MahmaX1-3}, we find
\begin{align*}
|\partial_\phi T_1(\phi,\varphi,\eta)|\lesssim &\frac{\sin^3(\varphi)|h(\varphi)|\sin^2(\eta)\left\{|\phi-\varphi|+\sin^2(\eta/2)\right\}}{\left\{(\phi-\varphi)^2+(\sin^2(\phi)+\sin^2(\varphi))\sin^2(\eta/2)\right\}^\frac52}\\
&+\frac{\sin^3(\varphi)|h(\varphi)|\sin^2(\eta)\left\{|\phi-\varphi|^2+(\sin(\varphi)+\sin(\phi))\sin^2(\eta/2)\right\}}{\left\{(\phi-\varphi)^2+(\sin^2(\phi)+\sin^2(\varphi))\sin^2(\eta/2)\right\}^3}\cdot
\end{align*}
It follows that
\begin{align*}
|\partial_\phi T_1(\phi,\varphi,\eta)|\lesssim &\frac{\sin(\varphi)|h(\varphi)|\left\{|\phi-\varphi|+\sin^2(\eta/2)\right\}}{\left\{(\phi-\varphi)^2+(\sin^2(\phi)+\sin^2(\varphi))\sin^2(\eta/2)\right\}^\frac32}\\
&+\frac{\sin(\varphi)|h(\varphi)|\left\{|\phi-\varphi|^2+(\sin(\varphi)+\sin(\phi))\sin^2(\eta/2)\right\}}{\left\{(\phi-\varphi)^2+(\sin^2(\phi)+\sin^2(\varphi))\sin^2(\eta/2)\right\}^2}\\
\lesssim &\frac{|h(\varphi)|}{\left\{(\phi-\varphi)^2+(\sin^2(\phi)+\sin^2(\varphi))\sin^2(\eta/2)\right\}}\cdot
\end{align*}
Putting together this estimate with   \eqref{estim-h} we infer
\begin{align*}
|\partial_\phi T_1(\phi,\varphi,\eta)|\lesssim& \frac{|\phi-\varphi|^{-1}|h(\varphi)|}{\left\{(\phi-\varphi)^2+(\sin^2(\phi)+\sin^2(\varphi))\sin^2(\eta/2)\right\}^{\frac12}}\\
\lesssim& C|\phi-\varphi|^{-(2-\beta)}\sin^{-\beta}(\eta/2),
\end{align*}
for any $\beta\in(0,1)$.

Concerning the estimate of the term 
 $T_2$, we first make appeal to   \eqref{estim-den} and \eqref{Cheby1} leading to
\begin{align*}
|T_2(\phi,\varphi,\eta)|\lesssim &\frac{(\sin^3(\varphi)+\sin^2(\varphi)|h(\varphi)|)\sin^2(\eta/2)}{\left\{(\phi-\varphi)^2+(\sin^2(\phi)+\sin^2(\varphi))\sin^2(\eta/2)\right\}^\frac32}\\
\lesssim&\frac{(\sin(\varphi)+|h(\varphi)|)}{\left\{(\phi-\varphi)^2+(\sin^2(\phi)+\sin^2(\varphi))\sin^2(\eta/2)\right\}^\frac12}\cdot
\end{align*}
Applying \eqref{Tad11} and \eqref{estim-h}, one finds
$$
|T_2(\phi,\varphi,\eta)|\leq C |\phi-\varphi|^{\beta-1}\sin^{-\beta}(\eta/2),
$$
for any $\beta\in(0,1)$. The next stage is devoted to the estimate of $\partial_\phi T_2$ and one gets from direct computations
\begin{align*}
|\partial_\phi T_2(\phi,\varphi,\eta)|\lesssim &\frac{(\sin^3(\varphi)+\sin^2(\varphi)|h(\varphi)|)\sin^2(\eta/2)\left|\partial_\phi \widehat{R}(\phi,\varphi,\eta)\right|}{\widehat{R}^\frac52(\phi,\varphi,\eta)}\cdot
\end{align*}
Using \eqref{zamma1} and \eqref{estim-den}, it implies
\begin{align*}
|\partial_\phi T_2(\phi,\varphi,\eta)|\lesssim &\frac{(\sin^3(\varphi)+\sin^2(\varphi)|h(\varphi)|)\sin^2(\eta/2)}{\left\{(\phi-\varphi)^2+(\sin^2(\phi)+\sin^2(\varphi))\sin^2(\eta/2)\right\}^2}\\
\lesssim &\frac{(\sin(\varphi)+|h(\varphi)|)}{(\phi-\varphi)^2+(\sin^2(\phi)+\sin^2(\varphi))\sin^2(\eta/2)}\cdot
\end{align*}
Therefore  \eqref{Tad11} and \eqref{estim-h} allows to get
$$
|\partial_\phi T_2(\phi,\varphi,\eta)|\lesssim |\phi-\varphi|^{-(2-\beta)}\sin^\beta(\eta/2),
$$
for any $\beta\in(0,1)$. Hence, by Proposition \ref{prop-potentialtheory}{, adapted to a one variable function,} we achieve that $\mathcal{F}_n(h)'\in \mathscr{C}^\beta$, for any $\beta\in(0,1)$, which achieves the proof of the announced result. \end{proof}}
}
{
\subsection{Fredholm structure}
In this section we shall be concerned with the Fredholm structure of the linearized operator $\partial_f \tilde{F}(\Omega,0)$ defined through \eqref{Prop-lin2-expression} and \eqref{operator}. Our  main result reads as follows.

\begin{pro}\label{cor-fredholm}
Let $m\geq2,$ $\alpha\in(0,1)$ and  $\Omega\in(-\infty,\kappa)$, then  $\partial_f \tilde{F}(\Omega,0):X_m^\alpha\rightarrow X_m^\alpha$ is a well--defined  Fredholm operator with zero index. In addition, for $\Omega=\Omega_m$, the kernel of $\partial_f \tilde{F}(\Omega_m,0)$ is one--dimensional and its range is closed and of co-dimension one.

Recall that the spaces $X_m^\alpha$ have been introduced in \eqref{space} and $\Omega_m$ in Proposition $\ref{prop-kernel-onedim}.$
\end{pro}
}

\begin{proof}
We shall first prove the second part, assuming the first one. The structure of the linearized operator is detailed in  \eqref{Prop-lin2-expression} and one has for $h(\phi,\theta)=\sum_{n\geq1} h_n(\phi)\cos(n\theta)$
\begin{equation*}
\partial_{f} \tilde{F}(\Omega,0)h(\phi,\theta)=\sum_{n\geq 1}\cos(n\theta)\mathcal{L}_n^\Omega (h_n)(\phi),
\end{equation*}
where
\begin{align*}
\mathcal{L}_n^\Omega(h_n)(\phi)=&\nu_\Omega(\phi)h_n(\phi)-\int_0^\pi H_n(\phi,\varphi)h_n(\varphi)d\varphi, \quad  \phi\in[0,\pi].
\end{align*}
In view of \eqref{operator} and \eqref{kerneleq}, {$\mathcal{L}_n^\Omega(h)=0$}  can be written in the form 
$$
\mathcal{K}_n^\Omega h=h.
$$
We define the dispersion set by 
$$
\mathcal{S}=\big\{\Omega\in(-\infty,\kappa),\quad\textnormal{Ker} \partial_{f} \tilde{F}(\Omega,0)\neq\{0\}\big\}.
$$
Hence $\Omega\in\mathcal{S}$ if and only if there exists $m\geq1$ such that the equation 
$$
\forall \phi\in[0,\pi],\quad \mathcal{K}_m^\Omega(h_m)(\phi)=h_m(\phi),
$$
admits a nontrivial  solution satisfying the regularity $h_m\in \mathscr{C}^{1,\alpha}(0,\pi)$ and the boundary condition $h_m(0)=h_m(\pi)=0.$ By virtue of Proposition \ref{prop-higher-reg} and Proposition \ref{HoldX1} the foregoing conditions are satisfied for any eigenvalue provided that $m\geq2.$ On the other hand, we have shown in Proposition  \ref{prop-operator}-$(4)$  that for  $\Omega=\Omega_m$ the kernel of $\mathcal{L}_m$ is one--dimensional.  Moreover,   Proposition  \ref{prop-operator}-$(5)$ ensures that for any $n>m$ we have $\lambda_n(\Omega_m)<\lambda_m(\Omega_m)=1$. Since by construction $\lambda_n(\Omega_m)$ is the largest eigenvalue of $\mathcal{K}_n^{\Omega_m}$, then $1$ could not be an eigenvalue of this operator and the equation
$$
\mathcal{K}_n^{\Omega_m} h=h,
$$
admits only the trivial solution. Thus the kernel of the restricted operator $\partial_f \tilde{F}(\Omega_m,0):X_m^\alpha\rightarrow X_m^\alpha$  is one-dimensional and is generated by the eigenfunction
$$
(\phi,\theta)\mapsto h_m(\phi) \cos( m\theta).
$$ 
We emphasize that this element belongs to the space $X_m^\alpha$ because it belongs to the function space $\mathscr{C}^{1,\alpha}((0,\pi)\times{\T})$  since  $\phi\mapsto h_m(\phi)\in \mathscr{C}^{1,\alpha}(0,\pi).$
That the range of $\partial_f \tilde{F}(\Omega_m,0) $  is closed and of co-dimension one follows from the fact this operator is Fredholm of zero index.

Next, let us show that $\partial_f \tilde{F}(\Omega,0) $ is Fredholm of zero index. 
By virtue of the computations developed  in Proposition \ref{Prop-lin1} and the expression of \eqref{operator}, we assert that 
$$
\partial_f \tilde{F}(\Omega,0)h(\phi,\theta)=\nu_\Omega(\phi)h(\phi,\theta)-\frac{1}{4\pi}G(h)(\phi,\theta),
$$
with
\begin{align}\label{G-hX1}
G(h)(\phi,\theta)=&\frac{1}{r_0(\phi)}\bigintsss_0^\pi\bigintsss_0^{2\pi}\frac{\sin(\varphi)r_0(\varphi)h(\varphi,\eta)d\eta d\varphi}{A(\phi,\theta,\varphi,\eta)^\frac12},\\
A(\phi,\theta,\varphi,\eta)=&(r_0(\phi)-r_0(\varphi))^2+2r_0(\phi)r_0(\varphi)(1-\cos(\theta-\eta))+(\cos(\phi)-\cos(\varphi))^2.\nonumber
\end{align}
Since $\Omega\in(-\infty,\kappa)$, the function $\nu_\Omega$ is not vanishing. Moreover, by Proposition \ref{Lem-meas} one has that $\nu_\Omega\in \mathscr{C}^{1,\beta}$, for any $\beta\in(0,1)$.

Define the linear operator $\nu_\Omega\textnormal{Id}: X_m^\alpha\rightarrow X_m^\alpha$ by 
$$
(\nu_\Omega\textnormal{Id})(h)(\phi,\theta)=\nu_\Omega(\phi)h(\phi,\theta)\cdot
$$
We shall check  that it defines  an isomorphism. The continuity of this operator follows from the regularity  $\nu_\Omega\in \mathscr{C}^{1,\alpha}(0,\pi)$ combined with the fact that $ \mathscr{C}^{1,\alpha}((0,\pi)\times{\T})$ is an algebra. The Dirichlet boundary condition, the  $m$--fold symmetry and the absence of the frequency zero are  immediate for the product $\nu_\Omega h$, which finally belongs to  $X_m^\alpha$. Moreover, since $\nu_\Omega$ is not vanishing, one has that $\nu_\Omega\textnormal{Id}$ is injective. In order to check that such an operator is an isomorphism, it is enough to check that it is surjective, as a consequence of the Banach theorem. Take $k\in X_m^\alpha$, and we will find $h\in X_m^\alpha$ such that $(\nu_\Omega\textnormal{Id})(h)=k$. Indeed, $h$ is given by
$$
h(\phi,\theta)=\frac{d(\phi,\theta)}{\nu_\Omega(\phi)}\cdot
$$
Using the regularity of $\nu_\Omega$ and the fact that it is not vanishing, it is easy to check that  its inverse $\frac{1}{\nu_\Omega}$ still belongs to $\mathscr{C}^{1,\alpha}(0,\pi)$. Similar arguments as before allow to get $h\in X_m^\alpha$. Hence $\nu_\Omega\textnormal{Id}$ is an isomorphism, and thus it is a Fredholm operator of zero index. From classical results on index theory, it is known that to get   $\partial_f \tilde{F}(\Omega,0)$ is Fredholm  of zero index, it is enough to establish that the perturbation ${G}: X_m^\alpha\to  X_m^\alpha$ is compact. To do so, we prove that for any $\beta\in(\alpha,1)$ one has the smoothing effect
$$
\forall\, h\in X_{m}^\alpha,\quad \|G(h)\|_{\mathscr{C}^{1,\beta}}\leq C\|h\|_{\mathscr{C}^{1,\alpha}},
$$
that we combine with the compact embedding $\mathscr{C}^{1,\beta}\big((0,\pi)\times{\T}\big)\hookrightarrow \mathscr{C}^{1,\alpha}\big((0,\pi)\times{\T}\big)$.

Take $h\in X_m^\alpha$ and let us show that $G(h)\in \mathscr{C}^{1,\beta}\big((0,\pi)\times{\T}\big)$, for any $\beta\in(0,1)$. We shall first deal with a preliminary fact. Define the following function
\begin{equation}\label{gtheta-def}
\forall\, \varphi\in[0,\pi],\, \theta,\eta\in\R,\quad g_\theta (\varphi,\eta)=\int_\theta^\eta h(\varphi,\tau)d\tau.
\end{equation}
By \eqref{platit1} we infer
$$
|g_\theta(\varphi,\eta)|\leq C\|h\|_{\textnormal{Lip}}|\theta-\eta|\sin(\varphi),
$$
According to the definition of the space $X_m^\alpha$, the partial function $\tau\mapsto h(\varphi,\tau)$ is $2\pi$-periodic and with zero  average, that is, $\displaystyle{\int_{0}^{2\pi}h(\varphi,\tau)d\tau=0}$.  This allows to get that $\eta\mapsto g_\theta(\varphi,\eta)$ is also $2\pi$-periodic, and from elementary arguments we find 
\begin{equation}\label{gtheta}
|g_\theta(\varphi,\eta)|\leq C\|h\|_{\textnormal{Lip}}\,|\sin((\theta-\eta)/2)|\sin(\varphi),
\end{equation}
for any $\varphi\in[0,\pi]$ and $\theta,\eta\in[0,2\pi]$. In addition, it is immediate  that $g_\theta\in\mathscr{C}^{1,\alpha}\big((0,\pi)\times{\T}\big)$ and 
$$
 \partial_\varphi g_\theta (\varphi,\eta)=\int_\theta^\eta \partial_\varphi h(\varphi,\tau)d\tau.
$$
The same arguments as before show that the partial function $\tau\mapsto \partial_\varphi h(\varphi,\tau)$ is $2\pi$-periodic and with zero  average. Moreover,  $\eta\mapsto \partial_\varphi g_\theta(\varphi,\eta)$ is also $2\pi$-periodic  and 
\begin{equation}\label{gthetaM1}
|\partial_\varphi g_\theta(\varphi,\eta)|\leq C\|h\|_{\textnormal{Lip}}\,|\sin((\theta-\eta)/2)|,
\end{equation}
for any $\varphi\in[0,\pi]$ and $\theta,\eta\in[0,2\pi]$.
Using the auxiliary function $g_\theta$, one can integrate by parts in $G(h)$ in the variable $\eta$ obtaining
\begin{align}\label{G-2}
G(h)(\phi,\theta)=\bigintsss_0^\pi\bigintsss_0^{2\pi}\frac{\sin(\varphi)r_0^2(\varphi)\sin(\eta-\theta)g_\theta(\varphi,\eta)}{A(\phi,\theta,\varphi,\eta)^\frac32}d\eta d\varphi.
\end{align}
{We can justify the integration by parts by splitting the integral in $\eta\in(0,\theta-\varepsilon)$ and $\eta\in(\theta+\varepsilon,2\pi)$ and later taking limits as $\varepsilon\rightarrow 0$.}
The boundary term in the above integration by parts is vanishing due to the periodicity in $\eta$ of the  involved functions.
It follows from  \eqref{estim-den}, 
\begin{equation}\label{estim-den2}
A(\phi,\theta,\varphi,\eta)\gtrsim (\phi-\varphi)^2
+(\sin^2(\phi)+\sin^2(\varphi))\sin^2((\theta-\eta)/2),
\end{equation}
for any $\phi,\varphi\in(0,\pi)$ and $\theta,\eta\in(0,2\pi)$, and this estimate is crucial in the proof.

The boundedness of $G(h)$ can be implemented  by using \eqref{gtheta} and \eqref{estim-den2}. Indeed, we write
\begin{align*}
|G(h)(\phi,\theta)|\lesssim& \|h\|_{\textnormal{Lip}}\bigintsss_0^\pi\bigintsss_0^{2\pi}\frac{\sin^2(\varphi)r_0^2(\varphi)|\sin(\theta-\eta)||\sin((\theta-\eta)/2)|d\eta d\varphi}{\left((\phi-\varphi)^2+(\sin^2(\varphi)+\sin^2(\phi))\sin^2((\theta-\eta)/2)\right)^\frac32}\\
\lesssim& \|h\|_{\textnormal{Lip}}\bigintsss_0^\pi\bigintsss_0^{2\pi}\frac{r_0^2(\varphi)\sin^2(\varphi)\sin^2((\theta-\eta)/2)d\eta d\varphi}{\left((\phi-\varphi)^2+(\sin^2(\varphi)+\sin^2(\phi))\sin^2((\theta-\eta)/2)\right)^\frac32}\cdot
\end{align*}
Therefore, we obtain
\begin{align*}
|G(h)(\phi,\theta)|
\lesssim& \|h\|_{\textnormal{Lip}}\bigintsss_0^\pi\bigintsss_0^{2\pi}\frac{r_0^2(\varphi)d\eta d\varphi}{\left((\phi-\varphi)^2+(\sin^2(\varphi)+\sin^2(\phi))\sin^2((\theta-\eta)/2)\right)^\frac12}\cdot
\end{align*}
From the assumption ${\bf (H2)}$ on $r_0$ combined with \eqref{Tad11}  we get for any $\beta\in(0,1)$, and then
\begin{align*}
|G(h)(\phi,\theta)|\lesssim& \|h\|_{\textnormal{Lip}}\bigintsss_0^\pi\bigintsss_0^{2\pi}|\phi-\varphi|^{\beta-1}|\sin((\theta-\eta)/2)|^{-\beta} d\eta d\varphi\lesssim \|h\|_{\textnormal{Lip}}.
\end{align*}
Therefore
 \begin{align}\label{boundG1}
\|G(h)\|_{L^\infty}\lesssim \|h\|_{\mathscr{C}^{1,\alpha}}.
\end{align}
The next step is to  check now that {$\partial_\phi G(h)\in \mathscr{C}^\beta$} by making appeal to  Proposition \ref{prop-potentialtheory}. From direct computations using \eqref{G-2} we find
\begin{align*}
\partial_\phi G(h)(\phi,\theta)=\frac32\bigintsss_0^\pi\bigintsss_0^{2\pi}\frac{\sin(\varphi)r_0^2(\varphi)\sin(\theta-\eta)g_\theta(\varphi,\eta)\partial_\phi A(\phi,\theta,\varphi,\eta)}{A(\phi,\theta,\varphi,\eta)^\frac52}d\eta d\varphi.
\end{align*}
Note that we can insert the derivative inside the integral, to make this rigorous, cut off the integral  in $\eta$ away from $\theta$ and take a limit.
Adding and subtracting in the numerator $\partial_\varphi A(\phi,\theta,\varphi,\eta)$, it can be written in the form
\begin{align*}
\partial_\phi G(h)(\phi,\theta)=&\frac32\bigintsss_0^\pi\bigintsss_0^{2\pi}\frac{\sin(\varphi)r_0^2(\varphi)\sin(\theta-\eta)g_\theta(\varphi,\eta)\big(\partial_\phi A(\phi,\theta,\varphi,\eta)+\partial_\varphi A(\phi,\theta,\varphi,\eta)\big)}{A(\phi,\theta,\varphi,\eta)^\frac52}d\eta d\varphi\\
&-\frac32\bigintsss_0^\pi\bigintsss_0^{2\pi}\frac{\sin(\varphi)r_0^2(\varphi)\sin(\theta-\eta)g_\theta(\varphi,\eta)\partial_\varphi A(\phi,\theta,\varphi,\eta)}{A(\phi,\theta,\varphi,\eta)^\frac52}d\eta d\varphi.
\end{align*}
Integrating by parts in $\varphi$ in the last term yields
\begin{align*}
\partial_\phi G(h)(\phi,\theta)=&\frac32\bigintsss_0^\pi\bigintsss_0^{2\pi}\frac{\sin(\varphi)r_0^2(\varphi)\sin(\theta-\eta)g_\theta(\varphi,\eta)\big(\partial_\phi A(\phi,\theta,\varphi,\eta)+\partial_\varphi A(\phi,\theta,\varphi,\eta)\big)}{A(\phi,\theta,\varphi,\eta)^\frac52}d\eta d\varphi\\
&-\bigintsss_0^\pi\bigintsss_0^{2\pi}\frac{\partial_\varphi\left(\sin(\varphi)r_0^2(\varphi)g_\theta(\varphi,\eta)\right)\sin(\theta-\eta)}{A(\phi,\theta,\varphi,\eta)^\frac32}d\eta d\varphi\\
=:&\frac32\mathscr{G}_1(\phi,\theta)-\mathscr{G}_2(\phi,\theta).
\end{align*}
The goal is to check the kernel assumptions for Proposition \ref{prop-potentialtheory} in order to prove that $\mathscr{G}_1$ and $\mathscr{G}_2$ belong to $\mathscr{C}^\beta$, for any $\beta\in(0,1)$. For this aim, we define the kernels
$$
{K}_1(\phi,\theta,\varphi,\eta):=\frac{\sin(\varphi)r_0^2(\varphi)\sin(\theta-\eta)g_\theta(\varphi,\eta)(\partial_\phi +\partial_\varphi)A(\phi,\theta,\varphi,\eta)}{A(\phi,\theta,\varphi,\eta)^\frac52},
$$
and
$$
K_2(\phi,\theta,\varphi,\eta):=\frac{\partial_\varphi\left(\sin(\varphi)r_0^2(\varphi)g_\theta(\varphi,\eta)\right)\sin(\theta-\eta)}{A(\phi,\theta,\varphi,\eta)^\frac32}\cdot
$$
Let us start with  $K_1$ and show that it  satisfies the hypothesis of Proposition \ref{prop-potentialtheory}. From straightforward calculus we obtain in view of the assumptions ${\bf{(H)}}$ and the mean value theorem 
\begin{align}\label{Aphivarphi}
|(\partial_\phi +\partial_\varphi)A(\phi,\theta,\varphi,\eta)|=&|2(r_0'(\phi)-r_0'(\varphi))(r_0(\phi)-r_0(\varphi))\nonumber\\
&+2(r_0(\phi)r_0'(\varphi)+r_0(\varphi)r_0'(\phi))(1-\cos(\theta-\eta))\nonumber\\
&-2(\sin(\phi)-\sin(\varphi))(\cos(\phi)-\cos(\varphi))|\nonumber\\
\lesssim& \,(\phi-\varphi)^2+\big(\sin\varphi+\sin\phi\big)\sin^2((\theta-\eta)/2).
\end{align}
Using the inequality $|ab|\leq \frac12(a^2+b^2)$ allows to get
\begin{align*}
\sin\varphi|(\partial_\phi +\partial_\varphi)A(\phi,\theta,\varphi,\eta)|
\lesssim& \,(\phi-\varphi)^2+\big(\sin^2\varphi+\sin^2\phi\big)\sin^2((\theta-\eta)/2).
\end{align*}
Thus, applying \eqref{estim-den2} we deduce that
\begin{align}\label{denV2}
\sin\varphi\big|(\partial_\phi +\partial_\varphi)A(\phi,\theta,\varphi,\eta)\big|\lesssim& |A(\phi,\theta,\varphi,\eta)|.
\end{align}
Then, putting together   \eqref{gtheta}, ${\bf{(H2)}}$, \eqref{estim-den2} and \eqref{denV2}  we find
\begin{align*}
|K_1(\phi,\theta,\varphi,\eta)|\lesssim& \|h\|_{\textnormal{Lip}}\frac{\sin(\varphi)r_0^2(\varphi)\sin^2((\theta-\eta)/2)}{\left((\phi-\varphi)^2
+(\sin^2(\phi)+\sin^2(\varphi))\sin^2((\theta-\eta)/2)\right)^\frac32}\\
 \lesssim& \|h\|_{\textnormal{Lip}}\frac{\sin(\varphi)}{\left((\phi-\varphi)^2
+(\sin^2(\phi)+\sin^2(\varphi))\sin^2((\theta-\eta)/2)\right)^\frac12}\cdot
\end{align*}
As a consequence of  \eqref{Tad11}, we immediately get
\begin{align}\label{denVM2}
|K_1(\phi,\theta,\varphi,\eta)|\lesssim \|h\|_{\textnormal{Lip}} |\phi-\varphi|^{\beta-1}|\sin((\theta-\eta)/2)|^{-\beta},
\end{align}
for any $\beta\in(0,1)$. Let us compute the derivative with respect to $\phi$ of $K_1$,
\begin{align*}
\partial_\phi K_1(\phi,\theta,\varphi,\eta)=&\frac{\sin(\varphi)r_0^2(\varphi)\sin(\theta-\eta)g_\theta(\varphi,\eta)\partial_\phi((\partial_\phi +\partial_\varphi)A(\phi,\theta,\varphi,\eta))}{A(\phi,\theta,\varphi,\eta)^\frac52}\\
&-\frac52 \frac{\sin(\varphi)r_0^2(\varphi)\sin(\theta-\eta)g_\theta(\varphi,\eta)((\partial_\phi +\partial_\varphi)A(\phi,\theta,\varphi,\eta)) \partial_\phi A(\phi,\theta,\varphi,\eta) }{A(\phi,\theta,\varphi,\eta)^\frac72}\cdot
\end{align*}
From direct computations, we easily get 
\begin{equation}\label{dphiA}
|\partial_\phi((\partial_\phi +\partial_\varphi)A(\phi,\theta,\varphi,\eta))|\lesssim |\phi-\varphi|+\sin^2((\theta-\eta)/2)
\end{equation}
and
\begin{equation}\label{dphiA2}
|\partial_\phi A(\phi,\theta,\varphi,\eta)|\lesssim |\phi-\varphi|+\sin(\varphi)\sin^2((\theta-\eta)/2).
\end{equation}
Then, it is clear from \eqref{estim-den2}  that
\begin{equation}\label{dphiA3}
|\partial_\phi A(\phi,\theta,\varphi,\eta)|\lesssim  A^\frac12 (\phi,\theta,\varphi,\eta).
\end{equation}
In addition, one may  check that
\begin{align}\label{AphivarphiM}
|(\partial_\phi +\partial_\varphi)A(\phi,\theta,\varphi,\eta)|
\lesssim& \,(\phi-\varphi)^2+\big(\sin\varphi+\sin\phi\big)\sin^2((\theta-\eta)/2)\nonumber\\
\lesssim& A^{\frac12}(\phi,\theta,\varphi,\eta)\big(|\varphi-\phi|+|\sin((\theta-\eta)/2)|\big).
\end{align}
By using  \eqref{gtheta}, ${\bf{(H2)}}$, \eqref{estim-den2}, \eqref{dphiA}, \eqref{dphiA3} and \eqref {AphivarphiM}, one achieves
\begin{align*}
|\partial_\phi K_1(\phi,\theta,\varphi,\eta)|\lesssim&\|h\|_{\textnormal{Lip}}\frac{\sin^2(\varphi)\left(|\phi-\varphi|+\sin^2((\theta-\eta)/2)\right)}{\left((\phi-\varphi)^2
+(\sin^2(\phi)+\sin^2(\varphi))\sin^2((\theta-\eta)/2)\right)^\frac32}\\
&+\|h\|_{\textnormal{Lip}} \frac{\sin^2(\varphi)\big(|\varphi-\phi|+|\sin((\theta-\eta)/2)|\big)}{\big((\phi-\varphi)^2
+(\sin^2(\phi)+\sin^2(\varphi))\sin^2((\theta-\eta)/2)\big)^{\frac32}}\\
\lesssim&\|h\|_{\textnormal{Lip}}\frac{\sin^2(\varphi)\left(|\phi-\varphi|+|\sin((\theta-\eta)/2)|\right)}{\left((\phi-\varphi)^2
+(\sin^2(\phi)+\sin^2(\varphi))\sin^2((\theta-\eta)/2)\right)^\frac32}\cdot
\end{align*}
Therefore, using some elementary inequalities allow to get
\begin{align*}
|\partial_\phi K_1(\phi,\theta,\varphi,\eta)|
\lesssim&\|h\|_{\textnormal{Lip}} \frac{\sin\varphi}{(\phi-\varphi)^2
+(\sin^2(\phi)+\sin^2(\varphi))\sin^2((\theta-\eta)/2)}\\
\lesssim& \|h\|_{\textnormal{Lip}}\frac{|\phi-\varphi|^{-1}\,\sin\varphi}{\left((\phi-\varphi)^2
+(\sin^2(\phi)+\sin^2(\varphi))\sin^2((\theta-\eta)/2)\right)^{\frac12}}\cdot
\end{align*}
Applying \eqref{Tad11} implies for any $\beta\in(0,1)$
$$
|\partial_\phi K_1(\phi,\theta,\varphi,\eta)|\leq C \|h\|_{\textnormal{Lip}}|\phi-\varphi|^{-(2-\beta)}|\sin((\theta-\eta)/2)|^{-\beta}.
$$
Now, let us move to the estimate of the partial derivative  $\partial_\theta K_1$, given by 
\begin{align*}
\partial_\theta K_1(\phi,\theta,\varphi,\eta)=&\frac{\sin(\varphi)r_0^2(\varphi)\partial_\theta(\sin(\theta-\eta)g_\theta(\varphi,\eta))(\partial_\phi +\partial_\varphi)A(\phi,\theta,\varphi,\eta)}{A(\phi,\theta,\varphi,\eta)^\frac52}\\
&+\frac{\sin(\varphi)r_0^2(\varphi)\sin(\theta-\eta)g_\theta(\varphi,\eta)\partial_\theta\left\{(\partial_\phi +\partial_\varphi)A(\phi,\theta,\varphi,\eta)\right\}}{A(\phi,\theta,\varphi,\eta)^\frac52}\\
&-\frac52 \frac{\sin(\varphi)r_0^2(\varphi)\sin(\theta-\eta)g_\theta(\varphi,\eta)((\partial_\phi +\partial_\varphi)A(\phi,\theta,\varphi,\eta))(\partial_\theta A(\phi,\theta,\varphi,\eta))}{A(\phi,\theta,\varphi,\eta)^\frac72}\cdot
\end{align*}
By definition of $g_\theta$ in \eqref{gtheta-def} and \eqref{gtheta}, one concludes in view of \eqref{platit1} and \eqref{estim-den2} that
\begin{align}\label{K1theta-1}
|\partial_\theta (\sin(\theta-\eta)g_\theta(\varphi,\eta))|\lesssim &\|h\|_{\textnormal{Lip}}\sin(\varphi)|\sin((\theta-\eta)/2)|\lesssim\|h\|_{\textnormal{Lip}} A^{\frac12}(\phi,\theta,\varphi,\eta).
\end{align}
Moreover, one gets
\begin{align}\label{Athetaphivarphi}
|\partial_\theta\left\{(\partial_\phi +\partial_\varphi)A(\phi,\theta,\varphi,\eta)\right\}|\lesssim& (\sin(\phi)+\sin(\varphi))|\sin(\theta-\eta)|\lesssim A^{\frac12}(\phi,\theta,\varphi,\eta).
\end{align}
Using also the definition of $A$, we obtain
\begin{equation}\label{Atheta}
|\partial_\theta A(\phi,\theta,\varphi,\eta)|\leq C\sin(\phi)\sin(\varphi)|\sin(\theta-\eta)|\lesssim \sin(\varphi)A^\frac12(\phi,\theta,\varphi,\eta).
\end{equation}
Then, with the help of \eqref{gtheta}, \eqref{estim-den2}, \eqref{Aphivarphi} \eqref{K1theta-1}, \eqref{Athetaphivarphi} and \eqref{Atheta}, we can estimate $\partial_\theta K_1$ as
\begin{align*}
|\partial_\theta K_1(\phi,\theta,\varphi,\eta)|\lesssim&\|h\|_{\textnormal{Lip}}\frac{\sin^3(\varphi)\left((\phi-\varphi)^2+(\sin(\varphi)+\sin(\phi))\sin^2((\theta-\eta)/2)\right)}{\left((\phi-\varphi)^2+(\sin^2(\phi)+\sin^2(\varphi))\sin^2((\theta-\eta)/2)\right)^2}\\
&+\|h\|_{\textnormal{Lip}}\frac{\sin^2(\varphi)}{(\phi-\varphi)^2+(\sin^2(\phi)+\sin^2(\varphi))\sin^2((\theta-\eta)/2)}\\
&+ \|h\|_{\textnormal{Lip}}\frac{\sin^3(\varphi)\left((\phi-\varphi)^2+(\sin(\varphi)+\sin(\phi))\sin^2((\theta-\eta)/2)\right)}{\left((\phi-\varphi)^2+(\sin^2(\phi)+\sin^2(\varphi))\sin^2((\theta-\eta)/2)\right)^2}\cdot
\end{align*}
Consequently we get
\begin{align*}
|\partial_\theta K_1(\phi,\theta,\varphi,\eta)|\lesssim&\frac{\|h\|_{\textnormal{Lip}}\,\sin^2(\varphi)}{(\phi-\varphi)^2+(\sin^2(\phi)+\sin^2(\varphi))\sin^2((\theta-\eta)/2)}\\
&\lesssim\frac{\|h\|_{\textnormal{Lip}}\,|\sin((\theta-\eta)/2)|^{-1}\,\sin(\varphi)}{\left((\phi-\varphi)^2+(\sin^2(\phi)+\sin^2(\varphi))\sin^2((\theta-\eta)/2)\right)^{\frac12}}\cdot
\end{align*}
Therefore we obtain by virtue of \eqref{Tad11} 
\begin{align}\label{partialthetaK1}
|\partial_\theta K_1(\phi,\theta,\varphi,\eta)|\leq C\|h\|_{\textnormal{Lip}}|\phi-\varphi|^{-\beta}|\sin((\theta-\eta)/2|^{-(2-\beta)},
\end{align}
for any $\beta\in(0,1)$. Hence, all the  hypothesis of Proposition \ref{prop-potentialtheory} are satisfied and therefore we deduce that $\mathscr{G}_1$ belongs to $\mathscr{C}^\beta\big((0,\pi)\times{\T}\big)$, for any $\beta\in(0,1)$. The estimates of the kernel  $K_2$ we are quite  similar to those of $K_1$ modulo some slight adaptations. We shall not develop all the estimates which are straightforward and tedious. We will restrict this discussion to the analogous estimate to  \eqref{denVM2} and \eqref{partialthetaK1}. First note that thanks  to  \eqref{gtheta} and \eqref{gthetaM1} one gets
$$
\big|\partial_\varphi(\sin(\varphi)\, r_0^2(\varphi)g_\theta(\varphi,\theta))\big|\lesssim \|h\|_{\textnormal{Lip}}\sin^3(\varphi)|\sin((\theta-\eta)/2)|.
$$
This implies that
\begin{align*}
\big|K_2(\phi,\theta,\varphi,\eta)\big|\lesssim&\|h\|_{\textnormal{Lip}}\frac{\sin^3(\varphi)\sin^2((\theta-\eta)/2)}{A(\phi,\theta,\varphi,\eta)^\frac32}\\
\lesssim&\|h\|_{\textnormal{Lip}}\frac{\sin(\varphi)}{A(\phi,\theta,\varphi,\eta)^\frac12}\cdot
\end{align*}
It follows from \eqref{estim-den2} and \eqref{Tad11} that
\begin{equation}\label{K2-CoX}
| K_2(\phi,\theta,\varphi,\eta)|\lesssim \|h\|_{\textnormal{Lip}}|\phi-\varphi|^{-\beta}|\sin((\theta-\eta)/2|^{-(1-\beta)},
\end{equation}
which is the announced estimate. As to the estimate of $\partial_\theta K_2$ we first write
\begin{align}\label{FormCoron1}
\partial_\theta K_2(\phi,\theta,\varphi,\eta)=&-\frac{\partial_\varphi\left(\sin(\varphi)r_0^2(\varphi)h(\varphi,\theta)\right)\sin(\theta-\eta)}{A(\phi,\theta,\varphi,\eta)^\frac32}+\frac{\partial_\varphi\left(\sin(\varphi)r_0^2(\varphi) g_\theta(\varphi,\eta)\right)\cos(\theta-\eta)}{A(\phi,\theta,\varphi,\eta)^\frac32}\nonumber\\
-&\frac32  K_2(\phi,\theta,\varphi,\eta)\frac{\partial_\theta A(\phi,\theta,\varphi,\eta)}{A(\phi,\theta,\varphi,\eta)}
\end{align}
Straightforward calculations using ${\bf{(H2)}}$ and \eqref{platit1} show that
\begin{align*}
\frac{\left|\partial_\varphi\left(\sin(\varphi)r_0^2(\varphi)h(\varphi,\theta)\right)\sin(\theta-\eta)\right|}{A(\phi,\theta,\varphi,\eta)^\frac32}\lesssim&\|h\|_{\textnormal{Lip}}
\frac{\sin^3(\varphi)|\sin((\theta-\eta)/2)|}{A(\phi,\theta,\varphi,\eta)^\frac32}\\
\lesssim&\|h\|_{\textnormal{Lip}}
\frac{\sin^2(\varphi)}{A(\phi,\theta,\varphi,\eta)}\cdot
\end{align*}
 Putting together  \eqref{estim-den2} and \eqref{Tad11} implies
 \begin{align}\label{EsW1}
 \frac{\sin^2(\varphi) }{A(\phi,\theta,\varphi,\eta)}\lesssim&|\sin((\theta-\eta)/2)|^{-1} \frac{\sin(\varphi) }{A^{\frac12}(\phi,\theta,\varphi,\eta)}\nonumber\\
 \lesssim& \frac{1}{|\phi-\varphi|^{\beta}|\sin((\theta-\eta)/2)|^{2-\beta}}\cdot
 \end{align} 
 Therefore we find
 \begin{align*}
\frac{\left|\partial_\varphi\left(\sin(\varphi)r_0^2(\varphi)h(\varphi,\theta)\right)\sin(\theta-\eta)\right|}{A(\phi,\theta,\varphi,\eta)^\frac32}\lesssim& \frac{\|h\|_{\textnormal{Lip}}
}{|\phi-\varphi|^{\beta}|\sin((\theta-\eta)/2)|^{2-\beta}}\cdot
\end{align*}
As to the second term of the right-hand side of \eqref{FormCoron1}, we get in view of ${\bf{(H2)}}$, \eqref{gtheta} and \eqref{gthetaM1}
\begin{align*}
\frac{\left|\partial_\varphi\left(\sin(\varphi)r_0^2(\varphi) g_\theta(\varphi,\eta)\right)\cos(\theta-\eta)\right|}{A(\phi,\theta,\varphi,\eta)^\frac32}\lesssim&\|h\|_{\textnormal{Lip}}
\frac{\sin^3(\varphi)|\sin((\theta-\eta)/2)|}{A(\phi,\theta,\varphi,\eta)^\frac32}\\
\lesssim&\|h\|_{\textnormal{Lip}}
\frac{\sin^2(\varphi)}{A(\phi,\theta,\varphi,\eta)}\cdot
\end{align*}
It follows from \eqref{EsW1} that
\begin{align*}
\frac{\left|\partial_\varphi\left(\sin(\varphi)r_0^2(\varphi) g_\theta(\varphi,\eta)\right)\cos(\theta-\eta)\right|}{A(\phi,\theta,\varphi,\eta)^\frac32}\lesssim&\|h\|_{\textnormal{Lip}}
 \frac{1}{|\phi-\varphi|^{\beta}|\sin((\theta-\eta)/2)|^{2-\beta}}\cdot
\end{align*}
Concerning the last term of \eqref{FormCoron1}, we put together \eqref{Atheta}, \eqref{K2-CoX}, \eqref{estim-den2} and \eqref{Tad11} that
\begin{align*}
  |K_2(\phi,\theta,\varphi,\eta)|\frac{|\partial_\theta A(\phi,\theta,\varphi,\eta)|}{A(\phi,\theta,\varphi,\eta)}\lesssim& \|h\|_{\textnormal{Lip}}|\phi-\varphi|^{-\beta}|\sin((\theta-\eta)/2)|^{-(1-\beta)}\frac{\sin\varphi}{A^{\frac12}(\phi,\theta,\varphi,\eta)}\\
  \lesssim& \|h\|_{\textnormal{Lip}}|\phi-\varphi|^{-\beta}|\sin((\theta-\eta)/2)|^{-(2-\beta)}.
  \end{align*}
  Therefore collecting the preceding estimates allows to get the suitable estimate for $\partial_\theta K_2$, 
  $$
  |\partial_\theta K_2(\phi,\theta,\varphi,\eta)|\lesssim \|h\|_{\textnormal{Lip}}|\phi-\varphi|^{-\beta}|\sin((\theta-\eta)/2)|^{-(2-\beta)}.
  $$

The estimate for  $\partial_\phi K_2$ can be done  similarly in a straightforward way. Consequently the assumptions of Proposition \ref{prop-potentialtheory} hold true and one deduces that  $\mathscr{G}_1\in \mathscr{C}^\beta\big((0,\pi)\times{\T}\big)$. Hence we obtain $\partial_\phi G\in \mathscr{C}^\beta\big((0,\pi)\times{\T}\big),$ with the estimate
\begin{align}\label{partialphiG}
 \|\partial_\phi G(h)\|_{\mathscr{C}^{\beta}}\lesssim \|h\|_{\mathscr{C}^{1,\alpha}}.
 \end{align}
The next stage is to show that  $\partial_\theta G(h)\in \mathscr{C}^\beta\big((0,\pi)\times{\T}\big)$ following the same strategy as before. From \eqref{G-hX1}, we get
\begin{align*}
\partial_\theta G(h)(\phi,\theta)&=-\frac12\frac{1}{r_0(\phi)}\bigintsss_0^\pi\bigintsss_0^{2\pi}\frac{\sin(\varphi)r_0(\varphi)h(\varphi,\eta)\partial_\theta A(\phi,\theta,\varphi,\eta)}{A(\phi,\theta,\varphi,\eta)^\frac32}d\eta d\varphi.
\end{align*}
Direct computations show that
\begin{align*}
\partial_\theta A(\phi,\theta,\varphi,\eta)=&2r_0(\phi) r_0(\varphi)\sin(\theta-\eta)=-\partial_\eta A(\phi,\theta,\varphi,\eta).
\end{align*}
It follows
\begin{align*}
\partial_\theta G(h)(\phi,\theta)&=\frac12\frac{1}{r_0(\phi)}\bigintsss_0^\pi\bigintsss_0^{2\pi}\frac{\sin(\varphi)r_0(\varphi)h(\varphi,\eta)\partial_\eta A(\phi,\theta,\varphi,\eta)}{A(\phi,\theta,\varphi,\eta)^\frac32}d\eta d\varphi.
\end{align*}
On the other hand, integration by parts in $\eta$  {(this can be done by cutting off the integral in $\eta$ away from $\theta$ and taking a limit)} yields
$$
\bigintsss_0^\pi\bigintsss_0^{2\pi}\frac{\sin(\varphi)r_0(\varphi)h(\varphi,\theta)\partial_\eta A(\phi,\theta,\varphi,\eta)}{A(\phi,\theta,\varphi,\eta)^\frac32}d\eta d\varphi=0.
$$
Thus we deduce by subtraction 
\begin{align}\label{G-3-theta}
\partial_\theta G(h)(\phi,\theta)&=\frac12\frac{1}{r_0(\phi)}\bigintsss_0^\pi\bigintsss_0^{2\pi}\frac{\sin(\varphi)r_0(\varphi)\big(h(\varphi,\eta)-h(\varphi,\theta)\big)\partial_\eta A(\phi,\theta,\varphi,\eta)}{A(\phi,\theta,\varphi,\eta)^\frac32}d\eta d\varphi\nonumber\\
&=\bigintsss_0^\pi\bigintsss_0^{2\pi}\frac{\sin(\varphi)r_0^2(\varphi)\big(h(\varphi,\eta)-h(\varphi,\theta)\big)\sin(\eta-\theta)}{A(\phi,\theta,\varphi,\eta)^\frac32}d\eta d\varphi.
\end{align}
Since $h\in\mathscr{C}^{1,\alpha}$, then  the mean value theorem implies 
$$
|h(\varphi,\theta)-h(\varphi,\eta)|\lesssim \|h \|_{\textnormal{Lip}}|\theta-\eta|.
$$
Moreover, by the $2\pi$-periodicity of $h$ in $\eta$ one  obtains 
\begin{equation}\label{G-3-theta-h-2}
|h(\varphi,\theta)-h(\varphi,\eta)|\lesssim \|h \|_{\mathscr{C}^{1,\alpha}}\big|\sin\big((\theta-\eta)/2\big)\big|.
\end{equation}
Define the kernel 
$$
K_3(\phi,\theta,\varphi,\eta):=\frac{\sin(\varphi)r_0^2(\varphi)\sin(\eta-\theta)\big(h(\varphi,\eta)-h(\varphi,\theta)\big) }{A(\phi,\theta,\varphi,\eta)^\frac32},
$$
and let us check the hypothesis of Proposition \ref{prop-potentialtheory}. First using \eqref{estim-den2},  ${\bf{(H2)}}$ and \eqref{G-3-theta-h-2} we obtain 
\begin{align*}
|K_3(\phi,\theta,\varphi,\eta)|\lesssim&\|h \|_{\mathscr{C}^{1,\alpha}} \frac{\sin^{3}(\varphi)\sin^2((\theta-\eta)/2)}{\left((\phi-\varphi)^2+(\sin^2(\phi)+\sin^2(\varphi))\sin^2((\theta-\eta)/2)\right)^\frac32}\\
\lesssim&\|h \|_{\mathscr{C}^{1,\alpha}}\frac{\sin(\varphi)}{\left((\phi-\varphi)^2+(\sin^2(\phi)+\sin^2(\varphi))\sin^2((\theta-\eta)/2)\right)^\frac12}.
\end{align*}
Applying  \eqref{Tad11} yields
\begin{align}\label{boundX-K3}
|K_3(\phi,\theta,\varphi,\eta)|\lesssim  \|h \|_{\mathscr{C}^{1,\alpha}}|\phi-\varphi|^{-(1-\beta)}|\sin((\theta-\eta)/2)|^{-\beta},
\end{align}
for any $\beta\in(0,1)$.  Let us estimate $\partial_\phi K_3$ which is explicitly given by,
\begin{align*}
\partial_\phi K_3(\phi,\theta,\varphi,\eta)=&-\frac32 \frac{\sin(\varphi)r_0^2(\varphi)\sin(\eta)\big(h(\varphi,\eta)-h(\varphi,\theta)\big)\partial_\phi A(\phi,\theta,\varphi,\eta) }{A(\phi,\theta,\varphi,\eta)^\frac52}\\
=&-\frac32 K_3(\phi,\theta,\varphi,\eta)\frac{\partial_\phi A(\phi,\theta,\varphi,\eta)}{A(\phi,\theta,\varphi,\eta)}\cdot
\end{align*}
By virtue of \eqref{dphiA3} and \eqref{boundX-K3}, we achieve
\begin{align*}
|\partial_\phi K_3(\phi,\theta,\varphi,\eta)|\lesssim& \|h \|_{\mathscr{C}^{1,\alpha}}|\phi-\varphi|^{-(1-\beta)}\big|\sin\big((\theta-\eta)/2\big)\big|^{-\beta} A^{-\frac12}(\phi,\theta,\varphi,\eta),\\
\lesssim &\frac{\|h \|_{\mathscr{C}^{1,\alpha}}}{|\phi-\varphi|^{2-\beta}|\sin((\theta-\eta)/2)|^{\beta}},
\end{align*}
for any $\beta\in(0,1)$.  It remains to establish the suitable  estimates for $\partial_\theta K_3$. First we have
\begin{align*}
\partial_\theta K_3(\phi,\theta,\varphi,\eta)=&-\frac32 K_3(\phi,\theta,\varphi,\eta)\frac{\partial_\theta A(\phi,\theta,\varphi,\eta)}{A(\phi,\theta,\varphi,\eta)}
- \frac{\sin(\varphi)r_0^2(\varphi)\sin(\eta-\theta)\partial_\theta h(\varphi,\theta) }{A(\phi,\theta,\varphi,\eta)^\frac32} \\
&-\frac{\sin(\varphi)r_0^2(\varphi)\cos(\eta-\theta)\big(h(\varphi,\eta)-h(\varphi,\theta)\big) }{A(\phi,\theta,\varphi,\eta)^\frac32}\cdot
\end{align*} 
Using \eqref{Atheta} and \eqref{boundX-K3} (where we exchange $\beta$ by $1-\beta$) we get 
\begin{align*}
\big|K_3(\phi,\theta,\varphi,\eta)\big|\frac{\big|\partial_\theta A(\phi,\theta,\varphi,\eta)\big|}{A(\phi,\theta,\varphi,\eta)}&\lesssim \frac{\|h \|_{\mathscr{C}^{1,\alpha}}}{|\phi-\varphi|^{\beta}|\sin((\theta-\eta)/2)|^{1-\beta}}\frac{\sin\varphi }{A^{\frac12}(\phi,\theta,\varphi,\eta)}\\
&\lesssim \frac{\|h \|_{\mathscr{C}^{1,\alpha}}}{|\phi-\varphi|^{\beta}|\sin((\theta-\eta)/2)|^{2-\beta}}\cdot
\end{align*} 
For the second term of the right-hand side of $\partial_\theta K_3$  we write in view of ${\bf{(H2)}}$ 
\begin{align*}
 \frac{\sin(\varphi)r_0^2(\varphi)|\sin(\eta-\theta)||\partial_\theta h(\varphi,\theta)| }{A(\phi,\theta,\varphi,\eta)^\frac32} \lesssim&\|h\|_{\textnormal{Lip}} \frac{\sin^3(\varphi)|\sin((\theta-\eta)/2)| }{A(\phi,\theta,\varphi,\eta)^\frac32}\\ 
 \lesssim& \|h\|_{\textnormal{Lip}} \frac{\sin^2(\varphi) }{A(\phi,\theta,\varphi,\eta)}\cdot
 \end{align*} 
 Applying \eqref{EsW1} yields
 \begin{align*}
 \frac{\sin(\varphi)r_0^2(\varphi)|\sin(\eta-\theta)||\partial_\theta h(\varphi,\theta)| }{A(\phi,\theta,\varphi,\eta)^\frac32} \lesssim& \frac{\|h\|_{\textnormal{Lip}}  }{|\phi-\varphi|^{\beta}|\sin((\theta-\eta)/2)|^{2-\beta}}\cdot \end{align*} 
 Concerning the last term of the right-hand side of $\partial_\theta K_3$, it is similar to the foregoing one. Indeed, using \eqref{G-3-theta-h-2} and ${\bf{(H2)}}$ we get
 \begin{align*}
 \frac{\sin(\varphi)r_0^2(\varphi)|\cos(\eta-\theta)|\big|h(\varphi,\eta)-h(\varphi,\theta)\big| }{A(\phi,\theta,\varphi,\eta)^\frac32}\lesssim&\|h\|_{\mathscr{C}^{1,\alpha}}  \frac{\sin^3(\varphi)|\sin((\eta-\theta)/2)|}{A(\phi,\theta,\varphi,\eta)^\frac32}\\
 \lesssim&\|h\|_{\mathscr{C}^{1,\alpha}}  \frac{\sin^2(\varphi)}{A(\phi,\theta,\varphi,\eta)}\cdot
 \end{align*} 
It suffices to use \eqref{EsW1} to obtain
 \begin{align*}
 \frac{\sin(\varphi)r_0^2(\varphi)|\cos(\eta-\theta)|\big|h(\varphi,\eta)-h(\varphi,\theta)\big| }{A(\phi,\theta,\varphi,\eta)^\frac32}\lesssim&\|h\|_{\mathscr{C}^{1,\alpha}}  \frac{\sin^3(\varphi)|\sin((\eta-\theta)/2)|}{A(\phi,\theta,\varphi,\eta)^\frac32}\\
 \lesssim& \frac{\|h\|_{\mathscr{C}^{1,\alpha}}  }{|\phi-\varphi|^{\beta}|\sin((\theta-\eta)/2)|^{2-\beta}}\cdot
 \end{align*} 
 Therefore we get from the preceding estimates
 \begin{align*}
 |\partial_\theta K_3(\phi,\theta,\varphi,\eta)|\lesssim&\frac{\|h\|_{\mathscr{C}^{1,\alpha}}  }{|\phi-\varphi|^{\beta}|\sin((\theta-\eta)/2)|^{2-\beta}}\cdot
 \end{align*} 
 Consequently, all the assumptions of Proposition \ref{prop-potentialtheory} are verified by the kernel $K_3$ and  thus we deduce that  $\partial_\theta G(h)\in \mathscr{C}^\beta\big((0,\pi)\times{\T}\big)$ for any $\beta\in(0,1),$ with the estimate
 $$
 \|\partial_\theta G(h)\|_{\mathscr{C}^{\beta}}\lesssim \|h\|_{\mathscr{C}^{1,\alpha}}. 
 $$
Putting together this estimate with \eqref{partialphiG} and \eqref{boundG1}  yields
$$
 \| G(h)\|_{\mathscr{C}^{1,\beta}}\lesssim \|h\|_{\mathscr{C}^{1,\alpha}},
 $$
and this achieves the proof of the proposition.
\end{proof}

\subsection{Transversality}
We have shown in Proposition \ref{cor-fredholm} that when $\Omega$ belongs to the discrete set $\{\Omega_m, m\geq2\}$ then the linearized operator $\partial_f \tilde{F}(\Omega,0)$ is of Fredholm type with one--dimensional kernel.  This property is not enough to bifurcate to nontrivial solutions for the nonlinear problem. A sufficient condition for that,  according to Theorem \ref{CR}, is the   the transversal assumption which  amounts to checking 
$$
\partial^2_{\Omega, f} \tilde{F}(\Omega_m,0)f_m^\star\notin \textnormal{Im}(\partial_f \tilde{F}(\Omega_m,0)),
$$
where $f_m^\star$ is a generator  of the kernel of $\partial_f \tilde{F}(\Omega_m,0)$.
Note that as a consequence of \eqref{operator} and \eqref{K-kernel}, for a function  $h:(\phi,\theta)\mapsto \sum_{n\geq1}h_n(\phi) \cos(n\theta) \in X_m^\alpha$, we get
$$
\partial_{f} \tilde{F}(\Omega,0)h(\phi,\theta)=\sum_{n\geq1}\mathcal{L}_n^\Omega h_n(\phi) \cos(n\theta),
$$
with
\begin{align*}
\mathcal{L}_n^\Omega h_n(\phi)=&\nu_\Omega(\phi) h_n(\phi)-\bigintsss_0^\pi H_n(\phi,\varphi) h_n(\phi,\varphi)d\varphi\\
=&\nu_\Omega(\phi) \big(h_n(\phi)-\mathcal{K}_n^\Omega h_n(\phi)\big),
\end{align*}
where $\mathcal{K}_n^\Omega$ is defined in \eqref{kerneleq}.
Hence, the second mixed derivative takes the form,
$$
\partial^2_{\Omega,f} \tilde{F}(\Omega,0)h(\phi,\theta)=-h(\phi,\theta).
$$
Our main result of this section reads as follows.
\begin{pro}\label{prop-transversal}
Let $m\geq2,$ then the transversal condition holds true, that is, 
$$
\partial^2_{\Omega, f} \tilde{F}(\Omega_m,0)f_m^\star\notin \textnormal{Im}(\partial_f \tilde{F}(\Omega_m,0)),
$$
where $f_m^\star$ is a generator  of the kernel of $\partial_f \tilde{F}(\Omega_m,0)$.
\end{pro}
\begin{proof}
Recall from the proof of Proposition \ref{cor-fredholm}  that the function $f_m^\star$ has the form
$$
f_m^\star(\phi,\theta)=h_m^\star(\phi)\cos(m\theta)
$$
and  $h_m^\star$ is a nonzero solution to the equation
$$ 
\mathcal{K}_m^{\Omega_m} h_m^\star(\phi)=h_m^\star(\phi).
$$
It follows that 
$$
\partial^2_{\Omega,f} \tilde{F}(\Omega_m,0)f_m^\star(\phi,\theta)=-h_m^\star(\phi)\cos(m\theta).
$$
Assume that this element belongs to the range of $\partial_{f} \tilde{F}(\Omega_m,0)$. Then we can find $h_m$ such that
$$
h_m^\star(\phi)=\nu_{\Omega_m}(\phi)\big( h_m(\phi)-\mathcal{K}_m^{\Omega_m} h_m(\phi)\big).
$$
Dividing this equality by  $\nu_{\Omega_m}$ and taking the inner product with $h_m^{\star}$, with respect to $\langle\cdot ,\cdot \rangle_{\Omega_m}$ defined in \eqref{scalar-prod1}    yields by the symmetry of $ \mathcal{K}_m^{\Omega_m}$
\begin{align*}
\Big\langle \frac{h_m^{\star}}{\nu_{\Omega_m}},h_m^{\star}\Big\rangle_{\Omega_m}=&\Big\langle h_m,h_m^{\star}\Big\rangle_{\Omega_m}-\Big\langle \mathcal{K}_m^{\Omega_m} h_m,h_m^{\star}\Big\rangle_{\Omega_m}\\
=&\Big\langle h_m,h_m^{\star}\Big\rangle_{\Omega_m}-\Big\langle  h_m, \mathcal{K}_m^{\Omega_m}h_m^{\star}\Big\rangle_{\Omega_m}\\
=&\Big\langle h_m,h_m^{\star}-\mathcal{K}_m^{\Omega_m}h_m^{\star}\Big\rangle_{\Omega_m}\\
=&0.
\end{align*}
Coming back to the definition of the inner product \eqref{scalar-prod1} and \eqref{signed-meas}, we find
$$
\bigintsss_0^\pi \left(h_m^\star(\varphi)\right)^2\sin(\varphi)\, r_0^2(\varphi) d\varphi=0.
$$
From the assumption ${\bf{(H)}}$  we know  that $r_0$  does not vanish in $(0,\pi)$. Then we get  from   the continuity of $h_m^\star$ that  this latter function should vanish everywhere in $(0,\pi)$, which is a contradiction. Hence, we deduce that $f_m^\star$  does not belong to the range of $\partial_f \tilde{F}(\Omega_m,0)$ and then the transversal condition is satisfied.
\end{proof}

\section{Nonlinear action}\label{sec-regularity}
This section is devoted to the regularity study of the nonlinear functional $\tilde{F}$ defined in \eqref{Ftilde} that we recall for the convenience of the reader,
\begin{equation*}
 \tilde{F}(\Omega,f)(\phi,\theta)=\frac{1}{r_0(\phi)}\left\{I(f)(\phi,\theta)-\frac{\Omega}{2}r^2(\phi,\theta)-m(\Omega,f)(\phi)\right\},
\end{equation*}
for any $(\phi,\theta)\in(0,\pi)\times(0,2\pi)$ and where
\begin{equation*}
I(f)(\phi,\theta)=-\frac{1}{4\pi}\bigintsss_{0}^{\pi}\bigintsss_0^{2\pi}\bigintsss_0^{r(\varphi,\eta)}\frac{r\sin(\varphi)drd\eta d\varphi}{|(re^{i\eta},\cos(\varphi))-(r(\phi,\theta)e^{i\theta},\cos(\phi))|},
\end{equation*}
the mean $m$ is defined in \eqref{meanT} and
$$
r(\phi,\theta)=r_0(\phi)+f(\phi,\theta).
$$

We would like in particular to analyze the symmetry/regularity  persistence of the function spaces  $X_m^\alpha$ defined in \eqref{space} and \eqref{spaceX1} through the action of the nonlinear functional $\tilde{F}$.

\subsection{Symmetry persistence}
The main task here is to check the symmetry persistence of the function  spaces  $X_m^\alpha$ defined in \eqref{space} through the nonlinear  action of $\tilde{F}.$ Notice that  at this level, we  do not raise the problem of  whether or not this functional is well-defined and this target is postponed later in Section \ref{sec-regularity}.  First recall that 

{\begin{align*}
X_m^\alpha=\Big\{f:[0,\pi]\times[0,2\pi]\rightarrow \R\, :\, \, &f\in \mathscr{C}^{1,\alpha}, f(0,\theta)=f(\pi,\theta)\equiv 0,\,\\
& f\left(\frac{\pi}{2}-\phi,\theta\right)=f\left(\frac{\pi}{2}+\phi,\theta\right), \, f\left(\phi,\theta\right)=\sum_{n\geq 1}f_n(\phi)\cos(nm\theta)\Big\}.
\end{align*}}

\begin{pro}\label{prop-SymX1}
Let $\Omega\in\R$, $f\in X_m^\alpha$ with  $m\geq 1$ and assume that $r_0$ satisfies the \mbox{conditions ${\bf{(H)}}$.}  Then the following assertions hold true.
\begin{enumerate}
\item The equatorial symmetry: 
$$\tilde{F}(\Omega,f)\left(\pi-\phi,\theta\right)=\tilde{F}(\Omega,f)\left(\phi,\theta\right),\quad \forall \, (\phi,\theta)\in[0,\pi]\times\R.
$$
\item We get the algebraic structure, 
$$
\tilde{F}(\Omega,f)(\phi,\theta)=\sum_{n\geq 1} f_n(\phi)\cos(n\theta),
$$
for some functions $f_n$ and for any $(\phi,\theta)\in[0,\pi]\times[0,2\pi]$.
\item The $m$-fold symmetry: $\tilde{F}(\Omega,f)(\phi,\theta+\frac{2\pi}{m})=\tilde{F}(\Omega,f)(\phi,\theta)$, for any $(\phi,\theta)\in[0,\pi]\times\R$.

\end{enumerate}
\end{pro}
{\begin{proof}
\medskip
\noindent
${\bf{(1)}}$ From the expression of $\tilde{F}$ in \eqref{Ftilde}, it is suffices to check the property for  $I(f)$. One can easily verify using the symmetry of the functions $\cos$ and $r$ combined with the change of variables $\varphi\mapsto \pi-\varphi$
\begin{align*}
I(f)\left(\pi-\phi,\theta\right)=&-\frac{1}{4\pi}\bigintsss_{0}^{\pi}\bigintsss_0^{2\pi}
\bigintsss_0^{r(\varphi,\eta)}\frac{r\sin(\varphi)drd\eta d\varphi}{|(re^{i\eta},\cos(\varphi))-(r(\pi-\phi,\theta)e^{i\theta},\cos(\pi-\phi))|}\\
=&-\frac{1}{4\pi}\bigintsss_{0}^{\pi}\bigintsss_0^{2\pi}\bigintsss_0^{r(\pi-\varphi,\eta)}
\frac{r\sin(\pi-\varphi)drd\eta d\varphi}{|(re^{i\eta},-\cos(\varphi))-(r(\phi,\theta)e^{i\theta},-\cos(\phi))|}\\
=&-\frac{1}{4\pi}\bigintsss_{0}^{\pi}\bigintsss_0^{2\pi}\bigintsss_0^{r(\varphi,\eta)}
\frac{r\sin(\varphi)drd\eta d\varphi}{|(re^{i\eta},\cos(\varphi))-(r(\phi,\theta)e^{i\theta},\cos(\phi))|}\\
=&I(f)\left(\phi,\theta\right).
\end{align*}

\medskip
\noindent
${\bf{(2)}}$ In order to get  the desired structure,  it suffices to check the following symmetry
$$
I(f)(\phi,-\theta)=I(f)(\phi,\theta),\quad \forall \, (\phi,\theta)\in[0,\pi]\times\R.
$$
To do that, we use the symmetry of $r$, that is  $r(\varphi,-\theta)=r(\varphi,\theta),$ combined with  the change of variables $\eta\mapsto-\eta$ allowing one to get
\begin{align*}
I(f)(\phi,-\theta)=&-\frac{1}{4\pi}\bigintsss_{0}^{\pi}\bigintsss_0^{2\pi}\bigintsss_0^{r(\varphi,\eta)}\frac{r\sin(\varphi)drd\eta d\varphi}{|(re^{i\eta},\cos(\varphi))-(r(\phi,-\theta)e^{-i\theta},\cos(\phi))|}\\
=&-\frac{1}{4\pi}\bigintsss_{0}^{\pi}\bigintsss_0^{2\pi}\bigintsss_0^{r(\varphi,-\eta)}\frac{r\sin(\varphi)drd\eta d\varphi}{|(re^{-i\eta},\cos(\varphi))-(r(\phi,\theta)e^{-i\theta},\cos(\phi))|}\\
=&-\frac{1}{4\pi}\bigintsss_{0}^{\pi}\bigintsss_0^{2\pi}\bigintsss_0^{r(\varphi,\eta)}\frac{r\sin(\varphi)drd\eta d\varphi}{|(re^{i\eta},\cos(\varphi))-(r(\phi,\theta)e^{i\theta},\cos(\phi))|}\\
=&I(f)(\phi,\theta).
\end{align*}

\medskip
\noindent
${\bf{(3)}}$ First, since $r$ belongs to $X_m^\alpha$ then it satisfies $r(\varphi, \theta+\frac{2\pi}{m})=r(\varphi,\theta).$ Thus we get by the change of variables $\eta\mapsto \eta+\frac{2\pi}{m}$
\begin{align*}
I(f)\left(\phi,\theta+\frac{2\pi}{m}\right)
&=-\frac{1}{4\pi}\bigintsss_{0}^{\pi}\bigintsss_0^{2\pi}\bigintsss_0^{r(\varphi,\eta)}\frac{r\sin(\varphi)drd\eta d\varphi}{|(re^{i\eta},\cos(\varphi))-(r(\phi,\theta+\frac{2\pi}{m})e^{i(\theta+\frac{2\pi}{m})},\cos(\phi))|}\\
&=-\frac{1}{4\pi}\bigintsss_{0}^{\pi}\bigintsss_0^{2\pi}\bigintsss_0^{r(\varphi,\eta+\frac{2\pi}{m})}\frac{r\sin(\varphi)drd\eta d\varphi}{|(re^{i(\eta+\frac{2\pi}{m})},\cos(\varphi))-(r(\phi,\theta)e^{i(\theta+\frac{2\pi}{m})},\cos(\phi))|}\\
&=-\frac{1}{4\pi}\bigintsss_{0}^{\pi}\bigintsss_0^{2\pi}\bigintsss_0^{r(\varphi,\eta)}\frac{r\sin(\varphi)drd\eta d\varphi}{|(re^{i\eta},\cos(\varphi))-(r(\phi,\theta)e^{i\theta},\cos(\phi))|}\\
&=I(f)(\phi,\theta).
\end{align*}
Notice that we have used the fact that the Euclidean distance {in} $\C$ is invariant by the rotation action $z\mapsto e^{i\frac{2\pi}{m}} z$.
\end{proof}}

The next discussion is devoted to the symmetry effects of the surface of the vortices on the velocity structure. We shall show the following.

{\begin{lem}\label{lemma-velocidad0}
If $r_0$ satisfies  ${\bf{(H)}}$ and $f\in X_m^\alpha$, with $m\geq 2$, then 
\begin{align*}
\forall\, z\in\R,\quad \bigintsss_0^\pi\bigintsss_0^{2\pi}\frac{\sin(\varphi)\partial_\eta(r(\varphi,\eta)\cos(\eta))d\eta d\varphi}{\left(r^2(\varphi,\eta)+\big(z-\cos\varphi\big)^2\right)^\frac12}=&0,\\
\forall\, z\in\R,\quad \bigintsss_0^\pi\bigintsss_0^{2\pi}\frac{\sin(\varphi)\partial_\eta(r(\varphi,\eta)\sin(\eta))d\eta d\varphi}{\left(r^2(\varphi,\eta)+\big(z-\cos\varphi\big)^2\right)^\frac12}=&0.
\end{align*}
As a consequence, the velocity field defined in \eqref{Veloc1} is vanishing at the vertical axis, that is, 
$$
U(0,0,z)=0,
$$
for any $z\in\R$. 
\end{lem}}
{\begin{proof}
Set for any $z\in\R,$
\begin{align*}
I_1(z):=&\bigintsss_0^\pi\bigintsss_0^{2\pi}\frac{\sin(\varphi)\partial_\eta(r(\varphi,\eta)\cos(\eta))d\eta d\varphi}{\left(r^2(\varphi,\eta)+\big(z-\cos\varphi\big)^2\right)^\frac12},\\
I_2(z):=&\bigintsss_0^\pi\bigintsss_0^{2\pi}\frac{\sin(\varphi)\partial_\eta(r(\varphi,\eta)\sin(\eta))d\eta d\varphi}{\left(r^2(\varphi,\eta)+\big(z-\cos\varphi\big)^2\right)^\frac12}.
\end{align*}
Observe that from the periodicity in $\eta$ we may write 
\begin{align*}
I_1(z)=&\bigintsss_0^\pi\bigintsss_{-\pi}^{\pi}\frac{\sin(\varphi)(\partial_\eta r)(\varphi,\eta)\cos(\eta)d\eta d\varphi}{\left(r^2(\varphi,\eta)+\big(z-\cos\varphi\big)^2\right)^\frac12}\\
&-\bigintsss_0^\pi\bigintsss_{-\pi}^{\pi}\frac{\sin(\varphi)r(\varphi,\eta)\sin(\eta)d\eta d\varphi}{\left(r^2(\varphi,\eta)+\big(z-\cos\varphi\big)^2\right)^\frac12}.
\end{align*}
Since $f\in X_m^\alpha$, then  $r(\varphi,-\eta)=r(\varphi,\eta)$ and so  $(\partial_\eta r)(\varphi,-\eta)=-(\partial_\eta r)(\varphi,\eta)$. Therefore making the change of variables $\eta\mapsto-\eta$ allows to get $I_1(z)=0$.

To check $I_2(z)=0$ we shall use the $m$-fold symmetry of $r$. In fact  by  the change of variables  $\eta\mapsto \eta+\frac{2\pi}{m}$ and using the $2\pi$-periodicity in $\eta$ and some elementary trigonometric identity, we find
\begin{align*}
I_2(z)=& \bigintsss_0^\pi\bigintsss_0^{2\pi}\frac{\sin(\varphi)\partial_\eta\big(r(\varphi,\eta)\sin(\eta+\frac{2\pi}{m})\big)}{(r(\varphi,\eta)^2+(z-\cos(\varphi)^2)^\frac12}d\eta d\varphi\\
=& \cos({2\pi}/{m})I_2(z)+\sin({2\pi}/{m}) I_1(z).
\end{align*}
Since $m\geq 2$ and $I_1(z)=0$ then  we get  $I_2(z)=0$.

Coming back to \eqref{Veloc1} and following the change of variables giving \eqref{VelXW1} we easily get
$$
U(0,0,z)=\big(I_1(z), I_2(z),0\big),
$$
which gives the announced result.
\end{proof}}

\subsection{Deformation of the Euclidean norm}
The spherical  change of coordinates used to recover both  the velocity and the stream function from the surface geometry of the patch yields a deformation of the Green function. Notice that in the usual  Cartesian coordinates the  Green kernel  is radial and thus it is isotropic with respect to all the variables. In the new coordinates we lose this property and the Green  kernel becomes anisotropic and the north and south poles are degenerating points. To deal with these defects one needs refined treatments in the behavior of the kernel or also the adaptation of the function spaces which are of Dirichlet type.  
The following lemma is crucial to deal with the anisotropy of the kernel. 

\begin{lem}\label{lemma-estimdenominator}
Let $m\geq1, \alpha\in(0,1)$, $r_0$ satisfies ${\bf{(H)}},$ $f\in X_m^\alpha$  such that  {$\|f\|_{X_m^\alpha}\leq\varepsilon$} with $\varepsilon$ small enough and set $r=r_0+f$. Define  for any $\phi\in[0,\frac{\pi}{2}]$, $\varphi\in[0,\pi]$, $\theta,\eta\in[0,2\pi]$ and $s\in[0,1]$
$$
J_s(\phi,\theta,\varphi,\eta):=(r(\varphi,\eta)-sr(\phi,\theta))^2+2sr(\phi,\theta)r(\varphi,\eta)(1-\cos(\theta-\eta))+(\cos(\phi)-\cos(\varphi))^2.
$$
Then 
\begin{align}
|J_0(\phi,\theta,\varphi,\eta)|\geq& C\sin^2(\varphi),\label{den2}\\
|J_s(\phi,\theta,\varphi,\eta)|\geq& C\Big(\big(\sin^2(\varphi)+s^2\phi^2\big){\sin^2((\theta-\eta)/2)}+(\varphi+\phi)^2(\phi-\varphi)^2\Big),\label{den3}
\end{align}
with $C$ an absolute constant. Remark that we have restricted $\phi$ to $\in[0,\pi/2]$ instead of $[0,\pi]$ because of the symmetry of $r$ with respect to $\frac\pi2.$

\end{lem}
\begin{proof}
Since $f\in B_{X_m^\alpha}(\varepsilon)$, for some $\varepsilon<1$, and $r_0$ verifies {\bf (H2)} then 
\begin{align*}
r(\varphi,\eta)=&r_0(\varphi)+f(\varphi,\eta)\geq 2C \sin\varphi-|f(\varphi,\eta)|.
\end{align*}
In addition $f$ satisfies \eqref{platit1} and in particular  
$$
\frac{|f(\varphi,\eta)|}{\sin(\varphi)}\le C_1\|f\|_{\textnormal{Lip}}.
$$
It follows that,
\begin{align*}
r(\varphi,\eta)\geq \left(2C -C_1\| f\|_{\textnormal{Lip}}\right)\sin(\varphi).
\end{align*}
By imposing $\|f\|_{\textnormal{Lip}}\leq \varepsilon= \frac{C}{C_1}$, we infer
\begin{equation}\label{Min11}
r(\varphi,\eta)\geq  C\sin(\varphi).
\end{equation}
Consequently, we obtain
\begin{align*}
J_0(\phi,\theta,\varphi,\eta)=&r^2(\varphi,\eta)+(\cos(\varphi)-\cos(\phi))^2\geq C\sin^2(\varphi),
\end{align*}
which gives the estimate \eqref{den2}. Let us now check the validity of \eqref{den3}. First, we remark that
$$
J_s(\phi,\theta,\varphi,\eta)=r^2(\varphi,\eta)+s^2r^2(\phi,\theta)-2sr(\varphi,\eta)r(\phi,\theta)\cos(\theta-\eta)+(\cos(\varphi)-\cos(\phi))^2.
$$
Denote
$$
g_1(x):=r^2(\varphi,\eta)+x^2-2xr(\varphi,\eta)\cos(\theta-\eta)+(\cos(\varphi)-\cos(\phi))^2,
$$
and therefore  we get the relation $g_1(sr(\phi,\theta))=J_s(\phi,\theta,\varphi,\eta)$. From variation arguments we infer that the function $g_1$ reaches its global  minimum at  the point 
$$
x_c=r(\varphi,\eta)\cos(\theta-\eta).
$$
Let us distinguish the cases $\cos(\theta-\eta)\in[0,1]$ and $\cos(\theta-\eta)\in[-1,0]$. In the first case, one has according to \eqref{Min11}
\begin{align*}
J_s(\phi,\theta,\varphi,\eta)=&g_1(sr(\phi,\theta))\nonumber\\\geq& g_1(x_c)
=r^2(\varphi,\eta)\sin^2(\theta-\eta)+(\cos(\varphi)-\cos(\phi))^2\nonumber\\
\geq& C(\sin^2(\varphi){\sin^2(\theta-\eta)}+(\cos(\varphi)-\cos(\phi))^2).
\end{align*}
Using that $\cos(\theta-\eta)\in[0,1]$, one gets
  $$\sin^2(\theta-\eta)=2\sin^2((\theta-\eta)/2)(1+\cos(\theta-\eta))\geq 2\sin^2((\theta-\eta)/2).$$
Moreover, since $\phi\in[0,\frac{\pi}{2}]$ and $\varphi\in[0,\pi]$, we obtain
\begin{align}\label{est-cos}
|\cos(\varphi)-\cos(\phi)|=&|(1-\cos(\phi))-(1-\cos(\varphi))|\nonumber\\
=&2|\sin^2(\phi/2)-\sin^2(\varphi/2)|\nonumber\\
=&2|\sin(\phi/2)-\sin(\varphi/2)|(\sin(\phi/2)+\sin(\varphi/2))\nonumber\\
\geq&C|\phi-\varphi||\phi+\varphi|.
\end{align}
Hence
\begin{align}\label{den3-1-1}
J_s(\phi,\theta,\varphi,\eta)\geq& C\Big(\sin^2(\varphi){\sin^2((\theta-\eta)/2)}+(\phi+\varphi)^2(\phi-\varphi)^2\Big).
\end{align}
In the second case where  $\cos(\theta-\eta)\in[-1,0]$, the critical point is negative,  $x_c\leq 0$, and one has from the variations of $g_1$, the estimate \eqref{Min11} and \eqref{est-cos}
\begin{align}\label{den3-1-2}
J_s(\phi,\theta,\varphi,\eta)=&g_1(sr(\phi,\theta))\nonumber\\\geq& g_1(0)
=r(\varphi,\eta)^2+(\cos(\varphi)-\cos(\phi))^2\nonumber\\
\geq& C(\sin^2(\varphi){\sin^2((\theta-\eta)/2)}+(\phi+\varphi)^2(\phi-\varphi)^2).
\end{align}
Putting together \eqref{den3-1-1} and \eqref{den3-1-2}, one deduces that
\begin{align}\label{den3-1}
J_s(\phi,\theta,\varphi,\eta)\geq& C\Big(\sin^2(\varphi){\sin^2((\theta-\eta)/2)}+(\phi+\varphi)^2(\phi-\varphi)^2\Big),
\end{align}
for any $\phi\in[0,\pi/2]$, $\varphi\in[0,\pi]$ and $\theta,\eta\in[0,2\pi]$.

Following the same ideas, we introduce the function
$$
g_2(x):=x^2+s^2{r^2(\phi,\theta)}-2sxr(\phi,\theta)\cos(\theta-\eta)+(\cos(\varphi)-\cos(\phi))^2,
$$
which satisfies  $g_2(r(\varphi,\eta))=J_s(\phi,\theta,\varphi,\eta)$. Then as before we can check easily that the function $g_2$ reaches its  minimum at the point  
$
\tilde{x}_c=sr(\phi,\theta)\cos(\theta-\eta).
$
Similarly we distinguishing between two cases $\cos(\theta-\eta)\in[0,1]$ and $\cos(\theta-\eta)\in[-1,0]$. For  the first case, using  \eqref{est-cos}, we have
\begin{align*}
J_s(\phi,\theta,\varphi,\eta)\geq& C(s^2\sin^2(\phi){\sin^2((\theta-\eta)/2)}+(\phi+\varphi)^2(\phi-\varphi)^2).
\end{align*}
Since $\phi\in[0,\pi/2]$, we have that $\sin(\phi)\geq \frac2\pi\phi$, and then
\begin{align}
J_s(\phi,\theta,\varphi,\eta)\geq& C(s^2\phi^2{\sin^2((\theta-\eta)/2)}+(\phi+\varphi)^2(\phi-\varphi)^2).\label{den3-2}
\end{align}
{In the other case, i.e. $\cos(\theta-\eta)\in[-1,0]$, one has that $\tilde{x}_c<0$ and then
\begin{align}
\nonumber J_s(\phi,\theta,\varphi,\eta)=g_2(r(\varphi,\eta))\geq g_2(0)=&s^2r(\phi,\theta)^2+(\cos(\varphi)-\cos(\phi))^2\\
\geq &s^2\sin^2(\phi)\sin^2((\theta-\eta)/2)+(\phi+\varphi)^2(\phi-\varphi)^2.\label{den3-3}
\end{align}

By summing up  \eqref{den3-1}--\eqref{den3-2}--\eqref{den3-3} we achieve \eqref{den3}.}
\end{proof}

\subsection{Regularity persistence}
In this section we shall investigate the regularity of the function $\tilde{F}$ introduced in \eqref{Ftilde}. The main result reads as follows.
\begin{pro}\label{prop-wellpos}
Let  $m\geq 2, \alpha\in(0,1)$ and $r_0$ satisfy ${\bf{(H)}}$. There exists $\varepsilon\in(0,1)$ small enough such that the functional 
$$\tilde{F}:\R\times B_{X_m^\alpha}(\varepsilon)\rightarrow X_m^\alpha
$$ is well-defined and of class $\mathscr{C}^1$. The function spaces $X_m^\alpha$ are defined in \eqref{space} and \eqref{spaceX1}.
\end{pro}
\begin{proof}
First we shall split the functional  $\tilde{F}$ into two pieces
$$
\tilde{F}(\Omega,f)(\phi,\theta)=F_1(f)(\phi,\theta)-\frac{\Omega}{2} F_2(f)(\phi,\theta)-\frac{1}{2\pi}\bigintsss_0^{2\pi}\left[F_1(f)(\phi,\theta)-\frac{\Omega}{2} F_2(f)(\phi,\theta)\right]d\theta,
$$
with
\begin{align*}
F_1(f)(\phi,\theta)=&\frac{I(f)(\phi,\theta)}{r_0(\phi)},\\
F_2(f)(\phi,\theta)=&2f(\phi,\theta)+\frac{f^2(\phi,\theta)}{r_0(\phi)}\cdot
\end{align*}
{Define also 
\begin{align*}
\mathscr{F}_1(f)(\phi,\theta):=F_1(f)(\phi,\theta)-\frac{1}{2\pi}\bigintsss_0^{2\pi}F_1(f)(\phi,\theta)d\theta,\\
\mathscr{F}_2(f)(\phi,\theta):=F_2(f)(\phi,\theta)-\frac{1}{2\pi}\bigintsss_0^{2\pi}F_2(f)(\phi,\theta)d\theta.
\end{align*}}
Note that $I(f)$ is defined \eqref{def-I} and it is nothing but the stream function $\psi_0$ associated to the domain parametrized by 
$$
(\phi,\theta)\in[0,\pi]\times[0,2\pi]\mapsto \left(\big( r_0(\phi)+f(\phi,\theta)\big)e^{i\theta}, \cos\phi\right)\cdot
$$
Thus
\begin{equation}\label{F1-X1}
F_1(f)(\phi,\theta)=\frac{\psi_0\left(\big(r_0(\phi)+f(\phi,\theta)\big)e^{i\theta}, \cos\phi\right)}{r_0(\phi)}\cdot
\end{equation}
We point out that according to  the general potential theory the stream function $\psi_0$  belongs  at least  to the space $\mathscr{C}^{1,\alpha}(\R^3)$. The proof will be divided into three steps.

\medskip
\noindent
{{\bf Step 1:} {$f\mapsto \mathscr{F}_2(f)$} is $\mathscr{C}^1$.}
In this step, we check that {$\mathscr{F}_2$} is well-defined and of class  $\mathscr{C}^1$. {Note that checking the regularity for $\mathscr{F}_2$ is equivalent to do it for $F_2$.} The first term {of $F_2$} is trivial to check. As to the second one, it is clear by Taylor's formula using the boundary conditions and ${\bf{(H2)}}$ that  the function $\frac{f^2}{r_0}$ is bounded and  vanishes at the points $\phi=0,\pi$. For the regularity, we differentiate  with respect to $\phi$, 
$$
\partial_\phi\left(\frac{f^2(\phi,\theta)}{{ r_0(\phi)}}\right)=-r_0'(\phi) \left(\frac{f(\phi,\theta)}{r_0(\phi)}\right)^2+2\frac{f(\phi,\theta)}{r_0(\phi)}\partial_\phi f(\phi,\theta).
$$
Using again Taylor's formula and the assumptions ${\bf{(H)}}$ on $r_0$ we deduce  that the functions $(\phi,\theta)\mapsto \frac{f(\phi,\theta)}{\sin\phi}$ and $(\phi,\theta)\mapsto \frac{\sin\phi }{r_0(\phi)}$ belongs to  $\mathscr{C}^{\alpha}$. Thus using the algebra structure of this latter space we infer that  $(\phi,\theta)\mapsto \frac{f(\phi,\theta)}{r_0(\phi)}$ belongs also to $\mathscr{C}^\alpha$.  The same algebra structure allows to get  $\partial_\phi\left(\frac{f^2}{{ r_0}}\right)\in \mathscr{C}^\alpha.$  Following  the same argument we obtain $\partial_\theta F_2$ belongs to $\mathscr{C}^\alpha.$ Concerning the symmetry; it  can be derived from  Proposition \ref{prop-SymX1} combined with the fact that frequency $n=0$ is eliminated in the definition of $\mathscr{F}_2$ by subtracting the mean value in $\theta$.

 Now let us check the   $\mathscr{C}^1$ dependence in  $f$ of $F_2$. First we can check that   its Frechet derivative takes the form 
$$
\partial_f F_2(f)h(\phi,\theta)=2h(\phi,\theta)+2\frac{f(\phi,\theta)h(\phi,\theta)}{r_0(\phi)}\cdot
$$
Using similar ideas as before, we can easily get that 
$$
\|\partial_f {\mathscr{F}_2(f_1)h-\partial_f \mathscr{F}_2(f_2)h}\|_{X_m^\alpha}\leq C\|f_1-f_2\|_{X_m^\alpha}\|h\|_{X_m^\alpha}.
$$
{This implies  that $f\mapsto \partial_f \mathscr{F}_2(f)$ is continuous and therefore $\mathscr{F}_2$ is of class $\mathscr{C}^1$.}

\medskip
\noindent
{{{\bf Step 2:} $f\mapsto \mathscr{F}_1(f)$ {\it  is well-defined}.} This is more involved than $\mathscr{F}_2.$ According to  Proposition  \ref{prop-SymX1} the functional $\mathscr{F}_1$  is symmetric with respect to $\phi=\frac{\pi}{2}$ and therefore it suffices to check the desired regularity in the range $\phi\in(0,\pi/2)$ { and check that the derivative is not discontinuity at $\pi/2$}. Let us emphasize  that   we need to check the regularity not for $F_1$ but for $\mathscr{F}_1$}
First, we shall check that $\mathscr{F}_1$ is bounded and satisfies the boundary condition $\mathscr{F}_1(0,\theta)=0$, for any $\theta\in(0,2\pi)$. The remaining boundary condition  $\mathscr{F}_1(\pi,\theta)=0$  follows from the symmetry with respect to the equatorial.  For this purpose, we write by virtue of Taylor's formula
\begin{align}\label{psi0X}
\forall x_h\in\R^2,\quad \psi_0(x_h,\cos\phi)
=&\psi_0(0,0,\cos\phi)+x_h\cdot\int_0^1\nabla_h\psi_0\big(\tau x_h,\cos\phi\big)d\tau.
\end{align}
Making the substitution $x_h=(r_0(\phi)+f(\phi,\theta))e^{i\theta}$ and using \eqref{F1-X1} we infer
\begin{align*}
F_1(f)(\phi,\theta)=&\frac{\psi(0,0,\cos\phi)}{r_0(\phi)}+\left(1+\frac{f(\phi,\theta)}{r_0(\phi)}\right)e^{i\theta}\cdot\int_0^1\nabla_h\psi_0\left(\tau(r_0(\phi)+f(\phi,\theta))e^{i\theta},\cos\phi\right)d\tau\\
=:&\frac{\psi(0,0,\cos\phi)}{r_0(\phi)}+\mathscr{F}_{1,1}(\phi,\theta).
\end{align*}
We observe that  the $\cdot$ denotes  the usual Euclidean  inner product  of $\R^2.$
Consequently, we obtain
\begin{equation}\label{Simplifi1}
\mathscr{F}_1(\phi,\theta)=\mathscr{F}_{1,1}(\phi,\theta)-\langle \mathscr{F}_{1,1}\rangle_\theta.
\end{equation}

Let us analyze the  term $ \mathscr{F}_{1,1}$ and check its continuity and the Dirichlet boundary condition. First we observe from the assumption ${\bf{(H2)}}$ that $0$ is a simple zero for $r_0$ and we know that $f(0,\theta)=0$, then    one may easily  obtain the bound 
$$
|\mathscr{F}_{1,1}(\phi,\theta)|\leq C(1+\|\partial_\phi f\|_{L^\infty})\|\nabla_h\psi_0\|_{L^\infty(\R^3)}.
$$
Furthermore, according to  Lebesgue dominated convergence theorem  we infer
$$
\lim_{\phi\to 0}\mathscr{F}_{1,1}(\phi,\theta)=\left(1+\frac{\partial_\phi f(0,\theta)}{r_0^\prime(0)}\right)e^{i\theta}\cdot\nabla_h\psi_0\left(0,0,1\right),
$$
and this convergence is uniform in $\theta\in(0,2\pi)$. Notice that the same tool gives the continuity of $ \mathscr{F}_{1,1}$ in $[0;\pi/2]\times[0;2\pi].$

Now, applying  Lemma \ref{lemma-velocidad0} we get $\nabla_h\psi_0\left(0,0, 1\right)=0$, and therefore
$$
\forall \,\theta\in(0,2\pi),\quad \lim_{\phi\to 0}\mathscr{F}_{1,1}(\phi,\theta)=\lim_{\phi\to 0}\langle\mathscr{F}_{1,1}\rangle_\theta=0.
$$
This implies that $\mathscr{F}_1$ is  continuous in $[0,\pi]\times[0,2\pi]$ and it  satisfies the required  Dirichlet  boundary condition $\mathscr{F}_1(0,\theta)=\mathscr{F}_1(\pi,\theta)=0$.

The next step is to establish that  $\partial_\theta\mathscr{F}_1$ and $\partial_\phi \mathscr{F}_1$ are $\mathscr{C}^\alpha$. We will relate such derivatives to the two-components velocity field $U=\nabla_h^\perp \psi_0$. Differentiating \eqref{F1-X1} with respect to $\theta$  leads to 
\begin{align}
\nonumber \partial_\theta\mathscr{F}_1(\phi,\theta)=&
\partial_\theta F_1(f)(\phi,\theta)\\
=&r_0^{-1}(\phi)\,\nabla_h \psi_0(r(\phi,\theta)e^{i\theta}, \cos(\phi))\cdot\left(r(\phi,\theta)ie^{i\theta}+\partial_\theta r(\phi,\theta)e^{i\theta}\right)\nonumber\\
=&{-}\frac{r(\phi,\theta)}{r_0(\phi)} U(r(\phi,\theta)e^{i\theta}, \cos(\phi))\cdot e^{i\theta}{+} \partial_\theta r(\phi,\theta) 
\frac{U(r(\phi,\theta)e^{i\theta},\cos(\phi))}{r_0(\phi)}\cdot ie^{i\theta},\label{F1theta}
\end{align}
where $r(\phi,\theta)=r_0(\phi)+f(\phi,\theta)$ and recall that $\cdot$ is the usual Euclidean  inner product  of $\R^2.$

Concerning  the regularity of the partial derivative  in $\phi$, we achieve 
\begin{align}
\partial_\phi F_1(f)(\phi,\theta)
=&-\frac{r_0'(\phi)}{r_0^2(\phi)}\psi_0(r(\phi,\theta)e^{i\theta}, \cos(\phi))+
\frac{\partial_\phi r(\phi,\theta)}{r_0(\phi)}\nabla_h \psi_0(r(\phi,\theta)e^{i\theta},\cos(\phi))\cdot e^{i\theta}\nonumber\\
&-\frac{\sin(\phi)}{r_0(\phi)}\partial_z\psi_0(r(\phi,\theta)e^{i\theta},\cos(\phi))\nonumber\\
=&-\frac{r_0'(\phi)}{r_0^2(\phi)}\psi_0(r(\phi,\theta)e^{i\theta}, \cos(\phi))-
\partial_\phi r(\phi,\theta)\frac{U(r(\phi,\theta)e^{i\theta},\cos(\phi))}{r_0(\phi)} \cdot ie^{i\theta}\nonumber\\
&-\frac{\sin(\phi)}{r_0(\phi)}\partial_z\psi_0(r(\phi,\theta)e^{i\theta},\cos(\phi)).\label{F1phi}
\end{align}
Define
$$
\mathscr{J}_{1}(\phi,\theta)=\frac{r_0'(\phi)}{r_0^2(\phi)}\psi_0(r(\phi,\theta)e^{i\theta}, \cos(\phi)),
$$
and
$$
\mathscr{J}_{2}(\phi,\theta)=\partial_\phi r(\phi,\theta)\frac{U(r(\phi,\theta)e^{i\theta},\cos(\phi))}{r_0(\phi)} \cdot ie^{i\theta}.
$$
Then from \eqref{F1phi} we may write
$$
\partial_\phi F_1(f)(\phi,\theta)=-\mathscr{J}_{1}(\phi,\theta)-\mathscr{J}_{2}(\phi,\theta)-\frac{\sin(\phi)}{r_0(\phi)}\partial_z\psi_0(r(\phi,\theta)e^{i\theta},\cos(\phi)).
$$
{Let us justify why we can restrict ourselves to prove the regularity for $\phi\in[0,\pi/2]$ by symmetry. Indeed, since $r(\pi-\phi,\theta)=r(\phi,\theta)$ and  $r_0(\pi-\phi)=r_0(\phi)$, then we obtain that
$$
r_0'(\pi-\phi)=-r_0'(\phi), \quad \partial_\phi r(\pi-\phi,\theta)=\partial_\phi r(\phi,\theta).
$$ 
That implies that $r_0'(\pi/2)=0$ and $\partial_\phi r(\pi/2,\theta)=0$. By this way, it is easy to check that $\mathscr{J}_1(\pi/2,\theta)=\mathscr{J}_2(\pi/2,\theta)=0$ yielding that there is not any jump at $\pi/2$.} Moreover, note that the last term  belongs to $\mathscr{C}^\alpha$ {for any $\phi\in[0,\pi]$ (here we do not need to resctrict ourselves to $\phi\in[0,\pi/2]$) and then, by symmetry, we can extend it to $(0,\pi)$}. Indeed, as  $(\phi,\theta)\mapsto \big(r(\phi,\theta)e^{i\theta},\cos(\phi)\big)$ belongs to $\mathscr{C}^{1,\alpha}$ and $\partial_z\psi_0\in \mathscr{C}^{\alpha}(\R^3)$ then by composition we infer $(\phi,\theta)\mapsto \partial_z\psi_0\big(r(\phi,\theta)e^{i\theta},\cos(\phi)\big)$ is in $\mathscr{C}^{\alpha}((0,\pi)\times{\T}\big)$. On the  other hand, the function  $\frac{\sin}{r_0}$ belongs to $ \mathscr{C}^\alpha$ and thus by the algebra structure of $\mathscr{C}^\alpha$  we obtain the desired result.

Concerning the  term $\mathscr{J}_{1},$ we  use Taylor's formula for the stream function $\psi_0$ as in \eqref{psi0X} finding that  
\begin{align*}
\mathscr{J}_{1}(\phi,\theta)=&\frac{r_0'(\phi)\psi_0(0,0,\cos\phi)}{r_0^2(\phi)}+r_0'(\phi)r_0^{-1}(\phi)\left(1+\frac{f(\phi,\theta)}{r_0(\phi)}\right)\,\bigintsss_0^1\nabla_h\psi_0\big(s\, r(\phi,\theta)e^{i\theta},\cos\phi\big)ds\cdot e^{i\theta}\\
=&\frac{r_0'(\phi)\psi_0(0,0,\cos\phi)}{r_0^2(\phi)}+r_0'(\phi)\left(1+\frac{f(\phi,\theta)}{r_0(\phi)}\right)\,r_0^{-1}(\phi) \bigintsss_0^1U\big(s\, r(\phi,\theta)e^{i\theta},\cos\phi\big)ds\cdot ie^{i\theta}.
\end{align*}
We observe that the first term is singular and depends only on $\phi$ and therefore it does not contribute in $\mathscr{J}_{1}-\langle \mathscr{J}_{1}\rangle_\theta.$ Since $(\phi,\theta)\mapsto \frac{r(\phi,\theta)}{r_0(\phi)}$ belongs to $\mathscr{C}^\alpha$ then to get  $\mathscr{J}_{1}-\langle \mathscr{J}_{1}\rangle_\theta\in \mathscr{C}^\alpha$ it suffices to prove that
\begin{equation}\label{CondXX1}
(\phi,\theta)\mapsto\bigintsss_0^1 \frac{U(s r(\phi,\theta)e^{i\theta}, \cos(\phi))}{r_0(\phi)}\cdot ie^{i\theta} d s\in \mathscr{C}^\alpha.
\end{equation}
On the other hand to obtain $\mathscr{J}_2\in \mathscr{C}^\alpha$ it is enough to get 
\begin{equation}\label{CondXX2}
(\phi,\theta)\mapsto \frac{U(r(\phi,\theta)r^{i\theta}, \cos(\phi))}{r_0(\phi)}\cdot i\,e^{i\theta}\in \mathscr{C}^\alpha.
\end{equation}
From   \eqref{F1theta} we get that $\partial_\theta F_1(f)\in \mathscr{C}^\alpha$ provided that \eqref{CondXX1} and  \eqref{CondXX2} are satisfied  together with 
\begin{equation}\label{CondXX3}
(\phi,\theta)\mapsto {U(r(\phi,\theta)r^{i\theta}, \cos(\phi))}\cdot e^{i\theta}\in \mathscr{C}^\alpha.
\end{equation}
 By virtue of \eqref{v} and the fact that $U=\nabla_h^\perp\psi$, we find that
\begin{equation}\label{v-2}
U(r(\phi,\theta)e^{i\theta},\cos(\phi))=\frac{1}{4\pi}\bigintsss_{0}^{\pi}\bigintsss_0^{2\pi}\frac{\sin(\varphi)\big(\partial_{\eta}r(\varphi,\eta)e^{i\eta}+ir(\varphi,\eta)e^{i\eta}\big)\,d\eta d\varphi}{|(r(\phi,
\theta)e^{i\theta},\cos(\phi))-(r(\varphi,\eta)e^{i\eta},\cos(\varphi))|}\cdot
\end{equation}
Next we intend to prove \eqref{CondXX1}, \eqref{CondXX2} and \eqref{CondXX3}.

\medskip
\noindent
{$\bullet$ {\it Proof of \eqref{CondXX3}.}
Using \eqref{v-2}, we deduce that
$$
U(r(\phi,\theta)e^{i\theta},\cos(\phi))\cdot e^{i\theta}=\frac{1}{4\pi}\bigintsss_{0}^{\pi}\bigintsss_0^{2\pi}\frac{\sin(\varphi)\partial_\eta\big(r(\varphi,\eta)\cos(\eta-\theta)\big)\,d\eta d\varphi}{|(r(\phi,
\theta)e^{i\theta},\cos(\phi))-(r(\varphi,\eta)e^{i\eta},\cos(\varphi))|}\cdot
$$
Using the notation of Lemma \ref{lemma-estimdenominator} we find that
$$
|(r(\phi,
\theta)e^{i\theta},\cos(\phi))-(r(\varphi,\eta)e^{i\eta},\cos(\varphi))|=J_1^{\frac12}(\phi,\theta,\varphi,\eta),
$$
and therefore we may write
$$
U(r(\phi,\theta)e^{i\theta},\cos(\phi))\cdot e^{i\theta}=\frac{1}{4\pi}\bigintsss_{0}^{\pi}\bigintsss_0^{2\pi}\frac{\sin(\varphi)\partial_\eta\big(r(\varphi,\eta)\cos(\eta-\theta)\big)\,d\eta d\varphi}{J_1^{\frac12}(\phi,\theta,\varphi,\eta)}.
$$
This can be split into two integral terms
\begin{align}\label{veitheta}
\nonumber U(r(\phi,\theta)e^{i\theta},\cos(\phi))\cdot e^{i\theta}
=&\bigintsss_0^\pi\bigintsss_0^{2\pi}\frac{\sin(\varphi)\partial_\eta r(\varphi,\eta)\cos(\eta-\theta)d\eta d\varphi}{J_1^\frac12(\phi,\theta,\varphi,\eta)}\\
&-\bigintsss_0^\pi\bigintsss_0^{2\pi}\frac{\sin(\varphi)r(\varphi,\eta)\sin(\eta-\theta)d\eta d\varphi}{J_1^\frac12(\phi,\theta,\varphi,\eta)}\nonumber\\
:=&\mathcal{I}_1(\phi,\theta)-\mathcal{I}_2(\phi,\theta).
\end{align}
Next, we shall prove  that ${\mathcal{I}}_1$ is $\mathscr{C}^\alpha$.
Notice  that the second term $\mathcal{I}_2$ is easier to deal with than $\mathcal{I}_1$ because its kernel is more regular on the diagonal  than that of $\mathcal{I}_1$. To get $\mathcal{I}_2\in \mathscr{C}^\alpha$  it suffices to use  in a standard way  Proposition \ref{prop-potentialtheory}. We shall skip this part and focus our attention on the proof to the delicate part $\mathcal{I}_1.$  For this aim let us define the kernel
\begin{equation*}
\mathscr{K}_1(\phi,\theta,\varphi,\eta)=\frac{\sin(\varphi)\partial_\eta r(\varphi,\eta)\cos(\eta-\theta)}{J_1^\frac12(\phi,\theta,\varphi,\eta)}\cdot
\end{equation*}
We shall start with checking that  $\mathscr{K}_1$ is bounded. For this goal  we use Lemma \ref{lemma-estimdenominator} which implies 
\begin{align*}
|\mathscr{K}_1(\phi,\theta,\varphi,\eta)|\leq &\frac{C\,\sin(\varphi)|\partial_\eta r(\varphi,\eta)|}{\{(\varphi+\phi)^2(\phi-\varphi)^2+(\sin^2(\varphi)+\phi^2)\sin^2((\theta-\eta)/2)\}^{\frac12}}\cdot
\end{align*}
It is easy to check the inequality
\begin{equation}\label{K1sin}
\frac{\sin(\varphi)}{\Big((\varphi+\phi)^2(\phi-\varphi)^2+(\sin^2(\varphi)+\phi^2)\sin^2((\theta-\eta)/2)\Big)^{\frac12}}\leq \frac{1}{\Big((\phi-\varphi)^2+\sin^2((\theta-\eta)/2)\Big)^\frac12}\cdot
\end{equation}
By interpolating between
$$
\Big((\phi-\varphi)^2+\sin^2((\theta-\eta)/2)\Big)^\frac12\geq |\phi-\varphi|,
$$
and
$$
\Big((\phi-\varphi)^2+\sin^2((\theta-\eta)/2)\Big)^\frac12\geq |\sin((\theta-\eta)/2)|,
$$
one finds that for any $\beta\in [0,1]$ and $\phi\leq \pi/2$ 
\begin{equation}\label{K1sin-2}
\frac{\sin(\varphi)}{\Big((\varphi+\phi)^2(\phi-\varphi)^2+(\sin^2(\varphi)+\phi^2)\sin^2((\theta-\eta)/2)\Big)^{\frac12}}\leq \frac{1}{|\phi-\varphi|^{1-\beta}|\sin((\theta-\eta)/2)|^\beta},
\end{equation}
implying that
\begin{align}\label{K1-est1}
|\mathscr{K}_1(\phi,\theta,\varphi,\eta)|\leq \frac{C}{|\phi-\varphi|^{1-\beta}|\sin((\theta-\eta)/2)|^\beta}\cdot
\end{align}
Therefore, we easily achieve that $\mathcal{I}_1\in L^\infty$. 
To establish that $\mathcal{I}_1\in\mathscr{C}^\alpha$, we proceed in a direct way using the definition. Before that we remark that to get the $\mathscr{C}^\alpha$ regularity in both variables $(\phi,\theta)$ it is enough to check the  $\mathscr{C}^\alpha$-regularity separately in the partial variables. Thus we shall check that $\phi\mapsto \mathcal{I}_1(\phi,\theta)$ is $\mathscr{C}^\alpha(0,\pi)$ uniformly in $\theta\in[0,2\pi]$. 
Take $\phi_1,\phi_2\in [0,\frac{\pi}{2}]$ with $0<\phi_1<\phi_2$, then it is easy to check from some algebraic considerations that 
\begin{align*}
\mathcal{I}_1(\phi_2,\theta)-\mathcal{I}_1(\phi_1,\theta)=&\bigintsss_0^\pi\bigintsss_0^{2\pi}\frac{\sin(\varphi)\partial_\eta r(\varphi,\eta)\cos(\eta-\theta)(J_1(\phi_1,\theta,\varphi,\eta)-J_1(\phi_2,\theta,\varphi,\eta))d\eta d\varphi}{J_1^\frac12(\phi_2,\theta,\varphi,\eta)J_1^\frac12(\phi_1,\theta,\varphi,\eta)(J_1^\frac12(\phi_1,\theta,\varphi,\eta)+J_1^\frac12(\phi_2,\theta,\varphi,\eta))}\cdot
\end{align*} 
Coming back to the definition of $J_1$ seen in Lemma \ref{lemma-estimdenominator}, we can check that 
\begin{align*}
J_1(\phi_1,\theta,\varphi,\eta)-J_1(\phi_2,\theta,\varphi,\eta)=&\big(r(\phi_1,\theta)-r(\phi_2,\theta)\big)\big(r(\phi_1,\theta)-r(\varphi,\eta)+r(\phi_2,\theta)-r(\varphi,\eta)\big)\\
&+2r(\varphi,\eta)\big(r(\phi_1,\theta)-r(\phi_2,\theta)\big)\big(1-\cos(\theta-\eta)\big)\\
&{-}\big(\cos\varphi-\cos\phi_1+\cos\varphi-\cos\phi_2\big)\big(\cos\phi_1-\cos\phi_2\big).
\end{align*}
Since $r\in \textnormal{Lip}$ we infer by interpolation
$$
\big|r(\phi_1,\theta)-r(\phi_2,\theta)\big|\le C|\phi_1-\phi_2|^\alpha\Big(|r(\phi_1,\theta)-r(\varphi,\eta)|^{1-\alpha}+|r(\phi_2,\theta)-r(\varphi,\eta)|^{1-\alpha}\Big),
$$
and
$$
\big|r(\phi_1,\theta)-r(\phi_2,\theta)\big|\le C|\phi_1-\phi_2|^\alpha\Big(r^{1-\alpha}(\phi_1,\theta)+r^{1-\alpha}(\phi_2,\theta)\Big).
$$
Consequently we find
\begin{align*}
|J_1(\phi_1,\theta,\varphi,\eta)-&J_1(\phi_2,\theta,\varphi,\eta)|\leq C|\phi_1-\phi_2|^\alpha\Big(|r(\phi_1,\theta)-r(\varphi,\eta)|^{2-\alpha}+|r(\phi_2,\theta)-r(\varphi,\eta)|^{2-\alpha}\\
&+r(\varphi,\eta)\big(r^{1-\alpha}(\phi_1,\theta)+r^{1-\alpha}(\phi_2,\theta)\big)\big(1-\cos(\eta-\theta)\big)\\
&+|\cos(\varphi)-\cos(\phi_1)|^{2-\alpha}+|\cos(\varphi)-\cos(\phi_2)|^{2-\alpha}\Big).
\end{align*}
From straightforward calculus we observe that
\begin{align*}
\frac{|r(\phi_i,\theta)-r(\varphi,\eta)|^{2-\alpha}+r(\varphi,\eta)r(\phi_i,\theta)^{1-\alpha}\big(1-\cos(\eta-\theta)\big)+|(\cos(\varphi)-\cos(\phi_i)|^{2-\alpha}}{J_1(\phi_i,\theta,\varphi,\eta)^\frac{2-\alpha}{2}}\leq C,
\end{align*}
and then we find
\begin{align*}
|J_1(\phi_1,\theta,\varphi,\eta)-&J_1(\phi_2,\theta,\varphi,\eta)|\leq C|\phi_1-\phi_2|^\alpha\Big(J_1^{\frac{2-\alpha}{2}}(\phi_1,\theta,\varphi,\eta)+J_1^{\frac{2-\alpha}{2}}(\phi_2,\theta,\varphi,\eta)\Big).\end{align*}
It follows that

\begin{align}\label{J1-K1}
&\frac{|J_1(\phi_1,\theta,\varphi,\eta)-J_1(\phi_2,\theta,\varphi,\eta)|}{J_1^\frac12(\phi_2,\theta,\varphi,\eta)J_1(\phi_1,\theta,\varphi,\eta)^\frac12(J_1(\phi_1,\theta,\varphi,\eta)^\frac12+J_1^\frac12(\phi_2,\theta,\varphi,\eta))}\nonumber\\
&\lesssim\frac{|\phi_1-\phi_2|^\alpha}{J_1^\frac{\alpha}{2}(\phi_2,\theta,\varphi,\eta)J_1^\frac12(\phi_1,\theta,\varphi,\eta)}+\frac{|\phi_1-\phi_2|^\alpha}{J_1^\frac{1}{2}(\phi_2,\theta,\varphi,\eta)J_1^\frac{\alpha}{2}(\phi_1,\theta,\varphi,\eta)}\cdot
\end{align}
Using  \eqref{J1-K1}, one finds
\begin{align*}
|\mathcal{I}_1(\phi_2,\theta)-\mathcal{I}_1(\phi_1,\theta)|\lesssim&|\phi_1-\phi_2|^\alpha\bigintsss_0^{\pi}\bigintsss_0^{2\pi}\frac{\sin(\varphi)|\partial_\eta r(\varphi,\eta)|d\eta d\varphi}{J_1^\frac{\alpha}{2}(\phi_2,\theta,\varphi,\eta)J_1^\frac12(\phi_1,\theta,\varphi,\eta)}\\
&+|\phi_1-\phi_2|^\alpha\bigintsss_0^{\pi}\bigintsss_0^{2\pi}\frac{\sin(\varphi)|\partial_\eta r(\varphi,\eta)| d\eta d\varphi}{J_1^\frac{1}{2}(\phi_2,\theta,\varphi,\eta)J_1^\frac{\alpha}{2}(\phi_1,\theta,\varphi,\eta)}\cdot
\end{align*}
By virtue of \eqref{K1sin-2}, for any $\beta\in(0,1)$ we obtain
\begin{align*}
\bigintsss_0^{\pi}\bigintsss_0^{2\pi}\frac{\sin(\varphi)|\partial_\eta r(\varphi,\eta)|d\eta d\varphi}{J_1^\frac{\alpha}{2}(\phi_2,\theta,\varphi,\eta)J_1^\frac12(\phi_1,\theta,\varphi,\eta)}\leq&\bigintsss_0^{\pi}\bigintsss_0^{2\pi}\frac{|\partial_\eta r(\varphi,\eta)|\,J_1^{-\frac{\alpha}{2}}(\phi_2,\theta,\varphi,\eta)d\eta d\varphi}{|\phi_1-\varphi|^{1-\beta}|\sin((\theta-\eta)/2)|^\beta }\cdot
\end{align*}
Hence, we get in view  of Lemma \ref{lemma-estimdenominator} and \eqref{platit1}
\begin{align*}
\frac{|\partial_\eta r(\varphi,\eta)|}{J_1(\phi_i,\theta,\varphi,\eta)^\frac{\alpha}{2}}\lesssim& \frac{\varphi^\alpha}{\left((\varphi+\phi_i)^2(\phi_i-\varphi)^2+(\sin^2(\varphi)+\phi_i^2)\sin^2((\theta-\eta)/2)\right)^\frac{\alpha}{2}}\\
\lesssim&{|\phi_i-\varphi|^{-\alpha}}.
\end{align*}
for any $i=1,2$. This implies that
\begin{align*}
\bigintsss_0^{\pi}\bigintsss_0^{2\pi}\frac{\sin(\varphi)|\partial_\eta r(\varphi,\eta)|d\eta d\varphi}{J_1^\frac{\alpha}{2}(\phi_2,\theta,\varphi,\eta)J_1^\frac12(\phi_1,\theta,\varphi,\eta)}\lesssim&\bigintsss_0^{\pi}\bigintsss_0^{2\pi}\frac{|\phi_2-\varphi|^{-\alpha}d\eta d\varphi}{|\phi_1-\varphi|^{1-\beta}|\sin((\theta-\eta)/2)|^\beta }\\
\lesssim&\bigintsss_0^{\pi}\frac{|\phi_2-\varphi|^{-\alpha}d\varphi}{|\phi_1-\varphi|^{1-\beta} }\cdot
\end{align*}
It is classical that 
$$
\sup_{\phi_1,\phi_2\in(0,\pi)}\bigintsss_0^{\pi}\frac{|\phi_2-\varphi|^{-\alpha}}{|\phi_1-\varphi|^{1-\beta} }d\varphi<\infty,
$$
provided that $0\le\alpha<\beta<1$. Similarly we prove  under the same condition that
$$
{\sup_{\phi_1,\phi_2\in(0,\pi/2)\\
\atop  \theta\in(0,2\pi)}}\bigintsss_0^{\pi}\bigintsss_0^{2\pi}\frac{\sin(\varphi)|\partial_\eta r(\varphi,\eta)| d\eta d\varphi}{J_1^\frac{1}{2}(\phi_2,\theta,\varphi,\eta)J_1^\frac{\alpha}{2}(\phi_1,\theta,\varphi,\eta)}<\infty.
$$
Putting together the preceding estimates yields for any $0<\alpha<\beta<1$
$$
\forall \theta\in(0,2\pi),{\phi_1,\phi_2\in(0,\pi/2)},\quad |\mathcal{I}_1(\phi_2,\theta)-\mathcal{I}_1(\phi_1,\theta)|\leq C|\phi_1-\phi_2|^\alpha,
$$
where the constant $C$ is independent of $\theta, \phi_1$ and $\phi_2.$
{
Let us move on the regularity of $\mathcal{I}_1$ in $\theta$. Here we shall use the estimate later proved \eqref{Jtheta} for $s=1$, that is, for $\theta_1,\theta_2\in[0,2\pi]$ we have
\begin{align}\label{Jtheta1}
\frac{|J_1(\phi,\theta_2,\varphi,\eta)-J_1(\phi,\theta_1,\varphi,\eta)|}{J_1^\frac12(\phi,\theta_2,\varphi,\eta)+J_1^\frac12(\phi,\theta_1,\varphi,\eta)}\lesssim|\theta_1-\theta_2|^\alpha\phi^{\alpha^2}(J_1^{\frac{1-\alpha}{2}}(\phi,\theta_1,\varphi,\eta)+J_1^{\frac{1-\alpha}{2}}(\phi,\theta_2,\varphi,\eta)).
\end{align}
Note that
\begin{align*}
\mathcal{I}_1(\phi,\theta_1)-\mathcal{I}_1(\phi,\theta_2)=&\int_0^\pi\int_0^{2\pi}\frac{\sin(\varphi)\partial_\eta r(\varphi,\eta)(\cos(\eta-\theta_1)-\cos(\eta-\theta_2))}{J_1^\frac12(\phi,\theta_1,\varphi,\eta)}d\eta d\varphi\\
&+\int_0^\pi\int_0^{2\pi}\frac{\sin(\varphi)\partial_\eta r(\varphi,\eta)\cos(\eta-\theta_2)(J_1(\phi,\theta_2,\varphi,\eta)-J_1(\phi,\theta_1,\varphi,\eta))}{J_1^\frac12(\phi,\theta_1,\varphi,\eta)J_1^\frac12(\phi,\theta_1,\varphi,\eta)(J_1^\frac12(\phi,\theta_1,\varphi,\eta)+J_1^\frac12(\phi,\theta_1,\varphi,\eta))}d\eta d\varphi\\
=:&\int_0^\pi\int_0^{2\pi}(\mathscr{K}_{1,1}+\mathscr{K}_{1,2})(\phi,\theta,\varphi,\eta)d\eta d\varphi
\end{align*}
On the one hand, using \eqref{K1sin-2} we easily get
\begin{align*}
|\mathscr{K}_{1,1}(\phi,\theta,\varphi,\eta)|\leq \frac{C|\theta_1-\theta_2|}{|\phi-\varphi|^{1-\beta}|\sin((\theta-\eta)/2)|^\beta},
\end{align*}
for any $\beta\in(0,1)$.  On the other hand, we use \eqref{Jtheta1} to estimate the second term:
\begin{align*}
|\mathscr{K}_{1,2}(\phi,\theta,\varphi,\eta)|\leq& C|\theta_1-\theta_2|^\alpha\frac{\sin(\varphi)^{1+\alpha}\phi^{\alpha^2}(J_1^{\frac{1-\alpha}{2}}(\phi,\theta_1,\varphi,\eta)+J_1^{\frac{1-\alpha}{2}}(\phi,\theta_2,\varphi,\eta))}{J_1^\frac12(\phi,\theta_1,\varphi,\eta)J_1^\frac12(\phi,\theta_1,\varphi,\eta)}\\
\leq &C|\theta_1-\theta_2|^\alpha\frac{\sin(\varphi)^{1+\alpha}\phi^{\alpha^2}}{J_1^\frac{\alpha}{2}(\phi,\theta_1,\varphi,\eta)J_1^\frac12(\phi,\theta_1,\varphi,\eta)}\\
&+C|\theta_1-\theta_2|^\alpha\frac{\sin(\varphi)^{1+\alpha}\phi^{\alpha^2}}{J_1^\frac{1}{2}(\phi,\theta_1,\varphi,\eta)J_1^\frac{\alpha}{2}(\phi,\theta_1,\varphi,\eta)}.
\end{align*}
By virtue of Lemma \ref{lemma-estimdenominator}, we arrive at
\begin{align*}
|\mathscr{K}_{1,2}(\phi,\theta,\varphi,\eta)|\leq& C|\theta_1-\theta_2|^\alpha \frac{\sin(\varphi)^{1+\alpha}}{\sin(\varphi)^{1+\alpha}(|\phi-\varphi|+|\sin((\theta_1-\eta)/2)|)^{\alpha}(|\phi-\varphi|+|\sin((\theta_2-\eta)/2)|)}\\
&+C|\theta_1-\theta_2|^\alpha \frac{\sin(\varphi)^{1+\alpha}}{\sin(\varphi)^{1+\alpha}(|\phi-\varphi|+|\sin((\theta_2-\eta)/2)|)^{\alpha}(|\phi-\varphi|+|\sin((\theta_1-\eta)/2)|)}\\
\leq &C|\theta_1-\theta_2|^\alpha \frac{1}{\sin(\varphi)^{\gamma_1}|\sin((\theta_1-\eta)/2)|^{\gamma_2}|\sin((\theta_1-\eta)/2)|^{\gamma_3}},
\end{align*}
for $\gamma_1,\gamma_2,\gamma_3\in(0,1)$ and $\gamma_2+\gamma_3<1$.

 At the end, we can find that
$$
|\mathcal{I}_1(\phi,\theta_2)-\mathcal{I}_1(\phi,\theta_1)|\leq C|\theta_1-\theta_2|^\alpha.
$$}
Finally, this allows to get that  $\mathcal{I}_1$ is $\mathscr{C}^\alpha((0,\pi)\times{\T})$. 

\medskip
\noindent
$\bullet$ {\it Proof of \eqref{CondXX2}}. In fact we shall establish a more refined result:
$$
(\phi,\theta)\mapsto \frac{U\big(sr(\phi,\theta)e^{i\theta},\cos(\phi)\big)}{r_0(\phi)}\cdot ie^{i\theta}\in \mathscr{C}^\alpha([0,\pi]\times[0,2\pi]).
$$
uniformly with respect to  $s\in[0,1].$ This allows to get the results \eqref{CondXX2} and \eqref{CondXX1}.

Coming back to \eqref{v-2} and using $J_s$ introduced in Lemma \ref{lemma-estimdenominator} we find the expression
\begin{align}
\frac{U\big(sr(\phi,\theta)e^{i\theta},\cos(\phi)\big)}{r_0(\phi)}\cdot ie^{i\theta}&
=\frac{1}{4\pi r_0(\phi)}\bigintsss_0^\pi\bigintsss_0^{2\pi}\frac{\sin(\varphi)\partial_\eta\big(r(\varphi,\eta)\sin(\eta-\theta)\big)}{J_s^\frac12(\phi,\theta,\varphi,\eta)}d\eta d\varphi.
\label{T2}
\end{align}
Moreover, by Lemma \ref{lemma-velocidad0} we have
$$
\forall\, (\phi,\theta)\in(0,\pi)\times(0,2\pi),\quad \bigintsss_0^\pi\bigintsss_0^{2\pi}\frac{\sin(\varphi)\partial_\eta\big(r(\varphi,\eta)\sin(\eta-\theta)\big)d\eta d\varphi}{J_0(\phi,\theta,\varphi,\eta)^\frac12}=0,
$$ 
and then we can subtract this vanishing term obtaining
\begin{align*}
&\qquad\qquad\qquad\qquad\frac{U\big(sr(\phi,\theta)e^{i\theta},\cos(\phi)\big)}{r_0(\phi)}\cdot ie^{i\theta}=\\
&-\frac{r(\phi,\theta)}{4\pi r_0(\phi)}\bigintsss_0^\pi\bigintsss_0^{2\pi}\frac{\sin(\varphi)\partial_\eta\big(r(\varphi,\eta)\sin(\eta-\theta)\big)s\big(sr(\phi,\theta)-2r(\varphi,\eta)\cos(\theta-\eta)\big)}{J_s^\frac12(\phi,\theta,\varphi,\eta)J_0^\frac12(\phi,\varphi,\eta)\big(J_s^\frac12(\phi,\theta,\varphi,\eta)+J_0^\frac12(\phi,\varphi,\eta)\big)}d\eta d\varphi.
\end{align*}
Notice that we eliminate .the variable $\theta$ from the definition of $J_0$ because it is independent of this parameter.
Since $(\phi,\theta)\mapsto \frac{r(\phi,\theta)}{r_0(\phi)}$ is $\mathscr{C}^\alpha$, then to get the desired regularity it is enough to check it for the  the integral term. Denote
\begin{equation}\label{Kern2X1}
\mathscr{K}_2(s,\phi,\theta,\varphi,\eta)=\frac{\sin(\varphi)\partial_\eta\big(r(\varphi,\eta)\sin(\eta-\theta)\big)s\big(sr(\phi,\theta)-2r(\varphi,\eta)\cos(\theta-\eta)\big)}{J_s^\frac12(\phi,\theta,\varphi,\eta)J_0^\frac12(\phi,\varphi,\eta)\big(J_s^\frac12(\phi,\theta,\varphi,\eta)+J_0^\frac12(\phi,\varphi,\eta)\big)},
\end{equation}
and let  us show first that 
$$
(\phi,\theta)\mapsto \bigintsss_0^\pi\bigintsss_0^{2\pi}\mathscr{K}_2(s,\phi,\theta,\varphi,\eta) d\eta d\varphi,
$$
belongs to $L^\infty$.  It is plain that 
\begin{align*}
|\mathscr{K}_2(s,\phi,\theta,\varphi,\eta)|\lesssim&\frac{\sin(\varphi)\big(r(\varphi,\eta)+|\partial_\eta r(\varphi,\eta)||\sin(\eta-\theta)|\big)s\big(sr(\phi,\theta)+r(\varphi,\eta)\big)}{J_s^\frac12(\phi,\theta,\varphi,\eta)J_0(\phi,\varphi,\eta)}\cdot
\end{align*}
Combined with the estimate \eqref{den2}, it yields
\begin{align*}
|\mathscr{K}_2(s,\phi,\theta,\varphi,\eta)|
\lesssim&\frac{(r(\varphi,\eta)+|\partial_\eta r(\varphi,\eta)||\sin(\eta-\theta)|)s(sr(\phi,\theta)+r(\varphi,\eta))}{\sin(\varphi)J_s^\frac12(\phi,\theta,\varphi,\eta)}\cdot
\end{align*}

As we have mentioned before at different stages,  the symmetry allows us to restrict the discussion    to the interval  $\phi\in(0,\pi/2)$. Then using  \mbox{Lemma \ref{lemma-estimdenominator}} and \eqref{platit1}  we achieve that
\begin{align}\label{srr-1}
\frac{sr(\phi,\theta)+r(\varphi,\eta)}{J_s^\frac12(\phi,\theta,\varphi,\eta)}\lesssim& \frac{s\phi+\sin\varphi}{\left\{(\varphi+\phi)^2(\phi-\varphi)^2+(\sin^2(\varphi)+s^2\phi^2)\sin^2((\theta-\eta)/2)\right\}^\frac12}\nonumber\\
\lesssim& \frac{1}{\big((\phi-\varphi)^2+\sin^2((\theta-\eta)/2)\big)^\frac12}\cdot
\end{align}
Hence, the estimate of \eqref{platit1} allows to get
\begin{align*}
|\mathscr{K}_2(s,\phi,\theta,\varphi,\eta)|
\lesssim&\frac{\sin(\varphi)+\sin^\alpha(\varphi)|\sin(\theta-\eta)|}{\sin(\varphi)\big((\phi-\varphi)^2+\sin^2((\theta-\eta)/2)\big)^\frac12}\\
\lesssim&\frac{1}{\big((\phi-\varphi)^2+\sin^2((\theta-\eta)/2)\big)^\frac12}+\frac{1}{\sin^{1-\alpha}(\varphi)}\cdot
\end{align*}
By interpolation we deduce for any $\beta\in(0,1)$,
\begin{align*}
|\mathscr{K}_2(s,\phi,\theta,\varphi,\eta)|
\lesssim&\frac{1}{|\phi-\varphi|^{1-\beta}|\sin((\theta-\eta)/2)|^\beta}+\sin^{\alpha-1}(\varphi).
\end{align*}
It follows that 
$$
(\phi,\theta)\mapsto\bigintsss_0^{\pi}\bigintsss_0^{2\pi}\mathscr{K}_2(s,\phi,\theta,\varphi,\eta) d\eta d\varphi\in L^\infty\big((0,\pi)\times{\T}\big).
$$
Let us move to the $\mathscr{C}^\alpha$-regularity of this latter function. This amounts to checking the partial regularity separately in $\phi$ and $\theta$. {The strategy is the same for both of them and to alleviate the discussion, we shall establish the regularity in the variable $\theta$, contrary to the preceding section where it was established for $\mathcal{I}_1$ in the direction of $\phi.$ } The goal is to get a convenient estimate for the difference 
$$
\bigintsss_0^{\pi}\bigintsss_0^{2\pi}\big(\mathscr{K}_2(s,\phi,\theta_1,\varphi,\eta)-\mathscr{K}_2(s,\phi,\theta_2,\varphi,\eta)\big)d\eta d\varphi,
$$
where $0\leq\theta_1<\theta_2\leq 2\pi$. Coming back to the definition of  the kernel $\mathscr{K}_2$ in \eqref{Kern2X1} one deduces through straightforward algebraic computations that
$$
\mathscr{K}_2(s,\phi,\theta_1,\varphi,\eta)-\mathscr{K}_2(s,\phi,\theta_2,\varphi,\eta)=\mathcal{I}_3+\mathcal{I}_4+\mathcal{I}_5+\mathcal{I}_6.
$$
with
$$
\mathcal{I}_3=\frac{\sin(\varphi)\partial_\eta\big[r(\varphi,\eta)\big(\sin(\eta-\theta_1)-\sin(\eta-\theta_2)\big)\big]s\big[sr(\phi,\theta_1)-2r(\varphi,\eta)\cos(\theta_1-\eta)\big]}{J_s^\frac12(\phi,\theta_1,\varphi,\eta)J_0^\frac12(\phi,\varphi,\eta)\big(J_s^\frac12(\phi,\theta_1,\varphi,\eta)+J_0^\frac12(\phi,\varphi,\eta)\big)},
$$
\begin{align*}
\mathcal{I}_4=&\frac{\sin(\varphi)\partial_\eta\big[r(\varphi,\eta)\sin(\eta-\theta_2)\big]s\big[sr(\phi,\theta_1)-2r(\varphi,\eta)\cos(\theta_1-\eta)\big]}{J_s^\frac12(\phi,\theta_1,\varphi,\eta)J_0^\frac12(\phi,\varphi,\eta)\big(J_s^\frac12(\phi,\theta_1,\varphi,\eta)+J_0^\frac12(\phi,\varphi,\eta)\big)}\\
&\times \frac{J_s(\phi,\theta_2,\varphi,\eta)-J_s(\phi,\theta_1,\varphi,\eta)}{\big[J_s^\frac12(\phi,\theta_2,\varphi,\eta)+J_0^\frac12(\phi,\varphi,\eta)\big]\big[J_s^\frac12(\phi,\theta_1,\varphi,\eta)+J_s^\frac12(\phi,\theta_2,\varphi,\eta)\big]},
\end{align*}
$$
\mathcal{I}_5=\frac{\sin(\varphi)\partial_\eta\big[r(\varphi,\eta)\sin(\eta-\theta_2)\big]s\big[sr(\phi,\theta_1)-sr(\phi,\theta_2)-2r(\varphi,\eta)\big(\cos(\theta_1-\eta)-\cos(\theta_2-\eta)\big)\big]}{J_s^\frac12(\phi,\theta_1,\varphi,\eta)J_0^\frac12(\phi,\varphi,\eta)\big(J_s^\frac12(\phi,\theta_2,\varphi,\eta)+J_0^\frac12(\phi,\varphi,\eta)\big)}
$$
and
\begin{align*}
\mathcal{I}_6=&\frac{\sin(\varphi)\partial_\eta\big(r(\varphi,\eta)\sin(\eta-\theta_2)\big)s\big[sr(\phi,\theta_2)-2r(\varphi,\eta)\cos(\theta_2-\eta)\big]}{J_s(\phi,\theta_1,\varphi,\eta)^\frac12J_0(\phi,\varphi,\eta)^\frac12(J_s(\phi,\theta_2,\varphi,\eta)^\frac12+J_0(\phi,\varphi,\eta)^\frac12)}\\
&\times \frac{J_s(\phi,\theta_2,\varphi,\eta)-J_s(\phi,\theta_1,\varphi,\eta)}{J_s^\frac12(\phi,\theta_2,\varphi,\eta)\big(J_s^\frac12(\phi,\theta_2,\varphi,\eta)+J_s^\frac12(\phi,\theta_1,\varphi,\eta)\big)}\cdot
\end{align*}
We shall estimate independently each one of those terms.  Concerning  the term $\mathcal{I}_3$ it  can be estimated using \eqref{den2} 
\begin{align*}
|\mathcal{I}_3|\lesssim& |\theta_1-\theta_2|\frac{\sin(\varphi)\big(r(\varphi,\eta)+|\partial_\eta r(\varphi,\eta)|\big)s\big(sr(\phi,\theta_1)+r(\varphi,\eta)\big)}{\sin^2(\varphi)J_s^\frac12(\phi,\theta_1,\varphi,\eta)}\\
\lesssim& |\theta_1-\theta_2|\frac{\big(r(\varphi,\eta)+|\partial_\eta r(\varphi,\eta)|\big)s\big(sr(\phi,\theta_1)+r(\varphi,\eta)\big)}{\sin(\varphi)J_s^\frac12(\phi,\theta_1,\varphi,\eta)}\cdot
\end{align*}
Then by virtue of \eqref{srr-1} and \eqref{platit1}, we find
\begin{align*}
|\mathcal{I}_3|
\lesssim&\frac{ |\theta_1-\theta_2|\big(\sin(\varphi)+\sin^\alpha(\varphi)\big)}{\sin(\varphi)\left((\phi-\varphi)^2+\sin^2((\theta_1-\eta)/2)\right)^\frac12}\\
\lesssim& \frac{ |\theta_1-\theta_2|\sin^{\alpha-1}(\varphi)}{\left((\phi-\varphi)^2+\sin^2((\theta_1-\eta)/2)\right)^\frac12}\cdot
\end{align*}
Combining this estimate with the interpolation inequality:   for any $\beta\in[0,1]$
$$
\frac{1}{\left\{(\phi-\varphi)^2+\sin^2((\theta_1-\eta)^2)\right\}^\frac12}\leq \frac{1}{|\varphi-\phi|^\beta|\sin((\eta-\theta_1)/2)|^{1-\beta}},
$$
we infer
\begin{align*}
|\mathcal{I}_3|
\lesssim& \frac{|\theta_1-\theta_2|}{\sin^{1-\alpha}(\varphi)|\varphi-\phi|^\beta|\sin((\eta-\theta_1)/2)|^{1-\beta}}\cdot
\end{align*}
Thus
\begin{align*}
\int_0^{\pi}\int_0^{2\pi}|\mathcal{I}_3|d\eta d\varphi\lesssim &\bigintsss_0^{\pi}\bigintsss_0^{2\pi}\frac{|\theta_1-\theta_2|d\eta d\varphi}{\sin^{1-\alpha}(\varphi)|\varphi-\phi|^\beta|\sin((\eta-\theta_1)/2)|^{1-\beta}}\\
\lesssim&|\theta_1-\theta_2|,
\end{align*}
uniformly in $\phi\in(0,\pi/2)$ and $\theta_1,\theta_2\in(0,2\pi)$, provided that $0<\beta<\alpha\leq 1$.

Concerning the term $\mathcal{I}_4$, we first use the definition of $J_s$ in Lemma \ref{lemma-estimdenominator} and one may check
\begin{align}\label{reg-est1}
|J_s(\phi,\theta_2,\varphi,\eta)-J_s(\phi,\theta_1,\varphi,\eta)|
\leq &|sr(\phi,\theta_1)-sr(\phi,\theta_2)|(|sr(\phi,\theta_1)-r(\varphi,\eta)|+|sr(\phi,\theta_2)-r(\varphi,\eta)|)\nonumber\\
&+2(sr(\phi,\theta_1)-sr(\phi,\theta_2))r(\varphi,\eta)(1-\cos(\theta_1-\eta))\nonumber\\
&+2sr(\phi,\theta_2)r(\varphi,\eta)\big|\cos(\theta_2-\eta)-\cos(\theta_1-\eta)\big|.
\end{align}
Using the trigonometric identity
\begin{equation}\label{Idw1}
1-\cos(\theta-\eta)=2\sin^2((\theta-\eta)/2),
\end{equation}
we get
\begin{align}\label{J-est}
|J_s(\phi,\theta_2,\varphi,\eta)-J_s(\phi,\theta_1,\varphi,\eta)|
\lesssim &|sr(\phi,\theta_1)-sr(\phi,\theta_2)|\Big(J_s^{\frac12}(\phi,\theta_2,\varphi,\eta)+J_s^{\frac12}(\phi,\theta_1,\varphi,\eta)\Big)\nonumber\\
&+2sr(\phi,\theta_2)r(\varphi,\eta)\big|\cos(\theta_2-\eta)-\cos(\theta_1-\eta)\big|.
\end{align}
From \eqref{platit1} combined with Taylor's formula we find
\begin{align*}
|r(\phi,\theta_2)-r(\phi,\theta_1)|&\leq\left|\int_{\theta_1}^{\theta_2}\partial_\eta r(\phi,\eta)d\eta\right|\\
&\lesssim|\theta_2-\theta_1| \sin^\alpha(\phi).
\end{align*}
Therefore we get  by interpolation inequality
\begin{align*}
|sr(\phi,\theta_1)-sr(\phi,\theta_2)|\lesssim |\theta_1-\theta_2|^\alpha s^\alpha \phi^{\alpha^2}&\Big[|sr(\phi,\theta_1)-r(\varphi,\eta)|^{1-\alpha}\nonumber\\
&+|sr(\phi,\theta_2)-r(\varphi,\eta)|^{1-\alpha}\Big].
\end{align*}
Hence
\begin{align}\label{reg-est2}
|sr(\phi,\theta_1)-sr(\phi,\theta_2)|\lesssim |\theta_1-\theta_2|^\alpha s^\alpha \phi^{\alpha^2}&\Big[J_s^{\frac{1-\alpha}{2}}(\phi,\theta_1,\varphi,\eta)+J_s^{\frac{1-\alpha}{2}}(\phi,\theta_2,\varphi,\eta)\Big].
\end{align}
Combining \eqref{reg-est2} together with \eqref{J-est} implies
\begin{align}\label{reg-est1M}
|J_s(\phi,\theta_2,\varphi,\eta)-J_s(\phi,\theta_1,\varphi,\eta)|
\lesssim & |\theta_1-\theta_2|^\alpha s^\alpha \phi^{\alpha^2}\Big(J_s^{1-\frac\alpha2}(\phi,\theta_2,\varphi,\eta)+J_s^{1-\frac\alpha2}(\phi,\theta_1,\varphi,\eta)\Big)\nonumber\\
+&2sr(\phi,\theta_2)r(\varphi,\eta)\big|\cos(\theta_2-\eta)-\cos(\theta_1-\eta)\big|.
\end{align}

{
Define
$$
\xi:=|\cos(\theta_2-\eta)-\cos(\theta_1-\eta)|.
$$
Using once again \eqref{Idw1} we get
\begin{align*}
\xi\lesssim& \Big|\sqrt{1-\cos(\theta_2-\eta)}-\sqrt{1-\cos(\theta_1-\eta)}\Big|\Big|\sqrt{1-\cos(\theta_2-\eta)}+\sqrt{1-\cos(\theta_1-\eta)}\Big|\\
\lesssim & \Big(|\sin((\theta_1-\eta)/2)|-|\sin((\theta_2-\eta)/2)|\Big)\Big(|\sin((\theta_1-\eta)/2)|+|\sin((\theta_2-\eta)/2)|\Big).
\end{align*}
Note that using the inequality
$$
||\sin x|-|\sin y||\leq|x-y|,
$$
and interpolating, one achieves
$$
 \Big(|\sin((\theta_1-\eta)/2)|-|\sin((\theta_2-\eta)/2)|\Big)\leq C|\theta_1-\theta_2|^\alpha \Big(|\sin((\theta_1-\eta)/2)|+|\sin((\theta_2-\eta)/2)|\Big)^{1-\alpha},
$$
for any $\alpha\in(0,1)$. Putting everything together, one finds
\begin{align}\label{xxa1}
\xi\lesssim &{|\theta_1-\theta_2|^\alpha\Big(|\sin((\theta_1-\eta)/2)|^{2-\alpha}+|\sin((\theta_2-\eta)/2)|^{2-\alpha}\Big).}
\end{align}
Now using the assumption \eqref{platit1} and \eqref{den3} we find
\begin{align*}
sr(\phi,\theta_2)r(\varphi,\eta)\xi&\lesssim s\phi r_0(\varphi){|\theta_1-\theta_2|^\alpha\Big(|\sin((\theta_1-\eta)/2)|^{2-\alpha}+|\sin((\theta_2-\eta)/2)|^{2-\alpha}\Big)}\\
&\lesssim s^\alpha\phi^\alpha{|\theta_1-\theta_2|^\alpha}\Big[J_s^{\frac{1-\alpha}{2}}(\phi,\theta_1,\varphi,\eta)+J_s^{\frac{1-\alpha}{2}}(\phi,\theta_2,\varphi,\eta)\Big].
\end{align*}
Inserting this inequality into \eqref{reg-est1M} implies
\begin{align*}
|J_s(\phi,\theta_2,\varphi,\eta)-J_s(\phi,\theta_1,\varphi,\eta)|
\lesssim&|\theta_1-\theta_2|^\alpha s^\alpha \phi^{\alpha^2}\Big(J_s^{1-\frac\alpha2}(\phi,\theta_2,\varphi,\eta)+J_s^{1-\frac\alpha2}(\phi,\theta_1,\varphi,\eta)\Big).
\end{align*}
It follows that
\begin{align}\label{Jtheta}
\frac{|J_s(\phi,\theta_2,\varphi,\eta)-J_s(\phi,\theta_1,\varphi,\eta)|}{J_s^\frac12(\phi,\theta_2,\varphi,\eta)+J_s^\frac12(\phi,\theta_1,\varphi,\eta)}
\lesssim &|\theta_1-\theta_2|^\alpha s^\alpha \phi^{\alpha^2}\Big(J_s^{\frac{1-\alpha}{2}}(\phi,\theta_2,\varphi,\eta)+J_s^{\frac{1-\alpha}{2}}(\phi,\theta_1,\varphi,\eta)\Big).
\end{align}}
Thus we get
\begin{align*}
&|\mathcal{I}_4|\lesssim |\theta_1-\theta_2|^\alpha\frac{\big(r(\varphi,\eta)+\partial_\eta r(\varphi,\eta)|\sin(\theta_2-\eta)|\big)s(sr(\phi,\theta_1)+r(\varphi,\eta))s^\alpha\phi^{\alpha^2}}{J_s^\frac12(\phi,\theta_1,\varphi,\eta)\big(\sin(\varphi)+J_s^\frac12(\phi,\theta_1,\varphi,\eta)\big)^\alpha\big(\sin(\varphi)+J_s^\frac12(\phi,\theta_2,\varphi,\eta)\big)}\\
&+|\theta_1-\theta_2|^\alpha\frac{(r(\varphi,\eta)+\partial_\eta r(\varphi,\eta)|\sin(\theta_2-\eta)|)s(sr(\phi,\theta_1)+r(\varphi,\eta))s^\alpha\phi^{\alpha^2}}{J_s^\frac12(\phi,\theta_1,\varphi,\eta)\big(\sin(\varphi)+J_s^\frac12(\phi,\theta_1,\varphi,\eta)\big)\big(\sin(\varphi)+J_s^\frac12(\phi,\theta_2,\varphi,\eta)\big)^\alpha}\cdot
\end{align*}
Applying  \eqref{srr-1} combined with \eqref{platit1}  we arrive at
\begin{align}\label{I4-2}
|\mathcal{I}_4|\lesssim& \frac{|\theta_1-\theta_2|^\alpha\big(\sin(\varphi)+\sin^\alpha(\varphi)|\sin(\theta_2-\eta)|\big)ss^\alpha\phi^{\alpha^2}}{\left\{(\phi-\varphi)^2+\sin^2((\theta_1-\eta)/2)\right\}^\frac12(\sin(\varphi)+J_s^\frac12(\phi,\theta_1,\varphi,\eta))^\alpha(\sin(\varphi)+J_s^\frac12(\phi,\theta_2,\varphi,\eta))}\nonumber\\
&+\frac{|\theta_1-\theta_2|^\alpha\big(\sin(\varphi)+\sin^\alpha(\varphi)|\sin(\theta_2-\eta)|\big)ss^\alpha\phi^{\alpha^2}}{\left\{(\phi-\varphi)^2+\sin^2((\theta_1-\eta)/2)\right\}^\frac12(\sin(\varphi)+J_s^\frac12(\phi,\theta_1,\varphi,\eta))(\sin(\varphi)+J_s^\frac12(\phi,\theta_2,\varphi,\eta))^\alpha}\nonumber \\
\lesssim& \quad \mathcal{J}_1+\mathcal{J}_2.
\end{align}
The right hand side  terms $\mathcal{J}_1$ and $\mathcal{J}_2$ are treated similarly and we shall only focus on the first one. We find that
\begin{align*}
&|\mathcal{J}_1|\lesssim \frac{|\theta_1-\theta_2|^\alpha}{\left\{(\phi-\varphi)^2+\sin^2((\theta_1-\eta)/2)\right\}^\frac12\sin^\alpha(\varphi)}\nonumber\\
&\qquad\qquad +\frac{|\theta_1-\theta_2|^\alpha|\sin(\theta_2-\eta)|\,s^{\alpha}\phi^{\alpha^2}}{\left\{(\phi-\varphi)^2+\sin^2((\theta_1-\eta)/2)\right\}^\frac12(\sin(\varphi)+J_s^\frac12(\phi,\theta_2,\varphi,\eta))}\cdot
\end{align*}
Using \eqref{den3} we deduce, since $\alpha\in(0,1),$ that
\begin{align}\label{lil1}
|\sin(\theta_2-\eta)|\,s^{\alpha}\phi^{\alpha^2}&\le |\sin(\theta_2-\eta)|^{\alpha^2}\,s^{\alpha^2}\phi^{\alpha^2}\nonumber \\
&\lesssim J_s^{\frac{\alpha^2}{2}}(\phi,\theta_2,\varphi,\eta).
\end{align}
Thus
\begin{align*}
&|\mathcal{J}_1|\lesssim \frac{|\theta_1-\theta_2|^\alpha}{\left\{(\phi-\varphi)^2+\sin^2((\theta_1-\eta)/2)\right\}^\frac12\sin^\alpha(\varphi)}\nonumber\\
&+\frac{|\theta_1-\theta_2|^\alpha}{\left\{(\phi-\varphi)^2+\sin^2((\theta_1-\eta)/2)\right\}^\frac12\sin^{1-\alpha^2}(\varphi)}\cdot
\end{align*}
The same estimate holds true for $\mathcal{J}_2.$ Therefore we find
\begin{align*}
&|\mathcal{I}_4|\lesssim |\theta_1-\theta_2|^\alpha \frac{\sin^{-\alpha}(\varphi)+\sin^{\alpha^2-1}(\varphi)}{\left\{(\phi-\varphi)^2+\sin^2((\theta_1-\eta)/2)\right\}^\frac12}\cdot
\end{align*}
Hence by interpolation inequality  we get for all $\gamma\in(0,1)$
\begin{align*}
&|\mathcal{I}_4|\lesssim |\theta_1-\theta_2|^\alpha \frac{\sin^{-\alpha}(\varphi)+\sin^{\alpha^2-1}(\varphi)}{|\phi-\varphi|^\gamma|\sin((\theta_1-\eta)/2)|^{1-\gamma}}\cdot\end{align*}
It follows that
\begin{align*}
\int_0^{\pi}\int_0^{2\pi}|\mathcal{I}_4|d\eta d\varphi\lesssim & |\theta_1-\theta_2|^\alpha\bigintsss_0^{\pi}\bigintsss_0^{2\pi}\frac{\sin^{-\alpha}(\varphi)+\sin^{\alpha^2-1}(\varphi)}{|\phi-\varphi|^\gamma|\sin((\theta_1-\eta)/2)|^{1-\gamma}}d\varphi d\eta\\
\lesssim&|\theta_1-\theta_2|^\alpha,
\end{align*}
uniformly in $\phi\in(0,\pi)$ and $\theta_1,\theta_2\in(0,2\pi)$, provided that $0<\gamma<\min(1-\alpha, \alpha^2)$.

{ Let us focus on the estimate of the term $\mathcal{I}_5.$} First we  make the decomposition
$$
\mathcal{I}_5=\mathcal{J}_3+\mathcal{J}_4,
$$
with
$$
\mathcal{J}_3=\frac{\sin(\varphi)\partial_\eta\big[r(\varphi,\eta)\sin(\eta-\theta_2)\big]s\big[sr(\phi,\theta_1)-sr(\phi,\theta_2))\big]}{J_s^\frac12(\phi,\theta_1,\varphi,\eta)J_0^\frac12(\phi,\varphi,\eta)\big(J_s^\frac12(\phi,\theta_2,\varphi,\eta)+J_0^\frac12(\phi,\varphi,\eta)\big)},
$$ 
and
$$
\mathcal{J}_4=-2\frac{\sin(\varphi)\partial_\eta\big[r(\varphi,\eta)\sin(\eta-\theta_2)\big]sr(\varphi,\eta)\big[\cos(\theta_1-\eta)-\cos(\theta_2-\eta)\big]}{J_s^\frac12(\phi,\theta_1,\varphi,\eta)J_0^\frac12(\phi,\varphi,\eta)\big(J_s^\frac12(\phi,\theta_2,\varphi,\eta)+J_0^\frac12(\phi,\varphi,\eta)\big)}\cdot
$$
Let us start with the last term $\mathcal{J}_4.$ Using Lemma \ref{lemma-estimdenominator} combined with \eqref{xxa1} and \eqref{platit1}  yields
\begin{align*}
|\mathcal{J}_4|\lesssim&|\theta_1-\theta_2|^\alpha\frac{\big[\sin\varphi+\sin^\alpha(\varphi)|\sin\big((\eta-\theta_2)/2\big)|\big]\sin\varphi}{J_s^\frac12(\phi,\theta_1,\varphi,\eta)\big(J_s^\frac12(\phi,\theta_2,\varphi,\eta)+\sin\varphi\big)}\cdot
\end{align*}
From \eqref{den3} we infer
\begin{align*}
|\mathcal{J}_4|\lesssim&|\theta_1-\theta_2|^\alpha\frac{\sin^2(\varphi)+\sin^\alpha(\varphi)J_s^\frac12(\phi,\theta_2,\varphi,\eta)|}{J_s^\frac12(\phi,\theta_1,\varphi,\eta)\big(J_s^\frac12(\phi,\theta_2,\varphi,\eta)+\sin\varphi\big)}\\
\lesssim&|\theta_1-\theta_2|^\alpha\frac{\sin^\alpha(\varphi)}{J_s^\frac12(\phi,\theta_1,\varphi,\eta)}\cdot
\end{align*}
Applying \eqref{srr-1} leads to
\begin{align*}
|\mathcal{J}_4|\lesssim&|\theta_1-\theta_2|^\alpha \frac{\varphi^{\alpha-1}}{\big((\phi-\varphi)^2+\sin^2((\theta_1-\eta)/2)\big)^\frac12}\\
\lesssim&|\theta_1-\theta_2|^\alpha \frac{\varphi^{\alpha-1}}{|\phi-\varphi|^\beta|\sin((\theta_1-\eta)/2)|^{1-\beta}}\cdot
\end{align*}
Therefore
\begin{align*}
\int_0^{\pi}\int_0^{2\pi}|\mathcal{J}_4|d\eta d\varphi&\lesssim |\theta_1-\theta_2|^\alpha\bigintsss_0^\pi\bigintsss_0^{2\pi} \frac{\varphi^{\alpha-1}d\eta d\varphi}{|\phi-\varphi|^\beta|\sin((\theta_1-\eta)/2)|^{1-\beta}}\\
\lesssim&|\theta_1-\theta_2|^\alpha,
\end{align*}
uniformly in $\phi\in(0,\pi)$ and $\theta_1,\theta_2\in(0,2\pi)$ provided that $0<\beta<\alpha<1.$ Next we shall deal with the term $\mathcal{J}_5$. Then by virtue of 
\eqref{reg-est2}  combined with \eqref{platit1} we  may write
\begin{align*}
|\mathcal{J}_3|\lesssim& |\theta_1-\theta_2|^\alpha\frac{\sin(\varphi)+\sin^\alpha(\varphi)|\sin(\theta_2-\eta)|}{J_s(\phi,\theta_1,\varphi,\eta)^\frac12(J_s(\phi,\theta_2,\varphi,\eta)^\frac12
+\sin\varphi)}\nonumber\\
&\times (s^{1+\alpha}\phi^{\alpha^2}\Big[J_s^{\frac{1-\alpha}{2}}(\phi,\theta_1,\varphi,\eta)+J_s^{\frac{1-\alpha}{2}}(\phi,\theta_2,\varphi,\eta)\Big].
\end{align*}
It follows that 
\begin{align*}
|\mathcal{J}_3|\lesssim&|\theta_1-\theta_2|^\alpha\frac{\big(\sin(\varphi)+\sin^\alpha(\varphi)|\sin(\theta_2-\eta)|\big)s^{1+\alpha}\phi^{\alpha^2}}{J_s^\frac12(\phi,\theta_1,\varphi,\eta)\big(J_s^\frac12(\phi,\theta_2,\varphi,\eta)+\sin\varphi\big)^\alpha}\nonumber\\
&+|\theta_1-\theta_2|^\alpha\frac{\big(\sin(\varphi)+\sin^\alpha(\varphi)|\sin(\theta_2-\eta)|\big)s^{1+\alpha}\phi^{\alpha^2}}{J_s^\frac{\alpha}{2}(\phi,\theta_1,\varphi,\eta) \big(J_s^\frac12(\phi,\theta_2,\varphi,\eta)+\sin\varphi\big)}\\
\lesssim &|\theta_1-\theta_2|^\alpha \mathcal{J}_{3,1}+ |\theta_1-\theta_2|^\alpha \mathcal{J}_{3,2}.
\end{align*}
According to  \eqref{lil1} and since $\alpha\in(0,1)$ we find that
\begin{align*}
|\mathcal{J}_{3,1}|\lesssim&\frac{\sin(\varphi)+\sin^\alpha(\varphi)|\sin(\theta_2-\eta)|^{\alpha^2} s^{\alpha^2}\phi^{\alpha^2} }{J_s^\frac12(\phi,\theta_1,\varphi,\eta)\big(J_s^\frac12(\phi,\theta_2,\varphi,\eta)+\sin\varphi\big)^\alpha}\\\lesssim&\frac{\sin(\varphi)+\sin^\alpha(\varphi)J_s^\frac{\alpha^2}{2}(\phi,\theta_2,\varphi,\eta) }{J_s^\frac12(\phi,\theta_1,\varphi,\eta)\big(J_s^\frac12(\phi,\theta_2,\varphi,\eta)+\sin\varphi\big)^\alpha}
\end{align*}
and similarly we obtain
\begin{align*}
|\mathcal{J}_{3,2}|\lesssim&\frac{\sin(\varphi)+\sin^\alpha(\varphi)J_s^\frac{\alpha^2}{2}(\phi,\theta_2,\varphi,\eta)}{J_s^\frac{\alpha}{2}(\phi,\theta_1,\varphi,\eta) \big(J_s^\frac12(\phi,\theta_2,\varphi,\eta)+\sin\varphi\big)}\cdot
\end{align*}
Consequently, we get from \eqref{srr-1} 
\begin{align*}
|\mathcal{J}_{3,1}|
\lesssim& \frac{\sin^{-\alpha}(\varphi)+\sin^{\alpha^2-1}(\varphi)}{\big((\phi-\varphi)^2+\sin^2((\theta_1-\eta)/2)\big)^\frac12}\\
\lesssim&\frac{\sin^{-\alpha}(\varphi)+\sin^{\alpha^2-1}(\varphi)}{|\phi-\varphi|^\gamma|\sin((\theta_1-\eta)/2)|^{1-\gamma}}\cdot\end{align*}
By integration, we obtain
\begin{align*}
\int_0^{\pi}\int_0^{2\pi}|\mathcal{J}_{3,1}|d\eta d\varphi\lesssim&\bigintsss_0^\pi\bigintsss_0^{2\pi} \frac{\sin^{-\alpha}(\varphi)+\sin^{\alpha^2-1}(\varphi)}{|\phi-\varphi|^\gamma|\sin((\theta_1-\eta)/2)|^{1-\gamma}}d\eta d\varphi\\
\lesssim&1,
\end{align*}
uniformly in $\phi\in(0,\pi)$ and $\theta_1\in(0,2\pi)$ provided that $0<\gamma<\min(\alpha^2,1-\alpha).$

Following the same ideas as before and using  \eqref{srr-1} we get
\begin{align*}
|\mathcal{J}_{3,2}|\lesssim&\left(\frac{\sin\varphi}{J_s^{\frac12}(\phi,\theta_1,\varphi,\eta)}\right)^\alpha\left(\sin^{-\alpha}(\varphi)+ \big(J_s^\frac12(\phi,\theta_2,\varphi,\eta)+\sin\varphi\big)^{\alpha^2-1}\right)\\
\lesssim& \frac{\sin^{-\alpha}(\varphi)+\sin^{\alpha^2-1}(\varphi)}{\big((\phi-\varphi)^2+\sin^2((\theta_1-\eta)/2)\big)^\frac\alpha2}\cdot
\end{align*}
It follows that 
\begin{align*}
|\mathcal{J}_{3,2}|
\lesssim&\frac{\sin^{-\alpha}(\varphi)+\sin^{\alpha^2-1}(\varphi)}{|\sin((\theta_1-\eta)/2)|^{\alpha}}\cdot
\end{align*}
By integration, we deduce that
\begin{align*}
\int_0^{\pi}\int_0^{2\pi}|\mathcal{J}_{3,2}|d\eta d\varphi\lesssim&\bigintsss_0^\pi\bigintsss_0^{2\pi} \frac{\sin^{-\alpha}(\varphi)+\sin^{\alpha^2-1}(\varphi)}{|\sin((\theta_1-\eta)/2)|^{\alpha}}d\eta d\varphi.\\
\lesssim&1,
\end{align*}
uniformly in $\phi\in(0,\pi)$ and $\theta_1\in(0,2\pi)$ provided that $0<\alpha<1.$ Putting together the preceding estimates allows to get
\begin{align*}
\int_0^{\pi}\int_0^{2\pi}|\mathcal{I}_{5}|d\eta d\varphi\lesssim&|\theta_1-\theta_2|^\alpha.
\end{align*}

{
Let us finish estimating the term $\mathcal{I}_6$, which is quite similar to $\mathcal{I}_4$. Applying \eqref{Jtheta}, Lemma \ref{lemma-estimdenominator} and the ideas done for $\mathcal{I}_4$ we obtain
\begin{align*}
|\mathcal{I}_6|\lesssim& |\theta_1-\theta_2|^\alpha\frac{\sin(\varphi)s^\alpha\phi^{\alpha^2}(\sin(\varphi)+\sin(\varphi)^\alpha|\sin(\eta-\theta_2)|)(s\phi+\sin(\varphi))}{\sin(\varphi)^2 J_s^\frac12(\phi,\theta_1,\varphi,\eta)J_s^\frac12(\phi,\theta_2,\varphi,\eta)}\\
&\times (J_s^\frac{1-\alpha}{2}(\phi,\theta_1,\varphi,\eta)+J_s^\frac{1-\alpha}{2}(\phi,\theta_2,\varphi,\eta))\\
\lesssim& |\theta_1-\theta_2|^\alpha\frac{s^\alpha\phi^{\alpha^2}(\sin(\varphi)+\sin(\varphi)^\alpha|\sin(\eta-\theta_2)|)(s\phi+\sin(\varphi))}{\sin(\varphi) J_s^\frac{\alpha}{2}(\phi,\theta_1,\varphi,\eta)J_s^\frac12(\phi,\theta_2,\varphi,\eta)}
\\
&+ |\theta_1-\theta_2|^\alpha\frac{s^\alpha\phi^{\alpha^2}(\sin(\varphi)+\sin(\varphi)^\alpha|\sin(\eta-\theta_2)|)(s\phi+\sin(\varphi))}{\sin(\varphi) J_s^\frac{1}{2}(\phi,\theta_1,\varphi,\eta)J_s^\frac{\alpha}{2}(\phi,\theta_2,\varphi,\eta)}.
\end{align*}
Using again Lemma \ref{lemma-estimdenominator} we achieve
\begin{align*}
|\mathcal{I}_6|\lesssim&|\theta_1-\theta_2|^\alpha\frac{s^\alpha\phi^{\alpha^2}(\sin(\varphi)+\sin(\varphi)^\alpha|\sin(\eta-\theta_2)|)(s\phi+\sin(\varphi))}{\sin(\varphi) (\sin(\varphi)+s\phi)^{1+\alpha}(|\phi-\varphi|+|\sin((\theta_1-\eta)/2)|)^\alpha(|\phi-\varphi|+|\sin((\theta_2-\eta)/2)|)}
\\
&+ |\theta_1-\theta_2|^\alpha\frac{s^\alpha\phi^{\alpha^2}(\sin(\varphi)+\sin(\varphi)^\alpha|\sin(\eta-\theta_2)|)(s\phi+\sin(\varphi))}{\sin(\varphi) (\sin(\varphi)+s\phi)^{1+\alpha}(|\phi-\varphi|+|\sin((\theta_2-\eta)/2)|)^\alpha(|\phi-\varphi|+|\sin((\theta_1-\eta)/2)|)}\\
\lesssim&|\theta_1-\theta_2|^\alpha\frac{1}{\sin(\varphi)^{\alpha-\alpha^2} (|\phi-\varphi|+|\sin((\theta_1-\eta)/2)|)^\alpha(|\phi-\varphi|+|\sin((\theta_2-\eta)/2)|)}\\
&+|\theta_1-\theta_2|^\alpha\frac{1}{\sin(\varphi)^{1-\alpha^2}(|\phi-\varphi|+|\sin((\theta_1-\eta)/2)|)^\alpha}\\
&+|\theta_1-\theta_2|^\alpha\frac{1}{\sin(\varphi)^{\alpha-\alpha^2} (|\phi-\varphi|+|\sin((\theta_2-\eta)/2)|)^\alpha(|\phi-\varphi|+|\sin((\theta_1-\eta)/2)|)}\\
&+|\theta_1-\theta_2|^\alpha\frac{1}{\sin(\varphi)^{1-\alpha^2}(|\phi-\varphi|+|\sin((\theta_2-\eta)/2)|)}\\
\lesssim& |\theta_1-\theta_2|^\alpha\frac{1}{\sin(\varphi)^{\gamma_1}|\phi-\varphi|^{\gamma_2}|\sin((\theta_2-\eta)/2)|^{\gamma_3}|\sin((\theta_2-\eta)/2)|^{\gamma_4}},
\end{align*}
for some $\gamma_1,\gamma_2,\gamma_3,\gamma_4\in(0,1)$ with $\gamma_1+\gamma_2<1$ and $\gamma_3+\gamma_4<1$. Notice that this is possible since $1+2\alpha-\alpha^2<2$ for any $\alpha\in(0,1)$.

}

Therefore, we obtain 
$$
\left|\int_0^{\pi}\int_0^{2\pi}(\mathscr{K}_2(s,\phi,\theta_1,\varphi,\eta)-\mathscr{K}_2(s,\phi,\theta_2,\varphi,\eta))d\eta d\varphi\right|\leq C|\theta_1-\theta_2|^\alpha,
$$
uniformly in $\phi\in(0,\pi)$. This concludes the proof of the stability of the function spaces by $\tilde{F}.$

\medskip
\noindent
{{{\bf Step 3:} $F_1$ {\it is} $\mathscr{C}^1$.}} In this last step, we check that $F_1$ is $\mathscr{C}^1$.

More precisely, we intend to prove  the following
\begin{align}
\|\partial_f F_1(f_1)h-\partial_f F_1(f_2)h\|_{L^\infty}\lesssim \|h\|_{\mathscr{C}^{1,\alpha}}\|f_1-f_2\|^\gamma_{\mathscr{C}^{1,\alpha}},\label{F2C1-1}\\
\|\partial_\theta\big(\partial_f F_1(f_1)h-\partial_f F_1(f_2)h\big)\|_{\mathscr{C}^\alpha}\lesssim \|h\|_{\mathscr{C}^{1,\alpha}}\|f_1-f_2\|^\gamma_{\mathscr{C}^{1,\alpha}},\label{F2C1-2}
\end{align}
and
\begin{align}
\|\partial_\phi\big(\partial_f F_1(f_1)h-\partial_f F_1(f_2)h\big)\|_{\mathscr{C}^\alpha}\lesssim& \|h|\|_{\mathscr{C}^{1,\alpha}}\|f_1-f_2\|^\gamma_{\mathscr{C}^{1,\alpha}},\label{F2C1-3}
\end{align}
for some $\gamma>0$,{ and where $f_1,f_2\in B_{X_m^\alpha}(\varepsilon)$, for some $\varepsilon<1$.

 Notice that \eqref{F2C1-1}--\eqref{F2C1-2}--\eqref{F2C1-3} will imply the $\mathscr{C}^1$-regularity of $F_1$.}
Denote $r_i(\phi,\theta)=r_0(\phi)+f_i(\phi,\theta)$, for $i=1,2$. We will check directly the estimates for the derivatives, i.e., \eqref{F2C1-2}--\eqref{F2C1-3} and leave the first estimate which is less delicate. From the expressions \eqref{F1theta} and \eqref{F1phi},
it is enough to check the estimates  for the terms: $U\cdot e^{i\theta}$ and $\frac{U}{r_0}\cdot ie^{i\theta}$.  

{
As we can guess the computations are very long,  tedious and share lot of similarities. For this reason we shall focus only on one significant term given by  \eqref{veitheta} to illustrate how the estimates work, and restrict the discussion to some terms of $\mathcal{I}_2$. Notice that the Frechet derivative of each of the previous terms will correspond again to a singular integral where the kernel has the same order of singularity and then the estimates work similarly, even if the computations are longer.

Then, let us show the idea for $\mathcal{I}_2$ and note that}
\begin{align*}
\partial_f \mathcal{I}_2(f)h(\phi,\theta)=&\bigintsss_0^\pi\bigintsss_0^{2\pi}\frac{\sin(\varphi) h(\varphi,\eta)\sin(\eta-\theta)d\eta d\varphi}{J_1^\frac12(f)(\phi,\theta,\varphi,\eta)}\\
&-\frac12\bigintsss_0^\pi\bigintsss_0^{2\pi}\frac{\sin(\varphi)r(\varphi,\eta)\sin(\eta-\theta)}{J_1^\frac32(f)(\phi,\theta,\varphi,\eta)}\partial_f J_1(f)h(\phi,\theta,\varphi,\eta)d\eta d\varphi\\
=:&\mathcal{T}_1(f)h(\phi,\theta)-\mathcal{T}_2(f)h(\phi,\theta),
\end{align*}
where 
$$
|(r(\phi,
\theta)e^{i\theta},\cos(\phi))-(r(\varphi,\eta)e^{i\eta},\cos(\varphi))|^2=J_1(f)(\phi,\theta,\varphi,\eta),\quad r=r_0+f,
$$
and
\begin{align}\label{DfJ1}
\frac12\partial_f J_1(f)h(\phi,\theta,\varphi,\eta)=&(r(\varphi,\eta)-r(\phi,\theta))(h(\varphi,\eta)-h(\phi,\theta))\\
&+(r(\varphi,\eta)h(\phi,\theta)+h(\varphi,\eta)r(\phi,\theta))(1-\cos(\eta-\theta)).\nonumber
\end{align}
{Let us exhibit the main ideas of the term $\mathcal{T}_1$}. The goal is to check 
\begin{align*}
\|\mathcal{T}_1(f_1)h-\mathcal{T}_1(f_2)h\|_{{\mathscr{C}^{\alpha}}}\lesssim &\|h\|_{\mathscr{C}^{1,\alpha}}\|f_1-f_2\|^\gamma_{\mathscr{C}^{1,\alpha}},
\end{align*}
for some $\gamma>0$. We observe  that
\begin{align}\label{J1f2f2-1}
(J_1(f_1)-J_1(f_2))(\phi,\theta,\varphi,\eta)=&(r_1(\varphi,\eta)-r_1(\phi,\theta))^2-(r_2(\varphi,\eta)-r_2(\phi,\theta))^2\nonumber\\
&+2(r_1-r_2)(\phi,\theta)r_1(\varphi,\eta)(1-\cos(\theta-\eta))\nonumber\\
&+2r_2(\phi,\theta)(r_1-r_2)(\varphi,\eta)(1-\cos(\theta-\eta)).
\end{align}
Now we write for any   $\phi,\varphi\in(0,\pi)$ and $\theta,\eta\in(0,2\pi)$
$$
|r(\phi,\theta)-r(\varphi,\eta)|\leq |r(\phi,\theta)-r(\varphi,\theta)|+|r(\varphi,\theta)-r(\varphi,\eta)|,
$$
By the $\mathscr{C}^{1,\alpha}$ regularity of $r$ one has
$$
|r(\phi,\theta)-r(\varphi,\theta)|\lesssim |\phi-\varphi| \|r\|_{\textnormal{Lip}}.
$$
In addition, we claim that 
\begin{equation}\label{claim-0}
|r(\varphi,\theta)-r(\varphi,\eta)|\lesssim |\sin((\theta-\eta)/2)| \|r\|_{\textnormal{Lip}}.
\end{equation}
Indeed, and without any restriction to the generality we can  impose that $0\leq\eta\leq\theta \le 2\pi$. We shall discuss two cases: $0\leq{\theta-\eta}\leq\pi$ and $\pi\leq{\theta-\eta}\leq2\pi$. In the first case, we simply write
$$
\frac{|r(\varphi,\theta)-r(\varphi,\eta)|}{|\sin((\theta-\eta)/2)|}= \frac{|r(\varphi,\theta)-r(\varphi,\eta)|}{|\theta-\eta|}\frac{|\theta-\eta|}{|\sin((\theta-\eta)/2)|} \leq C\|r\|_{\textnormal{Lip}},
$$
with $C$ a constant.
As to the second case $\pi\leq{\theta-\eta}\leq2\pi$, by setting  $\widehat{\eta}=\eta+2\pi$ we get
$$
\widehat{\eta}-\theta\in[0,\pi],\quad  \sin((\theta-\eta)/2)=-\sin((\theta-\widehat{\eta})/2).
$$
Since $\eta\mapsto r(\varphi,\eta)$ is $2\pi$-periodic then using the result of the first case yields
{\begin{align*}
\frac{|r(\varphi,\theta)-r(\varphi,\eta)|}{|\sin((\theta-\eta)/2)|}=&\frac{|r(\varphi,\theta)-r(\varphi,\widehat\eta)|}{|\sin((\theta-\widehat\eta)/2)|}\\
 \leq& C\|r\|_{\textnormal{Lip}}.
\end{align*}}
{This achieves the proof of the \eqref{claim-0}}. Consequently we find
\begin{equation}\label{r-esti}
|r(\phi,\theta)-r(\varphi,\eta)|\lesssim \|r\|_{\textnormal{Lip}}\big(|\phi-\varphi|+|\sin((\theta-\eta)/2)|\big).
\end{equation}
From algebraic calculus we easily get
\begin{align*}
|(r_1(\varphi,\eta)-r_1(\phi,\theta))^2-(r_2(\varphi,\eta)-r_2(\phi,\theta))^2|
=&\big|((r_1-r_2)(\varphi,\eta)-(r_1-r_2)(\phi,\theta)\big|\\
&\times \big|((r_1+r_2)(\varphi,\eta)-(r_1+r_2)(\phi,\theta))\big|.
\end{align*}
Therefore  we deduce successively  from \eqref{r-esti}
\begin{align*}
|(r_1(\varphi,\eta)-r_1(\phi,\theta))^2-(r_2(\varphi,\eta)-r_2(\phi,\theta))^2|
\lesssim &\|r_1-r_2\|_{\textnormal{Lip}}\big(|\phi-\varphi|+|\sin((\theta-\eta)/2)|\big)\\
&\times \Big(|r_1(\varphi,\eta)-r_1(\phi,\theta)|+|r_2(\varphi,\eta)-r_2(\phi,\theta)|\Big),
\end{align*}
and
\begin{align*}
|(r_1(\varphi,\eta)-r_1(\phi,\theta))^2-(r_2(\varphi,\eta)-r_2(\phi,\theta))^2|
&\lesssim |r_1(\varphi,\eta)-r_1(\phi,\theta)|^2+|r_2(\varphi,\eta)-r_2(\phi,\theta)|^2.
\end{align*}
By interpolation, we infer for any $\gamma\in[0,1]$,
\begin{align}\label{J1f2f2-2}
\Big|(r_1(\varphi,\eta)-r_1(\phi,\theta))^2-&(r_2(\varphi,\eta)-r_2(\phi,\theta))^2\Big|\lesssim \|r_1-r_2\|_{\textnormal{Lip}}^{\gamma}\Big(|\varphi-\phi|^\gamma+|\sin((\theta-{\eta})/2)|^\gamma|\Big)\nonumber \\
\times&\Big(|r_1(\varphi,\eta)-r_1(\phi,\theta)|^{2-\gamma}+|r_2(\varphi,\eta)-r_2(\phi,\theta)|^{2-\gamma}\Big),
\end{align}
On the other hand, coming back to the definition of $J_1$ we get
$$
J_1(f_1)(\phi,\theta,\varphi,\eta)\geq|r_1(\phi,\theta)-r_1(\varphi,\eta)|^2.
$$
Thus, putting together this inequality with \eqref{J1f2f2-2} and  \eqref{J1f2f2-1} yield
\begin{align}\label{J2f1f2}
&\frac{|(J_1(f_1)-J_1(f_2))(\phi,\theta,\varphi,\eta)|}{J_1^\frac12(f_1)(\phi,\theta,\varphi,\eta)+J_1^\frac12(f_2)(\phi,\theta,\varphi,\eta)}
\lesssim\|r_1-r_2\|^\gamma_{\mathscr{C}^1}\Big\{\big(|\varphi-\phi|^\gamma+|\sin((\theta-{\eta})/2)|^\gamma\big)\nonumber\\
&\qquad\qquad \times\big[J_1^{\frac{1-\gamma}{2}}(f_1)(\phi,\theta,\varphi,\eta)+J_1^{\frac{1-\gamma}{2}}(f_2)(\phi,\theta,\varphi,\eta)\big]+\phi|\sin((\theta-\eta)/2)|\Big\}.
\end{align}
Now, we shall  give an estimate of $\mathcal{T}_1(f_1)-\mathcal{T}_1(f_2)$ in $ L^\infty$.  For this purpose, define the quantity
$$
\mathscr{K}_3(f)(\phi,\theta,\varphi,\eta)=\frac{\sin(\varphi)h(\varphi,\eta)\sin(\eta-\theta)}{J_1^\frac12(f)(\phi,\theta,\varphi,\eta)},
$$
then one can easily check that 
\begin{align}\label{K3}
\mathcal{I}_7(\phi,\theta,\varphi,\eta):=&\mathscr{K}_3(f_1)(\phi,\theta,\varphi,\eta)-\mathscr{K}_3(f_2)(\phi,\theta,\varphi,\eta)\nonumber\\
=&\frac{\sin(\varphi) h(\varphi,\eta)\sin(\eta-\theta)}{J_1^\frac12(f_1)(\phi,\theta,\varphi,\eta)J_1^\frac12(f_2)(\phi,\theta,\varphi,\eta)} \frac{J_1(f_2)(\phi,\theta,\varphi,\eta)-J_1(f_1)(\phi,\theta,\varphi,\eta)}{J_1^\frac12(f_1)(\phi,\theta,\varphi,\eta)+J_1^\frac12(f_2)(\phi,\theta,\varphi,\eta)}\cdot
\end{align}
From this definition, it follows that
$$
\big(\mathcal{T}_1(f_1)-\mathcal{T}_1(f_2)\big)(\phi,\theta)=\int_0^\pi\int_0^{2\pi}\mathcal{I}_7(\phi,\theta,\varphi,\eta)d\varphi d\eta.
$$
According to \eqref{den3} we get
\begin{align*}
\phi|\sin((\theta-\eta)/2)|\lesssim&|\sin((\theta-\eta)/2)|^{\gamma} \phi^{1-\gamma}|\sin((\theta-\eta)/2)|^{1-\gamma}\\
\lesssim&|\sin((\theta-\eta)/2)|^{\gamma}J_1^{\frac{1-\gamma}{2}}(f_1)(\phi,\theta,\varphi,\eta).
\end{align*}
Combining this inequality with \eqref{J2f1f2} and  \eqref{K3} leads to
\begin{align}\label{I7}
|\mathcal{I}_7(\phi,\theta,\varphi,\eta)|\lesssim \|r_1-r_2\|_{\mathscr{C}^{1,\alpha}}^\gamma &\frac{\sin(\varphi)|h(\varphi,\eta)|\big(|\varphi-\phi)|^\gamma+|\sin((\theta-\eta)/2)|^\gamma\big)}{J_1^\frac12(f_1)(\phi,\theta,\varphi,\eta)J_1^\frac12(f_2)(\phi,\theta,\varphi,\eta)}\nonumber\\
&\times \Big(J_1^{\frac{1-\gamma}{2}}(f_1)(\phi,\theta,\varphi,\eta)+J_1^{\frac{1-\gamma}{2}}(f_2)(\phi,\theta,\varphi,\eta)\Big).
\end{align}
Applying  Lemma \ref{lemma-estimdenominator}, we infer
\begin{align*}
|\mathcal{I}_7(\phi,\theta,\varphi,\eta)|\lesssim &\|r_1-r_2\|_{\mathscr{C}^{1,\alpha}}^\gamma\frac{\sin(\varphi)|h(\varphi,\eta)|(|\varphi-\phi)|^\gamma+|\sin((\theta-\eta)/2)|^\gamma)}{J_1^\frac12(f_1)(\phi,\theta,\varphi,\eta)J_1^\frac{\gamma}{2}(f_2)(\phi,\theta,\varphi,\eta)}\\
&+\|r_1-r_2\|_{\mathscr{C}^{1,\alpha}}^\gamma\frac{\sin(\varphi)| h(\varphi,\eta)|(|\varphi-\phi)|^\gamma+|\sin((\theta-\eta)/2)|^\gamma)}{J_1^\frac{\gamma}{2}(f_1)(\phi,\theta,\varphi,\eta)J_1^\frac12(f_2)(\phi,\theta,\varphi,\eta)}\\
\lesssim&\|r_1-r_2\|_{\mathscr{C}^{1,\alpha}}^\gamma\frac{\sin(\varphi)|h(\varphi,\eta)|(|\varphi-\phi)|^\gamma+|\sin((\theta-\eta)/2)|^\gamma)}{\left\{(\varphi+\phi)^2(\phi-\varphi)^2+(\sin^2(\varphi)+\phi^2)\sin^2((\theta-\eta)/2)\right\}^\frac{1+\gamma}{2}}\cdot
\end{align*}
Using the inequality $\varphi^2\geq \sin^2(\varphi)$ for any $\varphi\in\R$, one achieves
\begin{align*}
|\mathcal{I}_7(\phi,\theta,\varphi,\eta)|\lesssim&\|r_1-r_2\|_{\mathscr{C}^{1,\alpha}}^\gamma\frac{\sin(\varphi)|h(\varphi,\eta)|(|\varphi-\phi)|^\gamma+|\sin((\theta-\eta)/2)|^\gamma)}{(\sin(\varphi)+\phi)^{1+\gamma}\left\{(\phi-\varphi)^2+\sin^2((\theta-\eta)/2)\right\}^\frac{1+\gamma}{2}}\\
\lesssim&\frac{|h(\varphi,\eta)|\|r_1-r_2\|_{\mathscr{C}^{1,\alpha}}^\gamma}{\sin^\gamma(\varphi)\left\{(\phi-\varphi)^2+\sin^2((\theta-\eta)/2)\right\}^\frac{1}{2}}\cdot
\end{align*}
The boundary conditions $h(0,\eta)=h(\pi,\eta)=0$  allow to cancel the singularity and one gets 
\begin{align*}
|\mathcal{I}_7(\phi,\theta,\varphi,\eta)|\lesssim&\frac{\|h\|_{\mathscr{C}^{1,\alpha}}\|r_1-r_2\|_{\mathscr{C}^{1,\alpha}}^\gamma}{\left\{(\phi-\varphi)^2+\sin^2((\theta-\eta)/2)\right\}^\frac{1}{2}}\cdot
\end{align*}
Interpolating we find that for any $\beta\in(0,1)$,
\begin{align*}
|\mathcal{I}_7(\phi,\theta,\varphi,\eta)|\lesssim&\frac{\|h\|_{\mathscr{C}^{1,\alpha}}\|r_1-r_2\|_{\mathscr{C}^{1,\alpha}}^\gamma}{|\phi-\varphi|^{1-\beta}|\sin((\eta-\theta)/2)|^\beta}\cdot
\end{align*}
Thus, we have that $\mathcal{I}_7$ is integrable in the variable $(\varphi,\eta)$ uniformly in $(\phi,\theta)$, and then 
\begin{align*}
\|\mathcal{T}_1(f_1)h-\mathcal{T}_1(f_2)h\|_{L^\infty}\lesssim &\|h\|_{\mathscr{C}^{1,\alpha}}\|f_1-f_2\|^\gamma_{\mathscr{C}^{1,\alpha}}.
\end{align*}
The next purpose is to establish  the partial $\mathscr{C}^\alpha$-regularity in $\phi$, and the partial regularity in $\theta$ can be done similarly. We want to prove the following 
\begin{align}\label{T1holder}
|(\mathcal{T}_1(f_1)-\mathcal{T}_1(f_2))h(\phi_1,\theta)-(T_1(f_1)-T_1(f_2))h(\phi_2,\theta)|\lesssim &\|h\|_{\mathscr{C}^{1,\alpha}}\|f_1-f_2\|^\gamma_{\mathscr{C}^{1,\alpha}}|\phi_1-\phi_2|^\alpha.
\end{align}
 For this goal  we need to study the kernel 
$
|\mathcal{I}_7(\phi_1)-\mathcal{I}_7(\phi_2)|.
$
To alleviate the notation we simply denote  $\mathcal{I}_7(\phi, \theta,\varphi,\eta)$ by  $\mathcal{I}_7(\phi)$ and  $J_1(f_i)(\phi_i,\theta,\varphi,\eta)$ by $J_1(f_i)(\phi_i)$. Adding and subtracting some appropriate terms, one finds
$$
|\mathcal{I}_7(\phi_1)-\mathcal{I}_7(\phi_2)|\lesssim \mathcal{I}_8+\mathcal{I}_9+\mathcal{I}_{10}+\mathcal{I}_{11}+\mathcal{I}_{12}
$$
with 
\begin{align*}
 \mathcal{I}_8=& \frac{\sin(\varphi)|h(\varphi,\eta)|}{J_1^\frac12(f_1)(\phi_1)J_1^\frac12(f_1)(\phi_2)J_1^\frac12(f_2)(\phi_1)}\frac{|J_1(f_2)(\phi_1)-J_1(f_1)(\phi_1)|}{J_1^\frac12(f_1)(\phi_1)+J_1^\frac12(f_2)(\phi_1)}\\
&\times\frac{|J_1(f_1)(\phi_1)-J_1(f_1)(\phi_2)|}{J_1^\frac12(f_1)(\phi_1)+J_1^\frac12(f_1)(\phi_2)},
\end{align*}

\begin{align*}
 \mathcal{I}_9=&
\frac{\sin(\varphi)|h(\varphi,\eta)|}{J_1(f_1)(\phi_2)^\frac12J_1(f_2)(\phi_1)^\frac12}\frac{|J_1(f_2)(\phi_1)-J_1(f_1)(\phi_1)|}{(J_1^\frac12(f_1)(\phi_1)+J_1^\frac12(f_2)(\phi_1))(J_1(f_1)(\phi_1)^\frac12+J_1(f_2)(\phi_2)^\frac12)}\\
&\times\frac{|J_1(f_2)(\phi_1)-J_1(f_2)(\phi_2)|}{J_1^\frac12(f_2)(\phi_1)+J_1^\frac12(f_2)(\phi_2)},\end{align*}
\begin{align*}
\mathcal{I}_{10}=&\frac{\sin(\varphi)|h(\varphi,\eta)|}{J_1^\frac12(f_1)(\phi_2)J_1^\frac12(f_2)(\phi_1)}\frac{|(J_1(f_2)-J_1(f_1))(\phi_1)-(J_1(f_2)-J_1(f_1))(\phi_2)|}{J_1^\frac12(f_1)(\phi_1)+J_1^\frac12(f_2)(\phi_2)},
\end{align*}
\begin{align*}
\mathcal{I}_{11}=&\frac{\sin(\varphi)|h(\varphi,\eta)|}{J_1^\frac12(f_1)(\phi_2)J_1^\frac12(f_2)(\phi_1)}\frac{|J_1(f_2)(\phi_2)-J_1(f_1)(\phi_2)|}{(J_1^\frac12(f_1)(\phi_2)+J_1^\frac12(f_2)(\phi_2))(J_1^\frac12(f_1)(\phi_1)+J_1^\frac12(f_2)(\phi_2))}\\
&\times\frac{|J_1(f_1)(\phi_1)-J_1(f_1)(\phi_2)|}{J_1^\frac12(f_1)(\phi_1)+J_1^\frac12(f_1)(\phi_2)}\\
\end{align*}
and
\begin{align*}
\mathcal{I}_{12}=&\frac{\sin(\varphi)|h(\varphi,\eta)|}{J_1^\frac12(f_1)(\phi_2)J_1^\frac12(f_2)(\phi_1)J_1^\frac12(f_2)(\phi_2)}\frac{|J_1(f_2)(\phi_2)-J_1(f_1)(\phi_2)|}{J_1^\frac12(f_1)(\phi_2)+J_1^\frac12(f_2)(\phi_2)}\\
&\times\frac{|J_1(f_2)(\phi_1)-J_1(f_2)(\phi_2)|}{J_1^\frac12(f_2)(\phi_1)+J_1^\frac12(f_2)(\phi_2)}.
\end{align*}
The estimate of those terms are quite similar and we shall restrict the discussion to the term $\mathcal{I}_{10}$ which involves more computations. The analysis is straightforward and we will just give the basic ideas.  First one should give a suitable  estimate for the quantity
$$
|(J_1(f_2)-J_1(f_1))(\phi_1)-(J_1(f_2)-J_1(f_1))(\phi_2)|.
$$
By using \eqref{J1f2f2-1}--\eqref{J1f2f2-2}, one finds
\begin{align*}
&|(J_1(f_2)-J_1(f_1))(\phi_1)-(J_1(f_2)-J_1(f_1))(\phi_2)|\\
\lesssim&|(r_1-r_2)(\phi_1,\theta)-(r_1-r_2)(\phi_2,\theta)||(r_1+r_2)(\phi_1,\theta)-(r_1+r_2)(\varphi,\eta)|\\
&+|(r_1-r_2)(\phi_2,\theta)-(r_1-r_2)(\varphi,\eta)||(r_1+r_2)(\phi_1,\theta)-(r_1+r_2)(\phi_2,\theta)|\\
&+|(r_1-r_2)(\phi_1,\theta)-(r_1-r_2)(\phi_2,\theta)|r_1(\varphi,\eta)\sin^2((\theta-\eta)/2)\\
&+|r_2(\phi_1,\theta)-r_2(\phi_2,\theta)||(r_1-r_2)(\varphi,\eta)|\sin^2((\theta-\eta)/2).
\end{align*}
Moreover,
\begin{align*}
&|(r_1-r_2)(\phi_1,\theta)-(r_1-r_2)(\phi_2,\theta)|\lesssim\|r_1-r_2\|^\alpha |\phi_1-\phi_2|^\alpha\left(|r_1(\phi_1,\theta)-r_1(\varphi,\eta)|^{1-\alpha}\right.\\
&\left.+|r_2(\phi_1,\theta)-r_2(\varphi,\eta)|^{1-\alpha}+|r_1(\phi_2,\theta)-r_1(\varphi,\eta)|^{1-\alpha}+|r_2(\phi_2,\theta)-r_2(\varphi,\eta)|^{1-\alpha}\right),
\end{align*}
and
\begin{align*}
|(r_1+r_2)(\phi_1,\theta)-&(r_1+r_2)(\phi_2,\theta)|\lesssim |\phi_1-\phi_2|^\alpha (|r_1(\phi_1,\theta)-r_1(\varphi,\eta)|^{1-\alpha}\\
&+|r_2(\phi_1,\theta)-r_2(\varphi,\eta)|^{1-\alpha}+|r_1(\phi_2,\theta)-r_1(\varphi,\eta)|^{1-\alpha}\\
&+|r_2(\phi_2,\theta)-r_2(\varphi,\eta)|^{1-\alpha}).
\end{align*}
In a similar way, we deduce first by triangular inequality 
\begin{align*}
|(r_1+r_2)(\phi_1,\theta)-(r_1+r_2)(\varphi,\eta)|\leq&|r_1(\phi_1,\theta)-r_1(\varphi,\eta)|+|r_2(\phi_1,\theta)-r_1(\varphi,\eta)|
\end{align*}
and second from \eqref{J1f2f2-2}
\begin{align*}
|(r_1-r_2)(\phi_2,\theta)-(r_1-r_2)(\varphi,\eta)|\lesssim&\|r_1-r_2|\|^\gamma(|\phi_2-\varphi|^\gamma+|\sin((\theta-\eta)/2)|^\gamma)\\
&\times\left(|r_1(\phi_2,\theta)-r_1(\varphi,\eta)|^{1-\gamma}+|r_2(\phi_2,\theta)-r_2(\varphi,\eta)|^{1-\gamma}\right).
\end{align*}
Combining the preceding estimate we achieve
{\begin{align*}
&|(J_1(f_2)-J_1(f_1))(\phi_1)-(J_1(f_2)-J_1(f_1))(\phi_2)|\\
\lesssim&|\phi_1-\phi_2|^\alpha\|f_1-f_2\|^\gamma\left(|r_1(\phi_1,\theta)-r_1(\varphi,\eta)|^{2-\alpha}+|r_2(\phi_1,\theta)-r_2(\varphi,\eta)|^{2-\alpha}\right.\\
&\left.+|r_1(\phi_2,\theta)-r_1(\varphi,\eta)|^{2-\alpha}+|r_2(\phi_2,\theta)-r_2(\varphi,\eta)|^{2-\alpha}\right)\\
\lesssim& |\phi_1-\phi_2|^\alpha\|f_1-f_2\|^\gamma(\mathscr{E}_{1,1}^{2-\alpha}+\mathscr{E}_{2,1}^{2-\alpha}+\mathscr{E}_{1,2}^{2-\alpha}+\mathscr{E}_{2,2}^{2-\alpha}).
\end{align*}}
where we use the notation
$$
\mathscr{E}_{i,j}=|r_i(\phi_j,\theta)-r_i(\varphi,\eta)|; \, i,j\in\{1,2\}
$$
Hence,
\begin{align*}
|\mathcal{I}_{10}|\lesssim &|\phi_1-\phi_2|^\alpha\|f_1-f_2\|^\gamma\frac{\sin(\varphi)| h(\varphi,\eta)|}{J_1^\frac12(f_1)(\phi_2)J_1^\frac12(f_2)(\phi_1)}\frac{\sum_{i,j=1}^2\mathscr{E}_{i,j}^{2-\alpha}}{J_1^\frac12(f_1)(\phi_1)+J_1^\frac12(f_2)(\phi_2)}\cdot
\end{align*}
By using the definition of $J_1$ in Lemma \ref{lemma-estimdenominator}, we immediately get
$$
\mathscr{E}_{i,j}\leq J_1^\frac12(f_i)(\phi_j),
$$
that we combine with  \eqref{r-esti} in order to get
$$
\mathscr{E}_{i,j}\lesssim |\phi_j-\varphi|+|\sin((\theta-\eta)/2)|.
$$
We shall analyze the term associated to $\mathscr{E}_{1,1}$ and the treatment of the other ones are quite similar. 
First we note 
\begin{align*}
\frac{\mathscr{E}_{1,1}^{2-\alpha}}{J_1^\frac12(f_1)(\phi_2)J_1^\frac12(f_2)(\phi_1)\big(J_1^\frac12(f_1)(\phi_1)+J_1^\frac12(f_2)(\phi_2)\big)}\lesssim \frac{|\phi_1-\varphi|^{1-\alpha}+|\sin((\theta-\eta)/2)|^{1-\alpha}}{J_1^\frac12(f_1)(\phi_2)J_1^\frac12(f_2)(\phi_1)}\cdot
\end{align*}
Making appeal to \eqref{srr-1} and \eqref{platit1}, we infer
$$
\frac{\sin(\varphi)|h(\varphi)|}{J_1^\frac12(f_1)(\phi_2)J_1^\frac12(f_2)(\phi_1)}\lesssim \frac{\|h\|_{\textnormal{Lip}}}{\left\{(\phi_1-\varphi)^2+\sin^2((\theta-\eta)/2)\right\}^\frac{1}{2}\left\{(\phi_2-\varphi)^2+\sin^2((\theta-\eta)/2)\right\}^\frac12}\cdot
$$
By interpolation we obtain for any $\gamma, \beta\in[0,1]$,
$$
\frac{\sin(\varphi)|h(\varphi)|}{J_1^\frac12(f_1)(\phi_2)J_1^\frac12(f_2)(\phi_1)}\lesssim \|h\|_{\textnormal{Lip}}\frac{|\phi_1-\varphi|^{-\gamma}|\phi_2-\varphi|^{-\beta}}{|\sin((\theta-\eta)/2)|^{2-\gamma-\beta}}\cdot
$$
Combining the preceding inequalities gives for any $\gamma_1, \gamma_2, \beta_1,\beta_2\in[0,1]$
\begin{align*}
\frac{\sin(\varphi)|h(\varphi,\eta)|}{J_1(f_1)(\phi_2)^\frac12J_1(f_2)(\phi_1)^\frac12}\frac{\mathscr{E}_{1,1}^{2-\alpha}}{J_1(f_1)(\phi_1)^\frac12+J_1(f_2)(\phi_2)^\frac12}
\lesssim &\|h\|_{\textnormal{Lip}}\frac{|\phi_1-\varphi|^{1-\alpha-\gamma_1}|\phi_2-\varphi|^{-\beta_1}}{|\sin((\theta-\eta)/2)|^{2-\gamma_1-\beta_1}}\\
&+ \|h\|_{\textnormal{Lip}}\frac{|\phi_1-\varphi|^{-\gamma_2}|\phi_2-\varphi|^{-\beta_2}}{|\sin((\theta-\eta)/2)|^{1+\alpha-\gamma_2-\beta_2}}\cdot\end{align*}
The majorant functions are integrable in the variable $(\varphi,\eta)$ uniformly in $\phi_1,\phi_2,\theta$ provided that
$$
1<\gamma_1+\beta_1<2-\alpha\quad\hbox{and}\quad \alpha<\gamma_2+\beta_2<1,
$$
and under these constraints  one can find admissible parameters. Consequently,
$$
\int_0^\pi\int_{0}^{2\pi} \mathcal{I}_{10} d\varphi d\eta\lesssim \|h\|_{\mathscr{C}^{1,\alpha}}\|f_1-f_2\|^\gamma_{\mathscr{C}^{1,\alpha}}|\phi_1-\phi_2|^\alpha.
$$
{

They, we are able to find that
$$
|\mathcal{I}_7(\phi_1)-\mathcal{I}_7(\phi_2)|\lesssim ||h||_{\mathscr{C}^{1,\alpha}}||f_1-f_2||^\gamma_{\mathscr{C}^{1,\alpha}}|\phi_1-\phi_2|^\alpha,
$$
for some $\gamma\in(0,1)$.

Let us now move on $\mathcal{T}_2(f)h(\phi,\theta)$ and show the main ideas. Define
$$
\mathcal{K}_4(f)(\phi,\theta,\varphi,\eta):=\frac{\sin(\varphi)r(\varphi,\eta)\sin(\eta-\theta)}{J_1^\frac32(\phi,\theta,\varphi,\eta)}\partial_f J_1(f)h(\phi,\theta,\varphi,\eta),
$$
and then
\begin{align}\label{T2-dec}
\mathcal{K}_4(f_1)(\phi,\theta,\varphi,\eta)-&\mathcal{K}_4(f_2)(\phi,\theta,\varphi,\eta)=\frac{\sin(\varphi)(r_1-r_2)(\varphi,\eta)\sin(\eta-\theta)}{J_1^\frac32(f_1)(\phi,\theta,\varphi,\eta)}\partial_f J_1(f_1)h(\phi,\theta,\varphi,\eta)\nonumber\\
&+\sin(\varphi)r_2(\varphi,\eta)\sin(\eta-\theta)\left(\frac{\partial_f J_1(f_1)h(\phi,\theta,\varphi,\eta)}{J_1(f_1)^\frac32(\phi,\theta,\varphi,\eta)}-\frac{\partial_f J_1(f_2)h(\phi,\theta,\varphi,\eta)}{J_2(f_1)^\frac32(\phi,\theta,\varphi,\eta)}\right)\nonumber\\
=:&(\mathcal{I}_{13}+\mathcal{I}_{14})(\phi,\theta,\varphi,\vartheta).
\end{align}
Let us analyze $\mathcal{I}_{13}$ and note that $\mathcal{I}_{14}$ is more involved (it includes more computations) but has same order of singularity. Moreover, recall the expression of $\partial_f J_1$ in \eqref{DfJ1}.
Define also 
$$
\mathcal{T}_{2,1}(\phi,\theta):=\int_0^\pi\int_0^{2\pi}\mathcal{I}_{13}(\phi,\theta,\varphi,\eta)d\varphi d\eta.
$$
Let us begin studying the $L^\infty$ norm of $\mathcal{T}_{2,1}$. Indeed, by \eqref{DfJ1} and \eqref{r-esti} we are able to find that
\begin{align}\label{DfJ1-2}
&\frac{|\partial_f J_1(f)h(\phi,\theta,\varphi,\eta)|}{J_1^\frac32(f)(\phi,\theta,\varphi,\eta)}\leq C ||h||_{\textnormal{Lip}}\nonumber\\
&\times\frac{|r_1(\varphi,\eta)-r_1(\phi,\theta)|(|\phi-\varphi|+|\sin((\theta-\eta)/2)|)+\sin(\phi)\sin(\varphi)\sin^2((\theta-\eta)/2))}{J_1^\frac32(f)(\phi,\theta,\varphi,\eta)}\nonumber\\
\leq&  C\frac{ ||h||_{\textnormal{Lip}}}{|\sin\varphi+\phi|J_1^\frac12(f)(\phi,\theta,\varphi,\eta)}.
\end{align}
Then, using Lemma \ref{lemma-estimdenominator} we achieve
\begin{align}\label{I13-bound}
|\mathcal{I}_{13}(\phi,\theta,\varphi,\eta)|\leq & C||h||_{\textnormal{Lip}}||r_1-r_2||_{\textnormal{Lip}}\frac{\sin^2(\varphi)|\sin(\eta-\theta)|}{(\sin\varphi+\phi)J_1^\frac12(\phi,\theta,\varphi,\eta)\nonumber}\\
\leq & C||h||_{\textnormal{Lip}}||r_1-r_2||_{\textnormal{Lip}}\frac{\sin^2(\varphi)|\sin(\eta-\theta)|}{(\sin\varphi+\phi)^2(|\phi-\varphi|+|\sin((\theta-\eta)/2)|)}\nonumber\\
\leq & C ||h||_{\textnormal{Lip}}||r_1-r_2||_{\textnormal{Lip}},
\end{align}
finding that $\mathcal{T}_{2,1}$ is bounded. 

Since we showed the estimates in $\phi$ of $\mathcal{T}_1$ (see \eqref{T1holder}), let us work here with the variable $\theta$. Indeed, our purpose will be showing
\begin{align}\label{T2holder}
|(\mathcal{T}_2(f_1)-\mathcal{T}_2(f_2))h(\phi,\theta_1)-(T_2(f_1)-T_2(f_2))h(\phi,\theta_2)|\lesssim &\|h\|_{\mathscr{C}^{1,\alpha}}\|f_1-f_2\|^\gamma_{\mathscr{C}^{1,\alpha}}|\theta_1-\theta_2|^\alpha,
\end{align}
for any $\theta_1,\theta_2\in[0,2\pi]$ with $\theta_1<\theta_2$. Since we have decomposed $\mathcal{T}_2(f_1)-\mathcal{T}_2(f_2)$ in two terms in \eqref{T2-dec}, let us work with $\mathcal{T}_{2,1}$ and show 
\begin{align}\label{T21holder}
|\mathcal{T}_{2,1}(\phi,\theta_1)-T_{2,1}(\phi,\theta_2)|\lesssim &\|h\|_{\mathscr{C}^{1,\alpha}}\|f_1-f_2\|^\gamma_{\mathscr{C}^{1,\alpha}}|\theta_1-\theta_2|^\alpha.
\end{align}
Here, we will use Proposition \ref{prop-potentialtheory} by fixing $\phi$. The kernel of $\mathcal{T}_{2,1}$, i.e. $\mathcal{I}_{13}$, has been already bounded in \eqref{I13-bound}. That gives us hypothesis \eqref{prop-potentialtheory-h0} of such proposition and it remains to estimate $\partial_\theta \mathcal{I}_{13}$ (see hypothesis \eqref{prop-potentialtheory-h2}). By using the expression of $\mathcal{I}_{13}$ we get
\begin{align}\label{I13-theta-2}
|\partial_\theta\mathcal{I}_{13}(\phi,\theta,\varphi,\eta)|\leq& C||r_1-r_2||_{\textnormal{Lip}}\frac{\sin^2(\varphi)|\partial_f J_1(f_1)h(\phi,\theta,\varphi,\eta)|}{J_1^\frac32(f_1)(\phi,\theta,\varphi,\eta)}\nonumber\\
&+C||r_1-r_2||_{\textnormal{Lip}}\frac{\sin^2(\varphi)|\sin(\eta-\theta)||\partial_\theta \partial_f J_1(f_1)h(\phi,\theta,\varphi,\eta)|}{J_1^\frac32(f_1)(\phi,\theta,\varphi,\eta)}\nonumber\\
&+C||r_1-r_2||_{\textnormal{Lip}}\frac{\sin^2(\varphi)|\sin(\eta-\theta)||\partial_f J_1(f_1)h(\phi,\theta,\varphi,\eta)||\partial_\theta J_1(f_1)h(\phi,\theta,\varphi,\eta)|}{J_1^\frac52(f_1)(\phi,\theta,\varphi,\eta)}\nonumber\\
=:&C||r_1-r_2||_{\textnormal{Lip}}(\mathcal{I}_{13,1}+\mathcal{I}_{13,2}+\mathcal{I}_{13,3}).
\end{align}
For $\mathcal{I}_{13,1}$ we use \eqref{DfJ1-2} and Lemma \ref{lemma-estimdenominator} finding
\begin{align*}
\mathcal{I}_{13,1}\leq &C||h||_{\textnormal{Lip}}\frac{\sin^2(\varphi)}{|\sin\varphi+\phi|J_1^\frac12(f)(\phi,\theta,\varphi,\eta)}\\
\leq &C ||h||_{\textnormal{Lip}}\frac{1}{(|\phi-\varphi|+|\sin((\theta-\eta)/2)|)}\\
\leq  &C ||h||_{\textnormal{Lip}}\frac{1}{|\sin((\theta-\eta)/2)|}.
\end{align*}
That gives us hypothesis \eqref{prop-potentialtheory-h2} for the first term $\mathcal{I}_{13,1}$. In order to work with $\mathcal{I}_{13,2}$, using \eqref{r-esti} note that
\begin{align*}
&\frac{|\partial_\theta \partial_f J_1(f_1)h(\phi,\theta,\varphi,\eta)|}{J_1^\frac32(f_1)(\phi,\theta,\varphi,\eta)}\leq C||h||_{\textnormal{Lip}}\\
&\times \frac{|\sin(\phi)|^\alpha(|\phi-\varphi|+|\sin((\theta-\eta)/2)|))+\sin(\varphi)\phi^\alpha\sin^2((\theta-\eta)/2)+\sin(\varphi)\phi|\sin((\theta-\eta)/2)|}{
J_1^\frac32(f_1)(\phi,\theta,\varphi,\eta)}\\
\leq &C\frac{||h||_{\textnormal{Lip}}}{(\phi+\sin\varphi)^{1-\alpha}J_1(f_1)(\phi,\theta,\varphi,\eta)},
\end{align*}
and then $\mathcal{I}_{13,2}$ follows as
\begin{align*}
\mathcal{I}_{13,2}\leq &C ||h||_{\textnormal{Lip}}||r_1-r_2||_{\textnormal{Lip}}\frac{\sin^2(\varphi)|\sin(\eta-\theta)|}{(\phi+\sin\varphi)^{3-\alpha}(|\phi-\varphi|+|\sin((\theta-\eta)/2)|)^2}\\
\leq &C ||h||_{\textnormal{Lip}}||r_1-r_2||_{\textnormal{Lip}}\frac{1}{(\phi+\sin\varphi)^{1-\alpha}(|\phi-\varphi|+|\sin((\theta-\eta)/2)|)}\\
\leq &C ||h||_{\textnormal{Lip}}||r_1-r_2||_{\textnormal{Lip}}\frac{1}{(\phi+\sin\varphi)^{1-\alpha}|\sin((\theta-\eta)/2)|}.
\end{align*}
Similarly, we can work with $\mathcal{I}_{13,3}$. First note that
\begin{align*}
\frac{|\partial_\theta J_1(f_1)h(\phi,\theta,\varphi,\eta)|}{J_1(f_1)(\phi,\theta,\varphi,\eta)}\leq C\frac{\phi^\alpha}{J_1^\frac12(f_1)(\phi,\theta,\varphi,\eta)},
\end{align*}
which, together with \eqref{DfJ1-2} implies
\begin{align*}
\mathcal{I}_{13,3}\leq& C ||h||_{\textnormal{Lip}}||r_1-r_2||_{\textnormal{Lip}}\frac{\sin^2(\varphi)|\sin(\eta-\theta)|\phi^\alpha}{|\sin(\varphi)+\phi|J_1(f_1)(\phi,\theta,\varphi,\eta)}\\
\leq & C ||h||_{\textnormal{Lip}}||r_1-r_2||_{\textnormal{Lip}}\frac{1}{|\sin((\theta-\eta)/2)|}.
\end{align*}
Putting everything together we achieves that $\partial_\theta \mathcal{I}_{13}$ satisfies hypothesis \eqref{prop-potentialtheory-h2} of Proposition \ref{prop-potentialtheory}. Then, such proposition can be applied to find \eqref{T21holder} concluding the proof.

}

}
\end{proof}

}

\section{Main result}
In this section we shall  provide a general statement that precise  Theorem \ref{theo-introduction} and give its  proof using all the previous results. Recall that the search of rotating solutions in the patch form to the equation \eqref{equation}, that is, solutions in the form 
$$
q(t,x)=q_0(e^{-i\Omega t}(x_1,x_2), x_3), \quad q_0={\bf 1}_{D},
$$
where $D$ is a bounded simply--connected domain surrounded by a surface parametrized by
$$
(\phi, \theta)\in[0,\pi]\times[0,2\pi]\mapsto ((r_0(\phi)+f(\phi,\theta))e^{i\theta}, \cos(\phi)), 
$$
 reduces to solving  the following infinite-dimensional equation
$$\tilde{F}(\Omega, f)=0$$
 with $f$ in a small neighborhood of the origin in the Banach space $X_m^\alpha$ and  $\tilde{F}$ is introduced in \eqref{Ftilde}. Notice that a solution is nontrivial means that the associated shape is not invariant by rotation along the vertical axis. Looking to the structure of the elements of space $X_m^\alpha$ one can easily see that a nonzero element guarantees a nontrivial shape. Our result stated below  asserts  that solutions to this functional equation do exist and are organized in a countable family of one-dimensional curves bifurcating  from the trivial solution at the largest eigenvalues of the linearized operator at the origin. More precisely, we have the following.
\begin{theo}\label{theorem}
 Let $m\geq 2$ be a fixed integer and $r_0:[0,\pi]\rightarrow\R$ satisfies the conditions:
{\begin{itemize}
\item[(H1)] $r_0\in \mathscr{C}^{2}([0,\pi])$, with $r_0(0)=r_0(\pi)=0$ and $r_0(\phi)>0$ for $\phi\in(0,\pi)$.
\item[(H2)] There exists $C>0$ such that
$$
\forall\, \phi\in[0,\pi],\quad C^{-1}\sin\phi\leq r_0(\phi)\leq C\sin(\phi).
$$
\item[(H3)] $r_0$ is symmetric with respect to $\phi=\frac{\pi}{2}$, i.e., $r_0\left(\frac{\pi}{2}-\phi\right)=r_0\left(\frac{\pi}{2}+\phi\right)$, for any $\phi\in[0,\frac{\pi}{2}]$.
\end{itemize}}
Then  there exist  $\delta>0$ and two one--dimensonal $\mathscr{C}^1$-curves 
$s\in(-\delta,\delta)\mapsto f_m(s)\in X_m^\alpha$ and $ s\in(-\delta,\delta)\mapsto \Omega_m(s)\in\R,
$  with 
$$f_m(0)=0, \quad f_m(s)\neq0, \, \forall \, s\neq 0\quad \hbox{and}\quad \Omega_m(0)=\Omega_m,
$$
where $\Omega_m$ is defined in Proposition $\ref{prop-kernel-onedim}$, such that 
$$
\forall\, s\in(-\delta,\delta),\quad \tilde{F}\big(\Omega_m(s),f_m(s)\big)=0.
$$
\end{theo}
\begin{proof}
The main material to prove this result is Crandall--Rabinowitz theorem, recalled  in Theorem \ref{CR}. First the well--possednes and the regularity of $\tilde{F}:X_m^\alpha\rightarrow X_m^\alpha$ were discussed  in Proposition \ref{prop-wellpos}. Thus it remains to check the  suitable  spectral properties of the linearized operator at the origin. The expression of this operator is detailed in  Proposition \ref{Prop-lin2} and it is a of  Fredholm type of zero index  according to Proposition \ref{cor-fredholm}. In addition for   $\Omega=\Omega_m$ the kernel is a one-dimensional vector space. Finally, the transversal condition is satisfied by virtue of Proposition \ref{prop-transversal}.
\end{proof}

\subsection{Special case: sphere and ellipsoid}\label{Sec-sphere}

In this section we aim to show the particular case of bifurcating from spherical or ellipsoidal shapes. The main particularity of these shapes is that their associated stream function is well--known in the literature, see \cite{Kellog}.
More specifically, let $\mathscr{E}$ be an ellipsoid inside the region
$$
\frac{x_1^2}{a^2}+\frac{x_2^2}{b^2}+\frac{x_3^2}{c^2}=1.
$$
The associated stream function given by
$$
\psi_0(x)=-\frac{1}{4\pi}\int_{\mathscr{E}}\frac{dA(y)}{|x-y|},
$$
can be computed inside the ellipsoid as
$$
\psi_0(x)=\frac{abc}{4}\int_0^\infty\left\{\frac{x_1^2}{a^2+s}+\frac{x_2^2}{b^2+s}+\frac{x_3^2}{c^2+s}-1\right\}\frac{ds}{\sqrt{(a^2+s)(b^2+s)(c^2+s)}}\cdot
$$
In the case that $a=b$ we have that the ellipsoid is invariant under rotations about the $z$--axis and then it defines a stationary patch, see Lemma \ref{Prop-trivial}. Moreover and without loss of generality we can take $c=1$. Note that in this case
$$
\psi_0(x)=\alpha_1(a) (x_1^2+x_2^3)+\alpha_2 (a)x_3^2+\alpha_3(a),
$$
where
$$
\alpha_1(a):=\frac{a^2}{4}\int_0^\infty\frac{ds}{(a^2+s)^2\sqrt{1+s}},
$$
$$
\alpha_2(a):=\frac{a^2}{4}\int_0^\infty\frac{ds}{(a^2+s)\sqrt{(1+s)^3}},
$$
and
$$
\alpha_3(a):=-\frac{a^2}{4}\int_0^\infty\frac{ds}{(a^2+s)^2\sqrt{1+s}}.
$$
The sphere coincides with the case $a=1$ having $\alpha_1(1)=\alpha_2(1)=\frac16$ and $\alpha_3(1)=\frac12$.\\
The above expression of the stream function together with Remark \ref{rem-streamfunction} gives us that
\begin{equation*}
\int_0^\pi H_1(\phi,\varphi)d\varphi=2\alpha_1(a),
\end{equation*}
for any $\phi\in[0,\pi]$. Recall  that $H_n$ is defined in \eqref{H-1}. Now, by virtue of Proposition \ref{Prop-lin2} one has
\begin{equation*}
\partial_{f} \tilde{F}(\Omega,0)h(\phi,\theta)=\sum_{n\geq 1}\cos(n\theta)\mathcal{L}_n^\Omega(h_n)(\phi),
\end{equation*}
where
\begin{align*}
\mathcal{L}_n^\Omega (h_n)(\phi)=&h_n(\phi)\left[2\alpha_1(a)-\Omega\right]-\bigintsss_0^\pi H_n(\phi,\varphi)h_n(\varphi)d\varphi, \quad  \phi\in(0,\pi).
\end{align*}
Moreover, the function $\nu_\Omega$ used in the spectral study and defined in \eqref{nu-function} agrees with
$$
\nu_\Omega(\phi)=2\alpha_1(a)-\Omega,
$$
which now is constant on $\phi$. Also the constant $\kappa$ in \eqref{kappa} equals now to $2\alpha_1(a)$. Hence, the key point in Section \ref{Symmetriz} is the symmetrization of the above operator. For that reason, we have defined the signed measure $d{\mu_\Omega}$ as
$$
d\mu_\Omega(\varphi)=\sin(\varphi)r_0^2(\varphi)\nu_\Omega(\varphi)d\varphi,
$$in \eqref{signed-meas} and the operator $\mathcal{K}_n^\Omega$ in \eqref{kerneleq}. However, since in this case $\nu_\Omega(\varphi)$ is constant on $\varphi$, there is no need to introduce it in the measure with the goal of symmetryzing the operator. Following the ideas developed above, we deduce that the kernel study of the linearized operator agrees in this case with the following eigenvalue problem
$$
\tilde{\mathcal{K}}_n(\phi)=(2\alpha_1(a)-\Omega)h(\phi).
$$
Here, we define
$$
\tilde{\mathcal{K}}_n(\phi):=(2\alpha_1(a)-\Omega)\mathcal{K}_n^\Omega(\phi),
$$
which does not depend now on $\Omega$ by definition of $\mathcal{K}_n^\Omega$. Note that both operators have similar properties. Hence $\tilde{\mathcal{K}}_n$ sets the properties given in Proposition \ref{prop-operator} taking the Lebesgue space $L^2_{\tilde{\mu}_\Omega}$ with 
$$d\tilde{\mu}_\Omega(\varphi)=\sin(\varphi)r_0^2(\varphi)d\varphi.$$
\\
Denote by $\beta_{n,i}$ the eigenvalues of $\tilde{\mathcal{K}}_n$ (for each $n$ we have a family of eigenvalues). Then, we have necessary that
$$
\Omega_n=2\alpha_1(a)-\beta_{n,i}.
$$
In Theorem \ref{theorem}, bifurcation occurs from $\Omega_n^\star$ given by
$$
\Omega_n^\star=2\alpha_1(a)-\beta_{n}^\star,
$$
with
$$
\beta_n^\star=\max_{i}\beta_{n,i}.
$$
Moreover, we know that $\beta_n^\star$ is positive and then $\Omega_n^\star<2\alpha_1(a)$. In particular, by Proposition \ref{prop-kernel-onedim}, we have that $\Omega_n^\star$ tends to $\kappa=2\alpha_1(a)$. Furthermore, $\Omega_n^\star$ increases in $n$ and then we can bound it below by $\Omega_1^\star$. Using the equation for $\beta_1^\star$, that is
$$
\int_0^\pi H_1(\varphi,\phi)h(\varphi)d\varphi=\beta_1^\star h(\phi),
$$
one finds that $\beta_1^\star\leq 2\alpha_1(a)$ and then $\Omega_1^\star$ is positive. This implies that $\Omega_n^\star$ is positive for any $n$. Then, in Theorem \ref{theorem} bifurcation holds  at  some $\Omega_n^\star\in(0,2\alpha_1(a))$. Let us remark that in the case of the sphere, meaning $a=1$, one has $2\alpha_1(a)=\frac13$.

There is an interesting open problem  concerning, first the spectral distribution of the eigenvalues $\beta_{n,i}$ (whether or not they are finite, simple or multiple), and second if   bifurcation occurs at the eigenvalues  $\Omega_n=2\alpha_1(a)-\beta_{n,i}$ (which is shown to  happen only  for the largest eigenvalue $\beta_n^\star$). Notice that the simplicity and the monotonicity  of the eigenvalues is a delicate problem and could be related to the geometry of the revolution shape. 
Finally we observe that  since $\beta_{n,i}<\beta_n^\star$ then $\Omega_n=2\alpha_1(a)-\beta_{n,i}>\frac13-\beta_n^\star>0$.

\appendix

\section{Gauss Hypergeometric function}\label{Ap-spfunctions}
We give a short  discussion on Gauss hypergeometric functions and illustrate some  of their basic properties. The formulas listed below were crucial in the computations of the linearized operator associated to the V--states equation and  in the analysis of the main feature of its spectral properties.  Recall that for any real numbers $a,b\in \mathbb{R},\, c\in \mathbb{R}\backslash(-\mathbb{N})$ the hypergeometric function $z\mapsto F(a,b;c;z)$ is defined on the open unit disc $\mathbb{D}$ by the power series
\begin{equation}\label{GaussF}
F(a,b;c;z)=\sum_{n=0}^{\infty}\frac{(a)_n(b)_n}{(c)_n}\frac{z^n}{n!}, \quad \forall z\in \mathbb{D}.
\end{equation}
The  Pochhammer's  symbol $(x)_n$ is defined by
$$
(x)_n = \begin{cases}   1,   & n = 0, \\
 x(x+1) \cdots (x+n-1), & n \geq1,
\end{cases}
$$
and verifies
\begin{equation*}
(x)_n=x\,(1+x)_{n-1},\quad (x)_{n+1}=(x+n)\,(x)_n.
\end{equation*}
The series converges absolutely for all values of $|z|<1.$ For $|z|=1$ we have  the absolute convergence  if $\textnormal{Re} (a+b-c)<0$ and it diverges if $1\leq \textnormal{Re}(a+b-c)$. See \cite{Erdelyi} for more details.

We recall the integral representation of the hypergeometric function, see for instance  \cite[p. 47]{Rainville}. Assume that  $ \textnormal{Re}(c) > \textnormal{Re}(b) > 0,$ then 
\begin{equation}\label{Ap-spfunctions-integ}
\hspace{1cm}F(a,b;c;z)=\frac{\Gamma(c)}{\Gamma(b)\Gamma(c-b)}\int_0^1 x^{b-1} (1-x)^{c-b-1}(1-zx)^{-a}~ dx,\quad \forall{z\in \C\backslash[1,+\infty)}.
\end{equation}
Notice that this representation shows that the hypergeometric function initially defined in the unit disc admits an analytic continuation to the complex plane cut along  $[1,+\infty)$.
 Another useful identity is the following:  
\begin{equation}\label{Ap-spfunctions-hyp}
F(a,b;c;z)=(1-z)^{-a}F\left(a,c-b;c;\frac{z}{z-1}\right),
\end{equation}
for $\textnormal{Re} (c)>\textnormal{Re}(b)>0$.
The function $\Gamma: \C\backslash\{-\N\} \to \C$ refers to the gamma function, which is the analytic continuation to the negative half plane of the usual gamma function defined on the positive half-plane $\{\textnormal{Re} z > 0\}$, and given by
$$
\Gamma(z)=\int_0^{+\infty}\tau^{z-1}e^{-\tau} d\tau,
$$
and satisfies the relation
$
\Gamma(z+1)=z\,\Gamma(z), \ \forall z\in \C \backslash(-\N).
$
From this we deduce the identities
\begin{equation*}
(x)_n=\frac{\Gamma(x+n)}{\Gamma(x)},\quad (x)_n=(-1)^n\frac{\Gamma(1-x)}{\Gamma(1-x-n)},
\end{equation*}
provided that all the quantities in the right terms are well-defined.

We can differentiate the hypergeometric function  obtaining
\begin{equation}\label{Diff41}
\frac{d^kF(a,b;c;z)}{dz^k}=\frac{(a)_k(b)_k}{(c)_k}F(a+k,b+k;c+k;z),
\end{equation}
for $k\in\N$. Depending on the parameters, the hypergeometric function behaves differently at $1$. When {$\textnormal{Re}(c)>\textnormal{Re}(b)>0$ and $\textnormal{Re} (c-a-b)>0 $}, it can be shown that it is absolutely convergent on the closed unit disc  and one finds the expression
\begin{equation}\label{id1}
F (a,b;c;1)= \frac{\Gamma(c)\Gamma(c-a-b)}{\Gamma(c-a)\Gamma(c-b)},
\end{equation}
see for example \cite{Rainville} for the proof. However, in the case $a+b=c$, the hypergeometric function exhibits a logarithmic singularity as follows
 \begin{align}\label{log1}
\lim_{z\rightarrow 1-}\frac{F(a,b;c;z)}{-\ln(1-z)}=\frac{\Gamma(a+b)}{\Gamma(a)\Gamma(b)},
\end{align}
see for instance \cite{Andrews} for more details. Next, we shall give a proof of the following classical result. 
\begin{lem}\label{Lem-integral}
Let $n\in\N$, $\beta\geq0$ and $A>1$, then
$$
\bigintsss_0^{2\pi}\frac{\cos(n\theta)}{(A-\cos(\theta))^{\frac{\beta}{2}}}d\theta=\frac{2\pi}{(1+A)^{\frac{\beta}{2}+n}}\frac{\left(\frac{\beta}{2}\right)_n 2^n\left(\frac12\right)_n}{(2n)!}F\left(n+\frac{\beta}{2}, n+\frac12; 2n+1; \frac{2}{1+A}\right ).
$$

\end{lem}
\begin{proof}
By a change of variables and using $\cos(2\theta)=2\cos^2(\theta)-1$, we arrive at
$$
\bigintsss_0^{2\pi}\frac{\cos(n\theta)}{(A-\cos(\theta))^{\frac{\beta}{2}}}d\theta=\frac{2}{(1+A)^{\frac{\beta}{2}}}\bigintsss_0^\pi \frac{\cos(2n\theta)}{(1-\frac{2}{1+A}\cos^2(\theta))^{\frac{\beta}{2}}}d\theta.
$$
Since $\frac{2}{1+A}<1$, we can use Taylor series in the following way,
$$
\left(1-\frac{2}{1+A}\cos^2(\theta)\right)^{-\frac{\beta}{2}}=\sum_{m\geq 0}\frac{\left(\frac{\beta}{2}\right)_m}{m!}\frac{2^m}{(1+A)^m}\cos^{2m}(\theta). 
$$
Then,
$$
\bigintsss_0^{2\pi}\frac{\cos(n\theta)}{(A-\cos(\theta))^{\frac{\beta}{2}}}d\theta=\frac{2}{(1+A)^{\frac{\beta}{2}}}\sum_{m\geq 0}\frac{\left(\frac{\beta}{2}\right)_m}{m!}\frac{2^m}{(1+A)^m}\bigintsss_0^\pi \cos(2n\theta)\cos^{2m}(\theta)d\theta.
$$
At this stage we use the identity, see \cite[p. 449]{Watson},
$$
\int_0^{\pi}\cos^x(\theta)\cos(y\theta)d\theta=\frac{\pi\Gamma(x+1)}{2^x\Gamma(1+\frac{x+y}{2})\Gamma(1+\frac{x-y}{2})}, 
$$
for $x>-1$ and $y\in\R$. That identity  For $x=2m$ and $y=2n$, we obtain
\begin{align*}
\bigintsss_0^{2\pi}\frac{\cos(n\theta)}{(A-\cos(\theta))^{\frac{\beta}{2}}}d\theta=&\frac{2\pi}{(1+A)^{\frac{\beta}{2}}}\sum_{m\geq n}\frac{\left(\frac{\beta}{2}\right)_m}{m!}\frac{2^m}{(1+A)^m}\frac{\Gamma(2m+1)}{2^{2m}\Gamma(1+m+n)\Gamma(1+m-n)}\\
=&\frac{2\pi}{(1+A)^{\frac{\beta}{2}}}\sum_{m\geq 0}\frac{\left(\frac{\beta}{2}\right)_{m+n}}{(m+n)!}\frac{1}{(1+A)^{m+n}}\frac{\Gamma(2m+2n+1)}{2^{m+n}\Gamma(1+m+2n)\Gamma(1+m)}.
\end{align*}
We can use some properties of Gamma functions in order to find
\begin{align*}
\Gamma(m+1+2n)\Gamma(m+1)=&(2n)!m!(2n+1)_m,\\
\frac{\Gamma(2m+2n+1)}{(m+n)!}=&2^{2m+2n}\left(\frac12\right)_{{m+n}},\\
\left(\frac{\beta}{2}\right)_{m+n}=&\left(\frac{\beta}{2}\right)_{n}\left(n+\frac{\beta}{2}\right)_{m},
\end{align*}
which implies
\begin{align*}
\frac{2\pi}{(1+A)^{\frac{\beta}{2}}}&\sum_{m\geq 0}\frac{\left(\frac{\beta}{2}\right)_{m+n}}{(m+n)!}\frac{1}{(1+A)^{m+n}}\frac{\Gamma(2m+2n+1)}{2^{m+n}\Gamma(1+m+2n)\Gamma(1+m)}\\
=&\frac{2\pi}{(1+A)^{\frac{\beta}{2}+n}}\frac{\left(\frac{\beta}{2}\right)_n2^n\left(\frac12\right)_n}{(2n)!}\sum_{m\geq 0}\frac{\left(n+\frac{\beta}{2}\right)_{m}\left(n+\frac{1}{2}\right)_{m}}{m!(2n+1)_m}\left(\frac{2}{1+A}\right)^m\\
=&\frac{2\pi}{(1+A)^{\frac{\beta}{2}+n}}\frac{\left(\frac{\beta}{2}\right)_n2^n\left(\frac12\right)_n}{(2n)!}F\left(n+\frac{\beta}{2}, n+\frac12; 2n+1; \frac{2}{1+A}\right ).
\end{align*}
\end{proof}
Now we propose to describe the boundary behavior of the  Hypergeometric functions in some suitable cases that were very useful in the preceding sections.
\begin{pro}\label{Prop-behav} The following assertions hold true.
{\begin{enumerate}
\item Bound for $F(a,a;2a;x):$  for $a>1$, there exists $C>0$ such that 
\begin{equation}\label{estimat-1}
\forall x\in[0,1),\quad F\left(a,a;2a; x\right)\le C\frac{|\ln(1-x)|}{x}\leq C+C|\ln(1-x)|.
\end{equation}
\item Bound for $F(a,a;2a-1,x):$ for $a>2$, there exists $C>0$ such that
\begin{equation}\label{estimat-2}
\forall x\in[0,1),\quad F\left(a,a;2a-1; x\right)\le C\frac{1}{|1-x|}\cdot
\end{equation}
\item Bound for $F(a,a;2a-2,x):$  for $a>3$, there exists $C>0$ such that
\begin{equation}\label{estimat-7}
\forall x\in[0,1),\quad F\left(a,a;2a-2; x\right)\le C\frac{1}{|1-x|^2}\cdot
\end{equation}
\item For $a>1,$ there exists $C>0$ such that
\begin{equation}\label{estimat-8}
\forall x\in[0,1),\quad 0\leq F(a,a;2a;x)-1\leq C x\big(1+|\ln(1-x)|\big).
\end{equation}

\item For $a>2$, there exists $C>0$ such that
$$
\forall x\in[0,1),\quad  | F(a,a;2a-1;x)-1|\leq C \frac{x}{1-x}\cdot
$$
\item For $a>1$, there exists $C>0$ such that  any $ \alpha\in[0,1]$
\begin{equation}\label{estimat-3}
 \forall \, x_2\le x_1\in[0,1),\quad |F(a,a;2a;x_1)-F(a,a;2a;x_2)|\leq C\frac{|x_1-x_2|^\alpha}{|1-x_1|^\alpha}\cdot
\end{equation}
\item For $a>2$, there exists $C>0$ such that  any $ \alpha\in[0,1]$
\begin{equation}\label{estimat-4}
 \forall \, x_2\le x_1\in[0,1), \quad |F(a,a;2a-1;x_1)-F(a,a;2a-1;x_2)|\leq C\frac{|x_1-x_2|^\alpha}{|1-x_1|^{1+\alpha}}\cdot
\end{equation}
\end{enumerate}}
\end{pro}
\begin{proof}
The main tool is  the integral representation of the Hypergeometric functions \eqref{Ap-spfunctions-integ}.

\medskip
\noindent
{\bf (1)} 
From  the integral representation  \eqref{Ap-spfunctions-integ}, it is easy to get 
\begin{align*}
|F(a,a,2a,x)|\leq & C\bigintsss_0^1 \frac{t^{a-1}(1-t)^{a-1}}{(1-xt)^a}dt\\
\leq & C\bigintsss_0^1 \left(\frac{t(1-t)}{1-xt}\right)^{a-1}\frac{1}{1-xt}dt\cdot
\end{align*}
Because $ t(1-t)\leq 1-tx$ for any $t,x\in[0,1]$, then we deduce
\begin{align*}
|F(a,a,2a,x)|\leq & C\int_0^1\frac{dt}{1-xt}\\
\leq &C\frac{|\ln(1-x)|}{x}\cdot
\end{align*}

\medskip
\noindent
{\bf (2)} 
As for (1), we find
\begin{align*}
|F(a,a,2a-1,x)|\leq & C\bigintsss_0^1 \frac{t^{a-1}(1-t)^{a-2}}{(1-xt)^a}dt\\
\leq & C\bigintsss_0^1 \left(\frac{t(1-t)}{1-xt}\right)^{a-2}\frac{dt}{(1-xt)^2}.\end{align*}
Consequently, we infer from direction calculation
\begin{align*}
|F(a,a,2a-1,x)|
\leq & C\int_0^1\frac{1}{(1-xt)^2}dt\\
\leq &\frac{C}{|1-x|}\cdot
\end{align*}
\medskip
\noindent
{\bf (3)} We omit here the details of the proof by similarity with (1) and (2).

{
\medskip
\noindent
{\bf (4)}
First note from the integral representation  that $F(a,a;2a;x)>0$ provided that $a>0$ and $x\in[0,1)$. Moreover, it is strictly increasing function  since from \eqref{Diff41}
$$
F^\prime(a,a;2a;x)=\frac{a}{2}F(a+1,a+1;2a+1;x)>0, \,\forall x\in[0,1).
$$
According to \eqref{GaussF} one may check by construction  that $F(a,a;2a;0)=1$ and therefore 
$$
F(a,a;2a;x)-1\geq 0.
$$
By the mean value theorem, we achieve
$$
F(a,a;2a;x)-1=\frac{a}{2} x\int_0^1 F(a+1,a+1,2a+1,\tau x)d\tau.
$$
Combining this representation with \eqref{estimat-2}, where we replace $a$ by $a+1$, we achieve
$$
F(a,a;2a;x)-1\leq C x\bigintsss_0^1\frac{d\tau}{1-\tau x}\leq Cx(1+|\ln(1-x)|).
$$

\medskip
\noindent
{\bf (5)}
By using similar arguments as the previous point, we obtain
$$
0\le F(a,a;2a-1;x)-1\leq Cx\int_0^1 F(a+1,a+1;2a;\tau x)d\tau.
$$
Applying \eqref{estimat-7} by changing $a$ with $a+1$ allows to get
$$
|F(a+1,a+1;2a;x)|\leq \frac{C}{(1-x)^2}, \forall \, x\in[0,1).
$$
Then,
$$
F(a,a;2a-1;x)-1\leq Cx\int_0^1 \frac{d\tau}{(1-\tau x)^2}\leq C\frac{x}{1-x}\cdot
$$
}
\medskip
\noindent
{\bf (6)}
Let $t\in[0,1)$ and set ${g_t}(x)=(1-tx)^{-a}.$ Take $0\le x_2<
 x_1<1$, then direct computations, using in particular the mean value theorem, show that
 \begin{align*}
 |g_t(x_1)-g_t(x_2)|\le&2(1-tx_1)^{-a}\\
 |g_t(x_1)-g_t(x_2)|\le&C(1-tx_1)^{-a-1}|x_1-x_2|.
 \end{align*}
 Let $\alpha\in[0,1]$ then by interpolation between the preceding inequalities  we deduce that
 \begin{align*}
 |g_t(x_1)-g_t(x_2)|=&|g_t(x_1)-g_t(x_2)|^{1-\alpha}|g_t(x_1)-g_t(x_2)|^\alpha\\
\le&C(1-tx_1)^{-a-\alpha}|x_1-x_2|^\alpha.
 \end{align*}
 It follows that
\begin{align*}
|F(a,a,2a,x_1)-F(a,a,2a,x_2)|\leq& C\left|\int_0^1 t^{a-1}(1-t)^{a-1}\left(g_t(x_1)-g_t(x_2)\right)dt\right|\\
\leq& C|x_1-x_2|^\alpha\bigintsss_0^1 \frac{t^{a-1}(1-t)^{a-1}}{(1-x_1t)^{a+\alpha}} dt.
\end{align*}
Since $a>1$ and for any $t,x_1\in[0,1),$
$$
0\leq \frac{ t^{a-1}(1-t)^{a-1}}{(1-x_1t)^{a+\alpha}}\le (1-x_1t)^{-1-\alpha},
$$
then 
\begin{align*}
|F(a,a,2a,x_1)-F(a,a,2a,x_2)|
\leq& C|x_1-x_2|^\alpha\left|\int_0^1{(1-x_1t)^{-1-\alpha}} dt\right|\\
\leq& C\frac{|x_1-x_2|^\alpha}{|1-x_1|^\alpha}. \, \, 
\end{align*}
\medskip
\noindent
{\bf (7)} This is quite similar to the proof of the preceding one. Indeed, 
\begin{align*}
|F(a,a,2a-1,x_1)-F(a,a,2a-1,x_2)|\leq& C\left|\int_0^1 t^{a-1}(1-t)^{a-2}\left(g_t(x_1)-g_t(x_2)\right)dt\right|\\
\leq& C|x_1-x_2|^\alpha\left|\int_0^1 {(1-x_1t)^{-2-\alpha}}dt\right|\\
\leq& C\frac{|x_1-x_2|^\alpha}{|1-x_1|^{1+\alpha}}\cdot
\end{align*}
\end{proof}

\section{Bifurcation theory}\label{Ap-bif}
We  shall briefly recall   some basic facts around  bifurcation theory  which mainly focuses on  the topological transitions of the phase portrait through the variation of some parameters. 
A particular case is to understand this transition in  the equilibria set for   the stationary problem    $F(\lambda,x)=0,$ where $F:\R\times X\rightarrow Y$ is a  smooth  function between Banach spaces  $X$ and $Y$. 
Assuming that one has a trivial  solution,  $F(\lambda,0)=0$ for any $\lambda\in\R$, we would like to explore the bifurcation diagram in the neighborhood  of this elementary solution,  and see whether multiple branches of solutions may bifurcate  
from  a given point  $(\lambda_0,0)$, called  a bifurcation point. When the linearized operator around this point generates a Fredholm operator, then one  may  use Lyapunov--Schmidt 
reduction in order to reduce the infinite-dimensional problem to a finite-dimensional one, known as  the bifurcation equation. For this latter problem we need some specific  transversal conditions so that the    Implicit Function Theorem can be applied.  For more discussion in this subject, 
we refer to see \cite{Kato, Kielhofer}. Notice that Theorem  \ref{CR} below  is one of those interesting results that can cover  various configurations and it is used in  this paper to prove our main result.
Before giving its precise statement, we need to recall some basic results on Fredholm operators.
\begin{defi}
Let $X$ and $Y$ be  two Banach spaces. A continuous linear mapping $T:X\rightarrow Y,$  is a  Fredholm operator if it fulfills the following properties,
\begin{enumerate}
\item $\textnormal{dim Ker}\,  T<\infty$,
\item $\textnormal{Im}\, T$ is closed in $Y$,
\item $\textnormal{codim Im}\,  T<\infty$.
\end{enumerate}
The integer $\textnormal{dim Ker}\, T-\textnormal{codim Im}\, T$ is called the Fredholm index of $T$.
\end{defi}
Next, we shall discuss  the index persistence through compact perturbations, see  \cite{Kato, Kielhofer}.
\begin{pro}
The index of a Fredholm operator remains unchanged under compact perturbations.
\end{pro}
Now, we recall the classical Crandall-Rabinowitz Theorem whose proof can be found  in \cite{CrandallRabinowitz}.

\begin{theo}[Crandall-Rabinowitz Theorem]\label{CR}
    Let $X, Y$ be two Banach spaces, $V$ be a neighborhood of $0$ in $X$ and $F:\mathbb{R}\times V\rightarrow Y$ be a function with the properties,
    \begin{enumerate}
        \item $F(\lambda,0)=0$ for all $\lambda\in\mathbb{R}$.
        \item The partial derivatives  $\partial_\lambda F_{\lambda}$, $\partial_fF$ and  $\partial_{\lambda}\partial_fF$ exist and are continuous.
        \item The operator $\partial_f F(0,0)$ is Fredholm of zero index and $\textnormal{Ker}(F_f(0,0))=\langle f_0\rangle$ is one-dimensional. 
                \item  Transversality assumption: $\partial_{\lambda}\partial_fF(0,0)f_0 \notin \textnormal{Im}(\partial_fF(0,0))$.
    \end{enumerate}
    If $Z$ is any complement of  $\textnormal{Ker}(\partial_fF(0,0))$ in $X$, then there is a neighborhood  $U$ of $(0,0)$ in $\mathbb{R}\times X$, an interval  $(-a,a)$, and two continuous functions $\Phi:(-a,a)\rightarrow\mathbb{R}$, $\beta:(-a,a)\rightarrow Z$ such that $\Phi(0)=\beta(0)=0$ and
    $$F^{-1}(0)\cap U=\{(\Phi(s), s f_0+s\beta(s)) : |s|<a\}\cup\{(\lambda,0): (\lambda,0)\in U\}.$$
\end{theo}

\section{Potential theory}
This last section is devoted to some results on  the continuity of specific  operators with singular kernels, taking the form 
\begin{equation}\label{int-op}
\mathcal{K}(f)(x_1,x_2)=\int_0^1\int_0^1 K(x_1,x_2,y_1,y_2)f(y_1,y_2)dy_1dy_2,
\end{equation}
with $(x_1,x_2)\in[0,1]^2$ and the kernel  $K:[0,1]^2\times[0,1]^2\rightarrow \R$ is smooth out the diagonal.

\begin{pro}\label{prop-potentialtheory}
Let $K:[0,1]^2\times[0,1]^2\rightarrow \R$ be smooth out the diagonal, satisfying
\begin{align}
|K(x_1,x_2,y_1,y_2)|&\leq \frac{C_0 }{|x_1-y_1|^{1-\alpha}}{g_1(x_2,y_2)},\label{prop-potentialtheory-h0}\\ 
|K(x_1,x_2,y_1,y_2)|&\leq \frac{C_0 }{|x_2-y_2|^{1-\alpha}}{g_2(x_1,y_1)},\label{prop-potentialtheory-h1}\\
|\partial_{x_1} K(x,y)| &\leq \frac{C_0 }{|x_1-y_1|^{2-\alpha}}{g_3(x_2,y_2)}\label{prop-potentialtheory-h2}\\
|\partial_{x_2} K(x,y)| &\leq \frac{C_0 }{|x_2-y_2|^{2-\alpha}}{g_4(x_1,y_1)}\label{prop-potentialtheory-h3},
\end{align}
with $\alpha\in(0,1)$ and {$g_1,g_2,g_3,g_4\in L^\infty([0,1],L^1([0,1]))$}. Then $\mathcal{K}:L^\infty([0,1]\times[0,1])\rightarrow \mathscr{C}^\alpha([0,1]\times[0,1])$ is well-defined and
$$
\|\mathcal{K}(f)\|_{\mathscr{C}^\alpha}\leq C C_0 \|f\|_{L^\infty},
$$
with $C$ an absolute constant.
\end{pro}
\begin{rem}\label{remark-potentialtheory}
{We give here this general proposition for a function of two variables, but let us remark that this can be also done for $\mathcal{K}$ depending only on one variable.
Moreover, note that condition \eqref{prop-potentialtheory-h2} (and also \eqref{prop-potentialtheory-h3}) can be replaced by
$$
|K(x_1,x_2, y_1,y_2)-K(\tilde{x_1},x_2,y_1,y_2)|\leq C|x_1-\tilde{x_1}|^\alpha g(x_1,\tilde{x_1},x_2,y_1,y_2),
$$
for $x_1<\tilde{x_1}$ and $3|x_1-\tilde{x_1}|\leq |y_1-x_1|$. The function $g$ must satisfy
$$
\left|\int_0^1\int_{{ |y_1-x_1|>3|x_1-\tilde{x_1}|}}^1g(x_1,\tilde{x_1},y_1,x_2,y_2)dy_1 dy_2\right|\leq C,
$$
uniformly in $x_1,\tilde{x_1},x_2$.}
\end{rem}
\begin{proof}
The $L^\infty$ norm of $\mathcal{K}(f)$ can be estimated as
\begin{align*}
\left|\mathcal{K}(f)(x)\right|\leq &C\|f\|_{L^\infty}\int_0^1\int_0^1 |K(x_1,x_2,y_1,y_2)|dy_1 dy_2\\
\leq & C C_0\|f\|_{L^\infty}\int_0^1 \frac{dy_1}{|x_1-y_1|^{1-\alpha}}{\int_0^1|g_1(x_2,y_2)|dy_2}\\
\leq & C C_0\|f\|_{L^\infty}.
\end{align*}
The convergence follows from the assumptions $\alpha, \gamma\in (0,1)$. Hence,
$$
\|\mathcal{K}(f)\|_{L^\infty}\leq C C_0\|f\|_{L^\infty}.
$$
For the H\"older regularity, take $x_1, \tilde{x_1}\in[0,1]$ with $x_1<\tilde{x_1}$. Define $d=|x_1-\tilde{x_1}|$, $B_{x_1}(r)=\left\{y_1\in[0,1] :  |y_1-x_1|<r\right\}$ and $B^c_{x_1}(r)$ its complement set.  Hence
\begin{align*}
\mathcal{K}(f)(x_1,x_2)&-\mathcal{K}(f)(\tilde{x_1},x_2)\\
=&\int_0^1\int_0^1K(x_1,x_2,y_1,y_2)f(y_1,y_2)dy_1dy_2-\int_0^1\int_0^1K(\tilde{x_1},x_2,y_1,y_2)f(y_1,y_2)dy_1dy_2\\
=&\int_0^1\int_{[0,1]\cap B_{x_1}(3d)}K(x_1,x_2,y_1,y_2)f(y_1,y_2)dy_1dy_2\\
&-\int_0^1\int_{[0,1]\cap B_{x_1}(3d)}K(\tilde{x_1},x_2,y_1,y_2)f(y_1,y_2)dy_1dy_2\\
&+\int_0^1\int_{[0,1]\cap B^c_{x_1}(3d)}(K(x_1,x_2,y_1,y_2)-K(\tilde{x_1},x_2,y_1,y_2))f(y_1,y_2)dy_1dy_2\\
=:& I_1+I_2+I_3.
\end{align*}
Using \eqref{prop-potentialtheory-h0}, we arrive at
\begin{align*}
|I_1|\leq& CC_0\|f\|_{L^\infty}\int_{[0,1]\cap B_{x_1}(3d)}\frac{1}{|x_1-y_1|^{1-\alpha}}dy_1\int_0^1{|g_1(x_2,y_2)|}dy_2\\
\leq& CC_0\|f\|_{L^\infty}\int_{B_{x_1}(3d)}\frac{1}{|x_1-y_1|^{1-\alpha}}dy_1\\
\leq & CC_0\|f\|_{L^\infty} d^\alpha\\
=&CC_0\|f\|_{L^\infty} |x_1-\tilde{x_1}|^\alpha.
\end{align*}
In order to work with $I_2$, note that $B_{x_1}(3d)\subset B_{\tilde{x_1}}(4d)$. Thus,
\begin{align*}
|I_2|\leq& CC_0\|f\|_{L^\infty}\int_{[0,1]\cap B_{x_1}(3d)}\frac{1}{|\tilde{x_1}-y_1|^{1-\alpha}}dy_1\int_0^1{|g_1(x_2,y_2)|}dy_2\\
\leq& CC_0\|f\|_{L^\infty}\int_{B_{x_1}(4d)}\frac{1}{|\tilde{x_1}-y_1|^{1-\alpha}}dy_1\\
\leq&CC_0\|f\|_{L^\infty} |x_1-\tilde{x_1}|^\alpha.
\end{align*}
For the last term $I_3$ we use the mean value theorem and \eqref{prop-potentialtheory-h2} achieving
\begin{align*}
|I_3|\leq &C \left|(x_1-\tilde{x_1})\int_0^1\int_0^1\int_{[0,1]\cap B^c_{x_1}(3d)}(\partial_{x_1}K)(x_1+(1-s)(\tilde{x_1}-x_1),x_2,y_1,y_2)f(y_1,y_2)dy_1dy_2ds\right|\\
\leq &C C_0 \|f\|_{L^\infty}|x_1-\tilde{x_1}| \int_0^1\int_{[0,1]\cap B^c_{x_1}(3d)}\frac{dy_1ds}{|x_1+(1-s)(\tilde{x_1}-x_1)-y_1|^{2-\alpha}}\int_0^1{|g_3(x_2,y_2)|dy_2}.
\end{align*}
Note that if $y_1\in B^c_{x_1}(3d)$, then
$$
|x_1+(1-s)(\tilde{x_1}-x_1)-y_1|\geq |x_1-y_1|-(1-s)d\geq |x_1-y_1|-\frac{(1-s)}{3}|x_1-y_1|\geq \frac23|x_1-{y_1}|,
$$
which implies
\begin{align*}
|I_3|\leq &C C_0 \|f\|_{L^\infty} |x_1-\tilde{x_1}|\int_{[0,1]\cap B^c_{x_1}(3d)}\frac{dy_1}{|x_1-y_1|^{2-\alpha}}\\
\leq &C C_0 \|f\|_{L^\infty} |x_1-\tilde{x_1}|\frac{1}{|x_1-\tilde{x_1}|^{1-\alpha}}\\
\leq &C C_0 \|f\|_{L^\infty} |x_1-\tilde{x_1}|^\alpha.
\end{align*}
Putting together the preceding estimates yields
$$
|\mathcal{K}(f)(\tilde{x}_1,x_2)-\mathcal{K}(f)(x_1,{x_2})|\leq C C_0 \|f\|_{L^\infty} |x_1-\tilde{x_1}|^\alpha.
$$
The same arguments enables to obtain 
$$
|\mathcal{K}(f)(x_1,x_2)-\mathcal{K}(f)(x_1,\tilde{x_2})|\leq C C_0 \|f\|_{L^\infty} |x_2-\tilde{x_2}|^\alpha.
$$
Then, we conclude that
$$
||\mathcal{K}(f)||_{\mathscr{C}^\alpha}\leq C C_0 \|f\|_{L^\infty}.
$$
\end{proof}


\begin{thebibliography}{99}


\bibitem{Andrews} G. R. Andrews, R. Askey, R. Roy, {\it Special Functions.} Cambridge University Press, 1999.

\bibitem{Erdelyi}{H. Bateman, } {\it Higher Transcendental Functions Vol. I--III.} McGraw-Hill Book Company, New York, 1953.

\bibitem{B-B} J. T. Beale, A. J. Bourgeois, {\it Validity of the quasi-geostrophic model for large-scale flow in the atmosphere and ocean,} SIAM J. Math. Anal. 25 (1994), no. 4, 1023–1068

\bibitem{B-C} A. L. Bertozzi, P. Constantin, {\it Global regularity for vortex patches.} Comm. Math. Phys.
{\bf 152}(1) (1993), 19--28.

\bibitem{Burbea} { J. Burbea,} {\it Motions of vortex patches.} Lett. Math. Phys. {\bf 6} (1982), 1--16.

\bibitem{Cas0-Cor0-Gom} {A. Castro, D. C\'ordoba, J. G\'omez-Serrano, }{\it Existence and regularity of rotating global solutions for the generalized surface quasi-geostrophic equations.} Duke Math. J. {\bf 165}(5) (2016), 935--984.
 
\bibitem{Cas-Cor-Gom} {A. Castro, D. C\'ordoba, J. G\'omez-Serrano, }{\it  Uniformly rotating analytic global patch solutions for active scalars}. J. Ann. PDE {\bf 2}(1) (2016),  Art. 1, 34.

\bibitem{CastroCordobaGomezSerrano} {A. Castro, D. C\'ordoba, J. G\'omez-Serrano, } {{\it Uniformly rotating smooth solutions for the incompressible 2D Euler equations.}} Arch. Ration. Mech. Anal. {bf 231}(2) (2019), 719--785.

{\bibitem{C-C-GS-2} A. Castro, D. C\'ordoba,  J. G\'omez-Serrano, {\it Global smooth solutions for the inviscid SQG equation,} arXiv:1603.03325, 2016.}

\bibitem{Charv}  F. Charve, {\it Convergence of weak solutions for the primitive system of the quasi-geostrophic equations.}  Asymptot. Anal. 42 (2005), no. 3-4, 173–209.
\bibitem{Chemin} {J.-Y. Chemin,} {{\it Persistance de structures g\'eometriques dans les  fluides incompressibles bidimensionnels. }} Ann. Sci. Ec. Norm. Sup. {\bf 26} (1993), 1--26.

\bibitem{CrandallRabinowitz} {M. G. Crandall, P. H. Rabinowitz, } {\it Bifurcation from simple eigenvalues.} J. Funct. Anal. {\bf 8} (1971), 321--340.




\bibitem{DeemZabusky} {G. S. Deem, N. J. Zabusky, } {\it Vortex waves: Stationary ``V-states'', Interactions, Recurrence, and Breaking.} Phys. Rev. Lett. {\bf 40} (1978), 859--862.


\bibitem{DelaHoz-Hassainia-Hmidi}{F. De la Hoz, Z. Hassainia, T. Hmidi, }{\it Doubly Connected V-States for the Generalized Surface Quasi-Geostrophic Equations.} Arch. Ration. Mech. Anal {\bf 220} (2016), 1209--1281. 

\bibitem{DelaHozHmidiMateuVerdera} {F. De la Hoz, T. Hmidi, J. Mateu, J. Verdera, } {\it Doubly connected V-states for the planar Euler equations. } SIAM J. Math. Anal. {\bf 48} (2016), 1892--1928.
 
{\bibitem{DHHM} {F. De la Hoz, Z. Hassainia, T. Hmidi, J. Mateu,} {\it An analytical and numerical study of steady patches in the disc.}  Anal. PDE {\bf 9}(7) (2016), 1609--1670.}
 
\bibitem{D-S-R} D. G. Dritschel, R. K. Scott, J. N. Reinaud,  {\it  The stability of quasi-geostrophic ellipsoidal vortices.} J. Fluid Mech. 536 (2005), 401--421.

\bibitem{Dristchel2}{D. G. Dritschel, }{\it An exact steadily rotating surface quasi--geostrophic elliptical vortex. } Geophysical and Astrophysical Fluid Dynamics {\bf 105} (2011), 368--376.

 \bibitem{D-H-R} {D. G. Dritschel, T. Hmidi, C. Renault, } {\it
Imperfect bifurcation for the quasi-geostrophic shallow-water equations.} Arch. Ration. Mech. Anal. {\bf 231}(3) (2019), 1853--1915.

\bibitem{Dristchel}{D. G. Dritschel, J. N. Reinaud, W. J. McKiver, }{\it The quasi--geostrophic ellipsoidal vortex model. } J. Fluid Mech. {\bf 505} (2004), 201--223.
 
\bibitem{G-KVS}{C. Garc\'ia, }{\it K\'arm\'an Vortex Street in incompressible fluid models}. Nonlinearity {\bf 33}(4) (2020), 1625--1676. 

{ \bibitem{GHS}{C. Garc\'ia, T. Hmidi, J. Soler, }{\it Non uniform rotating vortices and periodic orbits for the two--dimensional Euler equations. } Arch. Ration. Mech. Anal. {\bf 238} (2020), 929--1086.}

\bibitem{GPSY}{J.  G\'omez-Serrano, J. Park, J. Shi, Y.Yao, } {\it Symmetry in stationary and uniformly-rotating solutions of active scalar equations}, arXiv:1908.01722, 2019. 

\bibitem{Hassa-Hmi}  Z. Hassainia, T. Hmidi, {\it  On the V-states for the generalized quasi-geostrophic equations.} Comm. Math. Phys. {\bf 337}(1) (2015), 321--377.

\bibitem{HMW} Z. Hassainia, N. Masmoudi, M. H. Wheeler, {\it Global bifurcation of rotating vortex patches}. 	Comm. Pure Appl. Math.. doi:10.1002/cpa.21855, 2019.

\bibitem{Helms}{L. L. Helms, }{\it Potential theory, } Springer-Verlag London, 2009.

\bibitem{HmidiHozMateuVerdera} {T. Hmidi, F. De la Hoz, J. Mateu, J. Verdera, }{\it  Doubly connected V-states for the planar Euler equations. } SIAM J. Math. Anal. {\bf 48}(3), 1892--1928.

{\bibitem{HmidiMateu}{T. Hmidi, J. Mateu, } {\it Bifurcation of rotating patches from Kirchhoff vortices.} Discret. Contin. Dyn. Syst. {\bf 36} (2016), 5401--5422.}
 
 \bibitem{H-M} T. Hmidi, J. Mateu,  {\it Existence of corotating and counter-rotating vortex pairs for active
scalar equations.} Comm. Math. Phys. {\bf 350}(2) (2017), 699--747.

\bibitem{HmidiMateuVerdera} {T. Hmidi, J. Mateu, J. Verdera, } {\it Boundary regularity of rotating vortex patches. } Arch. Ration. Mech. Anal {\bf 209} (2013), 171--208. 

\bibitem{Iftimie2}{D. Iftimie, }{\it Approximation of the quasigeostrophic system with primitive systems. } Asymptotic Analysis {\bf 21} (1999), 89--97.

\bibitem{Kato}{ T. Kato, }{\it Perturbation Theory for Linear Operators. }  Springer-Verlag, Berlin-Heidelberg-New York, 1995.

 \bibitem{Kellog}{O.D. Kellog, }{\it Foundations of Potential Theory} Springer-Verlag, Berlin, 1967. 

\bibitem{Kielhofer}{ H. Kielh\"ofer, }{\it Bifurcation Theory: An Introduction with Applications to PDEs. }  Springer-Verlag, Berlin-Heidelberg-New York, 2004.


{\bibitem{MajdaBertozzi}{A. Majda, A. Bertozzi, }{\it Vorticity and Incompressible Flow.} Cambridge University Press, 2002.}

\bibitem{Meacham} S. P. Meacham, {\it  Quasi-geostrophic, ellipsoidal vortices in stratified fluid.}  Dyn. Armos. Oceans, 16 (1992) 189--223.

\bibitem{Ped} J. Pedlosky, {\it  Geophysical Fluid Dynamics,} second ed., Springer-Verlag, New York, 1987, pp. 1--710.

\bibitem{Rainville}{E. D. Rainville, }{\it Special Functions. } The Macmillan Co., 1973.

\bibitem {reed}{ M. Reed, B. Simon}, {{\it Methods of modern mathematical physics. IV. Analysis of operators.} }Academic Press, New York-London, 1978.
\bibitem{Rein-Drit} J. N. Reinaud, D. G. Dritschel, {\it The stability and nonlinear evolution of quasi-geostrophic toroidal vortices.} J. Fluid Mech. 863 (2019), 60--78.

\bibitem{Rein}  J. N. Reinaud, {\it Three-dimensional quasi-geostrophic vortex equilibria with m-fold symmetry}. J. Fluid Mech. 863 (2019), 32--59.

\bibitem{Serf} P. Serfati, {\it Une preuve directe d'existence globale des vortex patches 2D.}
C. R. Acad. Sci. Paris S\'er. I Math. {\bf 318}(6) (1994), 515--518.

\bibitem{Nagy} B. Sz\"okefalvi-Nagy, {{\it Perturbations des transformations lin\'eaires ferm\'ees.}} Acta Sci. Math. (Szeged) 14 (1951), 125--137.

\bibitem{Watson}  G. A. Watson, {\it A Treatise on the Theory of Bessel Functions. Cambridge University Press}, 1944.
Trans. Amer. Math. Soc. 299 (1987), no. 2, 581--599.

\bibitem{Yudovich}{Y. Yudovich, }{\it Nonstationary flow of an ideal incompressible liquid. } Zh. Vych. Mat. {\bf 3} (1963), 1032--1066.

\end{thebibliography}
\end{document}